\documentclass[11pt]{article}

\usepackage[a4paper,margin=2.5cm]{geometry}
\usepackage{amsfonts,amsthm,amssymb,amsopn,amsmath,mathrsfs,dsfont}
\usepackage{hyperref}
\usepackage{graphicx}
\usepackage{verbatim}
\usepackage{authblk}
\usepackage{color}
\usepackage{enumitem}

\newcounter{subsection1}[section]
\newtheorem{lemma}[subsection1]{Lemma}
\newtheorem{prop}[subsection1]{Proposition}
\newtheorem{theorem}[subsection1]{Theorem}
\newtheorem*{remark}{Remark}
\newtheorem{cor}[subsection1]{Corollary}

\newcommand{\E}[1]{{\mathbb E}\left[#1\right]}
\newcommand{\e}{{\mathbb E}}
\newcommand{\p}[1]{{\mathbb P}\left(#1\right)}
\newcommand{\psub}[2]{{\mathbb P}_{#1}\left(#2\right)}
\newcommand{\Esub}[2]{{\mathbb E}_{#1}\left[#2\right]}

\newcommand{\psubb}[2]{{\mathbf P}_{#1}\left(#2\right)}
\newcommand{\Esubb}[2]{{\mathbf E}_{#1}\left[#2\right]}

\newcounter{assumptions}
\numberwithin{equation}{section} 
\numberwithin{subsection1}{section}

\begin{document}

\def\P{\mathbb{P}}
\def\Pb{\mathbf{P}}
\def\R{\mathbb{R}} 
\def\C{\mathbb{C}} 
\def\N{\mathbb{N}}
\def\Q{\mathbb{Q}}
\def\Z{\mathbb{Z}}
\def\1{\mathds{1}}
\def\to{\rightarrow}

\allowdisplaybreaks

\author{Alison Etheridge\footnote{Department of Statistics, University of Oxford, UK}\, and Sarah Penington\footnote{Department of Mathematical Sciences, University of Bath, UK}}

\title{Genealogies in bistable waves}

\date{\today}
\maketitle

\begin{abstract}
We study a model of selection acting on a diploid population (one in which each individual carries two copies of each gene) living in one spatial dimension.
We suppose a particular gene appears in two forms (alleles) $A$ and $a$, and that individuals carrying $AA$ have a higher fitness than $aa$ individuals, while $Aa$ individuals have a lower fitness than both $AA$ and $aa$ individuals.
The proportion of advantageous $A$ alleles expands through the population approximately according to a travelling wave.
We prove that on a suitable timescale, the genealogy of a sample of $A$ alleles taken from near the wavefront converges to a Kingman coalescent as the population density goes to infinity.
This contrasts with the case of directional selection in which the corresponding limit is thought to be the Bolthausen-Sznitman coalescent.
The proof uses `tracer dynamics'.
\end{abstract}


\section{Introduction and main results} \label{sec:intro}

Our interest in this work is in modelling the pattern of genetic variation left
behind when a gene that is favoured by natural selection 
`sweeps' through a spatially structured population in a travelling
wave. 
The interaction between natural selection and spatial structure is a 
classical problem; the
novelty of what we propose here is that we replace the simple directional
selection considered in the majority of the mathematical work
in this area by a model of selection acting on diploid individuals 
(carrying two copies of the gene in question) that provides a toy
model for the dynamics
of so-called hybrid zones. 
Hybrid zones are widespread  
in naturally occurring populations,~\cite{barton/hewitt:1989}, 
and there is a wealth of
recent empirical work on their dynamics; see~\cite{arntzen:2019} 
for an example and a brief discussion.
In our simple model, we  
shall suppose that the population is living in one spatial dimension, and that
the gene has exactly two forms (alleles), $A$ and $a$,
and that type $AA$ individuals are at a selective advantage over $aa$
individuals, but that $Aa$ individuals are at a selective disadvantage
relative to both.

Our goal is to understand the genealogical trees that 
describe the relationships between individual genes sampled from  
the present day population. In the case of directional selection, there is a
large body of work, of varying degrees of rigour, that suggests that
if we take a sample of favoured individuals from close to
the wavefront then, on 
suitable timescales, their genealogy is
described by the so-called Bolthausen-Sznitman coalescent.
In our models, where expansion of the favoured type is
driven from the bulk of the wave, we shall see that 
the corresponding object is the classical Kingman coalescent.

Before giving a precise mathematical definition of our model in Section~\ref{subsec:modeldefn} and 
stating our main results in Section~\ref{subsec:mainresults}, we place our work in context.

\subsection*{Directional selection: the (stochastic) Fisher-KPP equation} 

The mathematical modelling 
of the way in which a genetic type favoured by natural
selection spreads through a population that is distributed across space can be 
traced back at least to Fisher~(\cite{fisher:1937}) and 
Kolmogorov, Petrovsky \& Piscounov~(\cite{kolmogorov/petrovsky/piscounov:1937}). 
They introduced the now classical
Fisher-KPP equation,
\begin{eqnarray}
\label{FKPP equation}
\frac{\partial p}{\partial t}(t,x)=\frac{m}{2}\Delta p(t,x)+s_0p(t,x)\big(1-p(t,x)\big)
\qquad && \text{for }x\in\R, \, t> 0,\\
\nonumber
0\leq p(0,x)\leq 1 \qquad &&\forall x\in\R,
\end{eqnarray}
as a model for the way in which the proportion $p(t,x)$ 
of genes that are of
the favoured type changes with time.
A shortcoming of this equation is that it does not take account of random
genetic drift, that is, the randomness due to reproduction in a finite
population. The classical way to introduce such randomness is through a 
Wright-Fisher noise term, so that the equation becomes
\begin{equation}
\label{stochastic FKPP}
dp(t,x)=\frac{m}{2}\Delta p(t,x) dt + s_0p(t,x)\big(1-p(t,x)\big)dt
+\sqrt{\frac{1}{\rho_e}p(t,x)\big(1-p(t,x)\big)}W(dt,dx),
\end{equation}
where $W$ is a space-time white noise and $\rho_e$ is an
effective population density. This is a continuous
space analogue of Kimura's stepping stone model~\cite{kimura:1953},
with the additional
non-linear term capturing selection.
This equation has the limitation
that it only makes sense in one space dimension, but like~(\ref{FKPP equation})
it exhibits travelling wave solutions (\cite{mueller/sowers:1995})
which can be thought of as modelling
a selectively favoured type `sweeping' through the population and,
consequently, it has been the object of intensive study.

From a biological perspective, the power of mathematical models is that they
can throw some light on the patterns of genetic variation that one might
expect to see in the present day population if it has been subject to 
natural selection. Neither of the models above is adequate for this task. 
If it survives at all, one can expect
a selectively favoured type to eventually be carried by all individuals in 
a population and from simply observing that type, we have no way of knowing
whether it is fixed in the population as a result of natural selection, or
purely by chance. 
However, in reality, it is not just a single letter in the DNA 
sequence that is modelled by the equation, but a whole stretch of genome
that is passed down intact from parent to offspring, and on which we can 
expect some neutral mutations to arise. The pattern of {\em neutral} variation
can be understood if we know how individuals
sampled from the population are related
to one another; that is, if we have a model for the genealogical trees 
relating individuals in a sample from the population. 
Equation~(\ref{FKPP equation}) assumes an infinite population 
density everywhere
so that a finite sample of individuals will be unrelated;
in order to understand genealogies we have to consider~(\ref{stochastic FKPP}).
The first step is to understand the effect of the stochastic 
fluctuations on the forwards in time dynamics of the waves.

Any solution to~\eqref{FKPP equation} with a front-like initial condition $p(0,x)$ which decays sufficiently fast as $x\to \infty$ converges to the travelling wave solution with minimal wavespeed $\sqrt{2ms_0}$ (\cite{uchiyama:1978, bramson:1983}).
Since the speed of this travelling wave
is determined by the behaviour in the `tip' of the wave, where the frequency of
the favoured type is very low, it is very sensitive to stochastic fluctuations.
A great deal of work has gone into understanding the effect of those
fluctuations on the progress of the `bulk' of the wave 
(\cite{brunet/derrida:1997, brunet/derrida:2001, vansaarloos:2003, brunet/derrida/mueller/munier:2006, hallatschek/nelson:2008, mueller/mytnik/quastel:2011, berestycki/berestycki/schweinsberg:2013}).
The 
first striking fact is that the wave is significantly slowed by the 
noise (\cite{brunet/derrida/mueller/munier:2006, mueller/mytnik/quastel:2011}).
The second ramification of the noise is that there really is a well-defined
`wavefront'; that is, assuming that the favoured type is spreading from left
to right in our one-dimensional spatial domain, 
there will be a rightmost point of the support of 
the stochastic travelling wave  (\cite{mueller/sowers:1995}). Moreover, the shape of the wavefront 
is well-approximated by a truncated Fisher wave 
(\cite{brunet/derrida:1997, mueller/mytnik/quastel:2011}).

If we were to take a sample of favoured individuals from a population
evolving according to the analogue of~(\ref{stochastic FKPP}) without
space, then, from~\cite{barton/etheridge/sturm:2004}, their genealogy
would be given by a `coalescent in a random background'; that is, it
would follow a Kingman coalescent but with the instantaneous
rate of coalescence of each pair of lineages at time $t$ before the 
present given by $1/(N_0\overleftarrow{p}(t))$, where $\overleftarrow{p}(t)$ is 
the proportion of the population that is of the favoured type at time $t$
before the present, and $N_0$ is the total population size. 
This suggests that in the spatial context, as we trace back ancestral
lineages, their instantaneous rate of coalescence on meeting at the point $x$
should be proportional to $1/\overleftarrow{p}(t,x)$.
In particular, this means that if several lineages are in the tip 
at the same time, then they can coalesce very quickly. 
In fact, principally because $p(t,x)$ is very rough,
it is difficult to study the genealogy directly by tracking ancestral
lineages and analysing when and where they meet. However,  
several plausible approximations (at least for the population close to
the wavefront) have been proposed for which the frequencies of 
different types in the population are approximated by~(\ref{stochastic FKPP})
and a consensus has emerged that for biologically reasonable models,
over suitable timescales, the genealogy will be determined by a 
Bolthausen-Sznitman coalescent 
(\cite{brunet/derrida/mueller/munier:2006, berestycki/berestycki/schweinsberg:2013}).
We emphasize that this arises as a further scaling of the Kingman
coalescent in a random background. It reflects a separation of timescales.
The `multiple merger' events correspond to bursts of coalescence when 
several lineages are close to the tip of the wave. 
This then is the third ramification of
adding genetic drift to~(\ref{FKPP equation}); the genealogy of a sample
of favoured alleles
from the wavefront will be dominated by `founder effects', resulting from
the fluctuations in the wavefront. The idea is that from time to time 
a fortunate individual gets ahead of the wavefront, where its 
descendants can reproduce uninhibited by competition, at least until 
the rest of the population catches up, by which time they form a significant
portion of the wavefront. 

\subsection*{Other forms of selection: pushed and pulled waves of expansion}

The Fisher-KPP equation, and its stochastic analogue~\eqref{stochastic FKPP}, model a situation in 
which each individual in the population carries one copy of a gene that 
can occur in one of two types, usually denoted $a$ and $A$ and referred to
as alleles. If the type $A$ has a small selective advantage (in a sense
to be made more
precise when we describe our individual based model below), then in a 
suitable scaling limit, 
$p(t,x)$ represents the proportion of the population at location $x$ at 
time $t$ that carries the $A$ allele. 
This can also be used as a model for the frequency of $A$ alleles in a diploid
population, provided that the advantage of carrying two copies of the
$A$ allele is twice that of carrying one. However, natural selection is 
rarely that simple; here our goal is to model a situation in which there
is selection against heterozygotes, that is, individuals carrying one $A$ allele
and one $a$ allele, and in which $AA$-homozygotes are fitter than 
$aa$. As we shall explain below, 
the analogue of the Fisher-KPP equation in this situation
takes the form
\begin{equation} \label{AC equation}
\begin{aligned}
\frac{\partial p}{\partial t}(t,x)=\frac{m}{2}\Delta p(t,x)+
s_0 f\big(p(t,x)\big)
\qquad & \text{for }x\in\R, \, t> 0, \\
0\leq p(0,x)\leq 1 \qquad &\forall x\in\R, \\
\text{where }\quad f(p)=p(1-p)(2p-1+\alpha), \qquad &
\end{aligned}
\end{equation}
with $\alpha >0$ a parameter which depends on the relative fitnesses of $AA$, $Aa$ and $aa$ individuals.

In the case $\alpha \in (0,1)$, the non-linear term $f$ is bistable (since $f(0)=0=f(1)$, $f'(0)<0$, $f'(1)<0$ and $f<0$ on $(0,(1-\alpha)/2)$, $f>0$ on $((1-\alpha)/2,1)$)
and the equation has a unique travelling wave solution given up to translation by the exact form
\begin{equation}
\label{AC stationary wave}
p(t,x)=g\big(x-\alpha \sqrt{\tfrac{ms_0}2}t\big),
\quad \text{where }
g(y)=\big(1+e^{\sqrt{\frac{2s_0}m}y}\big)^{-1}.
\end{equation}
For $\alpha \in [1,2)$, the travelling wave solution with minimal wavespeed is also given by~\eqref{AC stationary wave}.
In both cases, solutions of~\eqref{AC equation} with suitable front-like initial conditions converge to the travelling wave~\eqref{AC stationary wave}~\cite{fife/mcleod:1977,rothe:1981}.
The case $\alpha=0$ corresponds to $AA$ and $aa$ being equally fit, in which 
case, for suitable initial conditions, 
there is a stationary `hybrid zone' trapped between two regions composed almost
entirely 
of $AA$ and almost entirely of $aa$ individuals respectively. 
As observed, for example, by Barton~(\cite{barton:1979}), when $\alpha>2$ the 
symmetric wavefront of~(\ref{AC stationary wave}) is replaced by an
asymmetric travelling wavefront moving at speed $\sqrt{2ms_0(\alpha -1)}$. This 
transition from symmetric to asymmetric wave corresponds to the transition
from a `pushed' wave to a `pulled' wave, notions introduced by 
Stokes~(\cite{stokes:1976}).

Considering the equation~\eqref{AC equation} for general monostable $f$ (i.e.~$f$ satisfying $f(0)=0=f(1)$, $f'(0)>0$, $f'(1)<0$ and $f>0$ on $(0,1)$),
the travelling wave solution with minimal wavespeed $c$ is called a pushed wave if $c>\sqrt{2ms_0 f'(0)}$, and is a pulled wave if $c=\sqrt{2ms_0 f'(0)}$.
(Here, $\sqrt{2ms_0 f'(0)}$ is the spreading speed of solutions of the linearised equation.)
The travelling wave solutions in the bistable case can also be seen as pushed waves (see~\cite{garnier/giletti/hamel/roques:2012}).

The natural stochastic version of~\eqref{AC equation}, which was also discussed briefly by 
Barton~(\cite{barton:1979}), simply adds a Wright-Fisher noise as in~\eqref{stochastic FKPP}.
For $\alpha >1$, this is a reparametrisation of an equation considered by
Birzu et al.~(\cite{birzu/hallatschek/korolev:2018}). Their model is framed in the language of ecology. 
Let $n(t,x)$ denote the population density at point $x$ at time $t$.
They consider
\begin{equation}
\label{Birzu equation}
dn(t,x)=\frac{m}{2}\Delta n(t,x)dt
+ n(t,x)r\big(n(t,x)\big)dt +\sqrt{\gamma\big(n(t,x)\big)n(t,x)}W(dt,dx),
\end{equation}
where $W$ is space-time
white noise, $\gamma (n)$ quantifies the strength
of the fluctuations, and $r(n)$ is the (density dependent) per capita growth
rate. For example, for logistic growth, one would take $r=r_0(1-n/N)$ for 
some `carrying capacity' $N$. A pushed wave arises when species grow
best at intermediate population densities, known as an Allee effect in 
ecology. This effect is typically incorporated by adding a cooperative term to 
the logistic equation, for example by taking 
$$r(n)=r_0\left(1-\frac{n}{N}\right)\left(1+\frac{Bn}{N}\right)$$ 
for some $B>0$.
If we write $p=n/N$, then, writing
$$s_0\left(1-\frac{n}{N}\right)\left(\frac{2n}{N}-1+\alpha\right)=
s_0(\alpha-1)\left(1-\frac{n}{N}\right)\left(\frac{2}{\alpha-1}\frac{n}{N}+1
\right),$$
we see that for $\alpha>1$ we can recover~(\ref{Birzu equation}) 
from a stochastic version of~(\ref{AC equation}) by setting $B=2/(\alpha -1)$ and 
$r_0=s_0(\alpha -1)$.
Birzu et al.~(\cite{birzu/hallatschek/korolev:2018}) define the travelling wave solution 
with minimal wavespeed to the
deterministic equation with this form of $r$ to be pulled if $B\leq 2$, 
`semi-pushed' if $2<B<4$ and `fully pushed' if $B\ge 4$ (see equation~(7)
in~\cite{birzu/hallatschek/korolev:2018} for a more general definition). In our 
parametrisation this says that the wave is pulled for $\alpha\geq 2$ (as
observed by \cite{barton:1979}), semi-pushed for $3/2<\alpha<2$ and 
fully pushed for $\alpha\leq 3/2$. 
For $B\leq 2$ the wavespeed is determined by the growth rate 
in the tip 
(in particular it is independent of $B$), and just as for
the Fisher wave, one can expect the behaviour to be very sensitive
to stochastic fluctuations. For $B>2$,
the velocity of the wave increases with $B$, and also the region of 
highest growth rate shifts from the tip into the bulk of the wave.
These waves should be much less sensitive to fluctuations in the tip.
Moreover if we follow the ancestry of an allele of the favoured
type $A$, that is we follow an ancestral lineage, then in the
pulled case, we expect the lineage to spend most of its time
in the tip of
the wave, and in contrast, in the pushed case, it will spend more time in the bulk. 
Indeed, if the shape of the advancing wave is close to
that of $g$ in~(\ref{AC stationary wave}) and the speed is close to $\nu=\alpha \sqrt{ms_0/2}$, then we should expect the motion 
of the ancestral lineage {\em relative to the wavefront} to be approximately
governed by the stochastic differential equation
\begin{equation}
\label{sde for ancestral lineage}
dZ_t=
\nu dt+\frac{m\nabla g(Z_t)}{g(Z_t)}dt+\sqrt{m}dB_t,
\end{equation}
where $(B_t)_{t\geq 0}$ is a standard Brownian motion.
(We shall explain this in more detail in the context of our model 
in Section~\ref{heuristics} below.)
The stationary measure of this diffusion (if it exists) will be the 
renormalised speed measure,
\begin{equation}
\label{defn of pi}
\pi(x)=\frac{C}{m}g(x)^2 \exp\big(2\nu x/m \big)= \frac C m e^{\frac{2\nu}m x}
(1+e^{\sqrt{\frac{2s_0}m}x})^{-2}.
\end{equation}
Substituting for the wavespeed, $\nu=\alpha\sqrt{ms_0/2}$,
we find that $\pi$ is integrable for $0< \alpha<2$. In other
words, the diffusion defined by~(\ref{sde for ancestral lineage}) has 
a non-trivial stationary distribution when the wave is pushed, but not
when it is pulled.
The expression~\eqref{defn of pi} appears in equation S28 in~\cite{birzu/hallatschek/korolev:2018},
and earlier in~\cite{roques/garnier/hamel/klein:2012} (where the authors study the deterministic equation~\eqref{AC equation}) and in
 Theorem~2 of~\cite{garnier/giletti/hamel/roques:2012} (in relation to pushed wave solutions of general reaction-diffusion equations).   
In \cite{birzu/hallatschek/korolev:2018}, through 
a mixture of simulations and calculations, the authors
also conjecture that
the behaviour of the genealogical trees of a sample of $A$ alleles from near the wavefront will change at 
$B=2$ (corresponding to $\alpha=3/2$) from being, on appropriate 
timescales, a Kingman coalescent for $\alpha \in (0,3/2)$ to being a multiple
merger coalescent for $\alpha>3/2$.

Our calculation of the stationary distribution only tells us about a 
single ancestral lineage; to understand why there should be a further
transition at $\alpha =3/2$,
we need to understand the behaviour of multiple lineages.
We seek a `separation of timescales' in which ancestral lineages 
reach stationarity on a faster timescale than coalescence; 
c.f.~\cite{nordborg/krone:2002}.
Recalling that we are sampling type $A$ alleles from near the wavefront,
then just as for the Fisher-KPP case,
the instantaneous rate of coalescence of two lineages that meet at the 
position $x\in\R$ relative to the wavefront should be proportional to the inverse of the density 
of $A$ alleles at $x$, which we approximate as 
$1/(2N_0 g(x))$ for a large constant $N_0$ (corresponding to the population 
density).
If $N_0$ is sufficiently large, then the 
lineages will not coalesce before their spatial positions reach equilibrium, and so
the probability that the two lineages are both at position $x$ relative to the wavefront 
should be proportional to $\pi(x)^2$.
This suggests that in this scenario the time to 
coalescence should be approximately exponential, with
parameter proportional to
$\int_{-\infty}^\infty \pi(x)^2/g(x)dx$ (this calculation appears in \cite{birzu/hallatschek/korolev:2018} 
in their equation~S119). This quantity is finite precisely
when $\alpha \in (0,3/2)$. If we sample $k$ lineages, one can conjecture that,
because of the separation of timescales, once a first pair
of lineages coalesces, the additional time until the next merger is the same 
as if the remaining $k-1$
lineages were started from points sampled independently
according to the stationary distribution $\pi$. This then strongly 
suggests that in the regime $\alpha \in (0,3/2)$, after suitable scaling, the genealogy of a sample will
converge to a Kingman coalescent. 

Although we believe that the suitably timescaled genealogy of lineages
sampled from near the wavefront of the advance of the favoured type really will
converge to Kingman's coalescent for all $\alpha \in (0,3/2)$, 
our main results in this article will be restricted to the case $\alpha \in (0,1)$.
The difficulty
is that for $\alpha >1$, as $x\to\infty$, the stationary measure $\pi(x)$ does not
decay as quickly as the wave profile $g(x)$. Consequently, 
a diffusion driven by~(\ref{sde for ancestral lineage}) will spend 
a non-negligible proportion of its 
time in the region where $g$ is very small, which is precisely 
where the fluctuations of $p$ about $g$ (or rather fluctuations of $1/p$ 
about $1/g$) become significant and our approximations
break down. 
For this reason, in what follows, we shall restrict ourselves to the 
case $\alpha<1$. 
Unlike the parameter range corresponding to~(\ref{Birzu equation}), in
this setting, the growth rate in the tip of the wave is actually 
negative, and the non-linear term $f$ in~\eqref{AC equation} is bistable. In ecology this would correspond to a strong Allee effect;
for us, it means that we can control the time that the ancestral lineage of an $A$ allele
spends in the tip of the wave (from which it is repelled). 
In Section~\ref{heuristics} below, we will briefly discuss the case $\alpha \in [1,3/2)$ in the context of our model.

\subsection*{Some biological considerations}

Our goal is to write down a mathematically tractable, but biologically
plausible, individual based model for a population subject to selection
acting on diploids, and to show that when suitably scaled the genealogy
of a sample from near the wavefront of expansion of $A$ alleles 
converges to a Kingman coalescent.
As we will see below, for this model the proportion of $A$ alleles will be 
governed by a discrete space stochastic analogue of~(\ref{AC equation}) 
with $0<\alpha <1$. 

The model that we define and analyse below will be a modification of 
a classical Moran model for a spatially structured population with 
selection in which we treat each allele as an individual. 
In order to justify this choice, we first follow a more
classical approach by considering a variant of a model that is usually
attributed to Fisher and Wright, for a large (diploid)
population, evolving in
discrete generations. 

First we explain the form of the nonlinearity in~(\ref{AC equation}). 
For simplicity, let us temporarily consider a population without spatial
structure. 
We are following the fate of a gene with two alleles, 
$a$ and $A$. Individuals in the population each carry two copies of
the gene. 
During reproduction, each individual produces a very large number of 
germ cells
(containing a copy of all the genetic material of the parent) 
which then split into
gametes (each carrying just one copy of the gene). All the 
gametes produced in this way are pooled and, if the population
is of size $N_0$, then $2N_0$ gametes 
are sampled (without replacement) from the pool.
The sampled gametes fuse at random to form the next generation of diploid
individuals. 
To model selection, we suppose that the numbers of germ cells produced
by individuals are
in the proportion
$1+2\alpha s: 1+(\alpha -1)s:1$
for genetic types $AA$, $Aa$, $aa$ respectively.
Here $\alpha \in (0,1)$ is a positive constant and $s>0$ is small, with $(\alpha +1)s<1$.
Notice in particular that type $AA$ homozygotes 
are `fitter' than type $aa$ homozygotes,
in that they contribute more gametes to the pool (fecundity selection). 
Both are fitter than 
the heterozygotes ($Aa$ individuals). 

Suppose that the proportion of type $A$ alleles in the population is $w$. 
If the population is in Hardy-Weinberg proportions, then the proportions of 
$AA$, $Aa$ and $aa$ individuals are $w^2$, $2w(1-w)$ and $(1-w)^2$ respectively.
Hence the proportion of type $A$ in the (effectively infinite) pool of
gametes produced during reproduction is
\begin{align}
\nonumber
&\frac{(1+2\alpha s)w^2+\tfrac{1}{2}(1+(\alpha-1)s)2w(1-w)}{1+2\alpha sw^2+(\alpha-1)s \cdot 2w(1-w)}\\
&\quad =(1+\alpha s-s)w+(3-\alpha )sw^2 -2sw^3 +\mathcal O(s^2) \notag \\
\label{change over a single generation 1}
&\quad =(1-(\alpha+1)s)w+\alpha s (2w-w^2)+s(3w^2-2w^3)+\mathcal O (s^2)\\
&\quad =w+\alpha s w(1-w) +sw(1-w)(2w-1) +\mathcal O (s^2).
\label{change over a single generation 2}
\end{align}
We will assume that $s$ is sufficiently small that
terms of $\mathcal O (s^2)$ are negligible.
If the population were infinite, then the frequency of $A$ alleles would 
evolve deterministically, and if $s=s_0/K$ for some large $K$, then 
measuring time in units of $K$ generations, we see that $w$ will 
evolve approximately according to the differential equation
\begin{equation}
\label{ODE for w}
\frac{dw}{dt}=\alpha s_0w(1-w)+s_0 w(1-w)(2w-1)=s_0w(1-w)(2w-1+\alpha),
\end{equation}
and we recognise the nonlinearity in~(\ref{AC equation}).

The easiest way to incorporate spatial structure into the Wright-Fisher
model described above is to suppose that the population is subdivided 
into demes (islands of population) which we can, for example, take to be 
the vertices of a lattice, and in each generation a proportion of the 
gametes produced in a deme is distributed to its neighbours (plausible, for 
example, for a population of plants). If we assume that this dispersal is 
symmetric, the population size in each deme is the same, and the 
proportion of gametes that migrate scales as $1/K$, then this will
result in the addition of a term involving the discrete Laplacian to the 
equation~(\ref{ODE for w}).

Since we are interested in understanding the interplay of selection, spatial
structure, and random genetic drift, we must consider a finite population. 
We shall nonetheless assume that the 
population in each deme 
is large, so that our assumption that
the population is in Hardy-Weinberg equilibrium remains valid. When this assumption
is satisfied, to specify the evolution of the proportions of the types 
$AA$, $Aa$, $aa$, it suffices to track the proportion of $A$ gametes in
each deme. Moreover, because we assume that the chosen gametes fuse at
random to form the next generation, the genealogical trees relating a sample
of alleles from the population can also be recovered from tracing just
single types. The only role that pairing of genes in individuals plays is in determining what 
proportion of the gamete pool will be contributed by
a given allele in the parental
population. 

Suppose that the proportion of $A$ alleles in some generation $t$ is $w$
and recall that the population consists of $2N_0$ alleles.
The probability that two type $A$ alleles sampled from generation $t+1$
are both descendants of the same parental 
allele is approximately $1/(2N_0w)$ since $s$ is small, while the probability
that three or more are all descended from the same parent
is $\mathcal O(1/N_0^2)$. Recalling that $s=s_0/K$ for some large $K$,
if now we measure time in units of $K$ generations, 
the forwards in time model for allele frequencies will be approximated 
by a stochastic differential equation,
$$dw=s_0w(1-w)(2w-1+\alpha)dt+\sqrt{\frac{K}{2N_0}w(1-w)}dB_t,$$
where $(B_t)_{t\geq 0}$ is a Brownian motion, and 
the genealogy of a sample of
type $A$ alleles
from our population will be well-approximated by a time-changed Kingman 
coalescent in which the instantaneous rate of coalescence, when the 
proportion of type $A$ alleles in the population is $w$, is $K/(2N_0w)$.

The Wright-Fisher model is inconvenient mathematically, but we now see
that for the purpose of understanding the genealogy, we can replace
it by any other model in which, over large 
timescales, the allele 
frequencies evolve in (approximately) the same way and in which, as we trace
backwards in time, the genealogy of a sample of favoured alleles is 
(approximately) the same (time-changed) Kingman coalescent. 
This will allow us to replace the discrete generation 
(diploid) `Wright-Fisher' 
model by a much more
mathematically convenient `Moran model', in which changes in allele
frequencies in each deme will be driven by Poisson processes of 
reproduction events in which exactly one allele is born and exactly one
dies.

Because our Moran model deals directly with alleles, from now on we shall refer to
alleles as {\em individuals}.
To understand the form that our Moran model should take, 
let us first consider the non-spatial setting. Once again we
trace $2N_0$ individuals (alleles), but now we label them $1,2,\ldots , 2N_0$. 
Reproduction events will take place at the
times of a rate $2N_0K$ Poisson process. 
Inspired by~(\ref{change over a single generation 2}), we
divide events into three types: neutral events, which will take place
at rate $2N_0K(1-(\alpha +1)s)$, events capturing directional selection
at rate $2N_0K\alpha s$, and events capturing selection against 
heterozygosity, at rate $2N_0Ks$.
In a neutral event, an ordered pair of individuals 
is chosen uniformly at random 
from the population; the first dies and is replaced by an offspring of the 
second (and this offspring inherits the label of the first individual).
At an event corresponding to directional selection, an ordered
pair of individuals is
chosen uniformly at random from the population; if the type of the second 
is $A$, then it produces an offspring which replaces the first. At an 
event corresponding to selection against heterozygosity, an ordered 
triplet of individuals is picked from the population; if the second and third
are of the same type, then the second produces an offspring that replaces
the first. (Note that in such an event, the first individual is either replaced by or remains a type $A$ if and only if at least two of the triplet of individuals picked were type $A$.)

Noting that if $X_1$, $X_2$ and $X_3$ are i.i.d.~Bernoulli($w$) random variables then
$$
\P\left(X_1+X_2 \geq 1 \right)=2w-w^2
\quad \text{ and } \quad 
\P\left(X_1+X_2+X_3 \geq 2 \right)=3w^2-2w^3,
$$
and recalling that $s=s_0/K$,
using~(\ref{change over a single generation 1}), 
we see that for large $K$, the proportion of $A$ alleles
under this model will be close to that under our time-changed Wright-Fisher model. 
Moreover, since there is at most one birth event at a time, coalescence
of ancestral lineages is necessarily pairwise.
If in a reproduction event the parent is type $A$, then the probability 
that a pair of type $A$ ancestral lineages corresponds to the parent and
its offspring (and therefore merges in the event) is
$1/(2N_0w(2N_0w-1))$. Since $s$ is very small, the 
instantaneous rate at which events with a type $A$ parent
fall is approximately $2N_0Kw$.
Thus, the probability that a particular pair of two type $A$ individuals
sampled from the population at time $t+\delta t$ are descended from
the same type $A$ individual at time $t$ is (up to a lower order error)
$K/(2N_0w)\delta t$ and we see that the genealogy under this model will
be (up to a small error) the same as under the Wright-Fisher model.

In what follows, to avoid too many factors of two, we are going to write 
$N=2N_0$ for the number of individuals in our Moran model. 

\subsection{Definition of the model} \label{subsec:modeldefn}
We now give a precise definition of our model.
Take $\alpha \in (0,1)$, $s_0>0$ and $m>0$.
Let $n,N\in \N$. 
We are going to define our (structured) Moran model on $\frac1n \Z$ in such
a way that there are $N$ 
individuals in each site (or deme) and they are 
indexed by $[N]:=\{1,\ldots,N\}$.
We shall denote the type of the $i$th individual at site $x$ at time $t$ by 
$\xi_t^n(x,i)\in\{0,1\}$, with $\xi_t^n(x,i)=1$ meaning that the
individual is type $A$,
and $\xi_t^n(x,i)=0$ meaning that the individual is type $a$.
For $x\in \frac1n \Z$ and $t\geq 0$, let 
$$p^n_t(x)=\frac{1}{N}\sum_{i=1}^N \xi^n_t (x,i) $$
be the proportion of type $A$ at $x$ at time $t$.
We shall reserve the symbol $x$ for space and $i,j,k$ for the label of
an individual.

Let 
\begin{equation} \label{eq:snrndefn}
s_n=\frac{2s_0}{n^{2}} \quad \text{ and }\quad r_n= \frac{n^2}{2N}.
\end{equation}
(Here, $s_n$ is a selection parameter which determines the space scaling needed to see a non-trivial limit, and $r_n$ is a time scaling parameter.)

To specify the dynamics of the process, we define four
independent families of i.i.d.~Poisson processes.
These will govern neutral reproduction, directional selection, selection 
against heterozygotes and migration respectively. 
Let $((\mathcal P_t^{x,i,j})_{t\geq 0})_{x\in \frac1n \Z, i \neq j \in [N]}$ be i.i.d.~Poisson processes with rate $r_n(1-(\alpha+1)s_n)$.
Let $((\mathcal S_t^{x,i,j})_{t\geq 0})_{x\in \frac1n \Z, i \neq j \in [N]}$ be i.i.d.~Poisson processes with rate $r_n \alpha s_n$.
Let $((\mathcal Q_t^{x,i,j,k})_{t\geq 0})_{x\in \frac1n \Z, i,j,k \in [N]\text{ distinct}}$ be i.i.d.~Poisson processes with rate $\frac{1}{N}r_n s_n$.
Let $((\mathcal R_t^{x,i,y,j})_{t\geq 0})_{x,y \in \frac1n \Z, |x-y|=n^{-1} , i,j\in[N]}$ be i.i.d.~Poisson processes with rate $m r_n$.

For a given initial condition $p^n_0:\frac 1n \Z \to \frac 1 N \Z \cap [0,1]$, 
we assign labels to the type $A$ individuals in each site uniformly at random.
That is, we define 
$(\xi^n_0(x,i))_{x\in \frac 1n \Z, i\in [N]}$ as follows.
For each $x\in \frac 1n \Z$ independently, take $I_x\subseteq [N]$, where $I_x$ is chosen uniformly at random from $\{A\subseteq [N]:|A|=Np^n_0(x)\}$.
For $i\in [N]$, let $\xi^n_0(x,i)=\1_{\{i\in I_x\}}$.

The process $(\xi^n_t(x,i))_{x\in \frac 1n \Z, i\in [N], t\ge 0}$ evolves as follows.
\begin{enumerate}
\item
If $t$ is a point in $\mathcal P^{x,i,j}$, then at time $t$, the individual at $(x,i)$ is replaced by offspring of the individual at $(x,j)$, i.e.~we let
$\xi^n_t(x,i)=\xi^n_{t-}(x,j)$.
\item
If $t$ is a point in $\mathcal S^{x,i,j}$, then at time $t$, if the individual at $(x,j)$ is type $A$ then the individual at $(x,i)$ is replaced by offspring of the individual at $(x,j)$, i.e.~we let
$$\xi^n_t(x,i)=
\begin{cases}
\xi^n_{t-}(x,j) \quad \text{ if }\xi^n_{t-}(x,j)=1,\\
\xi^n_{t-}(x,i) \quad \text{ otherwise}.
\end{cases}$$
\item
If $t$ is a point in $\mathcal Q^{x,i,j,k}$, then at time $t$, if the individuals at $(x,j)$ and $(x,k)$ have the same type then the individual at $(x,i)$ is replaced by offspring of the individual at $(x,j)$, i.e.~we let
$$\xi^n_t(x,i)=
\begin{cases}
\xi^n_{t-}(x,j) \quad \text{ if }\xi^n_{t-}(x,j)=\xi^n_{t-}(x,k),\\
\xi^n_{t-}(x,i) \quad \text{ otherwise}.
\end{cases}$$
\item
If $t$ is a point in $\mathcal R^{x,i,y,j}$, then at time $t$, the individual at $(x,i)$ is replaced by offspring of the individual at $(y,j)$, i.e.~we let
$\xi^n_t(x,i)=\xi^n_{t-}(y,j)$.
\end{enumerate}
Ancestral lineages will be represented in the form of a pair with the first 
coordinate recording the spatial position and the second the label of the
ancestor. More precisely, 
for $T\ge 0$, $t\in [0,T]$, $x_0\in \frac1n \Z$ and $i_0\in [N]$, 
if the individual at site $y$ with label $j$ is the ancestor at time $T-t$ of the  individual at site $x_0$ with label $i_0$ at time $T$, then we
let
$(\zeta^{n,T}_t(x_0,i_0),\theta^{n,T}_t(x_0,i_0))=(y,j)$. 
The pair $(\zeta_t^{n,T}(x_0,i_0),\theta_t^{n,T}(x_0,i_0))_{t\in [0,T]}$ is a jump process with $$(\zeta^{n,T}_0(x_0,i_0),\theta^{n,T}_0(x_0,i_0))=(x_0,i_0).$$ For some $t\in (0,T]$, suppose that $(\zeta^{n,T}_{t-}(x_0,i_0),\theta^{n,T}_{t-}(x_0,i_0))=(x,i)$.
Then if $T-t$ is a point in $\mathcal P^{x,i,j}$ for some $j\neq i$, we let
$(\zeta^{n,T}_t(x_0,i_0),\theta^{n,T}_t(x_0,i_0))=(x,j)$.
If instead $T-t$ is a point in $\mathcal S^{x,i,j}$ for some $j\neq i$, we let
$$(\zeta^{n,T}_t(x_0,i_0),\theta^{n,T}_t(x_0,i_0))=
\begin{cases}
(x,j) \quad \text{ if }\xi^n_{(T-t)-}(x,j)=1,\\
(x,i) \quad \text{ otherwise}.
\end{cases}$$
If instead $T-t$ is a point in $\mathcal Q^{x,i,j,k}$ for some $j\neq k\in [N]\setminus \{i\}$, we let
$$(\zeta^{n,T}_t(x_0,i_0),\theta^{n,T}_t(x_0,i_0))=
\begin{cases}
(x,j) \quad \text{ if }\xi^n_{(T-t)-}(x,j)=\xi^n_{(T-t)-}(x,k),\\
(x,i) \quad \text{ otherwise}.
\end{cases}$$
Finally, if $T-t$ is a point in $\mathcal R^{x,i,y,j}$ for some $y\in \{x-n^{-1}, x+n^{-1}\}$, $j\in[N]$, we let
$(\zeta^{n,T}_t(x_0,i_0),\theta^{n,T}_t(x_0,i_0))=(y,j)$.
These are the only times at which the ancestral lineage process $(\zeta^{n,T}_s(x_0,i_0),\theta^{n,T}_s(x_0,i_0))_{s\in [0,T]}$ jumps.

\subsection{Main results} \label{subsec:mainresults}

Recall from~\eqref{AC stationary wave} that $g:\R\to \R$ is given by 
\begin{equation} \label{eq:gdefn}
g(x)=(1+e^{\sqrt{\frac{2s_0}m}x})^{-1}.
\end{equation}
In our main results, we will make the following assumptions on the initial condition $p^n_0$, for $b_1,b_2>0$ to be specified later:
\begin{align} \label{eq:conditionA}
&p^n_0(x)=0\; \forall x\ge N, \quad 
p^n_0(x)=1\; \forall x\le -N, \notag \\
&\sup_{x\in \frac 1n \Z}|p_0^n(x)-g(x)|\leq b_1
\qquad \text{and }\qquad 
\sup_{z_1,z_2\in \frac 1n \Z,|z_1-z_2|\leq n^{-1/3}}|p^n_0(z_1)-p^n_0(z_2)| \le n^{-b_2}. \tag{A}
\end{align}
We will assume throughout that there exists $a_0>0$ such that $(\log N)^{a_0}\le \log n$ for $n$ sufficiently large.
The idea is that we need $N\gg n\gg 1$, in order that we are close to the 
deterministic limit, but we do not want $N$ to tend to infinity so quickly that 
we don't see the effect of the stochastic perturbation at all. 

For $t\ge 0$, define the position of the random travelling front at time $t$ by letting
\begin{equation} \label{eq:muntdefn}
\mu^n_t =\sup\{x\in \tfrac 1n \Z : p^n_t(x)\ge 1/2\}.
\end{equation}
For $t\ge 0$ and $R>0$, let
\begin{equation} \label{eq:Gdefn}
G_{R,t}=\{(x,i)\in \tfrac 1n \Z \times [N] :|x-\mu^n_t|\le R, \, \xi^n_t(x,i)=1\},
\end{equation}
the set of type $A$ individuals which are near the front at time $t$.

Our first main result says that if at a large time $T_n$ we sample a type $A$ individual from near the front, then the position of its ancestor relative to the front at a much earlier time $T_n-T'_n$ has distribution approximately given by $\pi$ (as defined in~\eqref{eq:pidefn}).

\begin{theorem} \label{thm:statdist}
Suppose $\alpha \in (0,1)$ and,  for some $a_1>1$, $N\ge n^{a_1}$ for $n$ sufficiently large.
There exists $b_1>0$ such that for $b_2>0$ and $K_0<\infty$ the following holds.
Suppose condition~\eqref{eq:conditionA} holds,
$T_n \le N^2$
and $T'_n \to \infty$ as $n\to \infty$ with $T_n-T'_n \ge (\log N)^2$. 
Let $(X_0,J_0)\in \frac 1n \Z \times [N]$ be  measurable with respect to $\sigma((\xi^n_{T_n}(x,i))_{x\in \frac 1n \Z, i\in [N]})$ with $(X_0,J_0)\in G_{K_0,T_n}.$
Then
$$
\zeta^{n,T_n}_{T'_n}(X_0, J_0) - \mu^n_{T_n-T'_n} \stackrel{d}{\to} Z \quad \text{as }n\to \infty,
$$
where $Z$ is a random variable with density
\begin{equation} \label{eq:pidefn}
\pi(x)=\frac{g(x)^2 e^{\alpha \sqrt{\frac {2s_0}m} x}}{\int_{-\infty}^\infty g(y)^2 e^{\alpha  \sqrt{\frac {2s_0}m} y}dy}.
\end{equation}
\end{theorem}

Our second main result says that the genealogy of a sample of type $A$ individuals from near the front at a large time $T_n$ is approximately given by a Kingman coalescent (under a suitable time rescaling).

\begin{theorem} \label{thm:main}
Suppose $\alpha \in (0,1)$ and, for some $a_2>3$, $N \ge n^{a_2}$ for $n$ sufficiently large.
There exists $b_1>0$ such that for $b_2>0$, $k_0\in \N$ and $K_0<\infty$, the following holds.
Suppose condition~\eqref{eq:conditionA} holds, and take
$T_n\in [N,N^2]$. 
Let $(X_1,J_1), \ldots ,(X_{k_0},J_{k_0})$ be measurable with respect to $\sigma((\xi^n_{T_n}(x,i))_{x\in \frac 1n \Z, i\in [N]})$ and distinct, with $(X_i,J_i) \in G_{K_0,T_n}$ $\forall i \in [ k_0]$.

For $i, j\in [k_0],$
let $\tau^{n}_{i,j}$ denote the time at which the $i^{\text{th}}$ and $j^{\text{th}}$ ancestral lineages coalesce, i.e.~let
$$
\tau^{n}_{i,j}
=\inf\{t\geq 0:(\zeta^{n,T_n}_t(X_i,J_i),\theta^{n,T_n}_t(X_i,J_i))
=(\zeta^{n,T_n}_t(X_j,J_j),\theta^{n,T_n}_t(X_j,J_j))\}.
$$
Then 
$$
\left( \frac{(2m+1) n}{N}\frac{\int_{-\infty}^\infty g(x)^3 e^{2\alpha  \sqrt{\frac {2s_0}m}  x}dx}{\left( \int_{-\infty}^\infty g(x)^2 e^{\alpha  \sqrt{\frac {2s_0}m}  x}dx \right)^2}\tau^{n}_{i,j} \right)_{i,j \in [k_0]}
\stackrel{d}{\longrightarrow} (\tau_{i,j} )_{i,j \in [k_0]} \quad \text{as }n\to \infty,
$$
where $\tau_{i,j}$ is the time at which the $i^{\text{th}}$ and $j^{\text{th}}$ ancestral lineages coalesce in the Kingman ${k_0}$-coalescent.

\end{theorem}

\subsection{Strategy of the proof}
\label{heuristics}
We will show that if $N\gg n$, then if $n$ is large and $T_0$ is not too large, $(p^n_t)_{t\in [0,T_0]}$ is approximately given by the solution of the PDE
\begin{equation} \label{eq:PDE}
\frac{\partial u}{\partial t}=\tfrac 12 m \Delta u+s_0 u(1-u)(2u-1+\alpha).
\end{equation}
(Recall from our discussion of a non-spatial Moran model before Section~\ref{subsec:modeldefn} that the non-linear term in~\eqref{eq:PDE} comes from the events corresponding to the Poisson processes $(\mathcal S^{x,i,j})_{x,i,j}$ and $(\mathcal Q^{x,i,j,k})_{x,i,j,k}$.
The Laplacian term comes from the Poisson processes $(\mathcal R^{x,i,y,j})_{x,i,y,j}$ which cause migration between neighbouring sites and whose
rate was chosen to coincide with the diffusive rescaling.)

As noted in~\eqref{AC stationary wave}, $u(t,x):=g(x-\alpha \sqrt{\frac{ms_0}2} t)$ is a travelling wave solution of~\eqref{eq:PDE}.
In the case $\alpha \in (0,1)$, work of Fife and McLeod~\cite{fife/mcleod:1977} shows that for a front-like initial condition $u_0$ satisfying $\limsup_{x\to \infty}u_0(x)<\frac 12 (1-\alpha)$ and $\liminf_{x\to -\infty}u_0(x)>\frac 12 (1-\alpha)$, the solution of~\eqref{eq:PDE} converges to a moving front with shape $g$ and wavespeed $\alpha \sqrt{\frac{ms_0}2}$.
We can use this to show that if $N\gg n$, then for large $n$, with high probability,
\begin{equation} \label{eq:heurevent}
p^n_t(x)\approx g(x-\mu^n_t)\; , \forall x\in \tfrac 1n \Z, t\in [\log N,N^2]
\quad \text{ and }\quad 
\frac{\mu^n_t-\mu^n_s}{t-s}\approx \alpha \sqrt{\tfrac{ms_0}2} \; ,\forall s<t\in [\log N,N^2],
\end{equation}
where $\mu^n_t$ is the front location defined in~\eqref{eq:muntdefn}
(see Proposition~\ref{prop:eventE1}).

Suppose the event in~\eqref{eq:heurevent} occurs, and sample a type $A$ individual at time $T_n$ by taking $(X_0,J_0)$ with $\xi^n_{T_n}(X_0,J_0)=1$. We will show that the recentred ancestral lineage process $(\zeta^{n,T_n}_t(X_0,J_0)-\mu^n_{T_n-t})_{t\in [0,T_n]}$ moves approximately according to the diffusion
$$
dZ_t =\alpha \sqrt{\tfrac{ms_0}2} dt+\frac{m\nabla g(Z_t)}{g(Z_t)}dt +\sqrt m dB_t,
$$
where $(B_t)_{t\ge 0}$ is a Brownian motion (see Lemmas~\ref{lem:vnvbound} and~\ref{lem:qnvnonepoint}).
This can be explained heuristically as follows.
Observe first that $(\mu^n_{T_n-t}-\mu^n_{T_n-t-s})/s\approx \alpha \sqrt{\frac{ms_0}2}$ for $s>0$.
Then if $\zeta^{n,T_n}(X_0,J_0)$ jumps at some time $t$, and $\zeta^{n,T_n}_{t-}(X_0,J_0)=x_0$,
the conditional probability that $\zeta^{n,T_n}_t(X_0,J_0)=x_0+n^{-1}$ is 
$$
\frac{p^n_{T_n-t}(x_0+n^{-1})}{p^n_{T_n-t}(x_0-n^{-1})+p^n_{T_n-t}(x_0+n^{-1})}
\approx \frac 12 +\frac 12 \frac{\nabla g(x_0-\mu^n_{T_n-t})}{g(x_0-\mu^n_{T_n-t})}n^{-1}.
$$
Finally, the total rate at which $\zeta^{n,T_n}(X_0,J_0)$ jumps is given by $2mr_n N =mn^2$, and the jumps have increments $\pm n^{-1}$.

As we observed before in~\eqref{defn of pi}, $(Z_t)_{t\ge 0}$ has a unique stationary distribution given by $\pi$, as defined in~\eqref{eq:pidefn}.
In Theorem~\ref{thm:statdist}, we show rigorously that for large $t$, $\zeta^{n,T_n}_t(X_0,J_0)-\mu^n_{T_n-t}$ has distribution approximately given by $\pi$.
Theorem~\ref{thm:statdist} is not strong enough to give the precise estimates
that we need for Theorem~\ref{thm:main}, and so in fact we prove
Theorem~\ref{thm:main} 
first and then Theorem~\ref{thm:statdist} will follow from 
results that we have obtained along the way.

A pair of ancestral lineages can only coalesce if they are distance at most $n^{-1}$ apart.
Take a pair of type $A$ individuals at time $T_n$ by sampling $(X_1,J_1)\neq (X_2,J_2)$ with $\xi^n_{T_n}(X_1,J_1)=1=\xi^n_{T_n}(X_2,J_2)$.
Suppose at some time $T_n-t$ that their ancestral lineages are at the same site, i.e.~$\zeta^{n,T_n}_t(X_1,J_1)=x=\zeta^{n,T_n}_t(X_2,J_2)$ for some $x\in \frac 1n \Z$.
For $\delta_n >0$ small, on the time interval $[T_n-t-\delta_n,T_n-t]$, each type $A$ individual at $x$ produces offspring at $x$ at rate approximately $r_n N$, and not many types produce more than one offspring. Hence the number of pairs of type $A$ individuals at $x$ at time $T_n-t$ which have common ancestors at time $T_n-t-\delta_n$ is approximately $r_n N^2 \delta_n p^n_{T_n-t-\delta_n}(x)$ (see Lemma~\ref{lem:coalCB}).
Therefore, the probability that our pair of lineages coalesce within time $\delta_n$ (backwards in time), which is the same as the probability that it is
one such pair, is approximately
\begin{equation} \label{eq:heurcoal}
\frac{r_n N^2 \delta_n p^n_{T_n-t-\delta_n}(x)}{{N p^n_{T_n-t}(x) \choose 2}} \approx \frac{n^2 \delta_n}{Np^n_{T_n-t}(x)}.
\end{equation}
Similarly, if $\zeta^{n,T_n}_t(X_1,J_1)=x$ and $\zeta^{n,T_n}_t(X_2,J_2)=x+n^{-1}$
then, since an individual at $x$ produces offspring at $x+n^{-1}$ at rate $mr_nN$ and vice-versa,
the probability that the pair of lineages coalesce within time $\delta_n$ is approximately
\begin{equation} \label{eq:heurcoal2}
\frac{m r_n N^2 \delta_n (p^n_{T_n-t-\delta_n}(x)+p^n_{T_n-t-\delta_n}(x+n^{-1}))}{N  p^n_{T_n-t}(x)\cdot Np^n_{T_n-t}(x+n^{-1})} \approx \frac{m n^2 \delta_n}{Np^n_{T_n-t}(x)}.
\end{equation}
These heuristics suggest that for $x_0\in \frac 1n \Z$,
since $\pi(x_0)\pi(x_0+n^{-1})^{-1} \approx 1$ and $\pi(x_0)\pi(x_0-n^{-1})^{-1} \approx 1$,
the rate at which the pair of ancestral lineages of $(X_1,J_1)$ and $(X_2,J_2)$ coalesce with the ancestral lineage of $(X_1,J_1)$ at 
location $x_0$ relative to the front 
should be approximately
\begin{equation*} 
n^{-2} \pi(x_0)^2 \cdot \frac{n^2}{N g(x_0)}+2n^{-2} \pi(x_0)^2  \cdot \frac{m n^2}{N g(x_0)}
= (2m+1)\frac{\pi(x_0)^2}{N g(x_0)}.
\end{equation*}
Note that for some constants $C_1,C_2>0$,
\begin{equation} \label{eq:heurasym}
\frac{\pi(x_0)^2}{g(x_0)}\sim C_1 e^{(2\alpha -3)\sqrt{\frac{2s_0}m}x_0}\to 0 \; \text{as }x_0\to \infty
\quad \text{ and } \quad
\frac{\pi(x_0)^2}{g(x_0)}\sim C_2 e^{2\alpha \sqrt{\frac{2s_0}m} x_0}\to 0 \; \text{as }x_0\to -\infty .
\end{equation}
This suggests that coalescence only occurs (fairly) close to the front.
If a pair of lineages coalesce close to the front, then the rate at which they subsequently coalesce with any other lineage is $\mathcal O(n^2 N^{-1})$, which suggests that if $N\gg n^2$, their location relative to the front will have distribution approximately given by $\pi$ before any more coalescence occurs.
Hence the genealogy of a sample of type $A$ individuals from near the front should be approximately given by a Kingman coalescent with rate
$$
\sum_{x_0 \in \frac 1n \Z} (2m+1) \frac{\pi(x_0)^2}{Ng(x_0)}\approx (2m+1) \frac n N \int_{-\infty}^\infty \frac{\pi(y)^2}{g(y)}dy.
$$
This result is proved in Theorem~\ref{thm:main} (with the additional technical assumption that $N\gg n^3$).

For $\alpha \in [1,2)$, work of Rothe~\cite{rothe:1981} shows that for the PDE~\eqref{eq:PDE}, if the initial condition $u_0(x)$ decays sufficiently quickly as $x\to \infty$ then the solution converges to a moving front with shape $g$ and wavespeed $\alpha \sqrt{\frac{ms_0}2}$.
Moreover,~\eqref{eq:heurasym} holds for any $\alpha \in (0,3/2)$, which suggests that Theorem~\ref{thm:main} should hold for any $\alpha \in (0,3/2)$.
The main difficulty in proving the theorem is that $p^n_t(x)^{-1}$ is hard to control when $x-\mu^n_t$ is very large, i.e.~far ahead of the front.
This in turn makes it hard to control the motion of ancestral lineages if they are far ahead of the front.
For $\alpha \in (0,1)$, the non-linear term $f(u)=u(1-u)(2u-1+\alpha)$ in the PDE~\eqref{eq:PDE} satisfies $f(u)<0$ for $u\in (0,\frac 12 (1-\alpha))$, which means that far ahead of the front, the proportion of type $A$ decays.
This allows us to show that with high probability, no lineages of type $A$ individuals stay far ahead of the front for a long time (see Proposition~\ref{prop:eventE4}), which then gives us upper bounds on the probabilities of lineages being far ahead of the front at a fixed time (see Proposition~\ref{prop:intip}).
A proof of Theorem~\ref{thm:main} for $\alpha \in [1,3/2)$ would require a different method to bound these tail probabilities, along with more delicate estimates on $p^n_t(x)$ for large $x$ in order to apply~\cite{rothe:1981} and ensure that $p^n_t(\cdot)\approx g(\cdot - \mu^n_t)$ with high probability at large times $t$.

One of the main tools in the proofs of Theorems~\ref{thm:statdist}
and~\ref{thm:main} is the notion of tracers. In population genetics, this
corresponds to labelling a subset of individuals by a neutral genetic marker,
which is passed down from parent to offspring, and which has no
effect on the fitness of an individual by whom it is carried.
Such markers allow us to deduce which
individuals in the population are descended from a particular subset of
ancestors (c.f.~\cite{donnelly/kurtz:1999}).
The idea of using these markers, or `tracers',
in the context of expanding biological populations goes back at least to
Hallatschek and Nelson~\cite{hallatschek/nelson:2008}, and has
subsequently been used, for example, by
Durrett and Fan~\cite{durrett/fan:2016},
Birzu et al.~\cite{birzu/hallatschek/korolev:2018}
and Biswas et al.~\cite{biswas/etheridge/klimek:2018}.
The idea is that at some time $t_0$, a subset of the type $A$ individuals are labelled as `tracers'.
At a later time $t$, we can look at the subset of type $A$ individuals which are descended from the original set of tracers.
In particular, for $0\le t_0 \le t$ and $x_1,x_2 \in \frac 1n \Z$, we can record
the proportion of individuals at $x_2$ at time $t$ which are descended from type $A$ individuals at $x_1$ at time $t_0$.
This tells us the conditional probability that the time-$t_0$ ancestor of a randomly chosen type $A$ individual at $x_2$ at time $t$ was at $x_1$.
For $x_1,x_2 \in \frac 1n \Z$ and $t\ge 0$, and taking $\delta_n>0$ very small, we can also record the number of pairs of type $A$ individuals at $x_1$ and $x_2$ at time $t+\delta_n$ which have the same ancestor at time $t$.
This tells us the conditional probability that a randomly chosen pair of type $A$ lineages at $x_1$ and $x_2$ at time $t+\delta_n$ coalesce in the time interval $[t,t+\delta_n]$.

In Section~\ref{sec:mainproof}, we will define a `good' event $E$ in terms of these `tracer' random variables, and in Sections~\ref{sec:eventE1}-\ref{sec:eventE4}, we will show that the event $E$ occurs with high probability.
In Section~\ref{sec:mainproof}, we will show that conditional on the tracer random variables, if the event $E$ occurs, the locations of ancestral lineages relative to the front approximately have distribution $\pi$ (see Lemma~\ref{lem:fromxixj}), pairs of nearby lineages coalesce at approximately the rates given in~\eqref{eq:heurcoal} and~\eqref{eq:heurcoal2} (see Proposition~\ref{prop:coal}), and we are unlikely to see two pairs of lineages coalesce in a short time (see Proposition~\ref{prop:doublecoal}).
We can also prove bounds on the tail probabilities of lineages being far ahead of or far behind the front (see Propositions~\ref{prop:intip} and~\ref{prop:RlogN}).
These results combine to give a proof of Theorem~\ref{thm:main}.
Finally, in Section~\ref{sec:thmstatdist}, we use results from the earlier sections to complete the proof of Theorem~\ref{thm:statdist}.

\section{Proof of Theorem~\ref{thm:main}} \label{sec:mainproof}
Throughout Sections~\ref{sec:mainproof}-\ref{sec:thmstatdist}, we suppose $\alpha \in (0,1)$.
We let
\begin{equation} \label{eq:kappanu}
\kappa =\sqrt{\frac{2s_0}m} \qquad \text{and}\qquad \nu=\alpha \sqrt{\frac{ms_0}2}.
\end{equation}
For $k\in \N$, let $[k]=\{1,\ldots,k\}$.
For $0\le t_1 \le t_2$ and $x_1,x_2 \in \frac 1n \Z$, let
\begin{equation} \label{eq:qt1t2defn}
q^n_{t_1,t_2}(x_1,x_2)
=\frac 1 N |\{i\in [N] :\xi^n_{t_2}(x_2,i)=1, \,  \zeta^{n,t_2}_{t_2-t_1}(x_2,i)=x_1\}|,
\end{equation}
the proportion of individuals at $x_2$ at time $t_2$ which are type $A$ and are descended from an individual at $x_1$ at time $t_1$. Similarly, for
 $0\le t_1 \le t_2$ and $x_1\in \R$, $x_2 \in \frac 1n \Z$, let
\begin{align} \label{eq:qn+-defn}
q^{n,+}_{t_1,t_2}(x_1,x_2)
&=\frac 1 N |\{i\in [N]:\xi^n_{t_2}(x_2,i)=1, \, \zeta^{n,t_2}_{t_2-t_1}(x_2,i)\ge x_1\}| \notag \\
\text{and }\quad q^{n,-}_{t_1,t_2}(x_1,x_2)
&=\frac 1 N |\{i\in [N] :\xi^n_{t_2}(x_2,i)=1, \, \zeta^{n,t_2}_{t_2-t_1}(x_2,i)\le x_1\}|.
\end{align}
Fix a large constant $C>2^{13}\alpha^{-2}$, and let
\begin{equation} \label{eq:paramdefns}
\delta_n = \lfloor N^{1/2} n^2 \rfloor ^{-1} ,\;
\epsilon_n =\lfloor (\log N)^{-2}\delta_n ^{-1}\rfloor \delta_n,\;
\gamma_n =\lfloor (\log \log N)^4 \rfloor
\text{ and } d_n =\kappa^{-1} C\log \log N.
\end{equation}
For $t\ge 0$, $\ell \in \N$ and $x_1,\ldots ,x_\ell \in \frac 1n \Z$, let
\begin{equation} \label{eq:Cntdefn}
\begin{aligned}
&\mathcal C^n_t(x_1,x_2,\ldots, x_\ell )\\
&=\Big\{(i_1,\ldots, i_\ell) \in [N] ^\ell :
(x_j,i_j) \neq (x_{j'},i_{j'}) \, \forall j\neq j' \in [\ell], \; \xi^n_{t+\delta_n}(x_j,i_j)=1 \, \forall j\in [\ell ],\\
&\qquad \qquad \quad  
(\zeta^{n,t+\delta_n}_{\delta_n}(x_j,i_j),\theta^{n,t+\delta_n}_{\delta_n}(x_j,i_j))=(\zeta^{n,t+\delta_n}_{\delta_n}(x_1,j_1),\theta^{n,t+\delta_n}_{\delta_n}(x_1,j_1)) \, \forall j\in [\ell] \Big\},
\end{aligned}
\end{equation}
the set of $\ell$-tuples of distinct type $A$ individuals at $x_1,\ldots, x_\ell$ at time $t+\delta_n$ which all have a common ancestor at time $t$.
Recall the definition of $\mu^n_t$ in~\eqref{eq:muntdefn}.
For $y,\ell>0$, $0\le s \le t$ and $x \in \frac 1n \Z$, let
\begin{equation} \label{eq:rnystdefn}
r^{n,y,\ell}_{s,t}(x) = \frac 1N \big| \big\{ i \in [N] : \xi^n_t(x,i)=1, \; \zeta^{n,t}_{t'} (x,i) \ge \mu^n_{t-t'}+y \;\, \forall t' \in \ell \N_0 \cap [0,s] \big\} \big|,
\end{equation}
the proportion of individuals at $x$ at time $t$ which are type $A$ and whose ancestor at time $t-t'$ was to the right of $\mu^n_{t-t'}+y$ for each $t' \in \ell \N_0 \cap [0,s]$.

Fix $T_n \in [(\log N)^2,N^2]$ and 
define the sigma algebra
\begin{align*} 
\mathcal F &= \sigma\Big( (p^n_t(x))_{x\in \frac 1n \Z, t\le T_n},
(\xi^n_{T_n}(x,i))_{x\in \frac 1n \Z, i\in [N]},
(q^n_{T_n-t_1,T_n-t_2}(x_1,x_2))_{x_1,x_2\in \frac 1n \Z, t_1,t_2\in \delta_n \N_0, t_2\le t_1 \le T_n},\\
&\hspace{2.5cm} (\mathcal C^n_{T_n-t}(x_1,x_2))_{x_1,x_2 \in \frac 1n \Z, t\in \delta_n \N,\, t\le T_n},
(\mathcal C^n_{T_n-t}(x_1,x_2,x_3))_{x_1,x_2,x_3 \in \frac 1n \Z, t\in \delta_n \N, \, t\le T_n}\Big).
\end{align*}
We now define some `good' events, which occur with high probability, as we will show later.
Take $c_1,c_2>0$ small constants, and $t^*,K\in \N$ large constants, to be specified later.
The first event will allow us to show that the probability a lineage at $x_2$ at time $t+\gamma_n$ has an ancestor at $x_1$ at time $t$ is approximately $n^{-1} \pi(x_1-\mu^n_{t})$.
For $x_1, x_2 \in \tfrac 1n \Z$ and $0 \le t \le T_n$, define the event
$$
A^{(1)}_{t}(x_1,x_2)
=\left\{
\left| \frac{q^n_{t,t+\gamma_n}(x_1,x_2)}{p^n_{t+\gamma_n}(x_2)}-
n^{-1}\pi(x_1-\mu^n_{t})
\right| \le  n^{-1}(\log N)^{-3C}
\right\}.
$$
The next two events will allow us to control the probability that a lineage is far ahead of, or far behind, the front.
For $x_1,x_2\in \frac 1n \Z$ and $0 \le t \le T_n$, define the events
\begin{align*}
A^{(2)}_t(x_1,x_2) &=
\left\{
\frac{q^{n,+}_{t,t+t^*}(x_1,x_2)}{p^n_{t+t^*}(x_2)}\le c_1 e^{-(1+\frac 12 (1-\alpha))\kappa(x_1 -(x_2 -\nu  t^*)\vee (\mu^n_{t}+K)+2)}
\right\}\\
\text{and} \qquad
A^{(3)}_t(x_1,x_2) &=
\left\{
\frac{q^{n,-}_{t,t+t^*}(x_1,x_2)}{p^n_{t+t^*}(x_2)}\le c_1 e^{-\frac 12 \alpha \kappa ((x_2 -\nu  t^*)-x_1+1)}
\right\}.
\end{align*}
The next two events will give us a useful bound on the probability that a lineage is at the site $x$ at time $t$, conditional on its location at time $t+\epsilon_n$, and will allow us to show that lineages do not move more than distance $1$ in time $\epsilon_n$.
For $x \in \frac 1n \Z$ and $0 \le t \le T_n$, define the events
\begin{align*}
A^{(4)}_t(x) &=
\left\{
q^{n}_{t ,t+\epsilon_n}(x,x')\le n^{-1} \epsilon_n^{-1}p^n_{t+\epsilon_n}(x') \, \forall x' \in \tfrac 1n \Z
\right\}\\
\text{ and } \quad A^{(5)}_t(x) &=
\left\{
q^{n}_{t ,t+\epsilon_n}(x',x)\le \1_{|x-x'|\le 1} \, \forall x' \in \tfrac 1n \Z
\right\}.
\end{align*}
The next event will allow us to show that lineages do not move more than distance $(\log N)^{2/3}$ in time $t^*$.
For $x\in \frac 1n \Z$ and $0 \le t\le T_n$, define the event
$$
A^{(6)}_t(x) =
\left\{
q^{n}_{t ,t+k\delta_n}(x',x)\le \1_{|x-x'|\le (\log N)^{2/3}} \; \forall k\in [t^* \delta_n^{-1}],
x'\in \tfrac 1n \Z 
\right\}.
$$
The next four events will give us estimates on the probability that a pair of lineages at the same site or neighbouring sites coalesce in time $\delta_n$, and bounds on the probabilities that a pair of lineages further apart coalesce, or a set of three lineages coalesce.
For $x\in \frac 1n \Z$ and $0 \le t \le T_n $, define the events
\begin{align*}
B^{(1)}_t(x) 
&= \left\{ \frac{\big| \, |\mathcal C^n_t (x,x)|-n^2 N \delta_n p^n_t(x)\big|}{n^2 N \delta_n p^n_t(x)}\le 2n^{-1/5}
\right\}, \\
 B^{(2)}_t(x) 
&= \left\{ \frac{\big|\, |\mathcal C^n_t (x,x+n^{-1})|-\frac 12 mn^2 N \delta_n (p^n_t(x)+p^n_t(x+n^{-1}))\big|}{\frac 12 mn^2 N \delta_n (p^n_t(x)+p^n_t(x+n^{-1}))}\le 2n^{-1/5}
\right\},\\
B^{(3)}_t(x) 
&= \left\{ \frac{|\mathcal C^n_t (x,x')|}{n^2 N \delta_n p^n_t(x)} \le n^{-1/5}\1_{|x-x'|< Kn^{-1}}
\; \forall x' \in \tfrac 1n \Z \text{ with } |x'-x|>n^{-1}
\right\},\\
\text{and }\quad 
B^{(4)}_t(x) 
&= \left\{  \frac{|\mathcal C^n_t (x,y,y')|}{n^2 N \delta_n p^n_t(x)} \le n^{-1/5}\1_{|y-x|\vee |y'-x|< Kn^{-1}} \; \forall y,y' \in \tfrac 1n \Z
\right\}.
\end{align*}
Fix $c_0>0$ sufficiently small that $(1+\frac 14 (1-\alpha))(1-2c_0)>1$.
Let
\begin{equation} \label{eq:Dn+-defn}
D^+_n =(1/2-c_0)\kappa^{-1} \log (N/n)\quad \text{ and }\quad D^-_n=- 26 \kappa^{-1} \alpha^{-1} \log N
\end{equation}
and for $t\ge 0$ and $\epsilon \in (0,1)$, recalling~\eqref{eq:paramdefns}, let
\begin{equation} \label{eq:Intdefn}
I^n_t =\tfrac 1n \Z \cap [\mu^n_t-N^4,\mu^n_t+D_n^+], \;
I^{n,\epsilon}_t =\tfrac 1n \Z \cap [\mu^n_t+D_n^-,\mu^n_t+(1-\epsilon) D_n^+]
\; \text{and }
i^n_t =\tfrac 1n \Z \cap [\mu^n_t-d_n,\mu^n_t+d_n].
\end{equation}
We will show that with high probability, a pair of lineages are never both more than $D_n^+$ ahead of the front before they coalesce, and neither lineage is ever more than $|D_n^-|$ behind the front.

We now define an event which says that $(p^n_t)_{t\in [0,N^2]}$ is close to a moving front with shape $g$ and wavespeed approximately $\nu$. Let
\begin{equation} \label{eq:eventE1}
\begin{aligned}
E_1=E_1(c_2)&=
\Big\{ \sup_{x\in \frac 1n \Z, t\in [\log N,N^2]} |p^n_t(x)-g(x-\mu^n_t)|\le e^{-(\log N)^{c_2}}\Big\}\\
&\qquad \cap \big\{ p^n_t(x) \in [\tfrac 15 g(x-\mu^n_t), 5g(x-\mu^n_t)] \; \forall t \in [\tfrac 12 (\log N)^2,N^2], x\le \mu^n_t+D_n^++2 \big\}\\
&\qquad \cap \big\{ p^n_t(x) \le 5g(D^+_n) \; \forall t \in [\tfrac 12 (\log N)^2,N^2], x\ge \mu_t^n +D^+_n \big\}\\
& \qquad  \cap \big\{ |\mu^n_{t+s}-\mu^n_t -\nu  s |\le e^{-(\log N)^{c_2}} \; \forall  t\in  [\log N, N^2],s\in [0,1\wedge (N^2-t)]\big\}\\
&\qquad \cap \big\{ |\mu^n_{\log N  } |\le 2\nu  \log N \big\}.
\end{aligned}
\end{equation}
Let $T_n^-=T_n-(\log N)^2$ and define the event
\begin{align} \label{eq:eventE2}
E_2 &=E_2(c_1,t^*,K) \notag \\
&= E'_2 \cap \bigcap_{t\in \delta_n\N_0 \cap [0,T_n^-]}
\bigg( \bigcap_{x_1\in i^n_{T_n-t-\gamma_n}, \, x_2 \in i^n_{T_n-t}}A^{(1)}_{T_n-t-\gamma_n}(x_1,x_2)
\cap \bigcap_{x\in I^n_{T_n-t-\epsilon_n} }A^{(4)}_{T_n-t-\epsilon_n}(x) \bigg) ,
\end{align}
where
\begin{equation} \label{eq:eventE'2}
\begin{aligned}
E'_2 =E'_2(c_1,t^*,K)&=
 \bigcap_{t\in \delta_n\N_0 \cap [0,T_n^-]}
\bigcap_{x_1\in I^n_{T_n-t-t^*}, \, x_2 \in I^n_{T_n-t}, \, x_1-\mu^n_{T_n-t-t^*}\ge K} A^{(2)}_{T_n-t-t^*}(x_1,x_2)\\
&\quad \cap 
 \bigcap_{t\in \delta_n\N_0 \cap [0,T_n^-]}
\bigcap_{x_1\in I^n_{T_n-t-t^*}, \, x_2 \in I^n_{T_n-t}, \, x_1-\mu^n_{T_n-t-t^*}\le -K} A^{(3)}_{T_n-t-t^*}(x_1,x_2)\\
&\quad  \cap \bigcap_{t\in \delta_n\N_0 \cap [0,T_n^- +t^*]}
 \bigcap_{x\in \frac 1n \Z\cap [-N^5,N^5] }(A^{(5)}_{T_n-t-\epsilon_n}(x)\cap A^{(6)}_{T_n-t-\delta_n}(x)).
 \end{aligned}
\end{equation}
Define the event
\begin{align} \label{eq:eventE3}
E_3=E_3(K)
&= \bigcap_{t\in \delta_n\N_0 \cap [0,T_n^-]}
\bigcap_{x\in I^n_{T_n-t}}
\bigcap_{j=1}^4 B^{(j)}_{T_n-t-\delta_n}(x) .
\end{align}
Finally, we define an event which says that with high probability, no lineages stay distance $K$ ahead of the front for time $K\log N$. Let
$$
E_4 = E_4(t^*, K)=\bigcap_{t\in \delta_n \N_0 \cap [0,T_n^-]} \left\{\p{r^{n,K,t^*}_{  K\log N , T_n-t}(x)=0 \; \forall x\in \tfrac 1n \Z \Big|\mathcal F}\ge 1-\left( \frac n N \right)^2 \right\},
$$
and let $E=\cap_{j=1}^4 E_j.$
The following result will be proved in Sections~\ref{sec:eventE1}-\ref{sec:eventE4}.
\begin{prop} \label{prop:eventE}
Suppose for some $a_2>3$, $N\ge n^{a_2}$ for $n$ sufficiently large.
Take $c_1>0$.
There exist $t^*,K \in \N$ (with $K>104 \kappa^{-1} \alpha^{-1} t^*$) and $b_1,c_2>0$ such that for $b_2>0$, if condition~\eqref{eq:conditionA} holds, for $n$ sufficiently large,
$$
\p{E^c}\le \frac n N .
$$
\end{prop}
From now on in this section, we will take $c_1\in (0,1)$ sufficiently small that letting $\lambda = \frac 14 (1-\alpha)$,
\begin{equation} \label{eq:cchoice}
\begin{aligned} 
c_1 ((e^{\lambda \kappa}-1)^{-1}e^{\lambda \kappa }+e^{-(1+\lambda)\kappa}(1-e^{-(1+\lambda )\kappa})^{-1})^2 +e^{-2(1+\lambda )\kappa} &< 1,  \\
c_1 (e^{\lambda \kappa}-1)^{-1}e^{\lambda \kappa} +e^{-(1+\lambda)\kappa} &< 1,\\
c_1(1+  e^{3\alpha \kappa/4}(e^{\alpha \kappa/4}-1)^{-1})+e^{-\alpha \kappa /4} &<1, \\
\text{ and } \qquad \qquad \qquad   e^{-\alpha \kappa /4}+c_1 (1-e^{-\alpha \kappa /4})^{-1}&<e^{-\alpha \kappa /5},
\end{aligned}
\end{equation}
and then take $t^*$, $K$, $b_1$, $b_2$ and $c_2$ as in Proposition~\ref{prop:eventE}.

Take $K_0<\infty$, $k_0 \in \N$ and $(X_1,J_1)$, $(X_2,J_2), \ldots , (X_{k_0},J_{k_0}) \in \frac 1n \Z \times [N]$ measurable with respect to $\sigma((\xi^n_{T_n}(x,i))_{x\in \frac 1n \Z, i\in [N]})$ and distinct, with 
$(X_i,J_i) \in G_{K_0,T_n}$ $\forall i \in [k_0]$. 
For $t\in [0, T_n]$ and $i\in [k_0]$, let
\begin{equation} \label{eq:zetadefns}
\zeta^{n,i}_t = \zeta^{n,T_n}_t(X_i,J_i)
\quad \text{ and } \quad 
\tilde \zeta^{n,i}_t = \zeta^{n,T_n}_t(X_i,J_i)-\mu^n_{T_n-t},
\end{equation}
the location of the $i^{\text{th}}$ ancestral lineage at time $T_n-t$, and its location relative to the front.
For $i, j \in [k_0]$, let
$$
\tau^n_{i,j} = \inf\{t\ge 0: (\zeta^{n,T_n}_t(X_i,J_i),\theta^{n,T_n}_t(X_i,J_i))=(\zeta^{n,T_n}_t(X_j,J_j),\theta^{n,T_n}_t(X_j,J_j))\},
$$
the time at which the $i^{\text{th}}$ and $j^{\text{th}}$ lineages coalesce.
For $t \in [0,T_n]$, define the sigma algebra
$$
\mathcal F_t =\sigma \big(\mathcal F, \sigma((\zeta^{n,j}_s)_{s\le t, j\in [k_0]},(\1_{\tau_{i,j}^n \le s})_{s\le t, i, j \in [k_0]})\big).
$$
Then 
$((\zeta^{n,j}_{k\delta_n})_{j\in [k_0]},(\1_{\tau^n_{i,j}\le k\delta_n})_{i,j\in [k_0]})_{k\in \N_0, k\le T_n \delta_n^{-1}}$
is a strong Markov process with respect to the filtration $(\mathcal F_{k\delta_n})_{k\in \N_0, k\le T_n \delta_n^{-1}}$.

For $k\in \N_0$, let $t_k=k \lfloor (\log N)^C \rfloor$.
For $i,j \in [k_0]$, let
\begin{equation} \label{eq:tildetaudefn}
\tilde \tau^{n}_{i,j}
=
\begin{cases}
 \tau^{n}_{i,j} &\text{if }\tau^{n}_{i,j}\notin (t_k, t_k+2K \log N] \, \forall k \in \N_0
  \text{ and }|\tilde \zeta^{n,i}_{\lfloor \tau^{n}_{i,j} \delta_n^{-1} \rfloor \delta_n}|\wedge |\tilde \zeta^{n,j}_{\lfloor \tau^{n}_{i,j} \delta_n^{-1} \rfloor \delta_n}| \le \tfrac{1}{64} \alpha d_n, \\
 T_n  &\text{otherwise,}
\end{cases}
\end{equation}
i.e. $\tilde \tau^{n}_{i,j}$ only counts coalescence which happens fairly near the front and not too soon after $t_k$ (backwards in time from time $T_n$) for any $k$.
Let 
\begin{equation} \label{eq:betadefn}
\beta_n =(1+2m) \frac n N t_1 \frac{\int_{-\infty}^\infty g(y)^3 e^{2\alpha \kappa y} dy}{\left(\int_{-\infty}^\infty g(y)^2 e^{\alpha \kappa y} dy\right)^2}
=(1+2m) \frac n N t_1 \int_{-\infty}^\infty \pi(y)^2 g(y)^{-1}dy.
\end{equation}
Along with Proposition~\ref{prop:eventE}, the following three propositions are the main intermediate results in the proof of Theorem~\ref{thm:main}, and will be proved in Section~\ref{subsec:mainprops}. 
The first proposition says that if a pair of lineages $i$ and $j$ have not coalesced by time $t_k$, and one of them is not too far from the front, then the probability that $\tilde \tau^n_{i,j}\le t_{k+1}$ is approximately $\beta_n$.
\begin{prop} \label{prop:tauk}
Suppose for some $a_2>3$, $N\ge n^{a_2}$ for $n$ sufficiently large.
On
the event $E$, for $i,j \in [k_0]$, $\epsilon \in (0,1)$ and
$k\in \N_0$ with $t_{k+1}\le T_n^-$, if
 $\zeta^{n,i}_{t_k} \wedge \zeta^{n,j}_{t_k} \in I^{n,\epsilon}_{T_n-t_k}$ and
 $\tau^{n}_{i,j}>t_k$
 then
\begin{align*}
\p{\tilde \tau^{n}_{i,j} \in (t_k, t_{k+1}]  \Big| \mathcal F_{t_k}}
&=\beta_n
 (1+\mathcal O((\log N)^{-2})).
\end{align*}
\end{prop}
The second proposition says that two pairs of lineages are unlikely to coalesce in the same time interval $(t_k,t_{k+1}]$.
\begin{prop} \label{prop:doublecoaltk}
Suppose for some $a_2>3$, $N\ge n^{a_2}$ for $n$ sufficiently large.
For $\epsilon \in (0,1)$, there exists $\epsilon '>0$ such that on the event $E$, 
for $k\in \N_0$ with $t_{k+1}\le T_n^-$ the following holds.
For $i,j_1,j_2 \in [k_0]$ distinct,
if $ \zeta^{n,\ell }_{t_k} \wedge  \zeta^{n,\ell '}_{t_k} \in I^{n,\epsilon}_{T_n-t_k}$
and $\tau^{n}_{\ell, \ell '}>t_k$ $\forall \ell \neq \ell ' \in \{i,j_1,j_2\}$ then
\begin{equation} \label{eq:propdoubletk1}
\p{\tilde \tau^{n}_{i,j_1}, \tilde \tau^{n}_{i,j_2} \in (t_k,t_{k+1}] \Big| \mathcal F_{t_k} }
=\mathcal O(n^{1-\epsilon '} N^{-1}).
\end{equation}
For $i_1,i_2,j_1,j_2 \in [k_0]$ distinct,
if $\zeta^{n,\ell}_{t_k} \wedge \zeta^{n,\ell '}_{t_k}\in I^{n,\epsilon}_{T_n-t_k}$
and $\tau^{n}_{\ell, \ell '}>t_k$ $\forall \ell \neq \ell ' \in \{i_1,i_2,j_1,j_2\}$ then
\begin{equation} \label{eq:propdoubletk2}
\p{\tilde \tau^{n}_{i_1,j_1}, \tilde \tau^{n}_{i_2,j_2} \in (t_k,t_{k+1}] \Big| \mathcal F_{t_k} }
=\mathcal O(n^{1-\epsilon '} N^{-1}).
\end{equation}
\end{prop}
The last proposition says that for a pair of lineages $i$ and $j$, with high probability $\tilde \tau^n_{i,j}= \tau^n_{i,j}$, and at least one of the lineages is fairly near the front until they have coalesced.
\begin{prop} \label{prop:tautautilde}
Suppose $T_n\ge N$ and, for some $a_2>3$, $N\ge n^{a_2}$ for $n$ sufficiently large.
For $\epsilon \in (0,1)$ sufficiently small, for $n$ sufficiently large, on the event $E$, for $i\neq j \in [k_0]$,
$$\p{\tau^{n}_{i,j} \neq \tilde \tau^{n}_{i,j} \Big| \mathcal F_0 }  \le (\log N)^{-2}$$
and
$$
\p{\exists t\in \delta_n \N_0 \cap [0,  N n^{-1} \log N]
: \zeta^{n,i}_{t} \wedge \zeta^{n,j}_{t}\notin I^{n,\epsilon}_{T_n-t}, \; \tau^{n}_{i,j}>t \Big| \mathcal F_0}\le (\log N)^{-2}.
$$
\end{prop}

Before proving Propositions~\ref{prop:tauk}-\ref{prop:tautautilde}, we show how they can be combined with Proposition~\ref{prop:eventE} to prove Theorem~\ref{thm:main}.


\begin{proof}[Proof of Theorem~\ref{thm:main}]
Let $(B_{i,j,k})_{i<j \in [k_0], k\in \N_0}$ be i.i.d.~Bernoulli random variables with $$\p{B_{i,j,k}=1}=\beta_n,$$ and let
$B_{j,i,k}=B_{i,j,k}$ for $i<j \in [k_0]$.
For $k\in \N_0$, let
$$
P_k = \{i \in [k_0]: \tau^{n}_{i,j}>t_k \; \forall j \in [i-1]\},
$$
the set of lineages at time $T_n-t_k$ which have not coalesced with a lineage of lower index.
Take $\epsilon>0$ sufficiently small that Proposition~\ref{prop:tautautilde} holds, and take $\epsilon'>0$ as in Proposition~\ref{prop:doublecoaltk}.
Define the event
\begin{align*}
A_k &= \left\{\zeta^{n,i}_{t_k} \wedge \zeta^{n,j}_{t_k} \in I^{n,\epsilon}_{T_n-t_k} \; \forall i\neq j \in P_k\right\}.
\end{align*}
Take $k\in \N_0$ with $t_{k+1}\le T_n^-$, and
suppose the event $E\cap A_k$ occurs.
Then by Proposition~\ref{prop:tauk}, for each pair of lineages $i\neq j \in P_k$,
$$
\p{\tilde \tau^{n}_{i,j} \in (t_k, t_{k+1}] \Big| \mathcal F_{t_k}}= \beta_n  (1+\mathcal O((\log N)^{-2})),
$$
and by Proposition~\ref{prop:doublecoaltk},
$$
\p{|\{(i,j):i<j \in P_k \text{ and } \tilde \tau^{n}_{i,j}\in (t_k,t_{k+1}]\}|\ge 2 \Big| \mathcal F_{t_k}}
=\mathcal O(n^{1-\epsilon '}N^{-1})=o(\beta_n (\log N)^{-2})
$$
by the definition of $\beta_n$ in~\eqref{eq:betadefn}.
Therefore, conditional on $\mathcal F_{t_k}$, we can couple $(\tilde \tau^{n}_{i,j})_{i,j \in P_k}$ and $(B_{i,j,k})_{i<j \in [k_0]}$ in such a way that if $E \cap A_k$ occurs then
\begin{equation} \label{eq:mainproof1}
\p{\exists i\neq j \in P_k :B_{i,j,k}\ne \1_{\tilde \tau^{n}_{i,j} \in (t_k,t_{k+1}]} \Big| \mathcal F_{t_k}} =\mathcal O(\beta_n (\log N)^{-2}).
\end{equation}
Note that for $n$ sufficiently large, if the event $E$ occurs,
then by Proposition~\ref{prop:tautautilde}, 
\begin{align} \label{eq:mainproof2}
\p{\bigcup_{k=0}^{\lfloor Nn^{-1}t_1^{-1} \log N \rfloor } (A_k)^c \Bigg| \mathcal F_0}
&\le {{k_0} \choose 2} (\log N)^{-2}.
\end{align}
Now
define $(\sigma^n_{i,j})_{i, j \in [k_0]}$ iteratively as follows.
Let $\sigma^n_{i,i}=0$ $\forall i\in [k_0]$.
For $k\in \N_0$ and $i\in [k_0]$, let $\pi_k(i)=\min\{i' \in [k_0]:\sigma^n_{i',i}\le t_k\}$.
Then for each pair $i,j\in [k_0]$ with $\pi_k(i)\neq \pi_k(j)$,
 set $\sigma^n_{i,j}=t_{k+1}$ if $B_{\pi_k(i),\pi_k(j),k}=1$; otherwise $\sigma^n_{i,j}>t_{k+1}$.

Suppose $\tilde \tau^n_{i,j}=\tau^n_{i,j}$ $\forall i,j \in [k_0]$.
For some $k\in \N_0$, suppose
$\{(i,j):\tau^n_{i,j}>t_k\}=\{(i,j):\sigma^n_{i,j}>t_k\}$ and $B_{i,j,k}=\1_{\tilde \tau^n_{i,j}\in (t_k,t_{k+1}]}$ $\forall i \neq j\in P_k$.
Then for $i,j\in [k_0]$ with $\tau^n_{i,j}>t_k$ we have
that $\tau^n_{\pi_k(i),i}\le t_k$ and $\tau^n_{\pi_k(j),j}\le t_k$, and so
$$
\1_{\tau^n_{i,j}\in (t_k,t_{k+1}]}=\1_{\tilde \tau^n_{i,j}\in (t_k,t_{k+1}]}=\1_{\tilde \tau^n_{\pi_k(i),\pi_k(j)}\in (t_k,t_{k+1}]}
=B_{\pi_k(i),\pi_k(j),k}=\1_{\sigma^n_{i,j}=t_{k+1}},
$$
since $\pi_k(i),\pi_k(j)\in P_k$.
In particular, $\{(i,j):\tau^n_{i,j}>t_{k+1}\}=\{(i,j):\sigma^n_{i,j}>t_{k+1}\}$.
By induction, it follows that for $k^*\in \N$, if for each $k\in \{0\}\cup [k^*]$ we have $B_{i,j,k}=\1_{\tilde \tau^n_{i,j}\in (t_k,t_{k+1}]}$ $\forall i\neq j\in P_k$ then
$$
\{(i,j):\tau^n_{i,j}\in (t_k,t_{k+1}]\}=\{(i,j):\sigma^n_{i,j}=t_{k+1}\} \; \forall k\in \{0\}\cup [k^*].
$$
Therefore, if the event $E$ occurs, then by a union bound,
\begin{align*}
&\p{\exists i, j \in [k_0]: |\tau^{n}_{i,j}-\sigma^n_{i,j}|\ge (\log N)^C \Big|\mathcal F_0}\\
&\le \p{\exists i,  j \in [k_0]: \tau^{n}_{i,j} \neq \tilde \tau^{n}_{i,j} \Big|\mathcal F_0}\\
&\quad +\sum_{k=0}^{\lfloor Nn^{-1}t_1^{-1} \log N \rfloor }\p{\{\exists i\neq j \in P_k : B_{i,j,k}\neq \1_{\tilde \tau^{n}_{i,j}\in (t_k,t_{k+1}]} \} \cap A_k \Big| \mathcal F_0}\\
&\quad  + \p{\bigcup_{k=0}^{\lfloor Nn^{-1}t_1^{-1} \log N \rfloor } (A_k)^c \Bigg| \mathcal F_0}
+\p{\exists i,j \in [k_0] : \sigma^n_{i,j}> t_{\lfloor Nn^{-1}t_1^{-1} \log N \rfloor } \Big| \mathcal F_0}\\
&\le 2{{k_0}\choose 2} (\log N)^{-2}
+\sum_{k=0}^{\lfloor Nn^{-1}t_1^{-1} \log N \rfloor }\mathcal O(\beta_n (\log N)^{-2})
+{{k_0} \choose 2} (1-\beta_n)^{\lfloor Nn^{-1}t_1^{-1} \log N \rfloor }\\
&=\mathcal O((\log N)^{-1}),
\end{align*}
where the second inequality follows for $n$ sufficiently large by Proposition~\ref{prop:tautautilde},~\eqref{eq:mainproof1} and~\eqref{eq:mainproof2}, and the last inequality follows by the definition of $\beta_n$ in~\eqref{eq:betadefn}.
The result follows easily by Proposition~\ref{prop:eventE} and then by a coupling between $(\beta_n t_1^{-1}\sigma^n_{i,j})_{i,j\in [k_0]}$ and $(\tau_{i,j})_{i,j\in [k_0]}$.
\end{proof}

\subsection{Proof of Propositions~\ref{prop:tauk},~\ref{prop:doublecoaltk} and~\ref{prop:tautautilde}} \label{subsec:mainprops}
The next five results will be used in the proofs of Propositions~\ref{prop:tauk},~\ref{prop:doublecoaltk} and~\ref{prop:tautautilde}.
The first three results will also be used in Section~\ref{sec:thmstatdist} in the proof of Theorem~\ref{thm:statdist}.
The first result says that a pair of lineages are unlikely to be far ahead of the front, and will be proved in Section~\ref{subsec:tipbulkproofs}.

\begin{prop} \label{prop:intip}
Suppose for some $a_1>1$, $N\ge n^{a_1}$ for $n$ sufficiently large.
For $n$ sufficiently large, on the event $E_1\cap E'_2 \cap E_4$, for
$i,j \in [k_0]$, $s\le t \in \delta_n \N_0 \cap 
[0,T_n^-]$ and $\ell_1, \ell_2 \in \N \cap [K,D^+_n]$,
the following holds.
If $t-s \ge K \log N$ then
\begin{align}
\p{\tilde \zeta^{n,i}_t \ge \ell_1, \tilde \zeta^{n,j}_t \ge \ell_2, \tau^n_{i,j}> t \Big| \mathcal F_{s} }
&\le (\log N)^7
e^{-(1+\frac 14 (1-\alpha))\kappa(\ell_1+\ell_2)} \label{eq:propintipstat1} \\
\text{ and }\quad \p{\tilde \zeta^{n,i}_t \ge \ell_1 \Big| \mathcal F_{s} }
&\le (\log N)^3
e^{-(1+\frac 14 (1-\alpha))\kappa\ell_1}. \label{eq:propintipstat2}
\end{align}
If instead $t-s \in t^* \N_0 \cap [0,K \log N)$ then
\begin{align} 
\p{\tilde \zeta^{n,i}_t \ge \ell_1, \tilde \zeta^{n,j}_t \ge \ell_2, \tau^n_{i,j}> t \Big| \mathcal F_{s} }
&\le (\log N)^4 e^{(1+\frac 14 (1-\alpha))\kappa(\tilde \zeta^{n,i}_{s}\vee 0 -\ell_1 +\tilde \zeta^{n,j}_{s}\vee 0 - \ell_2)}  \label{eq:propintipstat*} \\
\text{ and }\quad \p{\tilde \zeta^{n,i}_t \ge \ell_1 \Big| \mathcal F_{s} }
&\le (\log N)^2
 e^{(1+\frac 14 (1-\alpha))\kappa(\tilde \zeta^{n,i}_{s}\vee 0 -\ell_1)}. \label{eq:propintipstat3}
\end{align}
\end{prop}
The next result says that lineages are unlikely to be far behind the front, and will be proved in Section~\ref{subsec:behindfront}.
\begin{prop} \label{prop:RlogN}
Suppose for some $a_1>1$, $N\ge n^{a_1}$ for $n$ sufficiently large.
For $n$ sufficiently large,
on the event $E_1\cap E'_2$ the following holds. For $i\in [k_0]$,
\begin{equation} \label{eq:propRlogN1}
\p{\exists t\in \delta_n \N_0 \cap[0, T_n^-]  : \tilde \zeta^{n,i}_{t}\le  D_n^- \Big|\mathcal F_0 } 
\le N^{-1}.
\end{equation}
For $i\in [k_0]$ and $s\le t \in \delta_n \N_0 \cap [0,T_n^-]$ with $t-s \ge K \log N$,
if $\tilde \zeta^{n,i}_{s } \ge D_n^-$ then
\begin{align} \label{eq:propRlogN2}
&\p{\tilde \zeta^{n,i}_t \le -d_n  \Big| \mathcal F_{s}} \le (\log N)^{2-\frac 18 \alpha C}
\quad \text{ and }\quad
\p{\tilde \zeta^{n,i}_t \le -\tfrac 1 {64} \alpha d_n +2 \Big| \mathcal F_{s}} \le (\log N)^{2-2^{-9} \alpha^2 C}.
\end{align}
For $i\in [k_0]$ and $t\in t^*\N_0\cap [0,T^-_n]$, 
\begin{equation} \label{eq:propRlogN3}
\p{\tilde \zeta^{n,i}_t \le -d_n \Big| \mathcal F_{0}} \le (\log N)^{-\frac 18 \alpha C}.
\end{equation}
\end{prop}
The next lemma gives estimates on the probability that a pair of lineages are at a particular pair of sites, and gives bounds on the increments of $\zeta^{n,i}$.
\begin{lemma} \label{lem:fromxixj}
Suppose for some $a_1>1$, $N\ge n^{a_1}$ for $n$ sufficiently large.
For $n$ sufficiently large, the following holds.
Suppose the event $E$ occurs.
Take $t\in \delta_n \N_0 \cap [0,T_n^-]$, $i,j \in [k_0]$ and $x_i,x_j \in \frac 1n \Z$.
If $x_i,x_j \in i^n_{T_n-t-\gamma_n}$, $\zeta^{n,i}_t, \zeta^{n,j}_t \in i^n_{T_n-t}$ and $\tau^n_{i,j}>t$ then
\begin{equation} \label{eq:lemfromxixj1}
\p{\zeta^{n,i}_{t+\gamma_n}=x_i, \zeta^{n,j}_{t+\gamma_n}=x_j \Big| \mathcal F_t}
=n^{-2} \pi(x_i -\mu^n_{T_n-t-\gamma_n}) \pi(x_j -\mu^n_{T_n-t-\gamma_n}) (1+\mathcal O((\log N)^{-C})).
\end{equation}
If $x_i,x_j \in I^n_{T_n -t-\epsilon_n}$ and $\tau^n_{i,j}>t$ then
\begin{equation} \label{eq:lemfromxixj2}
\p{\zeta^{n,i}_{t+\epsilon_n}=x_i, \zeta^{n,j}_{t+\epsilon_n}=x_j \Big| \mathcal F_t}
\le 2n^{-2} \epsilon_n^{-2}.
\end{equation}
Suppose instead the event $E_1 \cap E'_2$ occurs. For $t\in \delta_n \N_0 \cap [0,T_n^-]$, $i\in [k_0]$ and $t'\in \delta_n \N_0 \cap [t,t+t^*]$, 
\begin{equation} \label{eq:lemfromxixj3}
|\zeta^{n,i}_t- \zeta^{n,i}_{t'}|\le (\log N)^{2/3}, \quad 
|\zeta^{n,i}_t| \vee |\tilde \zeta^{n,i}_t|\le N^3 \quad
\text{ and }\quad 
|\zeta^{n,i}_t -\zeta^{n,i}_{t+\epsilon_n}|\le 1.
\end{equation}
\end{lemma}
\begin{proof}
Suppose the event $E$ occurs and $\tau^n_{i,j}>t$.
Then for $s\in \delta_n \N_0 \cap [0,T_n-t]$,
\begin{align} \label{eq:lemfromxpf1}
\p{\zeta^{n,i}_{t+s}=x_i, \zeta^{n,j}_{t+s}=x_j \Big| \mathcal F_t}
=\frac{q^n_{T_n-t-s, T_n-t}(x_i,\zeta^{n,i}_t)}{p^n_{T_n-t}(\zeta^{n,i}_t)}
\frac{q^n_{T_n-t-s, T_n-t}(x_j,\zeta^{n,j}_t) - N^{-1} \1_{\zeta^{n,i}_t =\zeta^{n,j}_t, \, x_i=x_j}}{p^n_{T_n-t}(\zeta^{n,j}_t)-N^{-1}\1_{\zeta^{n,i}_t=\zeta^{n,j}_t}}.
\end{align}
If $x_i,x_j \in i^n_{T_n-t-\gamma_n}$ and $\zeta^{n,i}_t, \zeta^{n,j}_t \in i^n_{T_n-t}$ then
by the definition of the event $E_2$ in~\eqref{eq:eventE2},
 the events $A^{(1)}_{T_n-t-\gamma_n}(x_i, \zeta^{n,i}_t)$ and $A^{(1)}_{T_n-t-\gamma_n}(x_j, \zeta^{n,j}_t)$ occur.
Moreover, $p^n_{T_n-t}(\zeta^{n,j}_t) \ge \frac 15 g(d_n)\ge \frac 1 {10} (\log N)^{-C}$
by the definition of the event $E_1$ in~\eqref{eq:eventE1} and the definition of $d_n$ in~\eqref{eq:paramdefns}, and so
\begin{align*}
&\p{\zeta^{n,i}_{t+\gamma_n}=x_i, \zeta^{n,j}_{t+\gamma_n}=x_j \Big| \mathcal F_t}\\
&=(n^{-1} \pi(x_i-\mu^n_{T_n-t-\gamma_n})+\mathcal O(n^{-1} (\log N)^{-3C}))\cdot (1+\mathcal O(N^{-1} (\log N)^C))\\
&\quad \cdot (n^{-1} \pi(x_j-\mu^n_{T_n-t-\gamma_n})+\mathcal O(n^{-1} (\log N)^{-3C})+\mathcal O(N^{-1} (\log N)^C)).
\end{align*}
Since  $\pi(x_i -\mu^n_{T_n-t-\gamma_n})^{-1} \vee \pi(x_j -\mu^n_{T_n-t-\gamma_n})^{-1} \le \pi(d_n)^{-1}\vee \pi(-d_n)^{-1}= \mathcal O((\log N)^{2C})$,~\eqref{eq:lemfromxixj1} follows.

If $x_i,x_j \in I^n_{T_n-t-\epsilon_n}$ then by the definition of the event $E'_2$ in~\eqref{eq:eventE'2}, the events $A^{(4)}_{T_n-t-\epsilon_n}(x_i)$ and $A^{(4)}_{T_n-t-\epsilon_n}(x_j)$ occur.
If $\zeta^{n,i}_t=\zeta^{n,j}_t$ then
$p^n_{T_n-t}(\zeta^{n,j}_t)-N^{-1} \ge \frac 12 p^n_{T_n-t}(\zeta^{n,j}_t)$, and so~\eqref{eq:lemfromxixj2} follows from~\eqref{eq:lemfromxpf1}. 

Suppose now that the event $E_1\cap E_2'$ occurs, and suppose for some $s\in \delta_n \N_0 \cap [0,T_n^-]$ that $|\zeta^{n,i}_s|\le N^3$.
Then the events $A^{(5)}_{T_n-s-\epsilon_n}(\zeta^{n,i}_s)$ and $\cap_{k\in [t^*\delta_n^{-1}]}A^{(6)}_{T_n-s-k\delta_n}(\zeta^{n,i}_s)$
occur, and so $|\zeta^{n,i}_{s+\epsilon_n}-\zeta^{n,i}_s|\le 1$ and $|\zeta^{n,i}_s-\zeta^{n,i}_{s'}|\le (\log N)^{2/3}$ $\forall s' \in \delta_n \N_0  \cap [s,s+t^*]$.
Since $|\tilde \zeta^{n,i}_0|\le K_0$ and $|\zeta^{n,i}_0|\le K_0 +|\mu^n_{T_n}|\le 2\nu N^2$ for $n$ sufficiently large, it follows by an inductive argument that $|\zeta^{n,i}_t|\vee |\tilde \zeta^{n,i}_t|\le N^3$ $\forall t\in \delta_n \N_0 \cap [0,T_n^-]$, which completes the proof.
\end{proof}
From now on in Section~\ref{subsec:mainprops}, we will assume for some $a_2>3$, $N\ge n^{a_2}$ for $n$ sufficiently large.
We will also need an estimate for the probability that a pair of lineages coalesce in a time interval of length $\delta_n$.
\begin{prop} \label{prop:coal}
Suppose the event $E$ occurs. Take $t\in \delta_n \N_0 \cap [0,T_n^-]$ and $x,y \in \frac 1n \Z$ with $|x-y|>n^{-1}$ and $x\in I^n_{T_n-t}$.
If $\zeta^{n,i}_t=x=\zeta^{n,j}_t$ and $\tau^n_{i,j}>t$ then
\begin{equation*}
\p{\tau^n_{i,j}\in (t,t+\delta_n] \big| \mathcal F_t } =
\begin{cases}
n^2 N^{-1} \delta_n g(x-\mu^n_{T_n-t})^{-1}\big(1+\mathcal O((\log N)^{-C})\big) \quad &\text{if }x\in i^n_{T_n-t},\\
\mathcal O(n^2 N^{-1} \delta_n g(x-\mu^n_{T_n-t})^{-1}) &\text{otherwise.}
\end{cases}
\end{equation*}
If instead $\zeta^{n,i}_t=x$, $\zeta^{n,j}_t=x+n^{-1}$ and $\tau^n_{i,j}>t$ then
\begin{equation*}
\p{\tau^n_{i,j}\in (t,t+\delta_n] \big| \mathcal F_t } 
=
\begin{cases}
m n^2 N^{-1} \delta_n g(x-\mu^n_{T_n-t})^{-1}\big(1+\mathcal O((\log N)^{-C})\big) \quad &\text{if }x\in i^n_{T_n-t},\\
\mathcal O(n^2 N^{-1} \delta_n g(x-\mu^n_{T_n-t})^{-1}) &\text{otherwise.}
\end{cases}
\end{equation*}
If instead $\zeta^{n,i}_t=x,\zeta^{n,j}_t=y$ and $\tau^n_{i,j}>t$ then
\begin{align*}
\p{\tau^n_{i,j}\in (t,t+\delta_n] \big|\mathcal F_t }
&= \mathcal O(n^{9/5} N^{-1} \delta_n g(x-\mu^n_{T_n-t})^{-1}\1_{|x-y|<Kn^{-1}}).
\end{align*}
\end{prop}

\begin{proof}
For $t\in \delta_n \N_0 \cap [0,T_n^-]$ and $x,x' \in \frac 1n \Z$,
if $\zeta^{n,i}_t=x$, $\zeta^{n,j}_t=x'$ and $\tau^n_{i,j}>t$, then
by the definition of $\mathcal C^n_{T_n-t-\delta_n}(x,x')$ in~\eqref{eq:Cntdefn},
\begin{align*}
\p{\tau^n_{i,j}\in (t,t+\delta_n] \big| \mathcal F_t }
&=
\begin{cases}
 \frac{|\mathcal C^n_{T_n-t-\delta_n}(x,x')|}{N p^n_{T_n-t}(x) \cdot Np^n_{T_n-t}(x')}
 \quad &\text{if }x\neq x' ,\\
  \frac{|\mathcal C^n_{T_n-t-\delta_n}(x,x)|}{N p^n_{T_n-t}(x) (Np^n_{T_n-t}(x)-1)}
 \quad &\text{if }x= x' .
 \end{cases}
\end{align*}
If $x\in I^n_{T_n-t}$ and $E$ occurs, then by the definition of the event $E_3$ in~\eqref{eq:eventE3}, $\cap_{j=1}^3 B^{(j)}_{T_n-t-\delta_n}(x)$ occurs.
Hence
\begin{align*}
|\mathcal C^n_{T_n-t-\delta_n}(x,x)|&=n^2 N \delta_n p^n_{T_n-t-\delta_n}(x)(1+\mathcal O(n^{-1/5})),\\
|\mathcal C^n_{T_n-t-\delta_n}(x,x+n^{-1})|&=\tfrac 12 m n^2 N \delta_n (p^n_{T_n-t-\delta_n}(x)+p^n_{T_n-t-\delta_n}(x+n^{-1}))(1+\mathcal O(n^{-1/5})),\\
\text{and }\quad
|\mathcal C^n_{T_n-t-\delta_n}(x,y)|&=\mathcal O(n^{9/5} N \delta_n) p^n_{T_n-t-\delta_n}(x)\1_{|x-y|<Kn^{-1}}
\; \forall y\in \tfrac 1n \Z \text{ with }|y-x|>n^{-1}.
\end{align*}
The result follows by the definition of the event $E_1$ in~\eqref{eq:eventE1}, and since $n^{-1/5}=o((\log N)^{-C})$, $Np^n_{T_n-t}(x)\ge \frac 15 N g(D^+_n)\ge \frac 1 {10} n^{1/2}N^{1/2}$ for $x\in I^n_{T_n-t}$ and $g(d_n+n^{-1})^{-1}=\mathcal O((\log N)^C)$.
\end{proof}

Finally, we need a bound on the probability that two pairs of lineages coalesce in the same time interval of length $\delta_n$.
\begin{prop} \label{prop:doublecoal}
Suppose the event $E$ occurs. 
For $t\in \delta_n \N_0 \cap [0,T_n^-]$, $x_1 \in i^n_{T_n-t}$, $x_2,x_3 \in \frac 1n \Z$,
and $i_1,i_2,i_3\in [k_0]$, if $\zeta^{n,i_k}_t=x_k$ for $k\in \{1,2,3\}$ and $\tau^n_{i_k,i_\ell}>t$ $\forall k\neq \ell \in \{1,2,3\}$ then
\begin{align} \label{eq:doublecoalstat1}
\p{\tau^n_{i_1,i_2},\tau^n_{i_1,i_3} \in (t,t+\delta_n] \Big| \mathcal F_t}
=\mathcal O(n^{9/5}N^{-2} \delta_n (\log N)^{2C} \1_{|x_1-x_2|\vee |x_1-x_3|< Kn^{-1}}).
\end{align}
For $x_1,x_3 \in i^n_{T_n-t}$, $x_2,x_4 \in \frac 1n \Z$ and
$i_1,i_2,i_3,i_4 \in [k_0]$, if
$\zeta^{n,i_k}_t=x_k$ for $k\in \{1,2,3,4\}$ and $\tau^n_{i_k,i_\ell}>t$ $\forall k\neq \ell \in \{1,2,3,4\}$ then
\begin{align}\label{eq:doublecoalstat2}
\p{\tau^n_{i_1,i_2},\tau^n_{i_3,i_4} \in (t,t+\delta_n] \Big| \mathcal F_t }
= \mathcal O(n^4 N^{-2} \delta_n^2 (\log N)^{2C}
\1_{|x_1-x_2|\vee |x_3-x_4|<Kn^{-1}}).
\end{align}
\end{prop}

\begin{proof}
For the first statement, since $B^{(4)}_{T_n-t-\delta_n}(x_1)$ occurs by the definition of the event $E_3$ in~\eqref{eq:eventE3},
\begin{align*}
&\p{\tau^n_{i_1,i_2},\tau^n_{i_1,i_3} \in (t,t+\delta_n] \big| \mathcal F_t}\\
&= \frac{|\mathcal C^n_{T_n-t-\delta_n}(x_1,x_2,x_3)|}{Np^n_{T_n-t}(x_1) (Np^n_{T_n-t}(x_2)-\1_{x_1=x_2})( Np^n_{T_n -t}(x_3)-\1_{x_1=x_3}-\1_{x_2=x_3})}\\
&\le \1_{|x_1-x_2|\vee |x_1-x_3|< Kn^{-1}}
\frac{6n^{9/5}N^{-2} \delta_n p^n_{T_n-t-\delta_n}(x_1)}{p^n_{T_n-t}(x_1)p^n_{T_n-t}(x_2)p^n_{T_n-t}(x_3) }.
\end{align*}
By the definition of the event $E_1$ in~\eqref{eq:eventE1} and since $x_1-\mu^n_{T_n-t}\le d_n$ and $g(d_n+Kn^{-1})^{-1}=\mathcal O((\log N)^C)$,~\eqref{eq:doublecoalstat1} follows.
For the second statement, since $B^{(3)}_{T_n-t-\delta_n}(x_1)$ and $B^{(3)}_{T_n-t-\delta_n}(x_3)$ occur,
\begin{align*}
&\p{\tau^n_{i_1,i_2},\tau^n_{i_3,i_4} \in (t,t+\delta_n] \big| \mathcal F_t }\\
&\le \frac{|\mathcal C^n_{T_n-t-\delta_n}(x_1,x_2)||\mathcal C^n_{T_n-t-\delta_n}(x_3,x_4)|}{Np^n_{T_n-t}(x_1) (Np^n_{T_n-t}(x_2) -\1_{x_1=x_2})(Np^n_{T_n-t}(x_3) -\sum_{j=1}^2\1_{x_j=x_3})( Np^n_{T_n-t}(x_4)-\sum_{j=1}^3 \1_{x_j=x_4}) }\\
&\le \1_{|x_1-x_2|\vee |x_3-x_4| <Kn^{-1} }
\frac{24 |\mathcal C^n_{T_n-t-\delta_n}(x_1,x_2)||\mathcal C^n_{T_n-t-\delta_n}(x_3,x_4)| }{N^4 p^n_{T_n-t}(x_1)p^n_{T_n-t}(x_2)p^n_{T_n-t}(x_3)p^n_{T_n-t}(x_4)}.
\end{align*}
Since $\cap_{j=1}^3 B^{(j)}_{T_n-t-\delta_n}(x_1)$ and $\cap_{j=1}^3 B^{(j)}_{T_n-t-\delta_n}(x_3)$ occur, and $(x_1-\mu^n_{T_n-t})\vee (x_3 -\mu^n_{T_n-t})\le d_n$,~\eqref{eq:doublecoalstat2} follows by the definition of the event $E_1$ in~\eqref{eq:eventE1}.
\end{proof}

We are now ready to prove Propositions~\ref{prop:tauk}-\ref{prop:tautautilde}.

\begin{proof}[Proof of Proposition~\ref{prop:tauk}]
Suppose $n$ is sufficiently large that $\gamma_n \le K \log N$.
Suppose the event $E$ occurs.
Take $t\in \delta_n \N \cap [t_k+2K \log N,t_{k+1})$, and take
 $x\in\frac 1n \Z$ such that $|x-\mu^n_{T_n-t}|\le \frac 1 {64}\alpha d_n+1$.
Then by conditioning on $\mathcal F_t$,
\begin{align} \label{eq:tauk*}
&\p{\tilde \tau^{n}_{i,j} \in (t,t+\delta_n ],
\zeta^{n,i}_{t}=x \Big| \mathcal F_{t_k} } \notag \\
&= \E{\p{\tilde \tau^{n}_{i,j} \in (t,t+ \delta_n ] \Big| \mathcal F_{t}}\1_{
\zeta^{n,i}_{t}=x}\1_{\tau^{n}_{i,j}>t} \bigg| \mathcal F_{t_k} } \notag \\
&\le  \mathbb E \Big[ n^2  N^{-1} \delta_n g(x-\mu^n_{T_n-t})^{-1} (1+\mathcal O((\log N)^{-C})) \notag \\
&\qquad \qquad \big(\1_{\zeta^{n,j}_{t}=x} +m \1_{|\zeta^{n,j}_{t}-x|=n^{-1}}+\mathcal O(n^{-1/5}) \1_{|\zeta^{n,j}_{t}-x|< Kn^{-1}}\big) 
 \1_{
\zeta^{n,i}_{t}=x}\1_{\tau^{n}_{i,j}> t} \Big| \mathcal F_{t_k} \Big] \notag \\
&=n^2 N^{-1} \delta_n g(x-\mu^n_{T_n-t})^{-1} (1+\mathcal O((\log N)^{-C}))\notag \\
&\qquad \Big(\p{\zeta^{n,i}_{t}=x=\zeta^{n,j}_{t}, \tau^{n}_{i,j}>t \Big| \mathcal F_{t_k}} 
+m\p{\zeta^{n,i}_{t}=x, |\zeta^{n,j}_{t}-x|=n^{-1}, \tau^{n}_{i,j}>t \Big| \mathcal F_{t_k}} \notag \\
&\qquad \qquad +\mathcal O (n^{-1/5}) \p{\zeta^{n,i}_{t}=x, |\zeta^{n,j}_{t}-x| < K n^{-1}, \tau^{n}_{i,j}>t \Big| \mathcal F_{t_k}}\Big),
\end{align}
where the inequality follows by Proposition~\ref{prop:coal} and the definition of $\tilde \tau^n_{i,j}$.
By conditioning on $\mathcal F_{t-\gamma_n}$ and then on $\mathcal F_{t-\epsilon_n}$,
\begin{align} \label{eq:tauk**}
&\p{\zeta^{n,i}_{t}=x=\zeta^{n,j}_{t}, \tau^{n}_{i,j}>t \Big| \mathcal F_{t_k}} \notag \\
&= \e \Big[\p{\zeta^{n,i}_{t}=x=\zeta^{n,j}_{t}, \tau^{n}_{i,j}>t \Big| \mathcal F_{t -\gamma_n}} 
 \1_{\tau^{n}_{i,j}>t -\gamma_n }\1_{|\tilde \zeta^{n,i}_{t -\gamma_n}| \vee |\tilde \zeta^{n,j}_{t -\gamma_n}|\le d_n} \Big| \mathcal F_{t_k}\Big] \notag \\
&\qquad + \e \Big[ \p{\zeta^{n,i}_{t}=x=\zeta^{n,j}_{t}, \tau^{n}_{i,j}>t \Big|\mathcal F_{t -\epsilon_n}} 
 \1_{\tau^{n}_{i,j}>t -\epsilon_n}\1_{|\tilde \zeta^{n,i}_{t -\gamma_n}| \vee |\tilde \zeta^{n,j}_{t -\gamma_n}|> d_n} \Big| \mathcal F_{t_k} \Big].
\end{align}
For the second term on the right hand side, note that
by a union bound, and then
by~\eqref{eq:propRlogN2} in Proposition~\ref{prop:RlogN} and~\eqref{eq:propintipstat2} in Proposition~\ref{prop:intip},
and since $\tilde \zeta^{n,i}_{t_k}\wedge \tilde \zeta^{n,j}_{t_k}\ge D_n^-$ by the definition of $I^{n,\epsilon}_{T_n-t_k}$ in~\eqref{eq:Intdefn},
 and $t-\gamma_n-t_k\ge K \log N$,
\begin{align} \label{eq:tautilde(A)}
\p{|\tilde \zeta^{n,i}_{t -\gamma_n}| \vee |\tilde \zeta^{n,j}_{t -\gamma_n}|> d_n \Big| \mathcal F_{t_k}} 
&\le \p{\tilde \zeta^{n,i}_{t -\gamma_n}\wedge \tilde \zeta^{n,j}_{t -\gamma_n}< -d_n \Big|
\mathcal F_{t_k}} 
 +\p{\tilde \zeta^{n,i}_{t -\gamma_n}\vee \tilde \zeta^{n,j}_{t -\gamma_n}> d_n \Big|\mathcal F_{t_k}} \notag \\
&\le 2(\log N)^{2-\frac 18 \alpha C}+2(\log N)^3 e^{-(1+\frac 14 (1-\alpha))\kappa\lfloor d_n \rfloor} \notag \\
&=\mathcal O((\log N)^{3-\frac 18 \alpha C})
\end{align}
by the definition of $d_n$ in~\eqref{eq:paramdefns}.
Therefore, by~\eqref{eq:tauk**} and by~\eqref{eq:lemfromxixj1} and~\eqref{eq:lemfromxixj2} from Lemma~\ref{lem:fromxixj},
\begin{align*}
&\p{\zeta^{n,i}_{t}=x=\zeta^{n,j}_{t}, \tau^n_{i,j}>t \Big|\mathcal F_{t_k}}\\
&\le n^{-2} \pi(x-\mu^n_{T_n-t})^2  \left(1+\mathcal O((\log N)^{-C})\right) +2n^{-2}\epsilon_n^{-2} \cdot \mathcal O((\log N)^{3-\frac 18 \alpha C})\\
&= n^{-2} \pi(x-\mu^n_{T_n-t})^2 (1+\mathcal O((\log N)^{-2})),
\end{align*}
since $\epsilon_n^{-2}=\mathcal O((\log N)^4)$, $\pi(x-\mu^n_{T_n-t})^{-2}=\mathcal O((\log N)^{\frac 1 {16}\alpha C})$ and we chose $C>2^{13}\alpha^{-2}$, so in particular $\frac 1 {16} \alpha C -7>2$.
Hence using the same argument for the other terms on the right hand side of~\eqref{eq:tauk*},
and since $\pi(y-\mu^n_{T_n-t}) =\pi(x-\mu^n_{T_n-t}) (1+\mathcal O(n^{-1}))$ if $|x-y|<Kn^{-1}$,
\begin{align*}
&\p{\tilde \tau^{n}_{i,j} \in (t,t+\delta_n ],
\zeta^{n,i}_{t}=x \Big| \mathcal F_{t_k} }\\
 &\le  N^{-1}\delta_n (1+2m) g(x-\mu^n_{T_n-t})^{-1} \pi(x-\mu^n_{T_n-t})^2  \left(1+\mathcal O((\log N)^{-2})\right).
\end{align*}
Note that if $\tilde \tau^n_{i,j} \in (t,t+\delta_n]$ then
$|\tilde \zeta^{n,i}_t|\wedge |\tilde \zeta^{n,j}_t |\le \frac 1 {64} \alpha d_n$ by the definition of $\tilde \tau^n_{i,j}$, and $|\tilde \zeta^{n,i}_t-\tilde \zeta^{n,j}_t|< Kn^{-1}$ by Proposition~\ref{prop:coal}, and so for $n$ sufficiently large, we must have $|\tilde \zeta^{n,i}_t|\le \frac 1 {64}\alpha d_n +1$.
Letting $\tilde i^n_s =\frac 1n \Z \cap [\mu^n_s-\frac 1 {64}\alpha d_n-1, \mu^n_s+\frac 1 {64}\alpha d_n +1]$ for $s\ge 0$, 
it follows that
\begin{align} \label{eq:taukupper}
& \p{\tilde \tau^{n}_{i,j} \in (t_k +2K \log n, t_{k+1}]  \Big| \mathcal F_{t_k}} \notag \\
&\le   N^{-1} \delta_n (1+2m) \left(1+\mathcal O((\log N)^{-2})\right)
\sum_{t\in \delta_n \N \cap  [t_k+ 2K \log N, t_{k+1})} \sum_{x\in \tilde i^n_{T_n-t}}
 g(x-\mu^n_{T_n-t})^{-1} \pi(x-\mu^n_{T_n-t})^{2} \notag \\
&\le \beta_n 
 \left(1+\mathcal O((\log N)^{-2})\right),
\end{align}
by the definition of $\beta_n  $ in~\eqref{eq:betadefn}.

For a lower bound, note that for $t\in \delta_n \N \cap [t_k+2K \log N,t_{k+1})$,
\begin{align} \label{eq:tauk(B)}
&\p{\tilde \tau^{n}_{i,j} \in (t,t+\delta_n ] \Big| \mathcal F_{t_k}} \notag \\
&\geq \sum_{x\in 2(\log N)^{-C} \Z, |x-\mu_{T_n-t}|\le \frac 1 {64}\alpha  d_n-1}
\p{\tilde \tau^{n}_{i,j} \in (t,t+\delta_n ], |\zeta^{n,i}_{t}-x|< (\log N)^{-C} \Big| \mathcal F_{t_k}}.
\end{align}
Now for $x\in 2(\log N)^{-C} \Z$ with $ |x-\mu_{T_n-t}|\le \frac 1 {64}\alpha  d_n-1$, by conditioning on $\mathcal F_t$,
\begin{align} \label{eq:taukdagger}
&\p{\tilde \tau^{n}_{i,j} \in (t,t+\delta_n ], |\zeta^{n,i}_{t}-x|< (\log N)^{-C} \Big| \mathcal F_{t_k}} \notag \\
& = \E{\p{\tilde \tau^{n}_{i,j} \in (t,t+\delta_n ]\Big| \mathcal F_{t}} \1_{\tau^{n}_{i,j}>t}\1_{ |\zeta^{n,i}_{t}-x|< (\log N)^{-C}} \Big| \mathcal F_{t_k}} \notag \\
& \geq \e \Big[ n^2 N^{-1}  \delta_n g(\zeta^{n,i}_t -\mu^n_{T_n-t})^{-1} (1-\mathcal O((\log N)^{-C}))
(\1_{\zeta^{n,i}_{t}=\zeta^{n,j}_{t}}
+m\1_{|\zeta^{n,i}_{t}-\zeta^{n,j}_{t}|=n^{-1}}) \notag \\
 &\hspace{9.5cm} \1_{\tau^{n}_{i,j}>t}\1_{ |\zeta^{n,i}_{t}-x|< (\log N)^{-C}} \Big| \mathcal F_{t_k} \Big] \notag \\
 & = n^2  N^{-1} \delta_n g(x-\mu^n_{T_n-t})^{-1}(1-\mathcal O((\log N)^{-C})) \notag \\
&\qquad \Big(\p{\zeta^{n,i}_{t}=\zeta^{n,j}_{t},  |\zeta^{n,i}_{t}-x|< (\log N)^{-C}, \tau^{n}_{i,j}>t \Big|  \mathcal F_{t_k}} \notag \\
&\qquad \qquad +m \p{|\zeta^{n,i}_{t}-\zeta^{n,j}_{t}|=n^{-1},
|\zeta^{n,i}_{t}-x|< (\log N)^{-C}, \tau^{n}_{i,j}>t \Big|  \mathcal F_{t_k}} \Big),
\end{align}
where the inequality follows by Proposition~\ref{prop:coal}. 
For the first term on the right hand side, by conditioning on $\mathcal F_{t-\gamma_n}$,
\begin{align} \label{tauklower*}
& \p{\zeta^{n,i}_{t}=\zeta^{n,j}_{t},  |\zeta^{n,i}_{t}-x|< (\log N)^{-C}, \tau^{n}_{i,j}>t \Big|  \mathcal F_{t_k}} \notag \\
&\ge \e \Big[ \p{\zeta^{n,i}_{t}=\zeta^{n,j}_{t},  |\zeta^{n,i}_{t}-x|< (\log N)^{-C}, \tau^{n}_{i,j}>t \Big| \mathcal F_{t -\gamma_n}} 
 \1_{\tau^{n}_{i,j}>t -\gamma_n}
\1_{|\tilde \zeta^{n,i}_{t -\gamma_n}| \vee |\tilde \zeta^{n,j}_{t -\gamma_n}| \le d_n} \Big|  \mathcal F_{t_k} \Big].
\end{align} 
By a union bound, if $\tau^n_{i,j}>t-\gamma_n$ then
\begin{align} \label{eq:taudagger(*)2}
\p{\tau^{n}_{i,j} \le t \Big| \mathcal F_{t -\gamma_n} } 
&\le \sum_{s\in \delta_n \N \cap [t-\gamma_n ,t)} 
\p{\tau^{n}_{i,j} \in (s,s+ \delta_n ],
\zeta^{n,i}_s \in I^n_{T_n-s} \text{ or }\zeta^{n,j}_s \in I^n_{T_n-s}
\Big| \mathcal F_{t -\gamma_n}} \notag \\
&\quad + \p{\exists s\in \delta_n \N \cap [t-\gamma_n, t):\zeta^{n,i}_{s}, \zeta^{n,j}_{s}\notin I^n_{T_n-s}, \, \tau^n_{i,j}>s \Big| \mathcal F_{t -\gamma_n}}.
\end{align}
Suppose $|\tilde \zeta^{n,i}_{t-\gamma_n}|\vee |\tilde \zeta^{n,j}_{t-\gamma_n}|\le d_n$.
Take $s\in \delta_n \N \cap [t-\gamma_n,t)$, and let
$I=2\Z \cap [\mu^n_{T_n-s}+(\log N)^{2/3}+K+\nu t^*+3, \mu^n_{T_n-s}+D^+_n]$; then
by conditioning on $\mathcal F_{s}$ and using Proposition~\ref{prop:coal},
\begin{align} \label{eq:tauktauupper}
&\p{\tau^{n}_{i,j} \in (s,s+ \delta_n ],
\zeta^{n,i}_s \in I^n_{T_n-s}
\Big| \mathcal F_{t -\gamma_n}} \notag \\
&\le  \e \Big[\mathcal O(n^2  N^{-1} \delta_n g(\zeta^{n,i}_s-\mu^n_{T_n-s})^{-1})
\1_{|\zeta^{n,i}_{s}-\zeta^{n,j}_{s}|< Kn^{-1}}
\1_{\tau^{n}_{i,j}>s}
\1_{\zeta^{n,i}_s \in I^n_{T_n-s}}
\Big| \mathcal F_{t -\gamma_n}\Big] \notag \\
&\le \mathcal O(n^2  N^{-1}\delta_n)
 \sum_{x' \in I}
g(x' +1-\mu^n_{T_n-s})^{-1}
 \P\Big(| \zeta^{n,i}_{s} - x'|\le 1, |\zeta^{n,j}_{s} -x'|\le 2,
\tau^{n}_{i,j}>s
\Big| \mathcal F_{t -\gamma_n}\Big) \notag \\
&\quad +\mathcal O(n^2 N^{-1}\delta_n g((\log N)^{2/3}+K+\nu t^*+4)^{-1}).
\end{align}
Take $s'\in [s-t^*,s]$ such that $s'-(t-\gamma_n)\in t^* \N_0$.
Then by~\eqref{eq:lemfromxixj3} in Lemma~\ref{lem:fromxixj},
for $x'\in I$,
\begin{align*}
&\P\Big( |\zeta^{n,i}_{s} - x'|\le 1, |\zeta^{n,j}_{s} - x'|\le 2,
\tau^{n}_{i,j}>s
\Big| \mathcal F_{t -\gamma_n}\Big)\\
&\le \P\Big( \zeta^{n,i}_{s' } \ge x'-1-(\log N)^{2/3}, \, \zeta^{n,j}_{s'}  \ge x'-2-(\log N)^{2/3},
 \tau^{n}_{i,j}>s' 
\Big| \mathcal F_{t -\gamma_n}\Big)\\
&\le (\log N)^4 e^{2(1+\frac 14 (1-\alpha))\kappa(d_n-(x'-3-(\log N)^{2/3}-\mu^n_{T_n-s'}))}
\end{align*}
by~\eqref{eq:propintipstat*} in Proposition~\ref{prop:intip} (since $s'-(t-\gamma_n)\le \gamma_n\le K \log N$ and we are assuming $\tilde \zeta^{n,i}_{t-\gamma_n}\vee \tilde \zeta^{n,j}_{t-\gamma_n}\le d_n$).
Therefore, by~\eqref{eq:tauktauupper},
\begin{align} \label{eq:tauearly*}
&\p{\tau^{n}_{i,j} \in (s,s+ \delta_n ],
\zeta^{n,i}_s \in I^n_{T_n-s}
\Big| \mathcal F_{t-\gamma_n}}  \notag \\
&\le \mathcal O(n^2  N^{-1}\delta_n)
\Big( \sum_{x'\in I}
g(x'+1-\mu^n_{T_n-s})^{-1} (\log N)^{4+4C} e^{4\kappa (\log N)^{2/3}}
e^{-2(1+\frac 14 (1-\alpha))\kappa(x'-3-\mu^n_{T_n-s'})} \notag \\
&\hspace{4cm} + 2e^{\kappa ((\log N)^{2/3}+K+\nu t^*+4)}\Big) \notag \\
&= \mathcal O(n^2  N^{-1}\delta_n (\log N)^{4+4C} e^{4\kappa(\log N)^{2/3}} )
\end{align}
since $g(y)^{-1} \le 2e^{\kappa y}$ for $y\ge 0$, and by the definition of the event $E_1$ in~\eqref{eq:eventE1}.
For the second term on the right hand side of~\eqref{eq:taudagger(*)2}, note that by~\eqref{eq:lemfromxixj3} in Lemma~\ref{lem:fromxixj} and by the definition of the event $E_1$,
\begin{align*}
&\p{\exists s\in \delta_n \N \cap [t-\gamma_n,t): \zeta^{n,i}_s, \zeta^{n,j}_s \notin I^n_{T_n-s},\,  \tau^{n}_{i,j}>s \Big| \mathcal F_{t -\gamma_n}}\\
&\le  \P \Big(\exists s ' \in [t-\gamma_n,t):  s'-(t-\gamma_n) \in t^*\N_0,
\tilde \zeta^{n,i}_{s'}\wedge \tilde \zeta^{n,j}_{s'} \ge D^+_n -(\log N)^{2/3}-2\nu t^*, \tau^{n}_{i,j}>s' \Big| \mathcal F_{t-\gamma_n}\Big) \\
&\le (t^*)^{-1}\gamma_n (\log N)^4 e^{2(1+\frac 14 (1-\alpha))\kappa (d_n-(D_n^+-(\log N)^{2/3}-2\nu  t^* -1))}
\end{align*}
by~\eqref{eq:propintipstat*} in Proposition~\ref{prop:intip} and since $\tilde \zeta^{n,i}_{t-\gamma_n} \vee \tilde \zeta^{n,j}_{t-\gamma_n}\le d_n$.
Note that $e^{-2(1+\frac 14 (1-\alpha)) \kappa D_n^+}=\left( \frac n N \right)^{(1+\frac 14 (1-\alpha))(1-2c_0)}\le \frac n N$ by~\eqref{eq:Dn+-defn} and our choice of $c_0$.
Hence, by~\eqref{eq:tauearly*}, substituting into~\eqref{eq:taudagger(*)2},
\begin{align*}
\p{\tau^{n}_{i,j} \le t \Big| \mathcal F_{t -\gamma_n} } 
&\le \mathcal O(n^2 N^{-1} \gamma_n (\log N)^{4+4C}e^{4\kappa(\log N)^{2/3}})
+ \mathcal O(\gamma_n (\log N)^{4+4C} e^{4\kappa(\log N)^{2/3}} nN^{-1})\\
&=\mathcal O(n^{-1-\frac 12 (a_2-3)}),
\end{align*}
since $N\ge n^{a_2}$ for $n$ sufficiently large, with $a_2>3$.
Therefore if $|\tilde \zeta^{n,i}_{t-\gamma_n}|\vee |\tilde \zeta^{n,j}_{t-\gamma_n}|\le d_n$ and $\tau^n_{i,j}>t-\gamma_n$,
\begin{align} \label{eq:taukstatdist}
&\p{\zeta^{n,i}_{t}=\zeta^{n,j}_{t},  |\zeta^{n,i}_{t}-x|< (\log N)^{-C}, \tau^{n}_{i,j}>t \Big| \mathcal F_{t -\gamma_n}} \notag  \\
&\ge \p{\zeta^{n,i}_{t}=\zeta^{n,j}_{t},  |\zeta^{n,i}_{t}-x|< (\log N)^{-C} \Big| \mathcal F_{t -\gamma_n}}
- \p{\tau^{n}_{i,j} \le t \Big| \mathcal F_{t -\gamma_n} } \notag \\
&\ge  
\pi(x-\mu^n_{T_n-t})^2  \cdot 2(\log N)^{-C} n^{-1} \left(1-\mathcal O((\log N)^{-C})\right)
-\mathcal O(n^{-1-\frac 12 (a_2-3)})
\end{align}
by~\eqref{eq:lemfromxixj1} in Lemma~\ref{lem:fromxixj} and since 
$\pi(y-\mu^n_{T_n-t})=\pi(x-\mu^n_{T_n-t})(1+\mathcal O((\log N)^{-C}))$ if $|y-x|<(\log N)^{-C}$.
To bound the other terms in~\eqref{tauklower*}, note first that by a union bound,
\begin{align} \label{eq:tauk*3}
\p{\tau^{n}_{i,j} \le t -\gamma_n \Big| \mathcal F_{t_k } } 
&\le \sum_{s\in \delta_n \N_0 \cap [t_k,t-\gamma_n)} 
 \p{\tau^{n}_{i,j} \in (s,s+ \delta_n ],
\zeta^{n,i}_s \in I^n_{T_n-s} \text{ or }\zeta^{n,j}_s \in I^n_{T_n-s}
\Big| \mathcal F_{t_k}} \notag \\
&\qquad + \p{\exists s'\in \delta_n \N_0 \cap [t_k,t-\gamma_n): \zeta^{n,i}_{s'}\wedge \zeta^{n,j}_{s'} \notin I^n_{T_n-s'} \Big| \mathcal F_{t_k }} .
\end{align}
By Proposition~\ref{prop:coal}, for $s\in \delta_n \N_0 \cap [t_k,t-\gamma_n)$,
\begin{align} \label{eq:tauearly**}
\p{\tau^{n}_{i,j} \in (s,s+\delta_n ],
\zeta^{n,i}_s \in I^n_{T_n-s}
\Big| \mathcal F_{t_k}}
&=\E{ \p{\tau^{n}_{i,j} \in (s,s+\delta_n ] \Big| \mathcal F_{s}}
\1_{\zeta^{n,i}_s \in I^n_{T_n-s}}
\Big| \mathcal F_{t_k}} \notag \\
&= \mathcal O(n^2 N^{-1}  \delta_n g(D_n^+)^{-1}) \notag \\
&= \mathcal O(n^{3/2}N^{-1/2} \delta_n)
\end{align}
since $\kappa D_n^+\le \frac 12 \log (N/n)$.
For the second term on the right hand side of~\eqref{eq:tauk*3},
by~\eqref{eq:lemfromxixj3} in Lemma~\ref{lem:fromxixj} and by the definition of the event $E_1$ in~\eqref{eq:eventE1},
\begin{align*}
 &\p{\exists s' \in \delta_n \N_0 \cap [t_k, t-\gamma_n): \zeta^{n,i}_{s'}\wedge \zeta^{n,j}_{s'} \notin I^n_{T_n-s'} \Big| \mathcal F_{t_k }}\\
 &\le  \p{\exists s'\in  [t_k,t-\gamma_n): s'-t_k \in t^*\N_0, \tilde \zeta^{n,i}_{s'}\wedge \tilde \zeta^{n,j}_{s'} \ge D^+_n - (\log N)^{2/3}-2\nu  t^* \Big| \mathcal F_{t_k }}\\
 &\le (t^*)^{-1} t_1 (\log N)^3 e^{(1+\frac 14 (1-\alpha))\kappa((1-\epsilon)D_n^+ -(D^+_n -(\log N)^{2/3}-2\nu t^*-1))}
\end{align*}
by~\eqref{eq:propintipstat2} and~\eqref{eq:propintipstat3} in Proposition~\ref{prop:intip} and since $\tilde \zeta^{n,i}_{t_k} \wedge \tilde \zeta^{n,j}_{t_k}\le (1-\epsilon)D_n^+$.
Hence by~\eqref{eq:tauk*3} and~\eqref{eq:tauearly**}, and since $\kappa(1+\frac 14 (1-\alpha))D^+_n \ge \frac 12 \log (N/n)$ by the definition of $D_n^+$ in~\eqref{eq:Dn+-defn},
\begin{align} \label{eq:taukA}
\p{\tau^{n}_{i,j} \le t-\gamma_n \Big| \mathcal F_{t_k } }
&\le \mathcal O(t_1 n^{3/2} N^{-1/2})+\mathcal O(t_1 (\log N)^3 e^{2\kappa(\log N)^{2/3}} n^{\epsilon/2} N^{-\epsilon/2})\notag \\ 
&=\mathcal O(n^{-(\frac 13 (a_2-3) \wedge \epsilon)}).
\end{align}
Therefore, substituting into~\eqref{tauklower*} and using~\eqref{eq:tautilde(A)} and~\eqref{eq:taukstatdist},
\begin{align*}
& \p{\zeta^{n,i}_{t}=\zeta^{n,j}_{t},  |\zeta^{n,i}_{t}-x|< (\log N)^{-C}, \tau^{n}_{i,j}>t  \Big|  \mathcal F_{t_k}} \notag \\
&\ge 
2\pi(x-\mu^n_{T_n-t})^2 (\log N)^{-C} n^{-1} \left(1-\mathcal O((\log N)^{-C})\right) 
(1-\mathcal O(n^{-(\frac 13 (a_2-3)\wedge \epsilon)})-\mathcal O((\log N)^{3-\frac 18 \alpha C})).
\end{align*}
Since we chose $C>2^{13}\alpha^{-2}$, we have $\frac 18 \alpha C -3>2$. Hence
by the same argument for the second term on the right hand side of~\eqref{eq:taukdagger},
and then substituting into~\eqref{eq:tauk(B)},
\begin{align*}
&\p{\tilde \tau^{n}_{i,j} \in (t,t+\delta_n ] \Big| \mathcal F_{t_k}}\\
&\geq \sum_{x\in 2(\log N)^{-C} \Z, |x-\mu^n_{T_n-t}|\le \frac 1 {64}\alpha d_n-1}
2(\log N)^{-C} n  N^{-1} \delta_n 
 \cdot (1+2m) 
\frac{\pi(x-\mu^n_{T_n-t})^2} {g(x-\mu^n_{T_n-t})}  \left(1-\mathcal O((\log N)^{-2})\right)\\
&=\beta_n t_1^{-1} \delta_n
 (1-\mathcal O((\log N)^{-2})),
\end{align*}
since $\frac 1 {32}\alpha^2 C >2$ and $\frac 1 {64}\alpha C>2$,
which, together with~\eqref{eq:taukupper}, completes the proof.
\end{proof}

\begin{proof}[Proof of Proposition~\ref{prop:doublecoaltk}]
Suppose $n$ is sufficiently large that $2K\log N \ge \epsilon_n$.
Suppose the event $E$ occurs.
We begin by proving the first statement~\eqref{eq:propdoubletk1}.
Take $s<t \in \delta_n \N\cap [t_k+2K \log N, t_{k+1})$.
Note that if for some $\ell,\ell'\in [k_0]$, $\tilde \tau^n_{\ell, \ell '}\in (t,t+\delta_n]$ then $|\tilde \zeta^{n,\ell}_t|\wedge |\tilde \zeta^{n,\ell '}_t|\le \frac 1 {64} \alpha d_n$ by the definition of $\tilde \tau^n_{\ell, \ell '}$ in~\eqref{eq:tildetaudefn}, and $|\tilde \zeta^{n,\ell}_t- \tilde \zeta^{n,\ell '}_t|<Kn^{-1}$ by Proposition~\ref{prop:coal}, so in particular $|\tilde \zeta^{n,\ell }_t|\le d_n$. Hence
by conditioning on $\mathcal F_{t}$ and applying Proposition~\ref{prop:coal},
\begin{align} \label{eq:doublecoaltk*}
&\p{\tilde \tau^{n}_{i,j_1} \in (s,s+\delta_n],
\tilde \tau^{n}_{i,j_2} \in (t,t+\delta_n]
\Big|\mathcal F_{t_k}} \notag \\
&\le \E{ \mathcal O(n^2  N^{-1}\delta_n  g(\tilde \zeta^{n,i}_t)^{-1}) \1_{| \tilde \zeta^{n,i}_{t}|\le d_n}
\1_{\tilde \tau^{n}_{i,j_1}\in (s,s+\delta_n]} 
\Big|\mathcal F_{t_k}} \notag \\
&\le \mathcal O( n^2  N^{-1} \delta_n (\log N)^C) 
\p{\tilde \tau^{n}_{i,j_1} \in (s,s+\delta_n] \Big|\mathcal F_{t_k}}.
\end{align}
By conditioning on $\mathcal F_{s}$ and applying Proposition~\ref{prop:coal},
\begin{align*}
&\p{\tilde \tau^{n}_{i,j_1} \in (s,s+\delta_n] \Big|\mathcal F_{t_k}}\\
&\le 
 \e \Big[\mathcal O(n^2  N^{-1} \delta_n g(\tilde \zeta^{n,i}_s)^{-1})
\1_{\tau^n_{i,j_1}>s} \1_{|\tilde \zeta^{n,i}_{s}|\le d_n}
 \1_{|\zeta^{n,i}_{s}-\zeta^{n,j_1}_{s}|< Kn^{-1}} \Big| \mathcal F_{t_k} \Big]\\
 &= \mathcal O( n^2  N^{-1} \delta_n (\log N)^{C}) 
 \p{|\tilde \zeta^{n,i}_{s}|\le d_n, 
|\zeta^{n,i}_{s}-\zeta^{n,j_1}_{s}|< Kn^{-1}, \tau^n_{i,j_1}>s \Big| \mathcal F_{t_k} }.
\end{align*}
Then since $s-t_k \ge \epsilon_n$, by conditioning on $\mathcal F_{s-\epsilon_n}$,
\begin{align} \label{eq:doublecoaltkS}
&\p{|\tilde \zeta^{n,i}_{s}|\le d_n, 
|\zeta^{n,i}_{s}-\zeta^{n,j_1}_{s}|< Kn^{-1}, \tau^n_{i,j_1}>s \Big| \mathcal F_{t_k} } \notag \\
&\le \e \Big[ \p{|\tilde \zeta^{n,i}_{s}|\le d_n, 
|\zeta^{n,i}_{s}-\zeta^{n,j_1}_{s}|< Kn^{-1}\Big| \mathcal F_{s-\epsilon_n}}
 \1_{\tau^n_{i,j_1}>s-\epsilon_n} 
 \Big| \mathcal F_{t_k} \Big] \notag \\
 &\le \E{\sum_{x\in i^n_{T_n-s},y \in \frac 1n \Z, |x-y|< Kn^{-1}}
 \p{\zeta^{n,i}_s =x, \zeta^{n,j}_s=y \Big| \mathcal F_{s-\epsilon_n}}
 \1_{\tau^n_{i,j_1}>s-\epsilon_n} 
 \Bigg| \mathcal F_{t_k}} \notag \\
 &\le (2n d_n+1) 2K\cdot 2 n^{-2} \epsilon_n^{-2}
\end{align}
by~\eqref{eq:lemfromxixj2} in Lemma~\ref{lem:fromxixj}.
Hence, by~\eqref{eq:doublecoaltk*}, and by the same argument for the case $s>t$, if $s\neq t \in \delta_n \N \cap [t_k+2K\log N,t_{k+1})$,
\begin{align} \label{eq:doublecoalA}
&\p{\tilde \tau^{n}_{i,j_1} \in (s,s+\delta_n],
\tilde \tau^{n}_{i,j_2} \in (t , t+\delta_n]
\Big|\mathcal F_{t_k}}
=\mathcal O(n^3  N^{-2}\delta_n^2 (\log N)^{2C+5}).
\end{align}
By Proposition~\ref{prop:doublecoal}, for $t \in \delta_n \N \cap [t_k+2K\log N,t_{k+1})$,
\begin{align} \label{eq:doublecoal1}
\p{\tilde \tau^{n}_{i,j_1} ,
\tilde \tau^{n}_{i,j_2} \in  (t,t+\delta_n] \Big|\mathcal F_{t_k}} 
&=\mathcal O( n^{9/5}N^{-2} \delta_n (\log N)^{2C})
+\p{\tilde \tau^{n}_{i,j_1}\in (t,t+\delta_n ], \tau^{n}_{j_1,j_2} \le t \Big| \mathcal F_{t_k}}.
\end{align}
By a union bound, and then by conditioning on $\mathcal F_t$ and using Proposition~\ref{prop:coal},
\begin{align*}
&\p{\tilde \tau^{n}_{i,j_1} \in (t,t+\delta_n ], \tau^{n}_{j_1,j_2} \in (t-\epsilon_n, t] \Big| \mathcal F_{t_k}}\\
&= \sum_{t'\in \delta_n \N\cap [t-\epsilon_n, t)}
\p{\tilde \tau^{n}_{i,j_1} \in (t,t+\delta_n ], \tau^{n}_{j_1,j_2} \in (t',t'+ \delta_n] \Big| \mathcal F_{t_k}}
\\
&\le \sum_{t'\in \delta_n \N\cap [t-\epsilon_n, t)}
\E{\mathcal O(n^2 N^{-1} \delta_ng(\tilde \zeta^{n,j_1}_t)^{-1})\1_{|\tilde \zeta^{n,j_1}_{t}|\le d_n}
\1_{\tau^{n}_{j_1,j_2}\in (t',t'+\delta_n]} \Big| \mathcal F_{t_k}} \\
&\le \sum_{t'\in \delta_n \N\cap [t-\epsilon_n, t)}
\mathcal O(n^2  N^{-1}\delta_n (\log N)^C)
\p{\tau^{n}_{j_1,j_2}\in (t',t'+\delta_n] , |\tilde \zeta^{n,j_1}_{t'}|\le d_n+(\log N)^{2/3}+1
\Big| \mathcal F_{t_k}}
\end{align*}
by~\eqref{eq:lemfromxixj3} in Lemma~\ref{lem:fromxixj} and the definition of the event $E_1$ in~\eqref{eq:eventE1}.
Then by Proposition~\ref{prop:coal} again, for $t'\in \delta_n \N \cap [t-\epsilon_n,t)$,
by conditioning on $\mathcal F_{t'}$,
\begin{align*}
\p{\tau^{n}_{j_1,j_2}\in (t',t'+\delta_n] , |\tilde \zeta^{n,j_1}_{t'}|\le d_n+(\log N)^{2/3}+1
\Big| \mathcal F_{t_k}}
&= \mathcal O(n^2  N^{-1}\delta_n g(d_n+(\log N)^{2/3}+1)^{-1}).
\end{align*}
Hence
\begin{align} \label{eq:doublecoal2}
\p{\tilde \tau^{n}_{i,j_1} \in (t,t+\delta_n ], \tau^{n}_{j_1,j_2} \in (t-\epsilon_n, t] \Big| \mathcal F_{t_k}}
&=\mathcal O(n^4  N^{-2} \delta_n \epsilon_n (\log N)^{C}e^{2\kappa(\log N)^{2/3}}) \notag \\
&=\mathcal O(n^{1-\frac 12 (a_2-3)}N^{-1} \delta_n).
\end{align}
Moreover, by Proposition~\ref{prop:coal}, conditioning on $\mathcal F_t$, and then conditioning on $\mathcal F_{t-\epsilon_n}$,
\begin{align} \label{eq:doublecoaltk*2}
&\p{\tilde \tau^{n}_{i,j_1} \in (t,t+\delta_n ], \tau^{n}_{j_1,j_2} \le t-\epsilon_n \Big| \mathcal F_{t_k}} \notag \\
&=\E{\mathcal O(n^2 N^{-1} \delta_n g(\tilde \zeta^{n,i}_t)^{-1})\1_{\tau^n_{i,j_1}>t} \1_{|\tilde \zeta^{n,i}_{t}|\le d_n}
\1_{|\zeta^{n,i}_{t}-\zeta^{n,j_1}_{t}|< Kn^{-1}}
\1_{\tau^{n}_{j_1,j_2}\le t -\epsilon_n} \Big| \mathcal F_{t_k}} \notag
\\
&\le \mathcal O(n^2  N^{-1} \delta_n(\log N)^C) \notag \\
&\qquad \qquad \cdot \e \Big[ \p{|\zeta^{n,i}_{t}-\zeta^{n,j_1}_{t}| < Kn^{-1}, |\tilde \zeta^{n,i}_{t}|\le d_n \Big| \mathcal F_{t -\epsilon_n}}  \1_{\tau^n_{i,j_1}>t-\epsilon_n}
\1_{\tau^{n}_{j_1,j_2} \le t-\epsilon_n} \Big| \mathcal F_{t_k} \Big].
\end{align}
By the same argument as in~\eqref{eq:doublecoaltkS}, if $\tau^n_{i,j_1}>t-\epsilon_n$ then
\begin{align*}
\p{|\zeta^{n,i}_{t}-\zeta^{n,j_1}_{t}| < Kn^{-1}, |\tilde \zeta^{n,i}_{t}|\le d_n \Big| \mathcal F_{t -\epsilon_n}}
&\le (2nd_n+1)2K \cdot 2 n^{-2} \epsilon_n^{-2}
=\mathcal O(n^{-1} (\log N)^5).
\end{align*}
By the same argument as in~\eqref{eq:taukA} in the proof of Proposition~\ref{prop:tauk},
$$
\p{\tau^{n}_{j_1,j_2}\le t -\epsilon_n \Big| \mathcal F_{t_k}} =\mathcal O(n^{-(\frac 13 (a_2-3)\wedge \epsilon)}).
$$
Hence by~\eqref{eq:doublecoaltk*2},
\begin{align} \label{eq:doublecoal3}
\p{\tilde \tau^{n}_{i,j_1} \in (t,t+\delta_n ], \tau^{n}_{j_1,j_2} \le t-\epsilon_n \Big| \mathcal F_{t_k}}
&= \mathcal O(n^{1-(\frac 13 (a_2-3)\wedge \epsilon)}  N^{-1} \delta_n (\log N)^{C+5}).
\end{align}
Therefore, by~\eqref{eq:doublecoal1},~\eqref{eq:doublecoal2} and~\eqref{eq:doublecoal3},
\begin{align*}
&\p{\tilde \tau^{n}_{i,j_1} , \tilde \tau^{n}_{i,j_2} \in (t,t+ \delta_n ]\Big| \mathcal F_{t_k}}\\
&=\mathcal O( n^{9/5}N^{-2} \delta_n (\log N)^{2C}) +\mathcal O(n^{1-\frac 12 (a_2-3)}N^{-1} \delta_n) +\mathcal O(n^{1-(\frac 13 (a_2-3)\wedge \epsilon)} N^{-1}\delta_n (\log N)^{C+5})\\
&= \mathcal O(n^{1-\frac 12 (\frac 13 (a_2-3)\wedge \epsilon)}N^{-1} \delta_n).
\end{align*}
Hence, by~\eqref{eq:doublecoalA} and a union bound, and since $N\ge n^3$,
\begin{align*}
\p{\tilde \tau^{n}_{i,j_1}  , \tilde \tau^{n}_{i,j_2} \in (t_k , t_{k+1}  ]\Big| \mathcal F_{t_k}}
&=\mathcal O(N^{-1}  (\log N)^{2C+5}t_1^2)+ \mathcal O(n^{1-\frac 12 (\frac 13 (a_2-3)\wedge \epsilon)}N^{-1} t_1),
\end{align*}
which completes the proof of the first statement~\eqref{eq:propdoubletk1}.

For the second statement~\eqref{eq:propdoubletk2}, by Proposition~\ref{prop:doublecoal}, for $t\in \delta_n \N \cap [t_k +2K \log N, t_{k+1})$,
\begin{align*}
&\p{\tilde \tau^{n}_{i_1,j_1}, \tilde \tau^{n}_{i_2,j_2} \in (t,t+\delta_n] \Big| \mathcal F_{t_k}}\\
&\le \mathcal O(n^4 N^{-2} \delta_n^2 (\log N)^{2C})
+ \sum_{i,j \in \{i_1,i_2,j_1,j_2\}, i\neq j} \p{\tilde \tau^{n}_{i_1,j_1}, \tilde \tau^n_{i_2,j_2} \in (t,t+\delta_n] , \tau^{n}_{i,j}\le t \Big| \mathcal F_{t_k}}.
\end{align*}
The second statement~\eqref{eq:propdoubletk2} then follows by the same argument as for~\eqref{eq:propdoubletk1}.
\end{proof}

\begin{proof}[Proof of Proposition~\ref{prop:tautautilde}]
Suppose the event $E$ occurs.
By the definition of $c_0$ before~\eqref{eq:Dn+-defn}, we can take
$\epsilon>0$ sufficiently small that $2(1+\frac 14 (1-\alpha))(1-2\epsilon)(\frac 12-c_0)>1$.
For $t\in \delta_n \N_0\cap [0,T_n^-]$ and $x\in I^{n,\epsilon}_{T_n-t}$, 
by conditioning on $\mathcal F_t$,
\begin{align} \label{eq:tildetau*}
&\p{\tau^{n}_{i,j} \in (t,t+ \delta_n], \zeta^{n,i}_{t}=x \Big| \mathcal F_0} \notag \\
&=\E{\p{\tau^{n}_{i,j} \in (t,t+\delta_n] \Big| \mathcal F_{t}}\1_{\tau^{n}_{i,j}>t}\1_{\zeta^{n,i}_{t}=x}\Big|\mathcal F_0} \notag \\
&= \E{\mathcal O( n^2  N^{-1}\delta_n g(x-\mu^n_{T_n-t})^{-1})\1_{\tau^{n}_{i,j}>t} \1_{|\zeta^{n,j}_{t}-x|<Kn^{-1}}\1_{\zeta^{n,i}_{t}=x}\Big|\mathcal F_0} \notag \\
&= \mathcal O(n^2  N^{-1} \delta_n g(x-\mu^n_{T_n-t})^{-1}) \p{|\zeta^{n,j}_{t}-x|<Kn^{-1}, \zeta^{n,i}_{t}=x, \tau^{n}_{i,j}>t \Big|\mathcal F_0},
\end{align} 
where the second equality follows by Proposition~\ref{prop:coal}.
If $t\ge \epsilon_n$, then for $y\in \frac 1n \Z$ with $|y-x|<Kn^{-1}$, by conditioning on $\mathcal F_{t-\epsilon_n}$, and
by~\eqref{eq:lemfromxixj3} in Lemma~\ref{lem:fromxixj},
\begin{align} \label{eq:tildetau**}
&\p{\zeta^{n,j}_{t}=y, \zeta^{n,i}_{t}=x,\tau^{n}_{i,j}>t \Big|\mathcal F_0} \notag \\
&= \E{\p{\zeta^{n,j}_{t}=y, \zeta^{n,i}_{t}=x, \tau^{n}_{i,j}>t \Big|\mathcal F_{t -\epsilon_n }} \1_{\tau^{n}_{i,j}>t -\epsilon_n} 
\1_{|\zeta^{n,j}_{t-\epsilon_n}-y|\le 1}\1_{|\zeta^{n,i}_{t-\epsilon_n}-x|\le 1} \Big|\mathcal F_0} \notag \\
&\le 2 n^{-2} \epsilon_n^{-2}
\p{|\zeta^{n,j}_{t -\epsilon_n}-x|\le 2,|\zeta^{n,i}_{t -\epsilon_n}-x|\le 1,
\tau^{n}_{i,j}>t -\epsilon_n \Big|\mathcal F_0 },
\end{align}
for $n$ sufficiently large, by~\eqref{eq:lemfromxixj2} in Lemma~\ref{lem:fromxixj}.
For $s\ge 0$, let 
$$i^{n,-}_s = \tfrac 1n \Z \cap [\mu^n_s+D^-_n, \mu^n_s-\tfrac 1 {64}\alpha d_n] \quad \text{ and } \quad
i^{n,+}_s = \tfrac 1n \Z \cap [\mu^n_s+\tfrac 1 {64} \alpha d_n, \mu^n_s-(1-\epsilon)D^+_n].$$
Suppose $x\in i^{n,+}_{T_n-t}$.
Since $x\le \mu^n_{T_n-t}+(1-\epsilon)D_n^+$, if $t\ge K \log N+\epsilon_n$ then by~\eqref{eq:propintipstat1} in Proposition~\ref{prop:intip},
\begin{align*}
\p{\zeta^{n,j}_{t -\epsilon_n}\ge x- 2,\zeta^{n,i}_{t -\epsilon_n}\ge x- 1, \tau^{n}_{i,j}>t -\epsilon_n \Big|\mathcal F_0 }
& \le (\log N)^7 e^{-2(1+\frac 14 (1-\alpha))\kappa(x-3-\mu^n_{T_n-t+\epsilon_n})}.
\end{align*}
Therefore, by~\eqref{eq:tildetau*} and~\eqref{eq:tildetau**}, if $t\ge K\log N+\epsilon_n$,
\begin{align} \label{eq:tildetauB}
&\p{\tau^{n}_{i,j} \in (t,t+\delta_n ], \zeta^{n,i}_{t}=x \Big| \mathcal F_0} \notag \\
&\le \mathcal O(n^2  N^{-1} \delta_n g(x-\mu^n_{T_n-t})^{-1}) \cdot 4K  n^{-2} \epsilon_n^{-2} \cdot (\log N)^7 e^{-2(1+\frac 14 (1-\alpha))\kappa(x-3-\mu^n_{T_n-t+\epsilon_n})} \notag \\
&=\mathcal O\left( (\log N)^{11} N^{-1} \delta_n e^{-(1+\frac 12 (1-\alpha))\kappa(x-\mu^n_{T_n-t})}\right)
\end{align}
by the definition of the event $E_1$ in~\eqref{eq:eventE1}, and since $g(z)^{-1}\le 2e^{\kappa z}$ for $z\ge 0$.
By~\eqref{eq:tildetau*} and~\eqref{eq:tildetau**}, if $t\ge \epsilon_n$ and $x\in i^{n,-}_{T_n-t}$,
\begin{align*}
\p{\tau^{n}_{i,j} \in (t,t+\delta_n], \zeta^{n,i}_{t}=x \Big| \mathcal F_0}
&=\mathcal O( n^2 N^{-1} \delta_n) \cdot 4K  n^{-2} \epsilon_n^{-2} \p{|\zeta^{n,i}_{t -\epsilon_n}-x|\le 1 \Big| \mathcal F_0}.
\end{align*}
Therefore, if $t\ge K \log N+\epsilon_n$,
\begin{align*}
\p{\tau^{n}_{i,j} \in (t,t+\delta_n ], \zeta^{n,i}_{t}\in i^{n,-}_{T_n-t} \Big|\mathcal F_0}
&\le \mathcal O(N^{-1} \delta_n \epsilon_n^{-2}) \sum_{x\in i^{n,-}_{T_n-t}} \p{| \zeta^{n,i}_{t-\epsilon_n}-x|\le 1 \Big|\mathcal F_0}\\
&=\mathcal O(nN^{-1} \delta_n \epsilon_n^{-2} (\log N)^{2-2^{-9}\alpha^2   C})
\end{align*}
by~\eqref{eq:propRlogN2} in Proposition~\ref{prop:RlogN} and by the definition of the event $E_1$.
By~\eqref{eq:tildetauB}, we now have that for $t\in \delta_n \N \cap [ K\log N +\epsilon_n,T_n^-]$,
\begin{align} \label{eq:tautautilde4}
&\p{\tau^{n}_{i,j} \in (t,t+\delta_n], |\tilde \zeta^{n,i}_{t}|\ge \tfrac 1 {64} \alpha d_n,
 \zeta^{n,i}_{t} \in I^{n,\epsilon}_{T_n-t} \bigg|\mathcal F_0} \notag \\
&= \mathcal O( n N^{-1}\delta_n  (\log N)^{6-2^{-9} \alpha^2 C})
+ \mathcal O( N^{-1} \delta_n (\log N)^{11}) \sum_{x\in i^{n,+}_{T_n-t}} e^{-(1+\frac 12 (1-\alpha))\kappa(x-\mu^n_{T_n-t})} \notag \\
&= \mathcal O(  n N^{-1}\delta_n (\log N)^{11-2^{-9} \alpha^2 C}).
\end{align}
For $t\in \delta_n \N \cap [ \epsilon_n,T^-_n]$ and $x\in \frac 1n \Z$ with $|x-\mu^n_{T_n-t}|\le \frac 1 {64}\alpha d_n$, by~\eqref{eq:tildetau*} and~\eqref{eq:tildetau**},
\begin{align*}
\p{\tau^{n}_{i,j}\in (t,t+\delta_n], \zeta^{n,i}_{t}=x \Big| \mathcal F_0}
&\le \mathcal O(n^2  N^{-1} \delta_n g(\tfrac 1 {64}\alpha d_n)^{-1}) \cdot 4K \epsilon_n^{-2} n^{-2}\\
&=\mathcal O(N^{-1} \delta_n (\log N)^{4+\frac 1 {64}\alpha C}).
\end{align*}
Therefore, by~\eqref{eq:tautautilde4} and since we chose $C>2^{13}\alpha^{-2}$, for $t\in \delta_n \N \cap [ K \log N+\epsilon_n, T^-_n]$, 
\begin{align} \label{eq:tautautilde3}
\p{\tau^{n}_{i,j} \in (t,t+\delta_n],
 \zeta^{n,i}_{t} \in I^{n,\epsilon}_{T_n-t}  \Big|\mathcal F_0} 
&=\mathcal O( n N^{-1}  \delta_n d_n (\log N)^{4+\frac 1 {64}\alpha C}).
\end{align}
Now note that for any $t\in \delta_n \N_0 \cap [0,T_n^-]$,
\begin{align} \label{eq:tautautilde2}
\p{\tau^{n}_{i,j} \in (t,t+\delta_n], \zeta^{n,i}_{t} \in I^{n,\epsilon}_{T_n-t} \Big| \mathcal F_0}
&=\E{\p{\tau^{n}_{i,j} \in (t,t+\delta_n]\Big| \mathcal F_{t}} \1_{\zeta^{n,i}_{t} \in I^{n,\epsilon}_{T_n-t}} \Big| \mathcal F_0} \notag \\
&=\mathcal O(n^2  N^{-1}\delta_n g(D_n^+)^{-1})
\end{align}
by Proposition~\ref{prop:coal}.
Finally, by~\eqref{eq:lemfromxixj3} in Lemma~\ref{lem:fromxixj},
 for $n$ sufficiently large,
\begin{align} \label{eq:tautautilde1}
& \p{\exists t \in \delta_n \N_0 \cap [0,Nn^{-1} \log N] :\zeta^{n,i}_{t} \wedge \zeta^{n,j}_{t} \notin I^{n,\epsilon}_{T_n-t} , \tau^{n}_{i,j}>t \Big| \mathcal F_0} \notag \\
&\le \p{\exists t \in t^*\N_0 \cap [0,Nn^{-1} \log N] :\tilde \zeta^{n,i}_{t} \wedge \tilde \zeta^{n,j}_{t}\ge (1-2\epsilon)D^+_n , \tau^{n}_{i,j}>t \Big| \mathcal F_0} \notag \\
&\quad + \p{\exists t\in\delta_n \N_0\cap [0,Nn^{-1}\log N]:\tilde \zeta^{n,i}_{t} \wedge \tilde \zeta^{n,j}_{t}\le  D_n^-  \Big| \mathcal F_0} \notag \\
&\le ( (t^*)^{-1} Nn^{-1} \log N+1)
(\log N)^7 e^{2(1+\frac 14 (1-\alpha))\kappa(K_0-(1-2\epsilon)D_n^+ -1)}+2N^{-1} \notag \\
& \le N^{-\epsilon'}
\end{align}
for some $\epsilon '>0$, 
where the second inequality follows by~\eqref{eq:propintipstat1} and~\eqref{eq:propintipstat*} in Proposition~\ref{prop:intip} and~\eqref{eq:propRlogN1} in Proposition~\ref{prop:RlogN}, and the last inequality since we chose $\epsilon>0$ sufficiently small that $2(1+\frac 14(1-\alpha))(1-2\epsilon)(\frac 12 -c_0)>1$
and since $\kappa D^+_n =(1/2-c_0)\log (N/n)$.
Hence by a union bound,
\begin{align} \label{eq:tautautildeconc}
&\p{\{\tau^{n}_{i,j}\neq \tilde \tau^{n}_{i,j}\}\cap \{\tau^n_{i,j}\le N n^{-1} \log N\} \Big| \mathcal F_0} \notag\\
&\le \p{\exists t \in \delta_n \N_0 \cap [0, N n^{-1} \log N]:
 \zeta^{n,i}_{t} \wedge \zeta^{n,j}_{t}\notin I^{n,\epsilon}_{T_n-t}, \tau^n_{i,j}>t \Big| \mathcal F_0} \notag \\
&\quad + \sum_{\{k\in \N_0:t_k \le N n^{-1} \log N\}} \sum_{t\in \delta_n \N_0\cap [t_k,t_k+2K \log N), i'\in \{i,j\}} \p{\tau^{n}_{i,j} \in (t,t+ \delta_n] , \zeta^{n,i'}_{t}\in I^{n,\epsilon}_{T_n-t}  \Big| \mathcal F_0} \notag \\
&\quad + \sum_{t\in \delta_n \N \cap [2K\log N, Nn^{-1} \log N], i'\in \{i,j\}}
\p{ \tau^{n}_{i,j} \in (t,t+\delta_n], 
|\tilde \zeta^{n,i'}_{t}|\ge \tfrac 1 {64}\alpha d_n, \zeta^{n,i'}_t \in I^{n,\epsilon}_{T_n-t} \Big| \mathcal F_0 }
\notag \\
&\le N^{-\epsilon '}+\mathcal O(n^2 N^{-1}g(D_n^+)^{-1}\log N) +
\mathcal O( n  N^{-1} d_n (\log N)^{4+\frac 1 {64}\alpha C} \cdot  N n^{-1} (\log N)^{2-C})\notag \\
&\qquad +\mathcal O(  n N^{-1}(\log N)^{11-2^{-9} \alpha^2 C} \cdot  N n^{-1} \log N) \notag \\
&\le \tfrac 12 (\log N)^{-2}
\end{align}
for $n$ sufficiently large, where the second inequality follows by~\eqref{eq:tautautilde1},~\eqref{eq:tautautilde2},~\eqref{eq:tautautilde3} and~\eqref{eq:tautautilde4},
and the last inequality since we chose $C>2^{13}\alpha^{-2}$ and so $2^{-9}\alpha^2 C -12>2$ and $\frac 12 C-6>2$,
and since $g(D_n^+)^{-1}\le 2e^{\kappa D_n^+}=\mathcal O\left( (\frac N n )^{1/2-c_0}\right)$ and $N\ge n^3$.
By a union bound and Proposition~\ref{prop:tauk}, for $n$ sufficiently large,
\begin{align*}
&\p{\tau^n_{i,j}>Nn^{-1} \log N \Big| \mathcal F_0}\\
&\le \p{\exists t\in \delta_n \N_0 \cap [0,Nn^{-1} \log N]: \zeta^{n,i}_t \wedge \zeta^{n,j}_t \notin I^{n,\epsilon}_{T_n-t}, \tau^n_{i,j}>t \Big| \mathcal F_0}
+(1-\tfrac 12 \beta_n)^{\lfloor (t_1)^{-1} N n^{-1} \log N \rfloor}\\
&\le \tfrac 12 (\log N)^{-2},
\end{align*}
for $n$ sufficiently large,
by~\eqref{eq:tautautilde1} and the definition of $\beta_n$ in~\eqref{eq:betadefn}.
By~\eqref{eq:tautautilde1} and~\eqref{eq:tautautildeconc}, this completes the proof.
\end{proof}

\subsection{Proof of Proposition~\ref{prop:intip}}
\label{subsec:tipbulkproofs}
Throughout the rest of Section~\ref{sec:mainproof}, we assume for some $a_1>1$, $N\ge n^{a_1}$ for $n$ sufficiently large.
We need two preliminary lemmas for the proof of Proposition~\ref{prop:intip}.
The first is an easy consequence of the definition of the event $E'_2$.
\begin{lemma} \label{lem:jumpbound}
For $n$ sufficiently large,
on the event $E_1\cap E'_2$,
for $t\in \delta_n \N_0 \cap [0,T^-_n]$, $i,j\in [k_0]$ and
$\ell_1 , \ell_2 \in \frac 1n \Z\cap [K,D^+_n]$, if
$ \zeta^{n,i}_t, \zeta^{n,j}_t \in I^n_{T_n-t}$,
\begin{align*}
\p{\tilde \zeta^{n,i}_{t+t^*}\ge \ell_1, \tilde \zeta^{n,j}_{t+t^*}\ge \ell_2 \Big| \mathcal F_t}\1_{\tau^n_{i,j}>t}
& \le c_1 e^{-(1+\frac 12 (1-\alpha))\kappa(\ell_1+1 -(\tilde \zeta^{n,i}_t \vee  K)+\ell_2+1 -(\tilde \zeta^{n,j}_t \vee  K))}\\
\text{and }\qquad \p{\tilde \zeta^{n,i}_{t+t^*}\ge \ell_1 \Big| \mathcal F_t}
& \le c_1 e^{-(1+\frac 12 (1-\alpha))\kappa(\ell_1+1-(\tilde \zeta^{n,i}_t \vee  K))}.
\end{align*}
\end{lemma}
\begin{proof}
Write $t'=T_n-(t+t^*)$.
By the definition of $q^{n,+}$ in~\eqref{eq:qn+-defn}, and the definition of $\tilde \zeta^{n,i}$ and $\tilde \zeta^{n,j}$ in~\eqref{eq:zetadefns}, for
$\ell_1, \ell_2 \in \frac 1n \Z$, if $\tau^n_{i,j}>t $,
\begin{align} \label{eq:tracerformula}
\p{\tilde \zeta^{n,i}_{t+t^*}\ge \ell_1, \tilde \zeta^{n,j}_{t+t^*}\ge \ell_2 \Big| \mathcal F_t } 
&\le  \frac{q^{n,+}_{t',t'+t^*}(\ell_1+\mu^n_{t'},\zeta^{n,i}_t)}{p^n_{t'+t^*}(\zeta^{n,i}_t)}
\frac{q^{n,+}_{t',t'+t^*}(\ell_2+\mu^n_{t' },\zeta^{n,j}_t)}{p^n_{t'+t^*}(\zeta^{n,j}_t)-N^{-1} \1_{\zeta^{n,j}_t =\zeta^{n,i}_t}}.
\end{align}
By the definition of the event $E'_2$ in~\eqref{eq:eventE'2}, for $\ell  \in I^n_{t'}$ and $z\in I^n_{t'+t^*}$ with $\ell -\mu^n_{t'}\ge K$, the event $A^{(2)}_{t'}( \ell , z)$ occurs, and so
\begin{align*} 
\frac{q^{n,+}_{t',t'+t^*}(\ell,z)}{p^n_{t'+t^*}(z)} \le c_1 e^{-(1+\frac 12 (1-\alpha))\kappa(\ell -(z-\nu t^*)\vee (\mu^n_{t'}+K)+2)}.
\end{align*}
Note that by the definition of the event $E_1$ in~\eqref{eq:eventE1}, if $\zeta^{n,j}_t \in I^n_{t'+t^*}$ then $p^n_{t'+t^*}(\zeta^{n,j}_t)\ge \frac 1 {10}\left( \frac n N \right)^{1/2}$.
Therefore by~\eqref{eq:tracerformula},
if $\tau^n_{i,j}>t$ and $ \zeta^{n,i}_t, \zeta^{n,j}_t \in I^n_{T_n-t}$, for $\ell_1,\ell_2 \in \frac 1n \Z \cap [K,D_n^+]$,
\begin{align} \label{eq:fromtipB}
&\p{\tilde \zeta^{n,i}_{t+t^*}\ge \ell_1, \tilde \zeta^{n,j}_{t+t^*}\ge \ell_2 \Big| \mathcal F_t}
\notag \\
&\le (1+\mathcal O(N^{-1/2}))c_1^2 e^{-(1+\frac 12 (1-\alpha))\kappa((\ell_1+ \mu^n_{t'})-(\zeta^{n,i}_t -\nu  t^*) \vee (\mu^n_{t'}+K)+2+(\ell_2+ \mu^n_{t'})-(\zeta^{n,j}_t -\nu  t^*) \vee (\mu^n_{t'}+K)+2)} \notag \\
&\le  (1+\mathcal O(N^{-1/2})) c_1^2 e^{-(1+\frac 12 (1-\alpha))\kappa((\ell_1- \tilde \zeta^{n,i}_t\vee K)-t^* e^{-(\log N)^{c_2}}+2
+(\ell_2- \tilde \zeta^{n,j}_t\vee K)-t^* e^{-(\log N)^{c_2}}+2)},
\end{align}
since, by the definition of the event $E_1$, $|(\mu^n_{t'}-\nu t^*)-\mu^n_{T_n-t}|\le t^* e^{-(\log N)^{c_2}}$. Since $c_1<1$, the first statement follows by taking $n$ sufficiently large.
The second statement follows by the same argument.
\end{proof}
We now use Lemma~\ref{lem:jumpbound} and an inductive argument to prove the following result.
\begin{lemma} \label{lem:intip}
For $t\in \delta_n \N_0 \cap [0,T_n^-]$ and $k\in [k_0]$, let
\begin{equation} \label{eq:tau+kdefn}
\tau^{+,k}_t=\inf\left\{s\ge t: s-t\in t^* \N_0, \tilde \zeta^{n,k}_{s} \ge D^+_n \right\}.
\end{equation}
Take $i,j \in [k_0]$ and let $\tau^{+}_t = \tau^{+,i}_t \wedge \tau^{+,j}_t \wedge \tau^n_{i,j}$.
On the event $E_1\cap E'_2 $, for $s\in [0,T^-_n]$ with $s-t\in t^* \N_0$, for $\ell_1,\ell_2 \in \N\cap [K,D^+_n]$,
\begin{align} \label{eq:indhyptail}
\p{\tilde \zeta^{n,i}_{s} \ge \ell_1, \tilde \zeta^{n,j}_{s} \ge \ell_2,  \tau_t^+\ge s \Big| \mathcal F_t } 
& \le 
e^{(1+\frac 14 (1-\alpha))\kappa (\tilde \zeta^{n,i}_t \vee K-\ell_1 + \tilde \zeta^{n,j}_t \vee K-\ell_2)} \\
\text{ and for }i'\in \{i,j\}, \quad 
\p{\tilde \zeta^{n,i'}_{s} \ge \ell_1 , \tau^{+,i'}_t\ge s \Big| \mathcal F_t }
&\le 
e^{(1+\frac 14 (1-\alpha))\kappa (\tilde \zeta^{n,i'}_t \vee K-\ell_1)} . \label{eq:indhyptail2}
\end{align}
\end{lemma}

\begin{proof}
Let $\lambda = \frac 14 (1-\alpha)$, and recall from~\eqref{eq:cchoice} that we chose $c_1>0$ sufficiently small that
\begin{equation} \label{eq:delta_cond}
\begin{aligned} 
c_1 ((e^{\lambda \kappa}-1)^{-1}e^{\lambda \kappa }+e^{-(1+\lambda)\kappa}(1-e^{-(1+\lambda )\kappa})^{-1})^2 +e^{-2(1+\lambda )\kappa} &< 1  \\
\text{ and }\quad c_1 (e^{\lambda \kappa}-1)^{-1}e^{\lambda \kappa} +e^{-(1+\lambda)\kappa} &< 1.
\end{aligned}
\end{equation}
The proof is by induction.
Take $t' \in [0,T^-_n]$ with $t'-t\in t^*\N_0$, and
suppose~\eqref{eq:indhyptail} and~\eqref{eq:indhyptail2} hold for $s=t'$.
Let $A=e^{(1+\lambda)\kappa(\tilde \zeta^{n,i}_t \vee K + \tilde \zeta^{n,j}_t \vee K)}$.
Note that by~\eqref{eq:lemfromxixj3} in Lemma~\ref{lem:fromxixj}, if $\tau^+_t > t'$ then $\zeta^{n,i}_{t'}, \zeta^{n,j}_{t'}\in I^n_{T_n-t'}$.
For $\ell_1,\ell_2 \in \N \cap [K, D^+_n]$,
let $J_{\ell_1,\ell_2}=\{(k_1,k_2): k_1,k_2 \in \N \cap (K,D^+_n], k_1\le \ell_1 \text{ or } k_2\le \ell_2\}$.
Then by Lemma~\ref{lem:jumpbound} and a union bound,
\begin{align*}
&\p{\tilde \zeta^{n,i}_{t'+t^*} \ge \ell_1, \tilde \zeta^{n,j}_{t'+t^*} \ge \ell_2, \tau^+_t\ge t'+t^* \Big| \mathcal F_t }\\
&\le \sum_{(k_1,k_2)\in J_{\ell_1,\ell_2}} 
c_1 e^{-(1+2\lambda)\kappa((\ell_1-k_1 )\vee 0+(\ell_2-k_2)\vee 0)} 
\p{\tilde \zeta^{n,i}_{t'}\in [ k_1,k_1+1) , \tilde \zeta^{n,j}_{t'}\in [k_2,k_2+1),
\tau^+_t>t' \Big| \mathcal F_t}\\
&\quad +\sum_{k\in \N \cap (K,D_n^+]} \Big(c_1 e^{-(1+2\lambda)\kappa((\ell_1-k) \vee 0+\ell_2-K)}
\p{\tilde \zeta^{n,i}_{t'}\in [k,k+1), \tau^{+,i}_t >t' \Big| \mathcal F_t}\\
&\hspace{4cm} +c_1 e^{-(1+2\lambda)\kappa ((\ell_2-k)\vee 0+\ell_1-K)}
\p{\tilde \zeta^{n,j}_{t'}\in [ k,k+1) , \tau^{+,j}_t >t' \Big| \mathcal F_t}\Big)
\\
&\quad + c_1 e^{-(1+2\lambda)\kappa (\ell_1-K+\ell_2-K)} 
+\p{\tilde \zeta^{n,i}_{t'}\ge \ell_1+1, \tilde \zeta^{n,j}_{t'}\ge \ell_2+1,
\tau^+_t>t' \Big| \mathcal F_t}\\
&\le \sum_{k_1,k_2\in \N \cap [K ,D^+_n]}Ae^{-(1+\lambda)\kappa (k_1+k_2)}c_1  e^{-(1+2\lambda)\kappa((\ell_1-k_1)\vee 0+(\ell_2-k_2)\vee   0)}
+A e^{-(1+\lambda)\kappa (\ell_1+\ell_2+2)}
\end{align*}
by the induction hypothesis and since by the definition of $A$, $e^{(1+\lambda)\kappa(\tilde \zeta^{n,i'}_t\vee K)}\le A e^{-(1+\lambda)\kappa K}$
for $i'\in \{i,j\}$ and $Ae^{-(1+\lambda)2\kappa K}\ge 1$.
Therefore
\begin{align} \label{eq:lemind1}
&\p{\tilde \zeta^{n,i}_{t'+t^*} \ge \ell_1, \tilde \zeta^{n,j}_{t'+t^*} \ge \ell_2, \tau^+_t \ge t'+t^* \Big| \mathcal F_t } \notag \\
&\le A c_1   \left(\sum_{k_1= K}^{\ell_1} e^{-(1+\lambda)\kappa k_1}e^{-(1+2\lambda)\kappa (\ell_1-k_1)}
+\sum_{k_1=\ell_1+1}^{\lfloor D^+_n \rfloor }e^{-(1+\lambda)\kappa k_1}\right) \notag \\
&\qquad \cdot \left(\sum_{k_2=K}^{\ell_2} e^{-(1+\lambda) \kappa k_2}e^{-(1+2\lambda)\kappa (\ell_2-k_2)}
+\sum_{k_2=\ell_2+1}^{\lfloor D^+_n \rfloor}e^{-(1+\lambda)\kappa k_2}\right)
+Ae^{-(1+\lambda)\kappa (\ell_1+\ell_2+2)}.
\end{align}
Note that
\begin{align*}
\sum_{k_1=K}^{\ell_1} e^{-(1+\lambda)\kappa k_1}e^{-(1+2\lambda)\kappa (\ell_1-k_1)}
<\sum_{k_1=0}^{\ell_1} e^{-(1+2\lambda)\kappa \ell_1}e^{\lambda \kappa k_1} 
&<e^{-(1+2\lambda)\kappa \ell_1}(e^{\lambda \kappa }-1)^{-1} e^{\lambda \kappa (\ell_1 +1)}\\
&=(e^{\lambda \kappa}-1)^{-1}e^{\lambda \kappa}e^{-(1+\lambda) \kappa \ell_1}.
\end{align*}
Hence, since $\sum_{k_1=\ell_1+1}^{\lfloor D^+_n \rfloor }e^{-(1+\lambda)\kappa k_1}<(1-e^{-(1+\lambda)\kappa})^{-1} e^{-(1+\lambda)\kappa(\ell_1+1)}$, substituting into~\eqref{eq:lemind1},
\begin{align*}
&\p{\tilde \zeta^{n,i}_{t'+t^*} \ge \ell_1, \tilde \zeta^{n,j}_{t'+t^*} \ge \ell_2, \tau_t^+\ge t'+t^* \Big| \mathcal F_t }\\
& \le A e^{-(1+\lambda)\kappa(\ell_1+\ell_2)} \left(c_1 ((e^{\lambda \kappa}-1)^{-1}e^{\lambda \kappa}+e^{-(1+\lambda)\kappa}(1-e^{-(1+\lambda )\kappa})^{-1})^2 +e^{-2(1+\lambda )\kappa}\right)\\
&\le A e^{-(1+\lambda)\kappa(\ell_1+\ell_2)}
\end{align*}
by~\eqref{eq:delta_cond}.
Similarly, letting $A_1=e^{(1+\lambda)\kappa(\tilde \zeta^{n,i}_t \vee K) }$,
for $\ell \in \N\cap [K,D^+_n]$, by Lemma~\ref{lem:jumpbound} and a union bound,
\begin{align*}
\p{\tilde \zeta^{n,i}_{t'+t^*} \ge \ell, \tau^{+,i}_t\ge t'+t^* \Big| \mathcal F_t }
&\le \sum_{k\in \N \cap (K, \ell]} 
c_1 e^{-(1+2\lambda)\kappa(\ell -k )} 
\p{\tilde \zeta^{n,i}_{t'}\in [ k,k+1), \tau^{+,i}_t>t' \Big| \mathcal F_t}\\
&\qquad +c_1 e^{-(1+2\lambda)\kappa (\ell -K)}
+\p{\tilde \zeta^{n,i}_{t'}\ge \ell+1, \tau^{+,i}_t>t' \Big| \mathcal F_t}\\
&\le \sum_{k\in \N \cap [K, \ell]} 
c_1 e^{-(1+2\lambda)\kappa(\ell -k )} 
A_1 e^{-(1+\lambda )\kappa k}
+A_1 e^{-(1+\lambda )\kappa (\ell+1)}
\end{align*}
by the induction hypothesis and since $A_1 e^{-(1+\lambda)\kappa K}\ge 1$.
Hence
\begin{align*}
\p{\tilde \zeta^{n,i}_{t'+t^*} \ge \ell, \tau^{+,i}_t\ge t'+t^* \Big| \mathcal F_t }
&\le A_1 \left(c_1 e^{-(1+2\lambda)\kappa \ell } (e^{\lambda \kappa}-1)^{-1} e^{\lambda \kappa (\ell +1)}+e^{-(1+\lambda )\kappa (\ell+1)}\right)\\
&= A_1 e^{-(1+\lambda )\kappa \ell} (c_1 (e^{\lambda \kappa}-1)^{-1} e^{\lambda \kappa}+e^{-(1+\lambda )\kappa} )\\
& \le  A_1 e^{-(1+\lambda )\kappa \ell}
\end{align*}
by~\eqref{eq:delta_cond}.
By the same argument, $\p{\tilde \zeta^{n,j}_{t'+t^*} \ge \ell, \tau^{+,j}_t\ge t'+t^* \Big| \mathcal F_t }
\le  e^{(1+\lambda )\kappa (\tilde \zeta^{n,j}_t \vee K- \ell)}$.
The result follows by induction.
\end{proof}

\begin{proof}[Proof of Proposition~\ref{prop:intip}]
If $t-s\ge K\log N$, for $i'\in \{i,j\}$, let 
$$
\sigma_{i'} =\inf\{s' : s'-(t-t^*\lfloor (t^*)^{-1}K \log N \rfloor)\in t^* \N_0, \tilde \zeta^{n,i'}_{s'}\le K \}.
$$
If instead $t-s < K \log N$ with $t-s\in t^* \N_0$, then let $\sigma_{i'} = s$ for $i'\in \{i,j\}$.
Note that in both cases $t-\sigma_{i'}\le K \log N$.
Let $\lambda = \frac 14 (1-\alpha)$.

Condition on $\mathcal F_{\sigma_i \vee \sigma_j}$ and suppose $\sigma_i\le \sigma_j \le t$. Recall the definition of $\tau^{+,i}_{\sigma_j}$ and $\tau^{+,j}_{\sigma_j}$ in~\eqref{eq:tau+kdefn}.
Then
 for $\ell_1,\ell_2 \in\N\cap [K,D_n^+]$, by a union bound and Lemma~\ref{lem:intip},
\begin{align} \label{eq:intipstar}
&\p{\tilde \zeta^{n,i}_t \ge \ell_1, \tilde \zeta^{n,j}_t \ge \ell_2, \tau^n_{i,j}> t \Big| \mathcal F_{\sigma_i \vee \sigma_j} } \notag \\
&\le e^{(1+\lambda )\kappa (\tilde \zeta^{n,i}_{\sigma_j} \vee K -\ell_1 +\tilde \zeta^{n,j}_{\sigma_j} \vee K-\ell_2)}
+\p{\tilde \zeta^{n,i}_t \ge \ell_1, \tau^n_{i,j}> t, \tau^{+,i}_{\sigma_j}\ge t, \tau^{+,j}_{\sigma_j}<t \Big| \mathcal F_{\sigma_i \vee \sigma_j} } \notag \\
&\quad +\p{\tilde \zeta^{n,j}_t \ge \ell_2, \tau^n_{i,j}> t, \tau^{+,j}_{\sigma_j}\ge t, \tau^{+,i}_{\sigma_j}<t \Big| \mathcal F_{\sigma_i \vee \sigma_j}  } 
 +\p{\tau^n_{i,j}> t, \tau^{+,i}_{\sigma_j}< t, \tau^{+,j}_{\sigma_j}<t \Big| \mathcal F_{\sigma_i \vee \sigma_j}    }.
\end{align}
We now bound the last three terms on the right hand side.
Recall that we let $\tau^+_{\sigma_j}=\tau^{+,i}_{\sigma_j}\wedge \tau^{+,j}_{\sigma_j}\wedge \tau^n_{i,j}$.
For $s' \in [\sigma_j,t]$ with $s'-\sigma_j \in t^* \N_0$, by conditioning on $\mathcal F_{s'}$,
\begin{align*}
&\p{\tilde \zeta^{n,i}_t \ge \ell_1, \tau^n_{i,j}> t, \tau^{+,i}_{\sigma_j}\ge t, \tau^{+,j}_{\sigma_j}=s' \Big| \mathcal F_{\sigma_i \vee \sigma_j} } \\
& \le \E{\p{\tilde \zeta^{n,i}_t \ge \ell_1, \tau^{+,i}_{s'} \ge t\Big| \mathcal F_{s'}}
\1_{\tilde \zeta^{n,j}_{s'} \ge D^+_n, \tau^+_{\sigma_j}=s'} \Big| \mathcal F_{\sigma_i \vee \sigma_j}}\\
&\le \sum_{ \ell '_1=K}^{ \ell_1-1}\p{\tilde \zeta^{n,i}_{s'} \in [ \ell '_1,\ell'_1+1), \tilde \zeta^{n,j}_{s'} \ge D^+_n, \tau^+_{\sigma_j}\ge {s'} \Big| \mathcal F_{\sigma_i \vee \sigma_j}}
\cdot e^{(1+\lambda )\kappa (\ell'_1+1-\ell_1)}\\
&\qquad +\p{\tilde \zeta^{n,i}_{s'}\le K,\tilde \zeta^{n,j}_{s'} \ge D^+_n, \tau^+_{\sigma_j}\ge s' \Big| \mathcal F_{\sigma_i \vee \sigma_j}}
\cdot e^{(1+\lambda )\kappa(K-\ell_1)}\\
&\qquad + \p{\tilde \zeta^{n,i}_{s'} \ge \ell_1, \tilde \zeta^{n,j}_{s'} \ge D^+_n, \tau^+_{\sigma_j}\ge {s'} \Big| \mathcal F_{\sigma_i \vee \sigma_j}}
\end{align*}
by~\eqref{eq:indhyptail2} in Lemma~\ref{lem:intip}.
Therefore, by Lemma~\ref{lem:intip} again,
\begin{align} \label{eq:zetatailtau+*}
&\p{\tilde \zeta^{n,i}_t \ge \ell_1, \tau^n_{i,j}> t, \tau^{+,i}_{\sigma_j}\ge t, \tau^{+,j}_{\sigma_j}=s' \Big| \mathcal F_{\sigma_i \vee \sigma_j} } \notag  \\
&\le  \sum_{\ell '_1=K}^{\ell_1}
e^{(1+\lambda )\kappa (\tilde \zeta^{n,i}_{\sigma_j} \vee K-\ell '_1 +\tilde \zeta^{n,j}_{\sigma_j} \vee K -\lfloor D_n^+ \rfloor )}
\cdot e^{(1+\lambda )\kappa (\ell'_1+1-\ell_1)}
+ e^{(1+\lambda )\kappa (\tilde \zeta^{n,j}_{\sigma_j} \vee K -\lfloor D^+_n \rfloor)}
\cdot e^{(1+\lambda )\kappa (K-\ell_1)} \notag \\
&\le e^{(1+\lambda )\kappa (\tilde \zeta^{n,i}_{\sigma_j} \vee K +\tilde \zeta^{n,j}_{\sigma_j} \vee K)}
 (\ell_1 e^{-(1+\lambda)\kappa(\ell_1 +\lfloor D^+_n \rfloor -1)}+e^{-(1+\lambda)\kappa(\ell_1+\lfloor D^+_n \rfloor)}) \notag \\
&\le  e^{(1+\lambda )\kappa(\tilde \zeta^{n,i}_{\sigma_j} \vee K +\tilde \zeta^{n,j}_{\sigma_j} \vee K +1)} e^{-(1+\lambda )\kappa (\ell_1+\lfloor D_n^+ \rfloor)}(D^+_n+1),
\end{align}
since $\ell_1 \le D_n^+$.
Therefore, for $n$ sufficiently large, since $t-\sigma_j\le K \log N$, 
\begin{align} \label{eq:propintipA}
\p{\tilde \zeta^{n,i}_t \ge \ell_1, \tau^n_{i,j}> t, \tau^{+,i}_{\sigma_j}\ge t, \tau^{+,j}_{\sigma_j}<t \Big| \mathcal F_{\sigma_i \vee \sigma_j} }   
&\le  e^{(1+\lambda )\kappa(\tilde \zeta^{n,i}_{\sigma_j} \vee K - \ell_1 +\tilde \zeta^{n,j}_{\sigma_j} \vee K -\lfloor D_n^+ \rfloor +1)} K \kappa^{-1} (\log N)^2,
\end{align}
and by the same argument,
\begin{align} \label{eq:propintipB}
\p{\tilde \zeta^{n,j}_t \ge \ell_2, \tau^n_{i,j}> t, \tau^{+,j}_{\sigma_j}\ge t, \tau^{+,i}_{\sigma_j}<t \Big| \mathcal F_{\sigma_i \vee \sigma_j} }  
&\le  e^{(1+\lambda )\kappa (\tilde \zeta^{n,i}_{\sigma_j} \vee K -\lfloor D^+_n \rfloor +\tilde \zeta^{n,j}_{\sigma_j} \vee K -\ell_2 +1)} K \kappa^{-1} (\log N)^2 .
\end{align}
For the last term on the right hand side of~\eqref{eq:intipstar}, note that for $\sigma_j\le s_1\le s_2\le t$ with $s_1-\sigma_j, s_2 -\sigma_j \in t^* \N_0$, by the same argument as for~\eqref{eq:zetatailtau+*},
\begin{align} \label{eq:zetataildagger}
\p{ \tau^n_{i,j}> t, \tau^{+,i}_{\sigma_j} =s_1, \tau^{+,j}_{\sigma_j}=s_2 \Big| \mathcal F_{\sigma_i \vee \sigma_j }}
&\le \p{ \tau^n_{i,j}>s_2, \tau^{+,i}_{\sigma_j} =s_1, \tau^{+,j}_{\sigma_j}\ge s_2, \tilde \zeta^{n,j}_{s_2}\ge \lfloor D^+_n \rfloor \Big| \mathcal F_{\sigma_i \vee \sigma_j }} \notag \\
&\le e^{(1+\lambda )\kappa (\tilde \zeta^{n,i}_{\sigma_j} \vee K -\lfloor D_n^+ \rfloor +\tilde \zeta^{n,j}_{\sigma_j} \vee K -\lfloor D_n^+\rfloor +1)} (D^+_n+1),
\end{align}
and by the same argument~\eqref{eq:zetataildagger} also holds for $s_1 \ge s_2$.
Hence by~\eqref{eq:intipstar},~\eqref{eq:propintipA} and~\eqref{eq:propintipB}, for $n$ sufficiently large, if $\sigma_i\le \sigma_j \le t$ then for $\ell_1,\ell_2 \in \N \cap [K, D_n^+]$,
\begin{align} \label{eq:propintip1}
\p{\tilde \zeta^{n,i}_t \ge \ell_1, \tilde \zeta^{n,j}_t \ge \ell_2, \tau^n_{i,j}> t \Big| \mathcal F_{\sigma_i \vee \sigma_j} } 
&\le e^{(1+\lambda )\kappa(\tilde \zeta^{n,i}_{\sigma_j} \vee 0 -\ell_1 +\tilde \zeta^{n,j}_{\sigma_j} \vee 0 -\ell_2 )}(\log N)^4.
\end{align}
By a simpler version of the same argument, for $i'\in \{i,j\}$ and $\ell \in \N \cap [K,D_n^+]$,
if $\sigma_i \le \sigma_j\le t$ then
\begin{align} \label{eq:propintiponel}
&\p{\tilde \zeta^{n,i'}_t \ge \ell \Big| \mathcal F_{\sigma_i \vee \sigma_j}}\notag \\
&\le \p{\tilde \zeta^{n,i'}_t \ge \ell, \tau^{+, i'}_{\sigma_j}\ge t \Big| \mathcal F_{\sigma_i \vee \sigma_j}}
+\sum_{s'\in [\sigma_j, t), s'-\sigma_j \in t^* \N_0} \p{\tilde \zeta^{n,i'}_{s'} \ge D_n^+, \tau^{+, i'}_{\sigma_j}\ge s' \Big| \mathcal F_{\sigma_i \vee \sigma_j}} \notag \\
&\le (\log N)^2 e^{(1+\lambda) \kappa(\tilde \zeta^{n,i'}_{\sigma_j}\vee 0-\ell)}
\end{align}
for $n$ sufficiently large, by~\eqref{eq:indhyptail2} in Lemma~\ref{lem:intip}.
Since we let $\sigma_i=\sigma_j =s$ in the case $t-s<K\log N$, this
completes the proof of~\eqref{eq:propintipstat*} and~\eqref{eq:propintipstat3}.

From now on, assume $t-s\ge K \log N$. 
Condition on $\mathcal F_{\sigma_i \wedge \sigma_j}$ and suppose $\sigma_i \wedge \sigma_j = \sigma_i \le t$; then
\begin{align} \label{eq:propintipsigmai}
&\E{e^{(1+\lambda)\kappa(\tilde \zeta^{n,i}_{\sigma_j}\vee 0)}\1_{\tau^{+,i}_{\sigma_i}> \sigma_j}\1_{\sigma_j \le t} \Big| \mathcal F_{\sigma_i\wedge \sigma_j}} \notag \\
&\le e^{(1+\lambda)\kappa K}
+ \sum_{\ell =K}^{\lfloor D^+_n \rfloor} e^{(1+\lambda)\kappa (\ell +1)} \sum_{s'-\sigma_i \in t^*\N_0, \, s'\le t} \p{\tilde \zeta^{n,i}_{s'} \in [\ell, \ell+1) , \tau^{+,i}_{\sigma_i}\ge s' \Big| \mathcal F_{\sigma_i\wedge \sigma_j}} \notag \\
&\le e^{(1+\lambda)\kappa K}
+ \sum_{\ell =K}^{\lfloor D^+_n \rfloor} e^{(1+\lambda)\kappa (\ell +1)} ((t^*)^{-1}K\log N +1) e^{(1+\lambda)\kappa (\tilde \zeta^{n,i}_{\sigma_i} \vee K-\ell)} \notag \\
&\le e^{(1+\lambda)\kappa (1+K)}K \kappa^{-1} (\log N)^2
\end{align}
for $n$ sufficiently large,
where the second inequality follows by~\eqref{eq:indhyptail2} in Lemma~\ref{lem:intip} and since $t-\sigma_i \le K \log N$, and the last inequality since $\tilde \zeta^{n,i}_{\sigma_i}\le K$.
Therefore, if $\sigma_i \wedge \sigma_j =\sigma_i \le t$, 
by conditioning on $\mathcal F_{\sigma_i \vee \sigma_j}$, and then by~\eqref{eq:propintip1},~\eqref{eq:propintiponel} and~\eqref{eq:propintipsigmai}, and
since $\tilde \zeta^{n,j}_{\sigma_j}\le K$,
\begin{align} \label{eq:propintipB2}
&\p{\tilde \zeta^{n,i}_t \ge \ell_1, \tilde \zeta^{n,j}_t \ge \ell_2, \tau^n_{i,j}> t \Big| \mathcal F_{\sigma_i \wedge \sigma_j} } \notag \\
&\le \E{\p{\tilde \zeta^{n,i}_t \ge \ell_1, \tilde \zeta^{n,j}_t \ge \ell_2, \tau^n_{i,j}> t \Big| \mathcal F_{\sigma_i \vee \sigma_j} }\1_{\sigma_j \le t}(\1_{\tau^{+,i}_{\sigma_i}>\sigma_j}+\1_{\tau^{+,i}_{\sigma_i}\le \sigma_j}) \Big| \mathcal F_{\sigma_i \wedge \sigma_j} } \notag \\
&\qquad + \p{\sigma_j >t \big| \mathcal F_{\sigma_i \wedge \sigma_j}} \notag \\
&\le e^{(1+\lambda )\kappa (1+2K)} K \kappa^{-1} (\log N)^2 \cdot (\log N)^4 e^{-(1+\lambda)\kappa (\ell_1+\ell_2)} \notag \\
&\quad +\E{(\log N)^2 e^{(1+\lambda)\kappa (K-\ell_2)} \1_{\sigma_j \le t} \1_{\tau^{+,i}_{\sigma_i}\le \sigma_j} \Big| \mathcal F_{\sigma_i \wedge \sigma_j}}
+\p{ \sigma_j  >t \big| \mathcal F_{\sigma_i \wedge \sigma_j}}.
\end{align}
By~\eqref{eq:indhyptail2} in Lemma~\ref{lem:intip}, if $\sigma_i \wedge \sigma_j = \sigma_i \le t$, then since $\tilde \zeta^{n,i}_{\sigma_i}\le K$,
\begin{align} \label{eq:tau+t}
\p{\tau^{+,i}_{\sigma_i}\le t \Big| \mathcal F_{\sigma_i \wedge \sigma_j}}
&\le \sum_{s'\le t, \, s' -\sigma_i \in t^* \N_0} \p{\tau^{+,i}_{\sigma_i}\ge s', \tilde \zeta ^{n,i}_{s'}\ge D_n^+ \Big| \mathcal F_{\sigma_i \wedge \sigma_j}} \notag
\\
&\le  ((t^*)^{-1} K \log N+1) e^{(1+\lambda )\kappa (K-\lfloor D^+_n \rfloor)}  .
\end{align}
Hence, for $n$ sufficiently large,
by a union bound and then by~\eqref{eq:propintipB2} (using the same argument for the case $\sigma_j \le \sigma_i$),
\begin{align} \label{eq:propintiptwol}
&\p{\tilde \zeta^{n,i}_t \ge \ell_1, \tilde \zeta^{n,j}_t \ge \ell_2, \tau^n_{i,j}> t \Big| \mathcal F_{s} } \notag \\
&\le \p{\sigma_i \wedge \sigma_j >t \Big| \mathcal F_{s} }
+ \E{\p{\tilde \zeta^{n,i}_t \ge \ell_1, \tilde \zeta^{n,j}_t \ge \ell_2, \tau^n_{i,j}> t \Big| \mathcal F_{\sigma_i \wedge \sigma_j}}\1_{\sigma_i \wedge \sigma_j \le t} \Big| \mathcal F_{s} } \notag \\
&\le \p{\sigma_i \wedge \sigma_j >t \Big| \mathcal F_{s} }+ \p{\sigma_i \vee \sigma_j >t \Big| \mathcal F_{s} }
+ \tfrac 12 (\log N)^7 e^{-(1+\lambda )\kappa(\ell_1 +\ell_2)}
\end{align}
for $n$ sufficiently large.
Finally, 
letting $t'=t-t^* \lfloor (t^*)^{-1}K \log N \rfloor \in \delta_n \N_0 \cap [0,T^-_n]$
with $t'\ge s$, since $(r^{n,K,t^*}_{K\log N, T_n-t'}(x))_{x\in \frac 1n \Z}$ only depends on the Poisson processes
$(\mathcal P^{x,i,j})_{x,i,j}$, $(\mathcal S^{x,i,j})_{x,i,j}$,
$(\mathcal Q^{x,i,j,k})_{x,i,j,k}$ and $(\mathcal R^{x,i,y,j})_{x,y,i,j}$ in the time interval $[0,T_n-t']$,
$$
\p{r^{n,K,t^*}_{ K\log N , T_n- t'}(x)=0 \; \forall x\in \tfrac 1n \Z \Big| \mathcal F_{s}}
=\p{r^{n,K,t^*}_{ K\log N, T_n- t'}(x)=0 \; \forall x\in \tfrac 1n \Z \Big| \mathcal F}\ge 1-\left( \frac n N \right)^2
$$
by the definition of the event $E_4$.
By the definition of $r^{n,K,t^*}_{K\log N,T_n-t'}(x)$ in~\eqref{eq:rnystdefn}, it follows that
$\p{\sigma_i \vee \sigma_j >t \big| \mathcal F_{s} }\le \left( \frac n N \right)^2$. By~\eqref{eq:propintiptwol},
and since $(1+\lambda)\kappa (\ell_1+\ell_2)\le 4(1/2-c_0)\log (N/n)$, this
 completes the proof of~\eqref{eq:propintipstat1}.
By a union bound and then by the same argument as in~\eqref{eq:propintiponel} and since $\tilde \zeta^{n,i}_{\sigma_i}\le K$,
\begin{align*}
\p{\tilde \zeta^{n,i}_t \ge \ell_1 \Big| \mathcal F_s } &\le \p{\sigma_i >t \Big| \mathcal F_s}
+ \E{\p{\tilde \zeta^{n,i}_t \ge \ell_1 \Big| \mathcal F_{\sigma_i}} \1_{\sigma_i \le t} \Big| \mathcal F_s}\\
&\le \left( \frac n N \right)^2 +(\log N)^2 e^{(1+\lambda) \kappa(K-\ell_1)},
\end{align*}
which completes the proof.
\end{proof}

\subsection{Proof of Proposition~\ref{prop:RlogN}} \label{subsec:behindfront}
We first prove two preliminary lemmas, similar to the lemmas in Section~\ref{subsec:tipbulkproofs}.
Write $d'_n=\frac 1 {64}\alpha d_n$.
\begin{lemma} \label{lem:jumpbulk}
For $n$ sufficiently large, on the event $E_1\cap E'_2$,  for $t\in \delta_n \N_0 \cap [0,T_n^-]$, $i\in [k_0]$ and $y,y' \le -\frac 12 d'_n$,
if $\tilde \zeta^{n,i}_t \ge y$ then 
$$
\p{\tilde \zeta^{n,i}_{t+t^*}\le y' \Big| \mathcal F_t } \le c_1 e^{-\frac 12 \alpha \kappa(y-y')}.
$$
\end{lemma}
\begin{proof}
Suppose $y'\ge -N^3$.
For $n$ sufficiently large,  by the definition of the event $E_1$ in~\eqref{eq:eventE1}, if $\tilde \zeta^{n,i}_t \ge y$ and $\zeta^{n,i}_t \in I^n_{T_n-t}$, 
\begin{align*}
\p{\tilde \zeta^{n,i}_{t+t^*}\le y' \Big| \mathcal F_t}
&\le \p{\zeta^{n,i}_{t+t^*}\le \mu^n_{T_n-t}-\nu t^* +1 +y' \Big| \mathcal F_t}\\
&= \frac{q^{n,-}_{T_n-t-t^*,T_n-t}(\mu^n_{T_n-t}-\nu t^* +1 +y', \tilde \zeta^{n,i}_t +\mu^n_{T_n-t})}{p^n_{T_n-t}(\tilde \zeta^{n,i}_t +\mu^n_{T_n-t})}
\\
&\le c_1 e^{-\frac 12 \alpha \kappa (y-y')}
\end{align*}
since the event $A^{(3)}_{T_n-t-t^*}(n^{-1} \lfloor n (\mu^n_{T_n-t}-\nu t^*+1+y')\rfloor , \zeta^{n,i}_t)$ occurs by the definition of the event $E'_2$ in~\eqref{eq:eventE'2}.
If instead $y'<-N^3$ or $\zeta^{n,i}_t \notin I^n_{T_n-t}$ then
 by~\eqref{eq:lemfromxixj3} in Lemma~\ref{lem:fromxixj}, $\p{\tilde \zeta^{n,i}_{t+t^*}\le y' \Big| \mathcal F_t}=0$.
\end{proof}

We now use Lemma~\ref{lem:jumpbulk} and an induction argument to prove the following result.

\begin{lemma} \label{lem:bulktail}
On the event $E_1\cap E'_2$, for $t\in  \delta_n \N_0 \cap [0,T^-_n]$, $i\in [k_0]$, $k\in \N_0$  and $t' \in [0,T_n^-]$ with $t'-t \in t^*\N_0$,
\begin{equation} \label{eq:lembulktailstat}
\p{\tilde \zeta^{n,i}_{t'}\le -\tfrac 12 d'_n -k
\Big| \mathcal F_t } \le e^{-\frac 14 \alpha \kappa( ( \frac 12 d'_n+\tilde \zeta^{n,i}_t)\wedge 0+ k)}.
\end{equation}
\end{lemma}
\begin{proof}
Recall from~\eqref{eq:cchoice} that we chose $c_1>0$ sufficiently small that
\begin{equation} \label{eq:lembulktailA}
c_1+ c_1 e^{3\alpha \kappa/4}(e^{\alpha \kappa/4}-1)^{-1}+e^{-\alpha \kappa /4}<1.
\end{equation}
Let $A=e^{-\frac 14 \alpha \kappa ((\frac 12 d'_n+\tilde \zeta^{n,i}_t)\wedge 0)}$.
Suppose, for an induction argument, that for some $t' \ge t$ with $t'\in [0,T_n^-]$ and $t'-t \in t^*\N_0$,~\eqref{eq:lembulktailstat} holds for all $k\in \N_0$.
Then by Lemma~\ref{lem:jumpbulk}, for $k\in \N_0$,
\begin{align*}
\p{\tilde \zeta^{n,i}_{t'+t^*}\le -\tfrac 12 d'_n -k
\Big| \mathcal F_t } 
&\le \sum_{k'=0}^k 
\p{\tilde \zeta^{n,i}_{t'} \in (-\tfrac 12 d'_n-k'-1,-\tfrac 12 d'_n-k'] \Big| \mathcal F_t}
c_1  e^{-\frac 12 \alpha \kappa (k-k'-1)}\\
&\qquad +\p{\tilde \zeta^{n,i}_{t'} \le -\tfrac 12 d'_n-k-1 \Big| \mathcal F_t} + c_1 e^{-\frac 12 \alpha \kappa k}
\\
&\le \sum_{k'=0}^k A e^{-\frac 14 \alpha \kappa k'}c_1  e^{-\frac 12 \alpha \kappa (k-k'-1)}
+A e^{-\frac 14 \alpha \kappa (k+1)} + c_1 e^{-\frac 12 \alpha \kappa k}
\end{align*}
by our induction hypothesis.
Therefore, since $A\ge 1$,
\begin{align*}
\p{\tilde \zeta^{n,i}_{t'+t^*}\le -\tfrac 12 d'_n -k
\Big| \mathcal F_t } 
&\le A \left( c_1  e^{-\frac 12 \alpha \kappa (k-1)}\sum_{k'=0}^k e^{\frac 14 \alpha \kappa k'}+e^{-\frac 14 \alpha \kappa (k+1)}+ c_1 e^{-\frac 12 \alpha \kappa k}\right)\\
&= A \left(c_1  e^{-\frac 12 \alpha \kappa (k-1)}\frac{e^{\frac 14 \alpha \kappa (k+1)}-1}{e^{\frac 14 \alpha \kappa}-1}+e^{-\frac 14 \alpha \kappa (k+1)}+ c_1 e^{-\frac 12 \alpha \kappa k} \right)\\
&< A e^{-\frac 14 \alpha \kappa k}\left(c_1 e^{\frac 34 \alpha \kappa}(e^{\frac 14 \alpha \kappa}-1)^{-1} +e^{-\frac 14 \alpha \kappa }+c_1 \right)\\
&\le A e^{-\frac 14 \alpha \kappa k}
\end{align*}
by~\eqref{eq:lembulktailA}.
The result follows by induction.
\end{proof}

\begin{proof}[Proof of Proposition~\ref{prop:RlogN}]
We begin by proving~\eqref{eq:propRlogN1}.
For $n$ sufficiently large, by~\eqref{eq:lemfromxixj3} in Lemma~\ref{lem:fromxixj} and then by
a union bound and Lemma~\ref{lem:bulktail}, and since $\tilde \zeta^{n,i}_0\ge -K_0$,
\begin{align*}
\p{\exists t \in \delta_n \N_0 \cap[0, T^-_n]  : \tilde \zeta^{n,i}_{t}\le D_n^- \Big|\mathcal F_0 } 
&\le \p{\exists t\in t^* \N_0 \cap [0,T^-_n] : \tilde \zeta^{n,i}_t \le \tfrac 12 D_n^- \Big| \mathcal F_0}
\\
&\le( (t^*)^{-1} T^-_n +1) e^{-\frac 14 \alpha \kappa (-\frac 12 D_n^- -\frac 12 d'_n)}
\\
& \le N^{-1}
\end{align*}
for $n$ sufficiently large, since, by~\eqref{eq:Dn+-defn}, $\frac 18 \alpha \kappa D_n^- =-\frac{13}4  \log N$ and since $T_n^-\le N^2$.

Note that the last statement~\eqref{eq:propRlogN3} follows directly from Lemma~\ref{lem:bulktail}
(since $\tilde \zeta^{n,i}_0\ge -K_0$ and $d_n>d'_n$).
We now prove~\eqref{eq:propRlogN2}.
Recall from~\eqref{eq:cchoice} that we chose $c_1>0$ sufficiently small that
\begin{equation} \label{eq:proplogNB}
e^{-\alpha \kappa /4}+c_1 (1-e^{-\alpha \kappa /4})^{-1}<e^{-\alpha \kappa /5}.
\end{equation}
Let $A\sim \text{Ber}(c_1)$ and $G\sim \text{Geom}(1-e^{-\alpha \kappa /2})$ be independent (with $\p{G\ge k}=e^{-\alpha \kappa  k/2}$ for $k\in \N_0$).
For $t' \in \delta_n \N_0 \cap [0,T^-_n]$, if
$\tilde \zeta^{n,i}_{t'} \le -\frac 12 d'_n$ then by Lemma~\ref{lem:jumpbulk}, for $k\in \N_0$,
\begin{equation} \label{eq:propRlogN*}
\p{\tilde \zeta^{n,i}_{t'} - \tilde \zeta^{n,i}_{t'+t^*} \ge k \Big|\mathcal F_{t'}}\le c_1 e^{-\frac 12 \alpha \kappa  k}
\le \p{AG-(1-A)\ge k}.
\end{equation}
Since $AG-(1-A)\ge -1$,~\eqref{eq:propRlogN*} holds for each $k\in \Z$.
Let $(A_j)_{j=1}^\infty$ and $(G_j)_{j=1}^\infty$ be independent families of i.i.d.~random variables with $A_1 \stackrel{d}{=}A$ and $G_1\stackrel{d}{=} G$. Suppose $\tilde \zeta^{n,i}_s \ge D_n^-$ and $t-s\ge K\log N$, and
take $s'\in [s,s+t^*]$ such that $t-s'\in t^*\N_0$. For $n$ sufficiently large, by~\eqref{eq:lemfromxixj3} in Lemma~\ref{lem:fromxixj}, we have $\tilde \zeta^{n,i}_{s'}\ge 2D_n^-$. Then using~\eqref{eq:propRlogN*} in the second inequality,
\begin{align*}
&\p{\tilde \zeta^{n,i}_{s'+\ell t^*}\le -\tfrac 12 d'_n \; \forall \ell  \in \{0\}\cup [ 4|D_n^-|] \Big|\mathcal F_{s'}}\\
&\le\mathbb P \Big( \tilde \zeta^{n,i}_{s'+\ell t^*}\le -\tfrac 12 d'_n \; \forall \ell  \in \{0\}\cup [ 4|D_n^-|-1],
\sum_{j=1}^{4|D_n^-|}(\tilde \zeta^{n,i}_{s'+(j-1)t^*}-\tilde \zeta^{n,i}_{s'+jt^*})\ge 2D_n^-
 \Big|\mathcal F_{s'} \Big) \\
&\le \p{ \sum_{j=1}^{4|D_n^-|}(A_j G_j - (1-A_j) ) \ge 2 D_n^- }.
\end{align*}
By Markov's inequality,
\begin{align*}
\p{ \sum_{j=1}^{4|D_n^-|}(A_j G_j - (1-A_j)) \ge 2 D_n^- }
&\le e^{\frac 14 \alpha \kappa \cdot 2  |D_n^-|} \E{e^{\frac 14 \alpha \kappa  (A_1 G_1-(1-A_1))}}^{4|D_n^-|}\\
&\le e^{\frac 12 \alpha \kappa  |D_n^-|} \left( (1-c_1) e^{-\frac 14 \alpha \kappa }+c_1 \frac{1-e^{-\alpha \kappa  /2}}{1-e^{-\alpha \kappa  /4}}
\right)^{4|D_n^-|}\\
&\le e^{-\frac 3{10} \alpha \kappa  |D^-_n|} 
\end{align*}
by~\eqref{eq:proplogNB}.
Therefore,
since $\alpha \kappa |D^-_n|=26 \log N$ by~\eqref{eq:Dn+-defn}, and since $K\log N>(4|D_n^-|+1)t^*$ by our choice of $K$ in Proposition~\ref{prop:eventE},
\begin{align*}
\p{\tilde \zeta^{n,i}_t \le -d_n \Big| \mathcal F_{s}} 
&\le N^{-7} + \sum_{\ell =0}^{4|D_n^-|}\E{\p{\tilde \zeta^{n,i}_{s'  +\ell t^*}\ge -\tfrac 12 d'_n, \tilde \zeta^{n,i}_t \le -d_n  \Big| \mathcal F_{s'}}\Big| \mathcal F_{s}}\\
&\le N^{-7} +\sum_{\ell=0}^{4|D_n^-|} e^{-\frac 14 \alpha \kappa  \cdot \frac 12 d_n }\\
&\le (\log N)^{2-\frac 18 \alpha C}
\end{align*}
for $n$ sufficiently large, where the second inequality follows by Lemma~\ref{lem:bulktail} and since $d_n>d'_n$.
Since $d'_n =2^{-6}\alpha d_n$,
by the same argument, for $n$ sufficiently large, $\p{\tilde \zeta^{n,i}_t \le -d'_n+2 \Big| \mathcal F_{s}} \le (\log N)^{2-2^{-9} \alpha^2 C}$.
\end{proof}

\section{Event $E_1$ occurs with high probability} \label{sec:eventE1}

In this section and the following three sections, we will prove Proposition~\ref{prop:eventE}.
We begin with some notation which will be used throughout the rest of the article.
For $h:\frac{1}{n}\Z \to \R$ and $x\in \frac1n \Z$, let
$$\nabla _n h(x)=n\left(h(x+n^{-1} )-h(x)\right)$$ 
and let $$\Delta _n h(x)=n^2\left(h(x+n^{-1} )-2h(x)+h(x-n^{-1} )\right).$$ 
Define $f:\R \to \R$ by letting 
\begin{equation} \label{eq:fdef}
f(u)=u(1-u)(2u-1+\alpha).
\end{equation}
Recall the definition of the event $E_1$ in~\eqref{eq:eventE1}.
In this section, we will prove the following result (along with some technical lemmas which will be used in later sections).
\begin{prop} \label{prop:eventE1}
For $t\ge 0$, let $(u^n_{t,t+s})_{s\ge 0}$ denote the solution of
\begin{equation} \label{eq:unttsdef}
\begin{cases}
\partial_s u^n_{t,t+s}=\tfrac 12 m \Delta_n u^n_{t,t+s}+s_0 f(u^n_{t,t+s}) \quad \text{for }s>0, \\
u^n_{t,t}=p^n_t.
\end{cases}
\end{equation}
For $c_2>0$, define the event
\begin{equation} \label{eq:E1'defn}
E '_1=E_1 \cap \left\{ \sup_{s\in [0,\gamma_n],x\in \frac 1n \Z}|u^n_{t,t+s}(x)-g(x-\mu^n_t -\nu s)|\le e^{-(\log N)^{c_2}}\; \forall t\in [\log N,N^2]\right\}.
\end{equation}
Suppose for some $a_1>1$, $N\ge n^{a_1}$ for $n$ sufficiently large.
For $\ell \in \N$, for $b_1,c_2>0$ sufficiently small and $b_2>0$, if condition~\eqref{eq:conditionA} holds then for $n$ sufficiently large,
$$
\p{(E'_1)^c}\le \left( \frac n N \right)^\ell .
$$
\end{prop}
From now on in this section, we will assume for some $a_1>1$, $N\ge n^{a_1}$ for $n$ sufficiently large.
We will need some more notation; we use notation similar to~\cite{durrett/fan:2016}.
For $f_1,f_2:\frac1n \Z \to \R$, write 
$$\langle f_1,f_2 \rangle _n:=n^{-1} \sum_{w\in \frac1n \Z} f_1(w) f_2(w).$$
Let $(X^n_t)_{t\geq 0}$ denote a continuous-time simple symmetric random walk on $\frac1n \Z$ with jump rate $n^2$.
For $z\in \frac1n \Z$, let $\mathbf P_z(\cdot):=\P \left(\cdot \left| X_0^n=z \right. \right)$.
Then for $z, w\in \frac1n \Z $ and $0\le s \le t$, let
\begin{equation} \label{eq:phitzdefq}
\phi_s^{t,z}(w):=
n \mathbf P_z \left(  X^n_{m(t-s)}=w  \right) .
\end{equation}
For $a\in \R$,
$z,w \in \frac 1n \Z$ and $0\le s \le t$,
 let
\begin{equation} \label{eq:phiadef}
\phi^{t,z,a}_s(w)=e^{-a(t-s)}\phi_s^{t,z}(w).
\end{equation}
Let $(u^n_t)_{t\geq 0}$ denote the solution of
\begin{equation} \label{eq:undef}
\begin{cases}
\partial_t u^n_t &= \tfrac 12 m \Delta_n u^n_t +s_0 f(u^n_t) \quad \text{for }t>0,\\
u^n_0 &= p^n_0.
\end{cases}
\end{equation}
We will prove in Proposition~\ref{prop:pnun} below that if $t$ is not too large, $p^n_t$ and $u^n_t$ are close with high probability.
By the comparison principle, $u^n_t\in [0,1]$.
Since $\partial_s \phi^{t,z}_s +\frac 12 m \Delta_n \phi^{t,z}_s =0$ for $s\in (0,t)$, we have that
for $a\in \R$, $z\in \frac 1n \Z$ and $t\ge 0$, by integration by parts,
\begin{align*}
&\langle u^n_t , \phi^{t,z,a}_t \rangle_n\\
&= \langle u^n_0, \phi^{t,z,a}_0 \rangle_n 
+\int_0^t \langle u^n_s, \partial_s \phi^{t,z,a}_s \rangle_n ds
+\int_0^t \langle u^n_s, \tfrac 12 m \Delta_n \phi^{t,z,a}_s \rangle_n ds
+s_0 \int_0^t \langle f(u^n_s), \phi^{t,z,a}_s \rangle_n ds\\
&= e^{-at} \langle p_0^n, \phi_0^{t,z}\rangle _n
+\int_0^t e^{-a(t-s)}\langle s_0 f(u^n_s)+a u^n_s,\phi_s^{t,z}\rangle_n ds.
\end{align*}
Therefore, since $\langle u^n_t , \phi^{t,z,a}_t \rangle_n =u^n_t(z)$, it follows that
for $a\in \R$, $z\in \frac 1n \Z$ and $t\ge 0$,
\begin{equation} \label{eq:ungreena}
u^n_t(z)=e^{-at} \langle p_0^n, \phi_0^{t,z}\rangle _n
+\int_0^t e^{-a(t-s)}\langle s_0 f(u^n_s)+a u^n_s,\phi_s^{t,z}\rangle_n ds.
\end{equation}
Note that by~\eqref{eq:ungreena} with $a=-(1+\alpha)s_0 $, since $f(u)\le (1+\alpha)u$ for $u\in [0,1]$,
\begin{equation} \label{eq:uneasybound}
u^n_t(z)\le e^{(1+\alpha)s_0 t}\langle p^n_0, \phi^{t,z}_0 \rangle _n.
\end{equation}

In this section, alongside proving Proposition~\ref{prop:eventE1}, we will prove some preliminary tracer dynamics results which will be used in later sections, so we need some notation for tracer dynamics with an arbitrary initial condition.
Take $\mathcal I_0 \subseteq \{(x,i):\xi^n_0(x,i)=1\}.$
Then for $t\ge 0$, let
\begin{equation} \label{eq:etandefn}
 \eta^n_t(x,i)=\1_{(\zeta^{n,t}_t(x,i),\theta^{n,t}_t(x,i))\in \mathcal I_0} \quad\text{for }x\in \tfrac 1n \Z, \, i\in [N],
\end{equation}
i.e.~$\eta^n_t(x,i)=1$ if and only if the $i^{\text{th}}$ individual at $x$ at time $t$ is descended from an individual in $\mathcal I_0$ at time 0.
For $t\ge 0$ and $x\in \frac 1n \Z$, let
\begin{equation} \label{eq:qndef}
q^n_t(x)=\frac 1N \sum_{i=1}^N \eta^n_t(x,i) ,
\end{equation}
i.e.~the proportion of individuals at $x$ at time $t$ which are descended from individuals in $\mathcal I_0$ at time 0.
Let $(v^n_t)_{t\geq 0}$ denote the solution of
\begin{equation} \label{eq:vndef}
\begin{cases}
\partial_t v^n_t &= \tfrac{1}{2}m \Delta_n v^n_t +s_0 v^n_t (1-u^n_t)(2u^n_t-1+\alpha ) \quad \text{for }t>0,\\
v^n_0 &= q^n_0.
\end{cases}
\end{equation}
We will prove in Proposition~\ref{prop:pnun} below that if $t$ is not too large, $q^n_t$ and $v^n_t$ are close with high probability.
Note that by the comparison principle, $0\le v^n_t \le u^n_t$.
Moreover, for $a\in \R$, $t\ge 0$ and $z\in \frac 1n \Z$,
by the same argument as for~\eqref{eq:ungreena},
\begin{equation} \label{eq:vngreena}
v^n_t(z)
= e^{-at}\langle q^n_0, \phi^{t,z}_0 \rangle_n 
+\int_0^t e^{-a(t-s)}\langle v^n_s (s_0 (1-u^n_s)(2u^n_s-1+\alpha )+a), \phi^{t,z}_s \rangle_n ds.
\end{equation}
For $t\ge 0$ and $z\in \frac 1n \Z$, by~\eqref{eq:vngreena} with $a=-(1+\alpha)s_0$ 
and since $(1-u)(2u-1+\alpha )\le 1+\alpha$ for $u\in [0,1]$,
\begin{align} \label{eq:veasybound}
v^n_t(z)
&\le e^{(1+\alpha)s_0 t}\langle q^n_0, \phi^{t,z}_0\rangle _n.
\end{align}
The following result says that if $t$ is not too large, $|p^n_t -u^n_t|$ and $|q^n_t -v^n_t|$ are small with high probability; the proof is postponed to Section~\ref{subsec:pnunproof}.
\begin{prop} \label{prop:pnun}
Suppose $c_3>0$ and $\ell \in \N$.
Then there exists $c_4=c_4(c_3,\ell)\in (0,1/2)$ such that for $n$ sufficiently large, for $T\leq 2(\log N)^{c_4}$,
\begin{align*}
\p{\sup_{x\in \frac1n \Z, |x|\leq N^5}\sup_{t\in [0,T]}|p^n_{t}(x)-u^n_t(x)| \geq \left(\frac{n}{N}\right)^{1/2-c_3}}
&\leq 
\left( \frac{n}{N}\right)^\ell 
\end{align*}
and 
for $t\le 2(\log N)^{c_4}$, 
$$
\p{\sup_{x\in \frac 1n \Z, |x|\le N^5} |q^n_t(x)-v^n_t(x)|\ge \left(\frac n N \right)^{1/2-c_3}}\le \left(\frac n N \right)^{\ell}.
$$
For $k\in \N$ with $k\ge 2$,
there exists a constant $C_1=C_1(k)<\infty$ such that for $t\geq 0$,
\begin{align} \label{eq:gronwall1stat}
\sup_{x\in \frac1n \Z} \E{ |p^n_{t}(x)-u^n_t(x)|^k} 
&\leq C_1 \left(\frac{n^{k/2}t^{k/4}}{N^{k/2}}+N^{-k}\right)
e^{C_1 t^k}.
\end{align}
\end{prop}
We also need to control $p^n_t(x)$ when $x$ is not in the interval $ [-N^5,N^5]$ covered by Proposition~\ref{prop:pnun}.
\begin{lemma} \label{lem:p01}
For $n$ sufficiently large, if
 $p^n_0(x)=0$ $\forall x\geq N$
and $p^n_0(x)=1$ $\forall x\leq -N$ then
\begin{align*}
\P\left(\exists t\in [0,2N^2],\, x\in \tfrac1n \Z \cap[ N^5, \infty): p^n_t(x)>0 \right)
&\leq e^{-N^5}\\
\text{and} \qquad \P\left(\exists t\in [0,2N^2],\, x\in \tfrac1n \Z\cap (-\infty, -N^5]: p^n_t(x)<1 \right)
&\leq e^{-N^5}.
\end{align*}
\end{lemma}
\begin{proof}
For $x\in \frac1n \Z$, let
$$
\tau_x :=\inf \{t\geq 0: p^n_t(x)>0\}.
$$
Let $(E_\ell)_{\ell =1}^\infty$ be a sequence of i.i.d.~random variables with $E_1\sim \text{Exp}(mr_n N^2)$.
For $x>N$, $\tau_x$ occurs at a jump time after time $\tau_{x-n^{-1}}$ in $\mathcal R^{x,i,x-n^{-1},j}$ for some $i,j\in [N]$.
Therefore we can couple the process $(\xi^n_t(x,i))_{x\in \frac1n \Z,\, i\in [N], t\geq 0}$
with $(E_\ell)_{\ell = 1}^\infty $
in such a way that for each $\ell \in \N$,
$$
\tau_{N+\ell n^{-1}}- \tau_{N+(\ell -1) n^{-1}}\geq E_\ell .
$$
It follows that
$$
\tau_{N^5}\geq \sum_{\ell =1}^{n(N^5-N)}E_\ell .
$$
Therefore, letting $Y_n$ denote a Poisson random variable with mean $2mr_n N^4$, we have that
\begin{align*}
\P\left(\tau_{N^5}\leq 2N^2\right)
&\leq \P\left(\sum_{\ell =1}^{n(N^5-N)}E_\ell\leq 2 N^2\right)\\
&= \P\left(Y_n \geq n(N^5-N) \right).
\end{align*}
By Markov's inequality, and then since $r_n =\frac12  n^2 N^{-1}$,
\begin{align*}
\P\left(Y_n \geq n(N^5-N) \right)
&\leq e^{-n(N^5-N)}\E{e^{Y_n}}
\leq e^{-n(N^5-N)}e^{ mn^2 N^3 (e-1)}
 \leq e^{-N^5}
\end{align*}
for $n$ sufficiently large, since $N\geq n$.
Therefore for $n$ sufficiently large,
$$\P\left(\tau_{N^5}\leq 2 N^2\right) \leq e^{-N^5}.$$
Letting
$
\sigma_x :=\inf \{t\geq 0: p^n_t(x)<1\}$
 for $x\in \frac1n \Z$, 
 by the same argument we have that 
$$\P\left(\sigma_{-N^5}\leq 2 N^2\right) \leq e^{-N^5}$$
for $n$ sufficiently large, which completes the proof.
\end{proof}

Recall from~\eqref{eq:gdefn} and~\eqref{eq:kappanu} that $g(x)=(1+e^{\kappa x})^{-1}$, and recall the definition of $f$ in~\eqref{eq:fdef}.
Note that $u(t,x):=g(x-\nu  t)$ is a travelling wave solution of the partial differential equation
$$
\partial_t u = \tfrac 12 m \Delta u +s_0 f(u).
$$
Since $\alpha\in (0,1)$, we have that $f(0)=f(1)=0$,
$f(u)<0$ for $u\in (0,\frac12 (1-\alpha))$,
$f(u)>0$ for $u\in (\frac12 (1-\alpha),1)$,
$f'(0)<0$ and $f'(1)<0$.
This allows us to apply results from~\cite{fife/mcleod:1977} as follows.
For an initial condition $u_0 :\R\to [0,1]$, let
$u(t,x)$ denote the solution of
\begin{equation} \label{eq:upde}
\begin{cases}
\partial_t u &=\tfrac 12 m \Delta u +s_0 f(u) \quad \text{for }t>0,\\
u(0,\cdot)&=u_0.
\end{cases}
\end{equation}
\begin{lemma} \label{lem:x0approx}
There exist constants $C_2<\infty$ and $c_5>0$ such that for $\epsilon\leq c_5$,
if $u_0$ is piecewise continuous with $0\leq u_0\leq 1$ and, for some $z_0\in \R$,
$|u_0(z)-g(z-z_0)|\leq \epsilon$ $\forall z\in \R$, then
$$|u(t,x)-g(x-\nu  t-z_0)|\leq C_2 \epsilon \quad \forall x\in \R,\;  t>0.$$
\end{lemma}
\begin{proof}
The result follows directly from Lemma~4.2 in~\cite{fife/mcleod:1977} and its proof.
\end{proof}
\begin{prop} \label{prop:expconvtog}
There exist constants $c_6>0$ and $C_3<\infty$ such that
if $u_0$ is piecewise continuous with $0\leq u_0\leq 1$ and $|u_0(z)-g(z)|\leq c_6$ $\forall z \in \R$, then for some $z_0\in \R$ with $|z_0|\leq 1$,
$$
|u(t,x)-g(x-\nu  t-z_0)|\leq C_3 e^{-c_6 t} \quad \forall x\in \R, \; t>0.
$$
\end{prop}
This is a slight modification of Theorem~3.1 in~\cite{fife/mcleod:1977}
(to ensure that $C_3$ and $c_6$ do not depend on the initial condition $u_0$, as long as $\|u_0-g\|_\infty $ is sufficiently small); we postpone the proof to Appendix~\ref{sec:append}.
The next lemma says that if the initial condition $p^n_0$ is not too rough, then $u^n_t$ is close to a solution of~\eqref{eq:upde}.
\begin{lemma} \label{lem:unu}
Let $(u_t)_{t\geq 0}$ denote the solution of
\begin{equation} \label{eq:ueq}
\begin{cases}
\partial_t u_t &= \tfrac 12 m \Delta u_t +s_0 f(u_t) \quad \text{for }t>0, \\
u_0&=\bar{p}^n_0,
\end{cases}
\end{equation}
for some $\bar{p}^n_0:\R\to [0,1]$ with $\bar{p}^n_0(y)=p^n_0(y)$ $\forall y\in \frac 1n \Z$.
There exists a constant $C_4<\infty$ such that for $T\geq 1$,
$$
\sup_{t\in [0,T],\, x\in \frac1n \Z}
|u_t^n(x)-u_t(x)|\leq \left(C_4 n^{-1/3}+\sup_{z_1,z_2\in\R,|z_1-z_2|\leq n^{-1/3}}|\bar{p}^n_0(z_1)-\bar{p}^n_0(z_2)| \right)T^2 
  e^{(1+\alpha)s_0 T}.
$$
\end{lemma}
\begin{proof}
For $t\geq 0$ and $z\in \frac1n \Z$,
by~\eqref{eq:ungreena} and since $p^n_0(y)=\bar{p}^n_0(y)$ $\forall y\in \frac 1n \Z$,
\begin{equation*} 
u^n_t(z)=\langle \bar{p}_0^n,\phi^{t,z}_0\rangle _n+ s_0 \int_0^t \langle f(u^n_s), \phi^{t,z}_s \rangle _n ds.
\end{equation*}
Let $G_t(x)=\frac{1}{\sqrt{2\pi t}}e^{-x^2/(2t)}$; then
 for $z\in \R$ and $t> 0$,
\begin{equation} \label{eq:lemunu*}
u_t(z)=G_{mt}\ast \bar{p}^n_0(z)+s_0 \int_0^t G_{m(t-s)}\ast f(u_s)(z)ds.
\end{equation}
Letting $(B_t)_{t\geq 0}$ denote a Brownian motion, and by the definition of $\phi^{t,z}_s$ in~\eqref{eq:phitzdefq}, it follows that for $z\in \frac1n \Z$ and $t> 0$,
\begin{align} \label{eq:unminusu0}
&|u_t^n(z)-u_t(z)| \notag\\
& \leq \left|\Esubb{z}{\bar{p}^n_0(X^n_{mt})}-\Esub{z}{\bar{p}^n_0(B_{mt})} \right|
+s_0 \int_0^t \left|\Esubb{z}{f(u^n_s(X^n_{m(t-s)}))}-\Esub{z}{f(u_s(B_{m(t-s)}))} \right| ds.
\end{align}
By a Skorokhod embedding argument (see e.g.~Theorem~3.3.3 in~\cite{lawler/limic:2010}), for $n$ sufficiently large, 
$(X^n_t)_{t\ge 0}$ and $(B_t)_{t\ge 0}$ can be coupled in such a way that $X^n_0=B_0$ and for $t\ge 0$,
\begin{align} \label{eq:couplingBM0}
\p{|X^n_{mt}-B_{mt}|\geq n^{-1/3}} 
\leq (t+1)n^{-1/2}.
\end{align}
Since $\bar{p}_0^n \in [0,1]$, it follows that
\begin{align} \label{eq:(*)unuproof0}
 \left|\Esubb{z}{\bar{p}^n_0(X^n_{mt})}-\Esub{z}{\bar{p}^n_0(B_{mt})} \right|
 &\leq (t+1)n^{-1/2}+\sup_{z_1,z_2\in\R,|z_1-z_2|\leq  n^{-1/3}}|\bar{p}^n_0(z_1)-\bar{p}^n_0(z_2)| .
\end{align}
For the second term on the right hand side of~\eqref{eq:unminusu0},
note that $\sup_{v\in [0,1]}|f(v)|<1$ and, since $f'(u)=6u(1-u)-1+\alpha (1-2u)$, we have
$\sup_{v\in [0,1]}|f'(v)|=1+\alpha$.
Therefore,
 using the triangle inequality and then by the same coupling argument as for~\eqref{eq:(*)unuproof0}, for $s\in [0,t]$,
\begin{align} \label{eq:(A)unapprox}
&\left|\Esubb{z}{f(u^n_s(X^n_{m(t-s)}))}-\Esub{z}{f(u_s(B_{m(t-s)}))} \right| \notag \\
&\leq \left|\Esubb{z}{f(u^n_s(X^n_{m(t-s)}))}-\Esubb{z}{f(u_s(X^n_{m(t-s)}))} \right|
+\left|\Esubb{z}{f(u_s(X^n_{m(t-s)}))}-\Esub{z}{f(u_s(B_{m(t-s)}))} \right| \notag \\
& \leq (1+\alpha) \sup_{x\in \frac1n \Z}|u^n_s(x)- u_s(x)|
+2(t+1)n^{-1/2}+(1+\alpha) \|\nabla u_s\|_\infty  n^{-1/3}.
\end{align}
We now bound $\|\nabla u_s\|_\infty$.
For $t> 0$ and $x\in \R$, by differentiating both sides of~\eqref{eq:lemunu*},
\begin{align} \label{eq:dut0}
\nabla u_t(x)
&=G'_{mt}\ast \bar{p}^n_0(x)+s_0 \int_0^t G'_{m(t-s)}\ast f(u_s)(x)ds .
\end{align}
For the first term on the right hand side, since $\bar{p}^n_0 \in [0,1]$,
\begin{align*}
|G'_{mt}\ast \bar{p}^n_0(x)|
&\le \int_{-\infty}^\infty |G'_{mt}(z)|dz
= 2G_{mt}(0)
=2(2\pi m t)^{-1/2}.
\end{align*}
For the second term on the right hand side of~\eqref{eq:dut0},
since $\sup_{v\in [0,1]}|f(v)|<1$,
\begin{align*}
\left|\int_0^t G'_{m(t-s)}\ast f(u_s)(x)ds\right|
&\leq \int_0^t \int_{-\infty}^\infty  |G'_{m(t-s)}(z)| dz ds
=4 (2\pi m)^{-1/2} t^{1/2}.
\end{align*}
Hence by~\eqref{eq:dut0}, for $t> 0$,
$$
\|\nabla u_t \|_\infty
\le (2\pi m )^{-1/2}(2t^{-1/2}+4s_0 t^{1/2}).
$$
Substituting into~\eqref{eq:(A)unapprox} and then into~\eqref{eq:unminusu0}, and using~\eqref{eq:(*)unuproof0},  we now have that for $t>0$ and $z\in \frac 1n \Z$,
\begin{align*}
&|u_t^n(z)-u_t(z)| \\
& \leq \left(t+1\right) n^{-1/2}+\sup_{z_1,z_2\in \R,|z_1-z_2|\leq n^{-1/3}}|\bar{p}^n_0(z_1)-\bar{p}^n_0(z_2)|\\
& \, +s_0 \int_0^t \bigg((1+\alpha)\sup_{x\in \frac1n \Z}|u^n_s(x)-u_s(x)|
+2(t+1)n^{-1/2} 
+ 2(2\pi m )^{-1/2}
(2s^{-1/2}+4s_0 s^{1/2})  n^{-1/3}\bigg) ds.
\end{align*}
Hence there exists a constant $C_4<\infty$ such that for $T\geq 1$, for $t\in [0,T]$,
\begin{align*}
&\sup_{x\in \frac1n \Z}|u_t^n(x)-u_t(x)|\\
&\leq \left(C_4 n^{-1/3}+\sup_{z_1,z_2\in\R ,|z_1-z_2|\leq n^{-1/3}}|\bar{p}^n_0(z_1)-\bar{p}^n_0(z_2)| \right)T^2 
 +(1+\alpha)s_0 \int_0^t \sup_{x\in \frac1n \Z}|u^n_s(x)-u_s(x)| ds.
\end{align*}
The result follows by Gronwall's inequality.
\end{proof}
The following lemma will be used in the proof of Proposition~\ref{prop:eventE1} to show that with high probability,
$\sup_{|z_1-z_2|\le n^{-1/3}}|p^n_t(z_1)-p^n_t(z_2)|$ is small at large times $t$, which will allow us to use Lemma~\ref{lem:unu}.
\begin{lemma} \label{lem:nablaunbound}
There exists a constant $C_5<\infty$ such that 
\begin{equation} \label{eq:phidiffbound}
n\langle 1,|\phi^{t,z+n^{-1}}_0-\phi^{t,z}_0|\rangle_n \le C_5 t^{-1/2}
\quad \forall \, t>0, z\in \tfrac 1n \Z,
\end{equation}
and
$\sup_{t\ge 1, x \in \frac 1n \Z}|\nabla_n u^n_t (x)|\le C_5$.
\end{lemma}
\begin{proof}
For $t> 0$, $z\in \frac 1n \Z$ and $t_0\in (0,t]$,
by~\eqref{eq:ungreena},
\begin{equation} \label{eq:**pnun}
\nabla_n u^n_t(z)=n\langle u^n_{t-t_0},\phi^{t_0,z+n^{-1}}_0-\phi^{t_0,z}_0\rangle _n+ ns_0 \int_0^{t_0} \langle f(u^n_{t-t_0+s}), \phi^{t_0,z+n^{-1}}_s-\phi^{t_0,z}_s \rangle _n ds.
\end{equation}
Since $u^n_{t-t_0}\in [0,1]$, we have that
\begin{align} \label{eq:(*)pnun}
|n\langle u^n_{t-t_0},\phi^{t_0,z+n^{-1}}_0-\phi^{t_0,z}_0\rangle _n|
&\leq n \langle 1, |\phi^{t_0,z+n^{-1}}_0-\phi^{t_0,z}_0|\rangle_n .
\end{align}
Let $(S_j)_{j=0}^\infty $ be a discrete-time  simple symmetric random walk on $\Z$ with $S_0=0$.
By Proposition~2.4.1 in~\cite{lawler/limic:2010} (which follows from the local central limit theorem), there exists a constant $K_1<\infty$ such that for $j\in \N$,
$$
\sum_{y\in \Z} \left|\p{S_j =y-1}-\p{S_j=y} \right|\leq K_1 j^{-1/2}.
$$
Let $(R_s)_{s\geq 0}$ denote a Poisson process with rate $1$.
Then by the definition of $\phi^{t,z}_s$ in~\eqref{eq:phitzdefq}, and
 since $(X^n_s)_{s\geq 0}$ jumps at rate $n^2$,
\begin{align} \label{eq:lemnabla*}
n\langle 1, |\phi_0^{t_0, z+n^{-1}}-\phi_0^{t_0,z} |\rangle_n &= n
\sum_{y\in \frac1n \Z}\left|\psubb{0}{X^n_{mt_0}=y-n^{-1}}-\psubb{0}{X^n_{mt_0}=y} \right| \notag \\
&\leq n
\sum_{y\in \frac1n \Z}\sum_{j=0}^\infty \p{R_{mn^2 t_0}=j}\left|\p{S_j=ny-1}-\p{S_j=ny} \right| \notag \\
&\leq n
\sum_{j=1}^\infty \p{R_{mn^2 t_0}=j}K_1 j^{-1/2}+2n \p{R_{mn^2t_0}=0}.
\end{align}
By Markov's inequality, and since $R_{mn^2 t_0}\sim\text{Poisson}(mn^2 t_0)$,
\begin{align*}
\p{R_{m n^2t_0}\leq \tfrac 12 mn^2 t_0}
=\p{e^{-R_{mn^2 t_0}\log 2} \ge e^{-\frac 12 mn^2 t_0 \log 2}}
&\le e^{\frac 12 mn^2 t_0 \log 2}e^{mn^2 t_0 (e^{-\log 2}-1)}\\
&=e^{-\frac 12 mn^2 t_0(1-\log 2)}.
\end{align*}
Therefore, by substituting into~\eqref{eq:lemnabla*},
\begin{align} \label{eq:probdiff}
n\langle 1, |\phi_0^{t_0, z+n^{-1}}-\phi_0^{t_0,z} |\rangle_n
&\leq n\left((K_1+2)\p{R_{mn^2t_0}\leq \tfrac 12 m n^2 t_0}+K_1(\tfrac 12 m n^2 t_0)^{-1/2} \right)\notag \\
&\leq t_0^{-1/2}\left((K_1+2)(n^2 t_0)^{1/2} e^{-\frac 12mn^2 t_0(1-\log 2)}+\sqrt{2} m^{-1/2}K_1\right) \notag \\
&\leq K_2 t_0^{-1/2},
\end{align}
where $K_2=(K_1+2)\sup_{s\geq 0}(s^{1/2} e^{-\frac 12 m(1-\log 2) s})+\sqrt{2}m^{-1/2}K_1<\infty $. This completes the proof of~\eqref{eq:phidiffbound}.
Since $|f(u^n_{t-t_0+s})|\leq 1$ for $s\in [0,t_0]$, and then by~\eqref{eq:probdiff},
\begin{align*}
\left| n s_0 \int_0^{t_0} \langle f(u^n_{t-t_0+s}), \phi^{t_0,z+n^{-1}}_s-\phi^{t_0,z}_s \rangle _n ds\right|
&\leq  s_0 \int_0^{t_0} n \langle 1, |\phi_0^{t_0-s, z+n^{-1}}-\phi_0^{t_0-s,z} |\rangle_n ds\\
&\le 2s_0 K_2 t_0^{1/2}.
\end{align*}
Therefore,
by~\eqref{eq:**pnun},~\eqref{eq:(*)pnun} and~\eqref{eq:probdiff}, for $t\ge 1$ and $t_0 \in (0,t]$ we have
 $\sup_{x\in \frac 1n \Z}|\nabla_n u^n_t(x)|\leq   K_2(t_0^{-1/2}+2s_0 t_0^{1/2})$,
and the result follows by taking $t_0=1$.
\end{proof}
We will use the following easy lemma repeatedly in the rest of this section, and in Section~\ref{sec:eventE2}.
\begin{lemma} \label{lem:Xnmgf}
For $a\in \R$ with $|a|\le n$ and $t\geq 0$,
\begin{align*}
\Esubb{0}{e^{aX^n_{mt}}}
&=e^{\frac 12 m a^2 t +\mathcal O (ta^3 n^{-1})}.
\end{align*}
\end{lemma}
\begin{proof}
Let $(R^+_s)_{s\geq 0}$ and $(R^-_s)_{s\geq 0}$ be independent Poisson processes with rate $1$.
Then for $a\in \R$, since $(X^n_t)_{t\geq 0}$ is a continuous-time simple symmetric random walk on $\frac 1n \Z$ with jump rate $n^2$,
\begin{align*}
\Esubb{0}{e^{aX^n_{mt}}}
&=\E{e^{an^{-1} (R^+_{m n^2 t/2}-R^-_{m n^2 t/2})}}\\
&=\exp(\tfrac 12 m n^2 t(e^{an^{-1}}-1))\exp(\tfrac 12 m n^2 t(e^{-an^{-1}}-1))\\
&=\exp\left(\tfrac 12 m n^2 t\left(an^{-1}+\tfrac{1}{2}a^2 n^{-2}+\mathcal O \left(a^3 n^{-3} \right)- an^{-1}+\tfrac{1}{2}a^2 n^{-2}+\mathcal O \left(a^3 n^{-3} \right)\right)\right)\\
&=e^{\frac 12 m a^2 t +\mathcal O (ta^3 n^{-1})},
\end{align*}
where the second line follows since $R^+_{m n^2t/2}$ and $R^-_{m n^2t/2}$ are both Poisson distributed with mean $\frac 12 m n^2 t$.
\end{proof}

The following two lemmas will allow us to control $p^n_t(x)$ for large $x$.
The first lemma gives us an upper bound.

\begin{lemma} \label{lem:untailinit}
There exists a constant $c_7\in (0,1)$ such that for $n$ sufficiently large, the following holds.
Suppose that $p^n_0(x)=0$ $\forall x \geq N^6$.
Take $c\in (0,1/2)$.
Suppose for some $R >0$ with  $R  \left(\frac{n}{N}\right)^{1/2-c}\leq c_7$ that 
\begin{equation} \label{eq:pexpupper}
p_0^n(x)\leq 3 e^{-\kappa(1-(\log N)^{-2}) x }+R\left(\frac{n}{N}\right)^{1/2-c}
\quad \forall x\in \tfrac1n \Z,
\end{equation}
and that for some $T\in (1,\log N]$, $\sup_{y\in \frac 1n \Z, |y|\leq N ,\, t\in [0,T]} |u^n_t(y)-g(y-\nu  t)| \leq c_7 (\log N)^{-2}$.
Then for $t\in [0,T]$, 
\begin{equation*} \label{eq:untailinit}
u^n_t(x)\leq \tfrac 43 \left(3 e^{-\kappa(1-(\log N)^{-2})  (x-\nu  t) }+R\left(\frac{n}{N}\right)^{1/2-c}\right)\quad \forall x\in \tfrac1n \Z,
\end{equation*}
and for $t\in[ 1,T]$,
$$
u^n_t(x)\leq (1-c_7 (\log N)^{-2}) 3 e^{-\kappa(1-(\log N)^{-2})  (x-\nu  t) }+(1-c_7)R \left(\frac{n}{N}\right)^{1/2-c}\quad \forall x\in \tfrac1n \Z.
$$
\end{lemma}

\begin{proof}
Take $d\in (0,1/3)$ such that
\begin{equation} \label{eq:ddef}
d< \min\left(\tfrac 1 {10} (2-\alpha)s_0,\tfrac 1 4 e^{-(1-\alpha)s_0}(1-\alpha) s_0\right).
\end{equation}
Suppose that
\begin{equation} \label{eq:A1nNsmall}
 R \left(\frac{n}{N}\right)^{1/2-c}<\tfrac 1{12}(1+d)^{-1} e^{-(1-\alpha)s_0}(1-\alpha)s_0 ,
\end{equation}
and that $T\in (1,\log N]$ with 
\begin{align} \label{eq:nearg}
\sup_{y\in \frac 1n \Z ,|y|\leq N,\, t\in [0,T]} |u^n_t(y)-g(y-\nu t)| 
& <\tfrac 1{73}
e^{-5s_0 }(2-\alpha) (\log N)^{-2}.
\end{align}
Let $\theta_N=(1-(\log N)^{-2})\kappa$, and let
$$\tau =T \wedge \inf\left \{t\geq 0:\exists \, x \in \tfrac1n \Z \text{ s.t. }u^n_t(x)\geq (1+d(\log N)^{-2})3 e^{-\theta_N (x-\nu t) }+(1+d)R\left(\frac{n}{N}\right)^{1/2-c}\right\}.
$$
By~\eqref{eq:uneasybound},
and then since $p_0^n(x)=0$ $\forall x\geq N^6$, for $t\ge 0$ and $z\in \frac 1n \Z$,
\begin{align} \label{eq:uncrudebound}
u^n_t(z)\leq e^{(1+\alpha)s_0 t}\langle p_0^n, \phi_0^{t,z}\rangle _n
&\leq e^{(1+\alpha)s_0 t}\psubb{z}{X^n_{mt}\leq N^6} \notag\\
&=e^{(1+\alpha)s_0 t}\psubb{0}{X^n_{mt}\geq z-N^6} \notag \\
&\leq e^{(1+\alpha)s_0 t}\Esubb{0}{e^{2\theta_N X^n_{mt}}}e^{-2\theta_N z+2\theta_N N^6} \notag\\
&\leq e^{(2s_0 +3m\theta_N^2 )t}e^{-2\theta_N z+2\theta_N N^6}
\end{align}
for $n$ sufficiently large, by Markov's inequality and Lemma~\ref{lem:Xnmgf}.
Therefore, since $u^n_t(x)\in [0,1]$, there exists $N'<\infty$ such that
$$
\tau =T \wedge \min_{x\in \frac1n \Z \cap [0,N']} \inf 
\left \{t\geq 0:u^n_t(x) \geq  (1+d(\log N)^{-2})3 e^{-\theta_N (x-\nu t) }+(1+d)R \left(\frac{n}{N}\right)^{1/2-c}\right\}.
$$
Hence (by continuity of $u^n_t(x)$ for each $x\in \frac 1n \Z$ and by our assumption on the initial condition in~\eqref{eq:pexpupper}) we have that $\tau>0$. Moreover, if $\tau<T$ then there exists $x\in \frac1n \Z\cap [0, N']$ such that 
\begin{equation} \label{eq:untau*}
u^n_\tau(x) \geq  (1+d(\log N)^{-2})3 e^{-\theta_N (x-\nu \tau) }+(1+d)R \left(\frac{n}{N}\right)^{1/2-c}.
\end{equation}
Note that for $u\in [0,1]$,
\begin{equation} \label{eq:f+bound}
f(u)+(1-\alpha)u=-2u^3+(3-\alpha)u^2\leq (3-\alpha)u^2.
\end{equation}
Now by~\eqref{eq:ungreena}, for $0<t\leq \tau$ and $x\in \frac 1n \Z$, for $0<t_0\leq t\wedge  1$,
\begin{align} \label{eq:pexpupperA}
u^n_t(x)
&= e^{-(1-\alpha)s_0 t_0}\langle u^n_{t-t_0},\phi_0^{t_0,x} \rangle _n
+s_0 \int_0^{t_0} e^{-(1-\alpha)s_0 (t_0-s)}\langle f(u^n_{t-t_0+s})+(1-\alpha)u^n_{t-t_0+s},\phi_s^{t_0,x}\rangle _n ds \notag \\
&\leq e^{-(1-\alpha)s_0 t_0}\langle u^n_{t-t_0},\phi_0^{t_0,x} \rangle _n +3s_0\int_0^{t_0} e^{-(1-\alpha)s_0(t_0-s)}\langle (u^n_{t-t_0+s})^2,\phi_s^{t_0,x}\rangle _n ds,
\end{align}
where the second line follows by~\eqref{eq:f+bound}.
Since $t\le \tau$, we have
\begin{align*}
\langle u^n_{t-t_0},\phi_0^{t_0,x} \rangle _n
&\le (1+d(\log N)^{-2}) \Esubb{x}{3 e^{-\theta_N (X^n_{mt_0}-\nu (t-t_0))}}+(1+d)R \left( \frac n N \right)^{1/2-c}\\
&\leq  (1+d(\log N)^{-2}) 3 e^{-\theta_N (x-\nu (t-t_0))} e^{\frac 12 m \theta_N^2 t_0 +\mathcal O(t_0 n^{-1})}+(1+d)R \left( \frac n N \right)^{1/2-c},
\end{align*}
by Lemma~\ref{lem:Xnmgf}.
For the second term on the right hand side of~\eqref{eq:pexpupperA}, we have that for $s\in [0,t_0)$,
\begin{align}
&\langle (u^n_{t-t_0+s})^2,\phi_s^{t_0,x}\rangle _n \notag \\
&\leq 2
\left((1+d(\log N)^{-2})^2\Esubb{x}{9 e^{-2\theta_N (X^n_{m(t_0-s)}-\nu(t-t_0+s))}} +
(1+d)^2 R^2 \left(\frac{n}{N}\right)^{1-2c}\right)  \notag \\
&\leq 2
\left((1+d(\log N)^{-2})^2 9 e^{-2\theta_N (x-\nu(t-t_0+s))} e^{2m\theta_N^2 (t_0-s) +\mathcal O (t_0 n^{-1})}
+
(1+d)^2 R^2 \left(\frac{n}{N}\right)^{1-2c}\right)  \notag
\end{align}
by Lemma~\ref{lem:Xnmgf}.
Note that by~\eqref{eq:kappanu}, $(1-\alpha)s_0 +\theta_N \nu -\frac 12 m \theta_N^2 =(2-\alpha -(\log N)^{-2})s_0 (\log N)^{-2}$. Hence
for $n$ sufficiently large, substituting into~\eqref{eq:pexpupperA},
\begin{align*}
& u^n_t(x)\\
&\leq e^{- ((1-\alpha)s_0+\theta_N \nu -\frac 12 m \theta_N^2)t_0+\mathcal O(t_0 n^{-1})} (1+d(\log N)^{-2}) 3 e^{-\theta_N (x-\nu t)}
+e^{-(1-\alpha)s_0 t_0}(1+d) R\left(\frac{n}{N}\right)^{1/2-c}\\ 
&\qquad +6s_0 (1+d(\log N)^{-2})^2 9 e^{-2\theta_N (x-\nu t)} e^{5s_0 t_0}t_0
  +6(1+d)^2 R^2  \left(\frac{n}{N}\right)^{1-2c} t_0\\
&\leq (1+d(\log N)^{-2}) 3 e^{-\theta_N (x-\nu t)}
+(1+d) R \left(\frac{n}{N}\right)^{1/2-c} \\ 
&\qquad +t_0 (1+d(\log N)^{-2}) 3 e^{-\theta_N (x-\nu t)}\Big( 18s_0 (1+d(\log N)^{-2})   e^{-\theta_N (x-\nu t)}e^{5s_0 t_0}\\
&\hspace{7cm}- e^{-\frac 12  (2-\alpha) s_0 (\log N)^{-2}t_0}
\tfrac 12 s_0 (2-\alpha)(\log N)^{-2}
\Big)\\
&\qquad  +t_0 (1+d) R \left(\frac{n}{N}\right)^{1/2-c}\left(6(1+d) R  \left(\frac{n}{N}\right)^{1/2-c}-e^{-(1-\alpha)s_0 t_0}(1-\alpha)s_0 \right),
\end{align*}
where the second inequality holds since for $y \geq 0$, $e^{-y}=1-(1-e^{-y})\leq 1-ye^{-y}$.
Suppose $x$ is such that
$$
18(1+d(\log N)^{-2})   e^{-\theta_N (x-\nu t)}e^{5s_0 t_0}
- \tfrac 14 e^{-\frac 12  (2-\alpha )s_0 (\log N)^{-2} t_0}
  (2-\alpha)(\log N)^{-2} \le 0.
$$
Then
since $t_0 \in (0,1]$, and
 by~\eqref{eq:A1nNsmall} and the definition of $d$ in~\eqref{eq:ddef}, if $n$ is sufficiently large
we have that
\begin{equation} \label{eq:unboundholds}
u_t^n(x)< (1+(d-2t_0 d) (\log N)^{-2})3 e^{-\theta_N (x-\nu t)}+(1+d-2t_0 d ) R \left(\frac{n}{N}\right)^{1/2-c}.
\end{equation}
If instead $x\geq \nu t$ and
\begin{equation} \label{eq:pntailcase2}
 18(1+d(\log N)^{-2})   e^{-\theta_N (x-\nu t)}e^{5s_0 t_0}
>\tfrac 14 e^{-\frac 12  (2-\alpha )s_0 (\log N)^{-2} t_0}
 (2-\alpha)(\log N)^{-2},
\end{equation}
then 
since $T\leq \log N$, for $n$ sufficiently large we have $|x|\leq N$. 
Since $d<1/3$ and $t_0\leq 1$, we have that for $n$ sufficiently large,
\begin{align*}
(1+(d-2t_0 d)(\log N)^{-2} )3 e^{-\theta_N (x-\nu t)}
&\geq e^{-\kappa (x-\nu t)}+ e^{-\theta_N (x-\nu t)}\\
&> g(x-\nu t)+\sup_{y\in \frac 1n \Z , |y|\leq N, s\in [0,T]} |u^n_s(y)-g(y-\nu s)|
\end{align*}
by~\eqref{eq:pntailcase2} and our assumption in~\eqref{eq:nearg}.
Therefore for $n$ sufficiently large, in this case we also have that~\eqref{eq:unboundholds} holds.
Finally, for $n$ sufficiently large, if $x< \nu t$ then since $d<1/3$, $t_0\leq 1$ and $u^n_t(x)\le 1$ we have that~\eqref{eq:unboundholds} holds. 

Suppose that $\tau<T$; then~\eqref{eq:untau*} holds, and by setting $t=\tau$ and $t_0=1\wedge \tau$, we have a contradiction by~\eqref{eq:unboundholds}.
It follows that $\tau=T$, and so the first statement of the lemma holds.
The second statement follows by setting $t_0=1$ in~\eqref{eq:unboundholds}.
\end{proof}
The next lemma will give us a corresponding lower bound on $p^n_t(x)$ for large $x$.
\begin{lemma} \label{lem:untailinitminus}
There exists a constant $c_8\in (0,1)$ such that the following holds for $n$ sufficiently large.
Take $c\in (0,1/2)$.
Suppose for some $R>0$ that 
\begin{equation} \label{eq:untailminusicd}
p_0^n(x)\geq \tfrac 13 e^{-\kappa(1+(\log N)^{-2}) x }\1_{x\ge 0}-R \left(\frac{n}{N}\right)^{1/2-c}
\quad \forall x\in \tfrac1n \Z,
\end{equation}
and that for some $T\in (1,\log N]$,
$\sup_{y\in \frac 1n \Z, |y|\le N, t\in [0,T]} |u^n_t(y)-g(y-\nu  t)|\le c_8 (\log N)^{-2}$.
Then for $t\in [0,T]$, 
\begin{equation*} \label{eq:untailinitminus}
u^n_t(x)\geq \tfrac 14 e^{-\kappa(1+(\log N)^{-2}) (x-\nu  t) }\1_{x\ge \nu  t}- R\left(\frac{n}{N}\right)^{1/2-c} \quad \forall x\in \tfrac1n \Z,
\end{equation*}
and for $t\in [1,T]$, $\forall x\in \tfrac1n \Z$,
$$
u^n_t(x)\geq (1+ c_8 (\log N)^{-2}) \tfrac 13 e^{-\kappa(1+(\log N)^{-2})  (x-\nu  t) }\1_{x\ge \nu  t -c_8}-(1-c_8)R \left(\frac{n}{N}\right)^{1/2-c} .
$$
\end{lemma}

\begin{proof}
Note that for $u\in [0,1]$,
\begin{equation} \label{eq:f+pos}
f(u)+(1-\alpha)u=-2u^3+(3-\alpha)u^2 \ge 0.
\end{equation}
Take $d\in \big(0, \min\big(\frac 1{100} e^{-4(\kappa+2s_0)} (1-e^{-\frac 12 \kappa})  (2-\alpha)s_0 , \log (10/9) \kappa^{-1} \big)\big)$, and suppose
\begin{equation} \label{eq:unguntailminusA}
\sup_{y\in \frac 1n \Z, |y|\le N, t\in [0,T]} |u^n_t(y)-g(y-\nu  t)| \le d(\log N)^{-2}.
\end{equation}
Let $\theta '_N=(1+(\log N)^{-2})\kappa$.
For some $t_1 \in [0, T]$, suppose 
\begin{equation} \label{eq:unindhyp}
u^n_{t_1}(x)\ge \tfrac 13 e^{-\theta '_N (x-\nu t_1)}\1_{x\ge \nu t_1}-R \left( \frac n N \right)^{1/2-c}
\quad \forall x\in \tfrac 1n \Z.
\end{equation}
Take $t\in (t_1, t_1+1]$ and let $t_0 =t-t_1$.
Then for $x\in \frac 1n \Z$, by~\eqref{eq:ungreena},
\begin{align*}
u^n_t(x)
&= e^{-(1-\alpha)s_0 t_0} \langle u^n_{t_1}, \phi_0^{t_0,x}\rangle _n +s_0 \int_0^{t_0} e^{-(1-\alpha)s_0 (t_0-s)}
\langle f(u^n_{t_1+s})+(1-\alpha) u^n_{t_1+s}, \phi^{t_0,x}_s \rangle_n ds \\
& \ge e^{-(1-\alpha)s_0 t_0} \langle u^n_{t_1}, \phi_0^{t_0,x}\rangle _n
\end{align*}
by~\eqref{eq:f+pos}. Hence by~\eqref{eq:unindhyp},
\begin{equation} \label{eq:untailminusge*}
u^n_t(x) \ge e^{-(1-\alpha)s_0 t_0}
\left(
\Esubb{x}{\tfrac 13 e^{-\theta_N ' (X^n_{mt_0}-\nu t_1)} \1_{X^n_{mt_0}\ge \nu t_1}} -R \left( \frac n N \right)^{1/2-c}
\right).
\end{equation}
Note that
\begin{align} \label{eq:untailminusdiff*}
\Esubb{x}{e^{-\theta '_N (X^n_{mt_0}-\nu t_1)} \1_{X^n_{mt_0}\ge \nu t_1}}
&=\Esubb{x}{e^{-\theta '_N (X^n_{mt_0}-\nu t_1)}}-\Esubb{x}{e^{-\theta '_N (X^n_{mt_0}-\nu t_1)}\1_{X^n_{mt_0}< \nu t_1}} \notag 
\\
&= e^{-\theta '_N (x-\nu t_1)} e^{\frac 12 m(\theta '_N)^2 t_0 +\mathcal O(n^{-1} t_0)}
-e^{\theta '_N \nu t_1}\Esubb{x}{e^{-\theta '_N X^n_{m t_0}} \1_{X^n_{m t_0}< \nu t_1}}
\end{align}
by Lemma~\ref{lem:Xnmgf}.
For the second term on the right hand side,
\begin{align*}
\Esubb{x}{e^{-\theta '_N X^n_{mt_0}} \1_{X^n_{mt_0}< \nu t_1}}
&\le \sum_{k= \lfloor x-\nu t_1 \rfloor} ^\infty e^{-\theta '_N (x-k-1)} \psubb{x}{X^n_{mt_0}\le x-k}\\
&\le e^{-\theta '_N x} \sum_{k= \lfloor x-\nu t_1 \rfloor}^\infty  e^{\theta '_N(k+1)}e^{-2\theta_N ' k} e^{2m(\theta '_N)^2 t_0+\mathcal O(t_0 n^{-1})}\\
&\le e^{-\theta '_N x} e^{\theta '_N+2m(\theta '_N)^2 t_0+\mathcal O(t_0 n^{-1})}e^{-\theta '_N \lfloor x-\nu t_1 \rfloor}(1-e^{-\theta '_N})^{-1},
\end{align*}
where the second inequality follows by Markov's inequality and Lemma~\ref{lem:Xnmgf}.
Suppose $x\ge \nu t_1$ with
\begin{equation} \label{eq:untailminusA}
e^{-\theta '_N(x-\nu t_1)} \le e^{-3(\theta '_N+m(\theta '_N)^2)} (1-e^{-\theta '_N}) \tfrac 15 (2-\alpha)s_0 (\log N)^{-2}.
\end{equation}
Then by~\eqref{eq:untailminusdiff*} and since $t_0\le 1$, for $n$ sufficiently large,
\begin{align*}
& e^{-(1-\alpha)s_0 t_0}
\Esubb{x}{\tfrac 13 e^{-\theta '_N (X^n_{m t_0}-\nu t_1)} \1_{X^n_{m t_0}\ge \nu t_1}} \\
&\ge e^{-(1-\alpha)s_0 t_0}
\tfrac 13 e^{-\theta '_N (x-\nu t_1)} (e^{\frac 12 m(\theta '_N)^2 t_0 +\mathcal O(t_0 n^{-1})} -e^{3(\theta '_N+m(\theta '_N)^2)} e^{-\theta '_N(x-\nu t_1)}(1-e^{-\theta_N '})^{-1}) \\
&\ge 
\tfrac 13 e^{-\theta '_N (x-\nu t)}
e^{((-1+\alpha)s_0 -\theta '_N \nu +\frac 12 m(\theta '_N)^2 +\mathcal O(n^{-1})) t_0 }
 (1 -e^{3(\theta '_N+m(\theta '_N)^2)} e^{-\theta '_N(x-\nu t_1)}(1-e^{-\theta '_N})^{-1})\\
 &\ge 
\tfrac 13 e^{-\theta '_N (x-\nu t)}
e^{\frac 12 (2-\alpha)s_0 (\log N)^{-2} t_0 }
 (1-\tfrac 15 (2-\alpha)s_0 (\log N)^{-2})
\end{align*}
for $n$ sufficiently large, 
where the second inequality holds since $t_1=t-t_0$ and the last inequality follows
since 
$(-1+\alpha)s_0 -\theta '_N \nu +\frac 12 m(\theta '_N)^2 \ge (2-\alpha) s_0 (\log N)^{-2}$
and by our assumption~\eqref{eq:untailminusA} on $x$.

By~\eqref{eq:untailminusge*}, it follows that for
$n$ sufficiently large, if $x\ge \nu t_1$ and~\eqref{eq:untailminusA} holds,
then for $t\in (t_1,t_1+1]$,
$$
u^n_t(x) \ge \tfrac 13 e^{-\theta '_N (x-\nu t)}
e^{\frac 12 (2-\alpha)s_0 (\log N)^{-2}(t-t_1)}
(1-\tfrac 15 (2-\alpha)s_0 (\log N)^{-2}) -e^{-(1-\alpha)s_0(t-t_1)}R \left( \frac n N \right)^{1/2-c}.
$$
If instead 
$t\in (t_1,(t_1+1)\wedge T]$ and $x\ge \nu t$ with
$e^{-\theta '_N(x-\nu t_1)}> e^{-3(\theta '_N+m(\theta '_N)^2)}(1-e^{-\theta '_N}) \frac 15 (2-\alpha)s_0 (\log N)^{-2}$,
then if $n$ is sufficiently large, we have $|x|\le N$ and so by~\eqref{eq:unguntailminusA},
$$
u^n_t(x)\ge g(x-\nu t)-d(\log N)^{-2} \ge \tfrac 12 e^{-\kappa(x-\nu t)} - \tfrac 1{20} e^{-\theta '_N (x-\nu t_1)}\ge\tfrac 9{20} e^{-\theta '_N(x-\nu t)},
$$
where the second inequality follows since $g(y)\ge \tfrac 12 e^{-\kappa y}$ $\forall y\ge 0$ and by the definition of $d$
and our assumption on $x$.
For $x\in [\nu t-d,\nu t]$, by~\eqref{eq:unguntailminusA},
$$
u^n_t(x)\ge \tfrac 12 -d(\log N)^{-2}\ge \tfrac 25 e^{\theta '_N d}\ge \tfrac 25 e^{-\theta '_N (x-\nu t)}
$$
for $n$ sufficiently large, since $e^{\kappa d} \le 10/9$ by the definition of $d$.
Since~\eqref{eq:unindhyp} holds for $t_1=0$ by our assumption in~\eqref{eq:untailminusicd}, for $n$ sufficiently large that
$e^{\frac 9 {40} (2-\alpha)s_0 (\log N)^{-2}}(1-\frac 15 (2-\alpha)s_0(\log N)^{-2})\ge 1$,~\eqref{eq:unindhyp} holds for each $t_1 \in \frac 12 \N_0\cap [0,T]$ by induction.
Then for $t\in [1,T]$, there exists $t_1\in [0,T]$ such that~\eqref{eq:unindhyp} holds and with $t-t_1\in [1/2,1]$, and the result follows.
\end{proof}
The following result will allow us to show that $|u^n_{t,t+s}(x)-g(x-\mu^n_t -\nu  s)|$ is small in the proof of Proposition~\ref{prop:eventE1}.
\begin{lemma} \label{lem:gronwallun}
Suppose $(u^{n,1}_t)_{t\ge 0}$ and  $(u^{n,2}_t)_{t\ge 0}$ solve~\eqref{eq:undef} with initial conditions $p_0^{n,1}$ and $p_0^{n,2}$ respectively. Then for $t\ge 0$,
$$
\sup_{x\in \frac 1n \Z}| u^{n,1}_t(x) -u^{n,2}_t(x)|\le e^{(1+\alpha)s_0 t} \sup_{y\in \frac 1n \Z}| p^{n,1}_0(y) -p^{n,2}_0(y) |.
$$
\end{lemma}
\begin{proof}
By~\eqref{eq:ungreena}, for $x\in \frac 1n \Z$ and $t\ge 0$,
\begin{align*}
| u^{n,1}_t(x) -u^{n,2}_t(x)|
&\le \langle |p^{n,1}_0 -p^{n,2}_0|, \phi^{t,x}_0\rangle_n
+s_0 \int_0^t \langle |f(u^{n,1}_s)-f(u^{n,2}_s)|, \phi^{t,x}_s \rangle_n ds\\
&\le \sup_{y\in \frac 1n \Z}| p^{n,1}_0(y) -p^{n,2}_0(y) |+(1+\alpha) s_0 \int_0 ^t \sup_{y\in \frac 1n \Z}| u^{n,1}_s(y) -u^{n,2}_s(y)|ds
\end{align*}
since $\sup_{u\in [0,1]}|f'(u)|=1+\alpha$. The result follows by Gronwall's inequality.
\end{proof}

We are now ready to prove Proposition~\ref{prop:eventE1}.
\begin{proof}[Proof of Proposition~\ref{prop:eventE1}]
Without loss of generality, assume $b_2\in (0,1/3)$ is sufficiently small that $\left( \frac n N \right)^{1/3} \le n^{-b_2}$ for $n$ sufficiently large.
Take $c_5, c_6>0$ as defined in Lemma~\ref{lem:x0approx} and Proposition~\ref{prop:expconvtog}.
Let $b_1 =\frac 12 (c_5\wedge c_6)$, and suppose condition~\eqref{eq:conditionA} holds.
Define the event 
$$
A=\left\{
p^n_t(x)=0 \;\forall t \in [0,2N^2], x\ge N^5
 \right\}\cap
 \left\{
p^n_t(x)=1\; \forall t \in [0,2N^2], x\le -N^5
 \right\}.
$$
Recall from~\eqref{eq:Dn+-defn} that $D_n^+= (1/2-c_0)\kappa^{-1} \log (N/n)$.
Take $c_3\in (0,c_0\wedge 1/6)$, and take $\ell ' \in \N$ sufficiently large that $N^2 \left( \frac n N \right)^{\ell '}\le \left( \frac n N \right)^{\ell+1}$ for $n$ sufficiently large.
Take $c_4=c_4(c_3,\ell ')\in (0,1/2)$ as defined in Proposition~\ref{prop:pnun}, and let $T_0=(\log N)^{c_4}$.
By making $c_4$ smaller if necessary, we can assume $c_4<a_0$ (recall that $(\log N)^{a_0}\le \log n$ for $n$ sufficiently large).
For $k\in \Z$, let $t_k=(k+1)T_0$, and for $k\in \N_0$, let $(u^{n,k}_t)_{t\geq 0}$ denote the solution of
\begin{equation*}
\begin{cases}
\partial_t u^{n,k}_t &= \tfrac 12 m \Delta_n u^{n,k}_t +s_0 f(u^{n,k}_t) \quad \text{for } t>0,\\
u^{n,k}_0 &= p^n_{t_{k-1}}.
\end{cases}
\end{equation*}
For $k\in \N_0$, define the event
$$
A_{k}=
\left\{
\sup_{x\in \frac 1n \Z, |x|\leq N^5}\sup_{t\in [0,2T_0]}|p^n_{t+ t_{k-1}}(x)-u^{n,k}_t(x)|\leq \left(\frac{n}{N} \right)^{1/2-c_3}
 \right\}.
$$
Let $j_0=\lfloor N^2 T_0^{-1}\rfloor$.
Note that by a union bound, and then by Proposition~\ref{prop:pnun} and Lemma~\ref{lem:p01}, for $n$ sufficiently large,
\begin{equation} \label{eq:eventE1prob}
\p{A^c \cup \bigcup_{j=0}^{j_0+1} A^c_j }
\le 2e^{-N^5}+(j_0+2) \left( \frac n N \right)^{\ell '} \le \left( \frac n N \right)^\ell
\end{equation}
by our choice of $\ell '$.
From now on, suppose that the event
$A \cap \bigcap_{j=0}^{j_0 +1}A_{j} $ occurs.

For $k\in \N_0$, let $(u^k_t)_{t\geq 0}$ denote the solution of
\begin{equation*}
\begin{cases}
\partial_t u^k_t &= \tfrac 12 m \Delta u^k_t +s_0 f(u^k_t) \quad \text{for }t>0,\\
u_0^k&=\bar{p}^n_{t_{k-1}},
\end{cases}
\end{equation*}
where $\bar{p}^n_{t_{k-1}}:\R\to [0,1]$ is the linear interpolation of $p_{t_{k-1}}^n:\frac1n \Z \to [0,1]$.

Now for an induction argument, for $k\in \N_0$ with $k\leq j_0+1$, suppose there exists $z_{k-1}\in \R$ with $|z_{k-1}|\leq k$ such that
\begin{align}
D_k := \sup_{x\in \frac 1n \Z}|p^n_{ t_{k-1}}(x)-g(x-\nu  t_{k-1}-z_{k-1})|
&\le \tfrac 12 (c_5 \wedge c_6)=b_1 \label{eq:pgclose} \\
 \text{and} \quad \sup_{x_1,x_2\in \frac 1n \Z, |x_1-x_2|\le n^{-1/3}} |p^n_{ t_{k-1}}(x_1)-p^n_{t_{k-1}}(x_2)|
&\le n^{-b_2}. \label{eq:prough}
\end{align}
(Note that~\eqref{eq:pgclose} and~\eqref{eq:prough} hold for $k=0$, by condition~\eqref{eq:conditionA}.)
Then by the triangle inequality,
\begin{align} \label{eq:(J)thmpn}
\|\bar{p}^n_{t_{k-1}}-g(\cdot-\nu   t_{k-1}-z_{k-1})\|_\infty &\leq 
D_k +n^{-1} \|\nabla g\|_\infty  +n^{-b_2} \notag 
\\
&\leq c_5 \wedge c_6
\end{align}
for $n$ sufficiently large.
Hence by Proposition~\ref{prop:expconvtog}, there exists $z_{k}\in \R$ with $|z_k|\le k+1$ such that
\begin{equation} \label{eq:(*)pn}
|u^k_t(x)-g(x-\nu  ( t_{k-1}+t)-z_{k})| \leq C_3 e^{-  c_6 t} \quad \forall x\in \R, \; t>0.
\end{equation}
Therefore by Lemma~\ref{lem:unu},
for $t\in [0,2T_0]$,
\begin{align} \label{eq:(**)pn}
\sup_{x\in \frac 1n \Z}|u^{n,k}_t(x)-g(x-\nu  ( t_{k-1}+t)-z_k)|\leq (C_4 n^{-1/3}+2n^{-b_2}  )4T_0^2 e^{2(1+\alpha)s_0 T_0}+C_3 e^{-c_6 t}.
\end{align}
Then by the definition of the event $A_{k}$, for $t\in [ T_0,2T_0]$,
\begin{align*}
&\sup_{x\in \frac 1n \Z, |x|\leq N^5}|p^n_{ t_{k-1}+t}(x)-g(x-\nu  ( t_{k-1}+t)-z_k)|\\
&\leq \left( \frac{n}{N}\right)^{1/2-c_3}+(C_4 n^{-1/3}+2n^{-b_2}) 4T_0^2 e^{2(1+\alpha)s_0 T_0}+C_3 e^{-c_6 T_0}\\
&\leq e^{-\frac 12 c_6 T_0}
\end{align*}
for $n$ sufficiently large.
Therefore, for $n$ sufficiently large, since $k\leq j_0+1$ and $|z_k|\leq k+1$, and by the definition of the event $A$, we have that for $t\in [ T_0,2T_0]$,
\begin{align} \label{eq:pgclose**}
&\sup_{x\in \frac 1n \Z}|p^n_{ t_{k-1}+t}(x)-g(x-\nu  (t_{k-1}+t)-z_k)| \notag \\
&\qquad \leq \max\left(e^{-\frac 12 c_6 T_0},\sup_{y\geq N^5-N^3}g(y),\sup_{y\leq -N^5 +N^2}(1-g(y)) \right) 
= e^{-\frac 12 c_6 T_0}.
\end{align}
By the definitions of the events $A_{k}$ and $A$, and then by Lemma~\ref{lem:nablaunbound} and our choice of $b_2$ and $c_3$, we have that
\begin{align*}
\sup_{x_1,x_2\in \frac 1n \Z, |x_1-x_2|\le n^{-1/3}} |p^n_{ t_k}(x_1)-p^n_{t_k}(x_2)|
&\le n^{-1} \lfloor n^{2/3} \rfloor \sup_{x\in \frac 1n \Z} |\nabla_n u^{n,k}_{T_0}(x)|+2\left( \frac n N \right)^{1/2-c_3}\\
&\le n^{-b_2}
\end{align*}
for $n$ sufficiently large.
By induction, we now have that for $n$ sufficiently large, for $k\in \N$ with $k\leq j_0+1$,
there exists $z_{k-1}\in \R$ with $|z_{k-1}|\le k$ such that~\eqref{eq:pgclose} and~\eqref{eq:prough} hold
with $D_k \le e^{-\frac 12 c_6 T_0}$.
By Lemma~\ref{lem:x0approx} and~\eqref{eq:(J)thmpn}, if $n$ is sufficiently large then for $t\geq 0$ and $x\in \R$,
$$
|u^k_t(x)-g(x-\nu  ( t_{k-1}+t)-z_{k-1})|\leq C_2(D_k +2n^{-b_2})
$$
and so by~\eqref{eq:(*)pn}, $\|g(\cdot-z_k)-g(\cdot-z_{k-1})\|_\infty \leq C_2(D_k +2n^{-b_2})$.
For $n$ sufficiently large, since $\nabla g(0)=-\kappa /4$, it follows that 
\begin{equation*} 
|z_{k-1}-z_k|\leq 5\kappa^{-1} C_2(D_k +2n^{-b_2})
\leq e^{-\frac 13 c_6 T_0}.
\end{equation*} 
Therefore, by~\eqref{eq:pgclose**}, for $n$ sufficiently large, for $k\in \N_0$ with $k\le j_0$,
\begin{align} \label{eq:(A)pn}
|z_{k+1}-z_{k}| \leq e^{-\frac 13 c_6 T_0}\quad 
\text{and} \quad \sup_{t\in [ t_k, t_{k+1}],\, x\in \frac 1n \Z}
|p^n_t(x)-g(x-\nu  t-z_k)|
&\leq e^{-\frac 12 c_6 T_0} .
\end{align}
Note that for $k\in \N_0$ with $k \le j_0$, by~\eqref{eq:(A)pn},
\begin{align} \label{eq:ungbound}
&\sup_{x\in \frac 1n \Z, |x-(z_k+\nu   t_k)|\le N,\, t\in [0, T_0]}
|u^{n,k+1}_t (x)-g(x-\nu  (t+ t_k)-z_k)| \notag \\
&\le e^{-\frac 12 c_6 T_0} +\sup_{|x|\le N^5,\, t\in [0, T_0]}|u^{n,k+1}_t(x)-p^n_{t+ t_k}(x)| \notag \\
&\le e^{-\frac 12 c_6 T_0}  +\left( \frac n N \right)^{1/2-c_3}
\end{align}
by the definition of the event $A_{k+1}$.

We now use Lemma~\ref{lem:untailinit} to prove an upper bound on $p^n_t(x)$ for large $x$.
Let $c_9=c_7\wedge c_8 \in (0,1)$ and $R_{0}= e^{-\frac 12 c_6 T_0}\left( \frac n N \right)^{-(1/2-c_3)}$.
Define $(R_k)_{k=1}^\infty$ inductively by letting
$R_{k}=(1-c_9)R_{k-1}+1$ for $k\ge 1$.
Let 
$$
k^* =\frac{\log (2c_9^{-1})-\log R_0}{\log (1-c_9/2)}.
$$
Then since $R_k \le (1-c_9/2)R_{k-1}$ if $R_{k-1}\ge 2c_9^{-1}$
and $R_k \le 2c_9^{-1}-1$ if $R_{k-1}\le 2c_9^{-1}$, we have
  $R_k \le 2c_9^{-1}$ for $k\ge k^*$.
Suppose $n$ is sufficiently large that
$e^{-\frac 12 c_6 T_0}\le c_9$
and $e^{-\frac 12 c_6 T_0}+\left( \frac n N \right)^{1/2-c_3}\le c_9 (\log N)^{-2}$.
Then by Lemma~\ref{lem:untailinit},~\eqref{eq:ungbound} and the definition of the event $A$,
for $k\in \N_0$ with $k\le j_0$, 
if
\begin{equation} \label{eq:pnsk}
p^n_{t_k}(x)\le 3 e^{-\kappa(1-(\log N)^{-2})(x-\nu  t_k -z_k)} +R_{k}\left( \frac n N \right)^{1/2-c_3}
\quad \forall x\in \tfrac 1n \Z,
\end{equation}
then for $t\in [0,T_0],$
$$
u^{n,k+1}_{t}(x)\le \tfrac 43 \left( 3 e^{-\kappa(1-(\log N)^{-2})(x-\nu  (t+ t_k) -z_k)} +R_{k}\left( \frac n N \right)^{1/2-c_3}\right) 
\quad \forall x\in \tfrac 1n \Z .
$$
Therefore, by the definition of the events $A_{k+1}$ and $A$, for $t\in [t_k,t_{k+1}]$
and $x\in \frac 1n \Z$,
\begin{equation} \label{eq:pntailE1*}
p^n_t(x) \le 4 e^{-\kappa(1-(\log N)^{-2})(x-\nu  t-z_k)}
+(1+\tfrac 43 R_k) \left( \frac n N \right)^{1/2-c_3}.
\end{equation}
Moreover, by Lemma~\ref{lem:untailinit} and~\eqref{eq:ungbound},
for $t\in [1,T_0]$ and $x\in \frac 1n \Z$,
$$
u_t^{n,k+1}(x) \le (1-c_7 (\log N)^{-2}) 3e^{-\kappa(1-(\log N)^{-2})(x-\nu  (t+t_k)-z_k)}+(1-c_7)R_k \left( \frac n N \right)^{1/2-c_3},
$$
and so by the definition of the events $A_{k+1}$ and $A$, for $x\in \frac 1n \Z$,
\begin{align*}
p^n_{t_{k+1}}(x) &\le(1-c_7 (\log N)^{-2}) 3 e^{-\kappa(1-(\log N)^{-2})(x-\nu  t_{k+1}-z_k)} +(1+(1-c_7)R_{k})\left( \frac n N \right)^{1/2-c_3}\\
&\le 3 e^{-\kappa(1-(\log N)^{-2})(x-\nu  t_{k+1}-z_{k+1})} +R_{k+1}\left( \frac n N \right)^{1/2-c_3}
\end{align*}
for $n$ sufficiently large, by the definition of $R_{k+1}$ and
since $|z_k-z_{k+1}|\le e^{-\frac 13 c_6 T_0}$ by~\eqref{eq:(A)pn}.
Note that~\eqref{eq:pnsk} holds for $k=0$ by~\eqref{eq:(A)pn} and the definition of $R_0$, and since $g(y)\le e^{-\kappa y}\wedge 1$ $\forall y\in \R$.
Hence by induction,~\eqref{eq:pnsk} holds for each $0\le k \le j_0$. Therefore, by~\eqref{eq:pntailE1*},
for $k\ge k^*$, for $t\in [t_k,t_{k+1}]$ and $x\in \frac 1n \Z$,
\begin{equation} \label{eq:pntailE1A}
p^n_t(x)\le 4e^{-\kappa(1-(\log N)^{-2})(x-\nu  t-z_k)}+(1+\tfrac 83 c_9^{-1}) \left( \frac n N \right)^{1/2-c_3}.
\end{equation}

We now use Lemma~\ref{lem:untailinitminus} to establish a corresponding lower bound.
By Lemma~\ref{lem:untailinitminus} and~\eqref{eq:ungbound}, if for some $k\in \N_0$ with $k \le j_0$
\begin{equation} \label{eq:plowerboundhyp}
p^n_{t_k}(x)\ge \tfrac 13 e^{-\kappa(1+(\log N)^{-2})(x-\nu  t_k -z_k)}\1_{x\ge \nu  t_k +z_k} -R_{k}\left( \frac n N \right)^{1/2-c_3}
\quad \forall x\in \tfrac 1n \Z,
\end{equation}
then for $t\in [0,T_0]$,
$$
u^{n,k+1}_{t}(x)\ge \tfrac 14 e^{-\kappa(1+(\log N)^{-2})(x-\nu  (t+t_k) -z_k)}\1_{x\ge \nu (t_k +t)+z_k} -R_{k}\left( \frac n N \right)^{1/2-c_3}
\quad \forall x\in \tfrac 1n \Z .
$$
Hence by the definition of the event $A_{k+1}$ and since $p^n_t\ge 0$, for $t\in [t_k,t_{k+1}]$ and $x\in \frac 1n \Z$,
\begin{equation} \label{eq:pntaileventE1dagger}
p^n_t(x) \ge \tfrac 14 e^{-\kappa(1+(\log N)^{-2})(x-\nu  t-z_k)}\1_{x\ge \nu  t+z_k}
-(1+R_k)\left( \frac n N \right)^{1/2-c_3}.
\end{equation}
Moreover, by Lemma~\ref{lem:untailinitminus} and~\eqref{eq:ungbound}, for $t\in [1,T_0]$ and $x\in \frac 1n \Z$,
\begin{align*}
u^{n,k+1}_t (x)
\ge  (1+c_8 (\log N)^{-2}) \tfrac 13 e^{-\kappa(1+(\log N)^{-2})(x-\nu  (t+t_k)-z_k)}&\1_{x\ge \nu  (t_k+t)+z_k-c_8}\\
&- (1-c_8) R_k\left( \frac n N \right)^{1/2-c_3},
\end{align*}
and so by the definition of the event $A_{k+1}$ and since $p^n_t\ge 0$, for $x\in \frac 1n \Z$,
\begin{align*}
 p^n_{t_{k+1}}(x) 
&\ge (1+c_8 (\log N)^{-2}) \tfrac 13 e^{-\kappa(1+(\log N)^{-2})(x-\nu  t_{k+1}-z_k)}\1_{x\ge \nu  t_{k+1} +z_k-c_8} \\
&\hspace{5cm} -((1-c_8)R_{k}+1)\left( \frac n N \right)^{1/2-c_3}\\
&\ge \tfrac 13 e^{-\kappa(1+(\log N)^{-2})(x-\nu  t_{k+1}-z_{k+1})} \1_{x\ge \nu  t_{k+1} +z_{k+1}} -R_{k+1}\left( \frac n N \right)^{1/2-c_3}
\end{align*}
for $n$ sufficiently large, by the definition of $R_{k+1}$ and
since $|z_k-z_{k+1}|\le e^{-\frac 13 c_6 T_0}$.
By~\eqref{eq:(A)pn} and the definition of $R_0$, and since $g(z)\ge \frac 12 e^{-\kappa z}$ for $z\ge 0$,~\eqref{eq:plowerboundhyp} holds for $k=0$.
Hence by induction,~\eqref{eq:plowerboundhyp} holds for each $0\le k\le j_0$.
Then by~\eqref{eq:pntaileventE1dagger}, for $k\ge k^*$, for $t\in [t_k,t_{k+1}]$ and $x\in \frac 1n \Z$,
\begin{equation} \label{eq:pntailE1B}
p^n_t(x)\ge \tfrac 14 e^{-\kappa(1+(\log N)^{-2})(x-\nu  t-z_k)}\1_{x\ge \nu  t +z_k}- (1+2c_9^{-1}) \left( \frac n N \right)^{1/2-c_3}.
\end{equation}

We are now ready to complete the proof.
Take $c_2\in (0,c_4)$. Recall that for $t\ge 0$, $\mu^n_t =\sup\{x\in \frac 1n \Z:p^n_t(x)\ge 1/2\}$.
By~\eqref{eq:(A)pn} and since $\nabla g(0)=-\kappa /4$, for $n$ sufficiently large, for $k\in \N_0$ with $k\le j_0$, for $t\in [t_k,t_{k+1}]$,
\begin{equation} \label{eq:pnunend*}
|(\nu  t+z_k)-\mu^n_t| \le 5 \kappa^{-1} e^{-\frac 12 c_6 T_0}.
\end{equation}
Therefore, for $n$ sufficiently large, by~\eqref{eq:(A)pn},
\begin{equation} \label{eq:pntailE1C}
\sup_{x\in \frac 1n \Z, t\in [T_0,N^2]}
|p^n_t(x)-g(x-\mu^n_t)|
\le e^{-\frac 12 c_6 T_0} +5 \kappa^{-1} e^{-\frac 12 c_6 T_0} \|\nabla g\|_\infty \le e^{-2(\log N)^{c_2}}
\end{equation}
since $c_2<c_4$.
By~\eqref{eq:pnunend*}
and since $|z_0|\le 1$ and $|z_k-z_{k-1}|\le e^{-\frac 13 c_6 T_0}$ $\forall k\in \N$ with $k\le j_0$, if $n$ is sufficiently large we have
$|\mu^n_{\log N}|\le 2\nu \log N$ and for $t\in  [\log N,N^2]$ and $s\in [0,1]$ with $t+s\le N^2$,
$$|\mu^n_{t+s}-\mu^n_t -\nu  s |\le 10 \kappa^{-1} e^{-\frac 12 c_6 T_0}+e^{-\frac 13 c_6 T_0} \le e^{-(\log N)^{c_2}}.
$$
Now for $ t\in  [\frac 12 (\log N)^2,N^2]$, take $x\in \frac 1n \Z$ such that $g(x-\mu^n_t)\le 2e^{-(\log N)^{c_2}}$.
Then for $n$ sufficiently large that $k^*\le \frac 12 (\log N)^{3/2}$, by~\eqref{eq:pntailE1A} and~\eqref{eq:pnunend*},
$$
p^n_t(x)\le 4e^{-\kappa(1-(\log N)^{-2})(x-\mu^n_t -5 \kappa^{-1} e^{-\frac 12 c_6 T_0})}+(1+\tfrac 83c_9^{-1}) \left( \frac n N \right)^{1/2-c_3} \le 5 g((x-\mu^n_t)\wedge D_n^+))
$$
for $n$ sufficiently large, since $\kappa D^+_n (\log N)^{-1}\le 1/2$, $c_3<c_0$ and $g(y) \sim e^{-\kappa y}$ as $y\to \infty$.
Similarly, for $n$ sufficiently large, by~\eqref{eq:pntailE1B} and~\eqref{eq:pnunend*}, if $x-\mu^n_t\le D^+_n+2$ then
\begin{align*}
p^n_t(x)
&\ge \tfrac 14 e^{-\kappa(1+(\log N)^{-2})(x-\mu^n_t+5\kappa^{-1} e^{-\frac 12 c_6 T_0})}-(1+2c_9^{-1})\left( \frac n N \right)^{1/2-c_3}
\ge \tfrac 15 g(x-\mu^n_t).
\end{align*}
If instead $g(x-\mu^n_t)\ge 2e^{-(\log N)^{c_2}}$, then $p^n_t(x) \in [\frac 12 g(x-\mu^n_t),\frac 32 g(x-\mu^n_t)]$ by~\eqref{eq:pntailE1C}.

Finally, for $t\in [\log N,N^2]$, let $(\tilde u^n_{t,t+s})_{s\ge 0}$ solve~\eqref{eq:unttsdef} with $\tilde u^n_{t,t}(x)=g(x-\mu^n_t)$ for $x\in \frac 1n \Z$. Then for $s\in [0,\gamma_n]$, by Lemma~\ref{lem:gronwallun} and~\eqref{eq:pntailE1C},
\begin{align*}
&\sup_{x\in \frac 1n \Z}|u^n_{t,t+s}(x)-g(x-\mu^n_t-\nu  s)|\\
&\quad \le e^{(1+\alpha)s_0 \gamma_n} e^{-2(\log N)^{c_2}}+\sup_{x\in \frac 1n \Z}|\tilde u^n_{t,t+s}(x)-g(x-\mu^n_t-\nu  s)|\\
&\quad \le e^{(1+\alpha)s_0\gamma_n} e^{-2(\log N)^{c_2}}+(C_4 +\|\nabla g\|_\infty) n^{-1/3}\gamma_n^2 e^{(1+\alpha)s_0 \gamma_n}\\
&\quad \le e^{-(\log N)^{c_2}}
\end{align*}
for $n$ sufficiently large, where the second inequality follows
by Lemma~\ref{lem:unu} and since $(g(\cdot -\mu^n_t -\nu  s))_{s\ge 0}$ solves~\eqref{eq:ueq}.
The result follows by~\eqref{eq:eventE1prob}.
\end{proof}

\subsection{Proof of Proposition~\ref{prop:pnun}} \label{subsec:pnunproof}

The proof of Proposition~\ref{prop:pnun} uses similar arguments to those in \cite{durrett/fan:2016}.
The following lemma is the main step in the proof.

\begin{lemma} \label{lem:qnphi}
Suppose $\phi:[0,\infty) \times \frac 1n \Z \to \R$ is continuously differentiable in $t$,
and write $\phi_t(x):= \phi(t,x)$. Suppose 
that for any $t>0$,
$\sup_{s\in [0,t]}\langle |\phi_s |,1\rangle_n<\infty$ and $\sup_{s\in [0,t]}\langle |\partial_s \phi_s |,1\rangle_n<\infty$.
Then for $t\ge 0$,
\begin{align} \label{eq:lemqnphi}
&\langle q^n_t, \phi_t \rangle_n -\langle q^n_0 , \phi_0 \rangle_n -\int_0^t \langle q^n_s, \partial_s \phi_s \rangle_n ds \notag \\
&\quad = s_0 \int_0^t \langle q^n_s (1-p^n_s)(2p^n_s-1+\alpha ), \phi_s \rangle_n ds
+\tfrac 12 m \int_0^t \langle q^n_s, \Delta_n \phi_s \rangle_n ds +M^n_t(\phi),
\end{align}
where $(M^n_t(\phi))_{t\ge 0}$ is a martingale with $M^n_0(\phi)=0$ and
$$
\langle M^n(\phi)\rangle_t \le \frac n N \int_0^t \langle (1+m) q^n_s(\cdot) +\tfrac 12 m (q^n_s(\cdot -n^{-1}) + q^n_s(\cdot +n^{-1})), \phi_s^2 \rangle_n ds.
$$
\end{lemma}
Before proving Lemma~\ref{lem:qnphi}, we prove the following useful consequence.
\begin{cor} \label{cor:qnMa}
For $a\in \R$, $t\ge 0$ and $z\in \frac 1n \Z$,
\begin{align}  \label{eq:qnC}
q^n_t(z) &= e^{-at}\langle q^n_0, \phi^{t,z}_0 \rangle _n
+ 
\int_0^t e^{-a(t-s)}\langle q^n_{s}(s_0(1-p^n_{s})(2p^n_{s}-1+\alpha )+a),\phi^{t,z}_s \rangle _n  ds 
+M^n_t(\phi^{t,z,a}).
\end{align}
\end{cor}

\begin{proof}
Recall the definitions of $\phi^{t,z}$ and $\phi^{t,z,a}$ in \eqref{eq:phitzdefq} and~\eqref{eq:phiadef}.
Note that $\partial_s \phi^{t,z}_s +\frac12 m \Delta _n \phi^{t,z}_s=0$ for $s\in (0,t)$. 
Hence
$$
\partial_s \phi_s^{t,z,a}+\tfrac 12 m \Delta_n \phi^{t,z,a}_s=a\phi_s^{t,z,a}.
$$
Therefore, by substituting
$\phi_s(x):=\phi_s^{t,z,a}(x)$  into~\eqref{eq:lemqnphi} in Lemma~\ref{lem:qnphi} we have 
\begin{align*}
\langle q^n_t, \phi^{t,z,a}_t \rangle _n &= \langle q^n_0, \phi^{t,z,a}_0 \rangle _n
+ 
 \int_0^t  \langle q^n_{s}(s_0(1-p^n_{s})(2p^n_{s}-1+\alpha )+a),\phi^{t,z,a}_s \rangle _n  ds
+M^n_t(\phi^{t,z,a}).
\end{align*}
Since $\phi_t^{t,z,a}(w)=n\1_{w=z}$,
the result follows. 
\end{proof}

\begin{proof}[Proof of Lemma~\ref{lem:qnphi}]
For $t\ge 0$, $x\in \frac 1n \Z$ and $i\in [N]$, by the definition of $\eta^n$ in~\eqref{eq:etandefn} we have that
\begin{align*}
\eta^n_t(x,i)
&= \eta_0^n(x,i)
+\sum_{j\in [N] \setminus \{i\}} \int_0^t (\eta_{s-}^n(x,j)-\eta_{s-}^n(x,i))d \mathcal P_s^{x,i,j}\\
&\hspace{1.8cm}+\sum_{j\in [N]\setminus \{i\}} \int_0^t \xi_{s-}^n(x,j)(\eta_{s-}^n(x,j)-\eta_{s-}^n(x,i))d \mathcal S_s^{x,i,j}\\
&\hspace{1.8cm}+\sum_{j\neq k\in [N]\setminus \{i\}} \int_0^t \1_{\xi_{s-}^n(x,j)=\xi_{s-}^n(x,k)}(\eta_{s-}^n(x,j)-\eta_{s-}^n(x,i))d \mathcal Q_s^{x,i,j,k}\\
&\hspace{1.8cm}+\sum_{j\in [N],y\in \{x-n^{-1},\, x+n^{-1}\}} \int_0^t (\eta_{s-}^n(y,j)-\eta_{s-}^n(x,i))d \mathcal R_s^{x,i,y,j}.
\end{align*}
Recall from~\eqref{eq:qndef} that $q^n_s(y)=N^{-1} \sum_{j\in [N]}\eta^n_s(y,j)$ for $y\in \frac 1n \Z$ and $s\ge 0$.
By integration by parts applied to $\eta^n_t(x,i) \phi_t(x)$, and then summing over $i$ and $x$,
using our assumptions on $\phi$,
\begin{align} \label{eq:qnA}
&\langle q^n_t, \phi_t \rangle _n - \langle q^n_0, \phi_0 \rangle _n
-\int_0^t \langle q^n_s, \partial_s\phi_s \rangle _n ds \notag \\
&\qquad = 
\frac{1}{Nn}\sum_{x\in \frac1n \Z}\sum_{i=1}^N \sum_{j\in [N] \setminus \{i\}} \int_0^t (\eta_{s-}^n(x,j)-\eta_{s-}^n(x,i))\phi_s(x)d \mathcal P_s^{x,i,j} \notag\\
&\hspace{1.3cm}+\frac{1}{Nn}\sum_{x\in \frac1n \Z}\sum_{i=1}^N\sum_{j\in[N]\setminus \{i\}} \int_0^t \xi_{s-}^n(x,j)(\eta_{s-}^n(x,j)-\eta_{s-}^n(x,i))\phi_s(x) d \mathcal S_s^{x,i,j} \notag\\
&\hspace{1.3cm}+\frac{1}{Nn}\sum_{x\in \frac1n \Z}\sum_{i=1}^N\sum_{j\neq k\in [N]\setminus \{i\}} \int_0^t \1_{\xi_{s-}^n(x,j)=\xi_{s-}^n(x,k)}(\eta_{s-}^n(x,j)-\eta_{s-}^n(x,i))\phi_s(x) d \mathcal Q_s^{x,i,j,k}\notag \\
&\hspace{1.3cm}+\frac{1}{Nn}\sum_{x\in \frac1n \Z}\sum_{i=1}^N\sum_{j\in [N] ,y\in \{x-n^{-1}, \, x+n^{-1}\}} \int_0^t (\eta_{s-}^n(y,j)-\eta_{s-}^n(x,i))\phi_s(x) d \mathcal R_s^{x,i,y,j}.
\end{align}
We shall consider each line on the right hand side of~\eqref{eq:qnA} separately. For the first line, 
\begin{align*}
A^1_t
&:= \frac{1}{Nn}\sum_{x\in \frac1n \Z}\sum_{i=1}^N \sum_{j\in[N]\setminus \{i\}} \int_0^t (\eta_{s-}^n(x,j)-\eta_{s-}^n(x,i))\phi_s(x)d \mathcal P_s^{x,i,j}\\
&=\frac{1}{Nn}\sum_{x\in \frac1n \Z}\sum_{i=1}^N \sum_{j\in[N]\setminus \{i\}} \int_0^t (\eta_{s-}^n(x,j)-\eta_{s-}^n(x,i))\phi_s(x)(d \mathcal P_s^{x,i,j}-r_n (1-(\alpha+1)s_n)ds)\\
&\quad + \frac{1}{Nn}\sum_{x\in \frac1n \Z}\sum_{i=1}^N \sum_{j\in [N] \setminus \{i\}} \int_0^t (\eta_{s-}^n(x,j)-\eta_{s-}^n(x,i))\phi_s(x)r_n (1-(\alpha+1)s_n)ds.
\end{align*}
Now for $x\in \frac1n \Z$ and $s\in [0,t]$,
$$
\sum_{i=1}^N \sum_{j\in [N]\setminus \{i\}} (\eta^n_{s-}(x,j)-\eta^n_{s-}(x,i))=0.
$$
Hence
\begin{align} \label{eq:A1q}
A^1_t
&=M^{n,1}_t(\phi) \notag\\
&:=\frac{1}{Nn}\sum_{x\in \frac1n \Z}\sum_{i=1}^N \sum_{j\in [N]\setminus \{i\}} \int_0^t (\eta_{s-}^n(x,j)-\eta_{s-}^n(x,i))\phi_s(x)(d \mathcal P_s^{x,i,j}-r_n (1-(\alpha+1)s_n)ds),
\end{align}
which is a martingale (since we assumed $\sup_{s\in [0,t']}\langle |\phi_s | ,1\rangle_n <\infty$ for any $t'>0$).
For the second line on the right hand side of~\eqref{eq:qnA},
\begin{align*}
A^2_t &:=
\frac{1}{Nn}\sum_{x\in \frac1n \Z}\sum_{i=1}^N\sum_{j\in [N]\setminus \{i\}} \int_0^t \xi_{s-}^n(x,j)(\eta_{s-}^n(x,j)-\eta_{s-}^n(x,i))\phi_s(x) d \mathcal S_s^{x,i,j}\\
&=\frac{1}{Nn}\sum_{x\in \frac1n \Z}\sum_{i=1}^N\sum_{j\in [N]\setminus \{i\}} \int_0^t \xi_{s-}^n(x,j)(\eta_{s-}^n(x,j)-\eta_{s-}^n(x,i))\phi_s(x) (d \mathcal S_s^{x,i,j}-r_n\alpha s_n ds)\\
&\quad + \frac{1}{Nn}\sum_{x\in \frac1n \Z}\sum_{i=1}^N\sum_{j\in [N]\setminus \{i\}} \int_0^t \xi_{s-}^n(x,j)(\eta_{s-}^n(x,j)-\eta_{s-}^n(x,i))\phi_s(x) r_n\alpha s_n ds .
\end{align*}
For the expression on the last line, for $x\in \frac 1n \Z$ and $s\in [0,t]$, since $\xi_{s-}^n(x,j)=1$ if $\eta_{s-}^n(x,j)=1$,
\begin{align*}
&\sum_{i=1}^N\sum_{j\in [N]\setminus \{i\}} \xi_{s-}^n(x,j)(\eta_{s-}^n(x,j)-\eta_{s-}^n(x,i))\\
&\quad =\sum_{i=1}^N\sum_{j\in [N] \setminus \{i\}} \eta_{s-}^n(x,j)-\sum_{i=1}^N \eta^n_{s-}(x,i)\left(\sum_{j=1}^N \xi^n_{s-}(x,j)-1\right)\\
&\quad = (N-1)N q^n_{s-}(x)-N q^n_{s-}(x)(Np^n_{s-}(x)-1)\\
&\quad = N^2 q^n_{s-}(x)(1-p^n_{s-}(x)).
\end{align*}
Therefore we can write
\begin{align*}
&\frac{1}{Nn}\sum_{x\in \frac1n \Z}\sum_{i=1}^N\sum_{j\in [N] \setminus \{i\}} \int_0^t \xi_{s-}^n(x,j)(\eta_{s-}^n(x,j)-\eta_{s-}^n(x,i))\phi_s(x) r_n\alpha s_n ds \\
&=\alpha Nr_n s_n \int_0^t \langle q_{s-}^n(1-p_{s-}^n),\phi_s \rangle _n ds.
\end{align*}
Hence, since $Nr_n s_n=s_0$,
\begin{equation} \label{eq:A2q}
A^2_t=\alpha s_0 \int_0^t \langle q_{s}^n(1-p_{s}^n),\phi_s \rangle _n ds
+M_t^{n,2}(\phi),
\end{equation}
where 
\begin{equation} \label{eq:M2defq}
M_t^{n,2}(\phi):=\frac{1}{Nn}\sum_{x\in \frac1n \Z}\sum_{i=1}^N\sum_{j\in [N]\setminus \{i\}} \int_0^t \xi_{s-}^n(x,j)(\eta_{s-}^n(x,j)-\eta_{s-}^n(x,i))\phi_s(x) (d \mathcal S_s^{x,i,j}-r_n\alpha s_n ds)
\end{equation}
is a martingale.
For the third line on the right hand side of~\eqref{eq:qnA},
\begin{align*}
A^3_t&:=
\frac{1}{Nn}\sum_{x\in \frac1n \Z}\sum_{i=1}^N\sum_{j\neq k\in [N] \setminus \{i\}} \int_0^t \1_{\xi_{s-}^n(x,j)=\xi_{s-}^n(x,k)}(\eta_{s-}^n(x,j)-\eta_{s-}^n(x,i))\phi_s(x) d \mathcal Q_s^{x,i,j,k}\\
&=\frac{1}{Nn}\sum_{x\in \frac1n \Z}\sum_{i=1}^N\sum_{j\neq k\in [N]\setminus \{i\}} \int_0^t \1_{\xi_{s-}^n(x,j)=\xi_{s-}^n(x,k)}(\eta_{s-}^n(x,j)-\eta_{s-}^n(x,i))\phi_s(x) (d \mathcal Q_s^{x,i,j,k}-\tfrac1N r_n s_n ds)\\
&\quad +\frac{1}{Nn}\sum_{x\in \frac1n \Z}\sum_{i=1}^N\sum_{j\neq k\in [N] \setminus \{i\}} \int_0^t \1_{\xi_{s-}^n(x,j)=\xi_{s-}^n(x,k)}(\eta_{s-}^n(x,j)-\eta_{s-}^n(x,i))\phi_s(x) \tfrac1N r_n s_n ds.
\end{align*}
For $x\in \frac 1n \Z$ and $s\in [0,t]$, since $\eta^n_{s-}(x,j)=0$ if $\xi^n_{s-}(x,j)=0$,
\begin{align*}
&\sum_{i=1}^N\sum_{j\neq k\in [N] \setminus \{i\}}  \1_{\xi_{s-}^n(x,j)=\xi_{s-}^n(x,k)}(\eta_{s-}^n(x,j)-\eta_{s-}^n(x,i))\\
&\quad =\sum_{i,j,k \in [N] \text{ distinct}}  \Big( \1_{\eta_{s-}^n(x,j)=\xi_{s-}^n(x,k)=1}
- \1_{\xi_{s-}^n(x,j)=\xi_{s-}^n(x,k)=\eta_{s-}^n(x,i)=1}\\
&\hspace{8cm}- \1_{\xi_{s-}^n(x,j)=\xi_{s-}^n(x,k)=0,\, \eta_{s-}^n(x,i)=1}
\Big) \\
&\quad = (N-2)N q^n_{s-}(x)(Np^n_{s-}(x)-1)-N q^n_{s-}(x)(Np^n_{s-}(x)-1)(Np^n_{s-}(x)-2)\\
&\hspace{1cm} -Nq^n_{s-}(x)(N-Np^n_{s-}(x))(N-N p^n_{s-}(x)-1)\\
&\quad = N^3 q^n_{s-}(x) (1-p^n_{s-}(x))(2p^n_{s-}(x)-1).
\end{align*}
Therefore, since $N r_n s_n=s_0$,
\begin{equation} \label{eq:A3q}
A^3_t= s_0 \int_0^t  \langle q^n_{s}(1-p^n_{s})(2p^n_{s}-1),\phi_s \rangle _n  ds +M^{n,3}_t(\phi),
\end{equation}
where
\begin{align} \label{eq:M3defq}
&M^{n,3}_t(\phi) \notag \\
&:=\frac{1}{Nn}\sum_{x\in \frac1n \Z}\sum_{i=1}^N\sum_{j\neq k\in [N]\setminus \{i\}} \int_0^t \1_{\xi_{s-}^n(x,j)=\xi_{s-}^n(x,k)}(\eta_{s-}^n(x,j)-\eta_{s-}^n(x,i))\phi_s(x) (d \mathcal Q_s^{x,i,j,k}-\tfrac1N r_n s_n ds) 
\end{align}
is a martingale.
Finally, for the fourth line on the right hand side of~\eqref{eq:qnA},
\begin{align*}
A^4_t&:=\frac{1}{Nn}\sum_{x\in \frac1n \Z}\sum_{i=1}^N\sum_{j\in [N] ,y\in \{x-n^{-1},\, x+n^{-1}\}} \int_0^t (\eta_{s-}^n(y,j)-\eta_{s-}^n(x,i))\phi_s(x) d \mathcal R_s^{x,i,y,j}\\
&=\frac{1}{Nn}\sum_{x\in \frac1n \Z}\sum_{i=1}^N\sum_{j\in [N],y\in \{x-n^{-1},\, x+n^{-1}\}} \int_0^t (\eta_{s-}^n(y,j)-\eta_{s-}^n(x,i))\phi_s(x) (d \mathcal R_s^{x,i,y,j}-mr_n ds)\\
&\quad +\frac{1}{Nn}\sum_{x\in \frac1n \Z}\sum_{i=1}^N\sum_{j\in [N],y\in \{x-n^{-1},\, x+n^{-1}\}} \int_0^t (\eta_{s-}^n(y,j)-\eta_{s-}^n(x,i))\phi_s(x) mr_n ds.
\end{align*}
For $x\in \frac 1n \Z$ and $s\in [0,t]$,
\begin{align*}
\sum_{i,j\in [N],y\in \{x-n^{-1},\, x+n^{-1}\}} (\eta_{s-}^n(y,j)-\eta_{s-}^n(x,i))
= N^2(q^n_{s-}(x-n^{-1})+ q^n_{s-}(x+n^{-1}))-2N^2 q^n_{s-}(x).
\end{align*}
Therefore we can write
\begin{align*}
&\frac{1}{Nn}\sum_{x\in \frac1n \Z}\sum_{i=1}^N\sum_{j\in [N],y\in \{x-n^{-1},\, x+n^{-1}\}} \int_0^t (\eta_{s-}^n(y,j)-\eta_{s-}^n(x,i))\phi_s(x) mr_n ds \\
&=\frac{mr_n}{Nn}\sum_{x\in \frac1n \Z} \int_0^t (N^2(q_{s-}^n(x-n^{-1})+q_{s-}^n(x+n^{-1}))-2N^2 q_{s-}^n(x))\phi_s(x) ds \\
&=\frac{N mr_n}{n}\sum_{x\in \frac1n \Z} \int_0^t q_{s-}^n(x)(\phi_s(x+n^{-1})+\phi_s(x-n^{-1})-2 \phi_s(x)) ds \\
&=\frac{N mr_n}{n^2}\int_0^t \langle q_{s}^n, \Delta_n \phi_s \rangle _n ds ,
\end{align*}
where the second equality follows by summation by parts.
Hence, since $Nr_n n^{-2}=\frac12$,
\begin{equation} \label{eq:A4q}
A^4_t=\tfrac12 m\int_0^t \langle q_{s}^n, \Delta_n \phi_s \rangle _n ds +M^{n,4}_t (\phi),
\end{equation}
where
\begin{equation} \label{eq:M4defq}
M^{n,4}_t (\phi):=\frac{1}{Nn}\sum_{x\in \frac1n \Z}\sum_{i=1}^N\sum_{j\in [N],y\in \{x-n^{-1},\, x+n^{-1}\}} \int_0^t (\eta_{s-}^n(y,j)-\eta_{s-}^n(x,i))\phi_s(x) (d \mathcal R_s^{x,i,y,j}-mr_n ds)
\end{equation}
is a martingale.
Combining~\eqref{eq:A1q},~\eqref{eq:A2q},~\eqref{eq:A3q} and~\eqref{eq:A4q} with~\eqref{eq:qnA}, we have that
\begin{align*}
&\langle q^n_t, \phi_t \rangle _n - \langle q^n_0, \phi_0 \rangle _n
-\int_0^t \langle q^n_s, \partial_s\phi_s \rangle _n ds \notag \\
&= 
s_0 \int_0^t  \langle q^n_{s}(1-p^n_{s})(2p^n_{s}-1+\alpha ),\phi_s \rangle _n  ds 
+\tfrac12 m\int_0^t \langle q_{s}^n, \Delta_n \phi_s \rangle _n ds
+M^n_t(\phi),
\end{align*}
where 
$M^{n}_t (\phi):=\sum_{i=1}^4 M^{n,i}_t (\phi)$ is a martingale with $M^n_0(\phi)=0$.

It remains to bound $\langle M^n(\phi)\rangle_t$.
Since $(\mathcal P^{x,i,j})$, $(\mathcal S^{x,i,j})$, $(\mathcal Q^{x,i,j,k})$ and $(\mathcal R^{x,i,y,j})$ are independent families of Poisson processes,
\begin{equation} \label{eq:Mqvsumq}
\langle M^{n}(\phi)\rangle _t
=\sum_{i=1}^4 \langle M^{n,i}(\phi)\rangle _t .
\end{equation}
By the definition of $M^{n,1}(\phi)$ in~\eqref{eq:A1q}, we have 
\begin{align} \label{eq:Mn1ineqq}
\langle M^{n,1}(\phi)\rangle _t
&=\frac{1}{N^2 n^2}r_n(1-(\alpha +1)s_n)
\sum_{x\in \frac1n \Z}\sum_{i=1}^N \sum_{j\in [N] \setminus \{i\}} \int_0^t (\eta_{s-}^n(x,j)-\eta_{s-}^n(x,i))^2\phi_s(x)^2 ds \notag\\
&=\frac{r_n}{n^2}(1-(\alpha +1)s_n)
\int_0^t \sum_{x\in \frac1n \Z} 2q_{s-}^n(x)(1-q_{s-}^n(x))\phi_s(x)^2 ds \notag\\
&\leq \frac{r_n}{n}(1-(\alpha +1)s_n) \int_0^t \langle 2q^n_s, \phi^2_s \rangle_n ds.
\end{align}
By the same argument, by the definition of $M^{n,2}(\phi)$ in~\eqref{eq:M2defq},
\begin{align*}
\langle M^{n,2}(\phi)\rangle _t
&\leq \frac{r_n}{n}\alpha s_n \int_0^t \langle 2q^n_s, \phi^2_s \rangle_n ds .
\end{align*}
Then by the definition of $M^{n,3}(\phi)$ in~\eqref{eq:M3defq},
\begin{align*}
&\langle M^{n,3}(\phi)\rangle _t \\
&\quad =
\frac{1}{N^2 n^2}\frac{r_n s_n}{N}\sum_{x\in \frac1n \Z}\sum_{i=1}^N\sum_{j\neq k\in [N] \setminus \{i\}} \int_0^t \1_{\xi_{s-}^n(x,j)=\xi_{s-}^n(x,k)}(\eta_{s-}^n(x,j)-\eta_{s-}^n(x,i))^2\phi_s(x)^2 ds \\
&\quad \leq 
\frac{1}{N^2 n^2}\frac{r_n s_n}{N}\sum_{x\in \frac1n \Z}N^3 \int_0^t 2q^n_{s-}(x)(1-q^n_{s-}(x)) \phi_s(x)^2 ds \\
&\quad \le
\frac{r_n}{n} s_n \int_0^t \langle 2q^n_s, \phi^2_s \rangle_n ds .
\end{align*}
Finally, by the definition of $M^{n,4}(\phi)$ in~\eqref{eq:M4defq},
\begin{align*}
\langle M^{n,4}(\phi)\rangle _t 
&\leq 
\frac{1}{N^2 n^2}m r_n\sum_{x\in \frac1n \Z}N^2 \int_0^t 
(q^n_{s-}(x-n^{-1})+2q^n_{s-}(x)+q^n_{s-}(x+n^{-1}))
\phi_s(x)^2 ds \\
&=
\frac{m r_n}{n} \int_0^t \langle q^n_s(\cdot-n^{-1})+2q^n_s(\cdot)+q^n_s(\cdot+n^{-1}), \phi^2_s \rangle_n ds .
\end{align*}
By~\eqref{eq:Mqvsumq}, and since $r_n n^{-1} =\frac 12 n N^{-1}$, the result follows.
\end{proof}
The following result, which is a version of the local central limit theorem in~\cite{lawler/limic:2010}, will be used several times in the rest of the article.
Recall that we let $(X^n_t)_{t\ge 0}$ denote a simple symmetric random walk on $\frac 1n \Z$ with jump rate $n^2$.
\begin{lemma}[Theorem~2.5.6 in~\cite{lawler/limic:2010}] \label{lem:lclt}
For $x\in \frac 1n \Z$ and $t>0$ with $|x|\le \frac 12 nt$,
$$
\psubb{0}{X^n_t=x}=\frac 1n \frac{1}{\sqrt{2\pi t}}e^{-\frac{x^2}{2t}}e^{\mathcal O(n^{-1}t^{-1/2}+n^{-1} |x|^3 t^{-2})}.
$$
\end{lemma}
The next lemma gives us useful bounds on $\langle M^n(\phi^{t,z})\rangle_t$.
\begin{lemma} \label{lem:Mqvarbound}
There exists a constant $C_6<\infty$ such that for $t\ge 0$, $s\in [0,t]$ and $z\in \frac 1n \Z$,
\begin{align} 
\langle 1, (\phi^{t,z}_s)^2 \rangle_n =n \psubb{0}{X^n_{2m(t-s)}=0},
\qquad 
 &\int_0^t \langle 1, (\phi^{t,z}_s)^2\rangle_n ds \leq C_6 t^{1/2} \label{eq:intphitzq}\\
\text{and } \qquad \quad 
\langle M^{n}(\phi^{t,z})\rangle _t 
&\leq 
C_6 t^{1/2} \frac{n }{N}. \label{eq:lemMqvar2}
\end{align}
\end{lemma}
\begin{proof}
For $s\in [0,t]$, by the definition of $\phi^{t,z}_s$ in~\eqref{eq:phitzdefq} and by translational invariance,
\begin{align} \label{eq:phitz2sumq}
\sum_{x\in \frac1n \Z} \phi^{t,z}_s(x)^2
&=n^2 \sum_{x\in \frac1n \Z}
\Pb_0 \left(  X^n_{m(t-s)}=x  \right)^2 \notag\\
&=n^2 \sum_{x\in \frac1n \Z}
\Pb_0 \left( X^n_{m(t-s)}=-x  \right)
\Pb_0 \left(  X^n_{m(t-s)}=x \right) \notag\\
&=n^2 
\Pb_0 \left( X^n_{2m(t-s)}=0  \right),
\end{align}
where the second line follows by symmetry.
(This argument is used in~(54) of \cite{durrett/fan:2016}.)
By Lemma~\ref{lem:lclt}, for $t_0>0$,
\begin{align*}
\int_0^{t_0} n \psubb{0}{X^n_s=0}ds \le \min(n t_0,n^{-1})+\int_{t_0\wedge n^{-2}}^{t_0} (2\pi s)^{-1/2} e^{\mathcal O(1)}ds
\le K_3 t_0^{1/2},
\end{align*}
for some constant $K_3$.
By~\eqref{eq:phitz2sumq}, the first statement~\eqref{eq:intphitzq} follows, and the second statement~\eqref{eq:lemMqvar2} follows by
Lemma~\ref{lem:qnphi} and since $q^n_s \in [0,1]$.
\end{proof}
We will use the following lemma in the proof of Proposition~\ref{prop:pnun}, and also later on in Section~\ref{sec:eventE2}.
\begin{lemma} \label{lem:qnvndet}
For $k\in \N$, $t\ge 0$ and $z\in \frac 1n \Z$,
\begin{align*}
&|q^n_t(z)-v^n_t(z)|^k\\
&\le 3^{2k-1} s_0^k t^{k-1}\left(
\int_0^t \langle |q^n_s - v^n_s|^k , \phi^{t,z}_s \rangle_n ds
+
\int_0^t \sup_{x\in \frac 1n \Z} v^n_s(x)^k \langle |p^n_s - u^n_s|^k, \phi^{t,z}_s \rangle_n ds\right) +3^{k-1}|M^n_t(\phi^{t,z})|^k.
\end{align*}
\end{lemma}
\begin{proof}
By
Corollary~\ref{cor:qnMa} and~\eqref{eq:vngreena}  with $a=0$, for $t\ge 0$ and $z\in \frac 1n \Z$,
\begin{align*}
|q^n_t(z)-v^n_t(z)| &\leq s_0 \int_0^t | \langle (q^n_s-v^n_s)(1-p^n_s)(2p^n_s-1+\alpha), \phi^{t,z}_s \rangle _n | ds\\
&\, +s_0 \int_0^t | \langle v^n_s((1-p^n_s)(2p^n_s-1+\alpha)-(1-u^n_s)(2u^n_s-1+\alpha)), \phi^{t,z}_s \rangle _n | ds
+|M^n_t (\phi^{t,z})|.
\end{align*}
Therefore, since $|(1-u)(2u-1+\alpha )|\le 1+\alpha$ for $u\in [0,1]$, and since
$|(1-x)(2x-1+\alpha )-(1-y)(2y-1+\alpha )|\le 3 |x-y|$ for $x,y \in [0,1]$, for $k\in \N$,
\begin{align} \label{eq:lemqnvnA}
 |q^n_t(z)-v^n_t(z)|^k
&\leq 3^{k-1}s_0^k \left( \int_0^t  \langle (1+\alpha) |q^n_s-v^n_s|, \phi^{t,z}_s \rangle _n ds\right)^k \notag \\
&\quad +3^{k-1} s_0^k \left( \int_0^t  \langle v^n_s\cdot 3 |p^n_s-u^n_s|, \phi^{t,z}_s \rangle _n  ds \right)^k
+3^{k-1} |M^n_t (\phi^{t,z})|^k.
\end{align}
Note that by the definition of $\phi^{t,z}$ in~\eqref{eq:phitzdefq},
for $s\in [0,t]$, $\langle 1,\phi^{t,z}_s\rangle_n =1$.
Hence by two applications of Jensen's inequality,
\begin{align*}
\left(\int_0^t \langle (1+\alpha) |q^n_{s}-v^n_s|, \phi^{t,z}_s \rangle _n ds\right)^k
&\leq  t^{k-1}(1+\alpha) ^k \int_0^t \langle |q^n_{s}-v^n_s|, \phi^{t,z}_s \rangle _n ^k ds \\
&\leq  t^{k-1}(1+\alpha) ^k \int_0^t \langle |q^n_{s}-v^n_s|^k, \phi^{t,z}_s \rangle _n  ds.
\end{align*}
Similarly,
\begin{align*}
\left( \int_0^t  \langle 3 v^n_s |p^n_s-u^n_s|, \phi^{t,z}_s \rangle _n  ds \right)^k
&\le   t^{k-1}3^k  \int_0^t \sup_{x\in \frac 1n \Z}v^n_s(x)^k \langle |p^n_{s}-u^n_s|^k, \phi^{t,z}_s \rangle _n  ds.
\end{align*}
The result follows by~\eqref{eq:lemqnvnA}.
\end{proof}
We will use the following form of the Burkholder-Davis-Gundy inequality (see the proof of Lemma~4 in~\cite{mueller/tribe:1995}) in the proof of Proposition~\ref{prop:pnun} and also later in Section~\ref{sec:eventE2}.
\begin{lemma}[Burkholder-Davis-Gundy inequality] \label{lem:BDG}
For $k\in \N$ with $k\geq 2$ there exists $C(k)<\infty$ such that for $(M_t)_{t\geq 0}$ a c\`adl\`ag martingale with $M_0=0$, for $t\geq 0$,
\begin{equation*} 
\E {\sup_{s\in [0,t]}|M_s|^k }
\leq C(k) \E {\langle M \rangle_t^{k/2} +\sup_{s\in [0,t]}|M_s-M_{s-}|^k }.
\end{equation*}
\end{lemma}
We are now ready to finish this section by proving Proposition~\ref{prop:pnun}.

\begin{proof}[Proof of Proposition~\ref{prop:pnun}]
For $t>0$ and $z\in \frac 1n \Z$, by Lemma~\ref{lem:qnphi} we have that almost surely
\begin{equation*} 
\sup_{s\in [0,t]}|M_s^n(\phi^{t,z})-M_{s-}^n(\phi^{t,z})|
= \sup_{s\in [0,t]} |\langle q^n_s, \phi^{t,z}_s \rangle_n - \langle q^n_{s-}, \phi^{t,z}_s \rangle_n |
\leq N^{-1}.
\end{equation*}
It follows by Lemma~\ref{lem:Mqvarbound} and Lemma~\ref{lem:BDG} that for $k\geq 2$,
\begin{equation*} 
\E {\sup_{s\in [0,t]}|M^n_s(\phi^{t,z})|^k }
\leq C(k) \left(\left(C_6 t^{1/2}\frac{n }{N}\right)^{k/2}+N^{-k} \right).
\end{equation*}
By Lemma~\ref{lem:qnvndet}, and since $\langle 1, \phi^{t,z}_s\rangle_n =1$ and $v^n_s\in [0,1]$ for $s\in [0,t]$,
\begin{align} \label{qnvnEbound}
&\E{|q^n_t(z)-v^n_t(z)|^k} \notag \\
&\leq 3^{2k-1}s_0^k  t^{k-1} \left(\int_0^t \sup_{x\in \frac1n \Z} \E{ |q^n_{s}(x)- v^n_s(x)|^k}  ds 
+\int_0^t \sup_{x\in \frac1n \Z} \E{ |p^n_{s}(x)-u^n_s(x)|^k}  ds \right) \notag \\
&\qquad +3^{k-1}C(k) \left(\left(C_6 t^{1/2} \frac{n }{N}\right)^{k/2}+N^{-k} \right).
\end{align}
Temporarily setting $\eta^n_0=\xi^n_0$ and so $q^n_0=p^n_0$, we have $p^n_s=q^n_s$ and $v^n_s=u^n_s$ $\forall s\ge 0$, and by Gronwall's inequality, for $t\ge 0$,
\begin{align*}
\sup_{x\in \frac1n \Z} \E{ |p^n_{t}(x)-u^n_t(x)|^k}
&\leq 3^{k-1} C(k)\left( \left(C_6 t^{1/2} \frac{n}{N} \right)^{k/2}+N^{-k}\right)
e^{3^{2k-1}2s_0^k t^k}.
\end{align*}
It follows that there exists a constant $C_1=C_1(k)<\infty$ such that for $t\geq 0$,
\begin{align} \label{eq:gronwall1}
\sup_{x\in \frac1n \Z} \E{ |p^n_{t}(x)-u^n_t(x)|^k} 
&\leq C_1 \left(\frac{n^{k/2}t^{k/4}}{N^{k/2}}+N^{-k}\right)
e^{C_1 t^k},
\end{align}
which establishes~\eqref{eq:gronwall1stat}.
Then substituting into~\eqref{qnvnEbound}, 
\begin{align*}
\E{|q^n_t(z)-v^n_t(z)|^k}
&\leq 3^{2k-1}s_0^k t^{k-1} \int_0^t \sup_{x\in \frac1n \Z} \E{ |q^n_{s}(x)- v^n_s(x)|^k}  ds\\
&\quad + 3^{2k-1} s_0^k t^{k-1} \int_0^t C_1\left(\frac{n^{k/2}s^{k/4}}{N^{k/2}}+N^{-k} \right) e^{C_1 s^k}ds\\
&\qquad +3^{k-1}C(k) \left(\left(C_6 t^{1/2}\frac{n }{N}\right)^{k/2}+N^{-k} \right).
\end{align*}
Hence by Gronwall's inequality, there exists a constant $K_4=K_4(k)<\infty$ such that for $t\ge 0$,
\begin{align} \label{eq:gronwall2}
\sup_{x\in \frac1n \Z} \E{ |q^n_{t}(x)-v^n_t(x)|^k}
&\leq K_4 (t^{5k/4}+1)e^{C_1 t^k}\left(\frac n N \right)^{k/2} e^{3^{2k-1}s_0^k  t^k}.
\end{align}
Note that for $x\in \frac1n \Z$, the rate at which $(p_t^n(x))_{t\geq 0}$ jumps is bounded above by 
$$
N^2 r_n(1-(\alpha+1)s_n)+N^2 r_n\alpha s_n +N^3 \cdot \tfrac{1}N r_n s_n+2N^2 m r_n
=N^2 r_n (1+2m)
=\tfrac12 N n^2 (1+2m).
$$
Therefore, for $t\geq 0$ and $x\in \frac 1n \Z$, letting $Z\sim \text{Poisson}(\frac12 (1+2m))$ and then using Markov's inequality,
\begin{align*} 
\p{\sup_{s\in [0,n^{-2} N^{-1}]}|p^n_{t+s}(x)-p^n_t(x)| \geq N^{-1/2}}
&\leq \p{Z \geq N^{1/2}} 
\leq e^{-2N^{1/2}} \E{e^{2Z}}
\leq e^{-N^{1/2}}
\end{align*}
for $n$ sufficiently large.
Suppose $T\le N$. Then
by a union bound,
\begin{align} \label{eq:pntchange}
&\p{\exists t\in n^{-2}N^{-1}\N_0\cap [0,T] , x\in \tfrac1n \Z\cap [- N^5,N^5]: \sup_{s\in [0,n^{-2}N^{-1}]}|p^n_{t+s}(x)-p^n_t(x)| \geq N^{-1/2}} \notag\\
&\leq \sum_{t\in n^{-2}N^{-1}\N_0\cap [0,T]}
\sum_{ x\in \frac1n \Z\cap [- N^5,N^5]}\p{\sup_{s\in [0,n^{-2}N^{-1}]}|p^n_{t+s}(x)-p^n_t(x)| \geq N^{-1/2}}\notag\\
&\leq (n^2 N T +1)(2N^5 n+1)e^{-N^{1/2}}\notag\\
&\leq e^{-N^{1/2}/2}
\end{align}
for $n$ sufficiently large.
For 
$t_1,t_2 \geq 0$ and $x\in \frac1n \Z$, since $\sup_{u\in [0,1]}|f(u)|< 1$,
\begin{align*}
|u^n_{t_1}(x)-u^n_{t_2}(x)|&\leq \tfrac12 m \sup_{s\geq 0, y\in \frac 1n \Z}|\Delta_n u^n_s(y)|  |t_1-t_2|
+s_0 |t_1-t_2|\\
&\leq (mn^2+s_0)|t_1-t_2|. \notag
\end{align*}
Therefore for $n$ sufficiently large, for $t\geq 0$ and $x\in \frac1n \Z$, 
\begin{equation} \label{eq:unchange}
\sup_{s\in [0,n^{-2}N^{-1}]}|u^n_{t+s}(x)-u^n_t(x)|
\leq 2m N^{-1}.
\end{equation}
Then by~\eqref{eq:pntchange},~\eqref{eq:unchange} and a union bound, for $c_3\in (0,1/2)$, for $n$ sufficiently large that $2mN^{-1}+N^{-1/2}\le \frac 12 \left( \frac n N \right)^{1/2-c_3}$,
\begin{align*}
&\p{\sup_{x\in \frac1n \Z,\, |x|\leq N^5}\sup_{t\in [0,T]}|p^n_{t}(x)-u^n_t(x)| \geq \left(\frac{n}{N}\right)^{1/2-c_3}}
\\
&\leq 
\sum_{t\in n^{-2}N^{-1}\N_0\cap [0,T]}
\sum_{ x\in \frac1n \Z,\, |x|\leq N^5}\p{|p^n_{t}(x)-u^n_t(x)| \geq \tfrac12\left(\frac{n}{N}\right)^{1/2-c_3}}+e^{-N^{1/2}/2}.
\end{align*}
Hence for $k\in \N$ with $k\geq 2$, by Markov's inequality,
\begin{align*}
&\p{\sup_{x\in \frac1n \Z,\, |x|\leq N^5}\sup_{t\in [0,T]}|p^n_{t}(x)-u^n_t(x)| \geq \left(\frac{n}{N}\right)^{1/2-c_3}}
\\
&\leq 
\sum_{t\in n^{-2}N^{-1}\N_0\cap [0,T]}
\sum_{ x\in \frac1n \Z,\, |x|\leq N^5}\E{|p^n_{t}(x)-u^n_t(x)|^k }2^k\left(\frac{n}{N}\right)^{-k(1/2-c_3)}+e^{-N^{1/2}/2}\\
&\leq 
\sum_{t\in n^{-2}N^{-1}\N_0\cap [0,T]}
\sum_{ x\in \frac1n \Z,\, |x|\leq N^5}C_1 \left(\frac{n^{k/2}t^{k/4}}{N^{k/2}}+N^{-k}\right)
e^{C_1  t^k}2^k\left(\frac{n}{N}\right)^{-k(1/2-c_3)}+e^{-N^{1/2}/2}\\
&\leq 
(n^2 N T+1)
(2n  N^5+1)C_1 \left(\frac{n^{k/2}T^{k/4}}{N^{k/2}}+N^{-k}\right)
e^{C_1  T^k}2^k\left(\frac{n}{N}\right)^{-k(1/2-c_3)}+e^{-N^{1/2}/2},
\end{align*}
where the second inequality follows by~\eqref{eq:gronwall1}.

Take $\ell'\in \N$ sufficiently large that 
$n^4 N^7 e^{2^k(C_1+3^{2k-1}s_0^k)(\log N)^{1/2}}\left( \frac n N \right)^{\ell'}\ \le 1$ for $n$ sufficiently large.
For $\ell \in \N$,
take $c_4\in (0,\frac 12 c_3(\ell+\ell'+1)^{-1})$. Since $1/(2c_4)>(\ell+\ell'+1)/c_3$ and $c_3<1/2$ we can take $k\in \N \cap ((\ell+\ell')/c_3,1/(2c_4))$ with $k\ge 2$.
Therefore for $T\leq 2(\log N)^{c_4}$, for $n$ sufficiently large,
\begin{align*}
&\p{\sup_{x\in \frac1n \Z,\, |x|\leq N^5}\sup_{t\in [0,T]}|p^n_{t}(x)-u^n_t(x)| \geq \left(\frac{n}{N}\right)^{1/2-c_3}}\\
&\quad \leq 
n^4 N^7\left( \frac{n}{N}\right)^{k/2} e^{C_1 2^k(\log N)^{c_4 k}}\left(\frac n N \right)^{-k(1/2-c_3)}+e^{-N^{1/2}/2}\\
&\quad \leq 
\left( \frac{n}{N}\right)^\ell
\end{align*}
for $n$ sufficiently large, since $kc_3>\ell+\ell'$ and $c_4 k<1/2$.
Similarly, by a union bound and Markov's inequality, and then by~\eqref{eq:gronwall2},
for $t \le 2 (\log N)^{c_4}$,
\begin{align*}
\p{\sup_{x\in \frac1n \Z,\, |x|\leq N^5}|q^n_{t}(x)-v^n_t(x)| \geq \left(\frac{n}{N}\right)^{1/2-c_3}}
&\le \sum_{x\in \frac 1n \Z, |x|\le N^5} \E{|q^n_t(x)-v^n_t(x)|^k} \left( \frac n N \right)^{-k(1/2-c_3)}
\\
&\leq (2 n N^5 +1)
K_4 (t^{5k/4}+1)e^{C_1 t^k} e^{3^{2k-1}s_0^k t^k}
\left(\frac n N \right)^{kc_3}\\
&\leq 
\left( \frac{n}{N}\right)^\ell
\end{align*}
for $n$ sufficiently large, which completes the proof.
\end{proof}

\section{Event $E_2$ occurs with high probability} \label{sec:eventE2}
Recall the definitions of the events $E_2$ and $E'_2$ in~\eqref{eq:eventE2} and~\eqref{eq:eventE'2}.
In this section, we will prove the following result.
\begin{prop} \label{prop:eventE2}
For $c_1,c_2>0$, for $t^*\in \N$ sufficiently large and $K\in \N$ sufficiently large (depending on $t^*$), the following holds.
If $a_1>1$ and $N\ge n^{a_1}$ for $n$ sufficiently large, then for $n$ sufficiently large,
$$
\p{(E'_2)^c\cap E'_1}\le \left( \frac n N \right)^2.
$$
Moreover, if $a_2>3$ and $N\ge n^{a_2}$ for $n$ sufficiently large, then for $n$ sufficiently large,
$$
\p{(E_2)^c\cap E'_1}\le \left( \frac n N \right)^2.
$$
\end{prop}
Suppose from now on in this section that for some $a_1>1$, $N\ge n^{a_1}$ for $n$ sufficiently large,
and fix $c_1, c_2>0$.
We begin by proving that for $t$, $x_1$ and $x_2$ such that $x_1$ and $x_2$ are not too far from the front, the event $A^{(1)}_{t}(x_1,x_2)$ occurs with high probability.
Recall the definition of $(v^n_t)_{t\ge 0}$ in~\eqref{eq:vndef}.
We begin by showing
that the solution of a PDE closely related to~\eqref{eq:vndef} can be written in terms of a diffusion $(Z_t)_{t\ge 0}$.
\begin{lemma} \label{lem:vtSDE}
Suppose $h:\R \to [0,1]$ is measurable, and take $t_0>0$.
For $x\in \R$ and $t\ge t_0$, let
$$
v_t(x)=g(x-\nu  t)\Esub{x-\nu  t}{\frac{h(Z_{t-t_0}+\nu  t_0)}{g(Z_{t-t_0})}},
$$
where under $\mathbb P_{x_0}$, $(Z_t)_{t\ge 0}$ solves the SDE
\begin{equation} \label{eq:SDE}
dZ_t =\nu  \,dt+\frac{m\nabla g(Z_t)}{g(Z_t)}\,dt+\sqrt m \,dB_t, \quad Z_0 =x_0,
\end{equation}
and $(B_t)_{t\ge 0}$ is a Brownian motion.
Then $v_{t_0}=h$ and 
$$
\partial_t v_t(x)=\tfrac 12 m \Delta v_t(x)+s_0 v_t(x)(1-g(x-\nu  t))(2g(x-\nu  t)-1+\alpha ) \quad \text{for }t>t_0, \, x\in \R.
$$
\end{lemma}

\begin{proof}
For $t\ge t_0$ and $x\in \R$, let
$$
v^{(1)}_t(x)=\Esub{x-\nu t}{\frac{h(Z_{t-t_0}+\nu t_0)}{g(Z_{t-t_0})}}
=v_t(x)g(x-\nu t)^{-1}.
$$
Since $\mathcal Af(x):= \frac 12 m\Delta f(x)+\left(\nu +\frac{m \nabla g(x)}{g(x)}\right) \nabla f(x)$ is the generator of the diffusion $(Z_t)_{t\ge 0}$ as defined in~\eqref{eq:SDE},
 for $t>t_0$ and $x\in \R$,
\begin{align*}
\partial_t v^{(1)}_t(x)&=
\tfrac 12 m \Delta v^{(1)}_t (x) +\left( \nu+\frac{m\nabla g(x-\nu t)}{g(x-\nu t)}\right)\nabla v^{(1)}_t (x)-\nu \nabla v^{(1)}_t (x)
\end{align*}
(see for example Theorem~7.1.5 in~\cite{durrett:1996}).
Therefore
\begin{align*}
\partial_t v_t(x)&=-\nu \nabla g(x-\nu t)v^{(1)}_t (x)+ \tfrac 12 m g(x-\nu t)\Delta v^{(1)}_t (x)+m \nabla g(x-\nu t)\nabla v^{(1)}_t (x)\\
&=\tfrac 12 m \Delta v_t(x)-\tfrac 12 m\frac{\Delta g(x-\nu t)}{g(x-\nu t)}v_t(x)-\nu \frac{\nabla g(x-\nu t)}{g(x-\nu t)}v_t(x).
\end{align*}
Since $\Delta g=-\kappa^2 g(1-g)(2g-1)$ and $\nabla g=-\kappa g(1-g)$, the result follows by~\eqref{eq:kappanu}.
\end{proof}
We now show that for $(u^n_t)_{t\ge 0}$ and $(v^n_t)_{t\ge 0}$ defined as in~\eqref{eq:undef} and~\eqref{eq:vndef}, if $\sup_{s\in [0,t],\,x\in \frac 1n \Z}|u^n_s(x)-g(x-\nu  s)|$ is small then $v^n_t$ is approximately given by an expectation of a function of $Z_t$.
The proof is similar to the proof of Lemma~\ref{lem:unu}.
\begin{lemma} \label{lem:vnvbound}
Take $\delta,\epsilon \in (0,1)$. For $t\ge 0$ and $x\in \R$, let
$$
v_t(x)=g(x-\nu  t) \Esub{x-\nu  t}{\bar{q}^n_0 (Z_t) g(Z_t)^{-1}},
$$
where $\bar{q}^n_0:\R\to [0,1]$ is the linear interpolation of $q^n_0: \frac 1n \Z \to [0,1]$,
and $(Z_t)_{t\ge 0}$ is defined in~\eqref{eq:SDE}.
Suppose $T\ge 1$, $\sup_{x\in \frac 1n \Z, s\in [0,T]}|u^n_s(x)-g(x-\nu  s)|\le \delta$
and $\sup_{x_1,x_2\in\frac 1n \Z,|x_1-x_2|\leq n^{-1/3}}|q^n_0(x_1)-q^n_0(x_2)|\le \epsilon$.
There exists a constant $C_7<\infty$ such that for $n$ sufficiently large, for $t\in [0,T]$,
\begin{align*}
\sup_{x\in \frac1n \Z}|v_t^n(x)-v_t(x)|
&\leq  \left(C_7 (n^{-1/3}+\delta)\sup_{x\in \frac 1n \Z}q^n_0(x) +2\epsilon \right)
e^{5s_0 T} T^2.
\end{align*}
\end{lemma}
\begin{proof}
For $t>0$ and $x\in \R$,
let $G_t(x)=\frac{1}{\sqrt{2\pi t}}e^{-x^2/(2t)}.$ 
For $s\ge 0$ and $x\in \R$, let $f_s(x)=v_s(x)(1-g(x-\nu  s))(2g(x-\nu  s)-1+\alpha )$.
By Lemma~\ref{lem:vtSDE}, for $a\in \R$, $z\in \R$ and $t> 0$,
\begin{equation} \label{eq:vtgreen*}
v_t(z)=e^{-at} G_{mt}\ast v_0(z)+\int_0^t e^{-a(t-s)} G_{m(t-s)}\ast (s_0 f_s+av_s)(z)ds.
\end{equation}
Therefore, by~\eqref{eq:vtgreen*} with $a=-(1+\alpha)s_0 $, and since $(1-u)(2u-1+\alpha )\le 1+\alpha$ for $u\in [0,1]$, 
\begin{align} \label{eq:lemvnv*}
v_t(z)
&\le e^{(1+\alpha)s_0 t} G_{mt}\ast v_0(z).
\end{align}
Letting $(B_t)_{t\geq 0}$ denote a Brownian motion, it follows from~\eqref{eq:vngreena} and~\eqref{eq:vtgreen*} with $a=0$ that for $z\in \frac1n \Z$ and $t\ge 0$,
\begin{align} \label{eq:unminusu}
|v_t^n(z)-v_t(z)| 
& \leq \left|\Esubb{z}{q^n_0(X^n_{mt})}-\Esub{z}{v_0(B_{mt})} \right| \notag\\
&\quad  +s_0 \int_0^t \Big|\Esubb{z}{v^n_s(1-u^n_s)(2u^n_s-1+\alpha )(X^n_{m(t-s)})}
-\Esub{z}{f_s(B_{m(t-s)})} \Big| ds.
\end{align}
Recall from~\eqref{eq:couplingBM0} in the proof of Lemma~\ref{lem:unu} that for $n$ sufficiently large, $(X^n_t)_{t\ge 0}$ and $(B_t)_{t\ge 0}$ can be coupled in such a way that $X^n_0=B_0$ and for $t\ge 0$,
\begin{align} \label{eq:couplingBM}
\p{|X^n_{mt}-B_{mt}|\ge  n^{-1/3}}\leq  (t+1)n^{-1/2}.
\end{align}
Since $v_0=\bar{q}^n_0$, which is the linear interpolation of $q^n_0$, it follows that for $z\in \frac 1n \Z$ and $t\ge 0$,
\begin{align} \label{eq:(*)unuproof}
 \left|\Esubb{z}{q^n_0(X^n_{mt})}-\Esub{z}{v_0(B_{mt})} \right|
 &\leq  (t+1)n^{-1/2} \sup_{x\in \frac 1n \Z}q^n_0(x) +\sup_{x_1,x_2\in\R,|x_1-x_2|\leq  n^{-1/3}}|\bar{q}^n_0(x_1)-\bar{q}^n_0(x_2)| \notag \\
 &\leq   (t+1)n^{-1/2} \sup_{x\in \frac 1n \Z}q^n_0(x) +2\epsilon
\end{align}
for $n$ sufficiently large.
For the second term on the right hand side of~\eqref{eq:unminusu}, 
note that if $t\le T$ then for $s\in [0,t]$ and $y\in \frac 1n \Z$,
$$
|(1-u^n_s(y))(2u^n_s(y)-1+\alpha )-(1-g(y-\nu  s))(2g(y-\nu  s)-1+\alpha )|\le 3 \delta.
$$
Hence
by the triangle inequality and then by~\eqref{eq:couplingBM}, for $s\in [0,t]$,
\begin{align} \label{eq:intvdiffbound}
&\left|\Esubb{z}{v^n_s(1-u^n_s)(2u^n_s-1+\alpha )(X^n_{m(t-s)})}-\Esub{z}{f_s(B_{m(t-s)})} \right| \notag \\
&\quad \le \Esubb{z}{(|(v^n_s-v_s)(1-u^n_s)(2u^n_s-1+\alpha )| +3\delta v_s )(X^n_{m(t-s)})} \notag
\\
&\qquad +\left|\Esubb{z}{f_s(X^n_{m(t-s)})}-\Esub{z}{f_s(B_{m(t-s)})} \right| \notag \\
&\quad \le 3\left( \sup_{x\in \frac 1n \Z}|v^n_s(x)-v_s(x)|+\delta \sup_{x\in \R} v_s(x) \right)
+2 (t+1)n^{-1/2}\sup_{x\in \R}|f_s(x)| + n^{-1/3} \sup_{x\in \R} |\nabla f_s(x)| \notag \\
&\quad \leq
3 \Big(
\sup_{x\in \frac 1n \Z}|v^n_s(x)-v_s(x)|+(\delta  
+ 2(t+1)n^{-1/2}) e^{(1+\alpha)s_0 s}\|v_0\|_\infty \notag \\
&\hspace{4cm} + n^{-1/3}(\|\nabla v_s\|_{\infty}+e^{(1+\alpha)s_0 s}\|v_0\|_\infty \|\nabla g\|_{\infty}) \Big)
\end{align}
by~\eqref{eq:lemvnv*}.
It remains to bound $\|\nabla v_s\|_\infty$.
For $t> 0$ and $x\in \R$, by differentiating both sides of~\eqref{eq:vtgreen*},
\begin{align} \label{eq:dut}
\nabla v_t(x)
&=G'_{mt}\ast v_0(x)+s_0 \int_0^t G'_{m(t-s)}\ast f_s(x)ds .
\end{align}
For the first term on the right hand side,
\begin{align*}
|G'_{mt}\ast v_0(x)|
&\le \|v_0\|_\infty \int_{-\infty}^\infty |G'_{mt}(z)|dz
= 2\|v_0\|_\infty G_{mt}(0)
=2 \|v_0\|_\infty(2\pi m t)^{-1/2}.
\end{align*}
For the second term on the right hand side of~\eqref{eq:dut},
since $|f_s(x)|\le (1+\alpha)e^{(1+\alpha)s_0 s}\|v_0\|_\infty$ by~\eqref{eq:lemvnv*},
\begin{align*}
\left|\int_0^t G'_{m(t-s)}\ast f_s(x)ds\right| 
&\leq (1+\alpha) e^{(1+\alpha)s_0 t}\|v_0\|_\infty  \int_0^t 2G_{m(t-s)}(0) ds,
\end{align*}
and so by~\eqref{eq:dut}, for $t> 0$,
$$
\|\nabla v_t \|_\infty
\leq (2t^{-1/2}+4s_0 (1+\alpha) e^{(1+\alpha)s_0 t}t^{1/2})(2\pi m)^{-1/2}\|v_0\|_\infty.
$$
Substituting into~\eqref{eq:intvdiffbound} and then into~\eqref{eq:unminusu}, using~\eqref{eq:(*)unuproof}, we now have that  for $t\in [0,T]$ and $z\in \R$,
\begin{align*}
|v_t^n(z)-v_t(z)| 
& \leq (t+1)n^{-1/2}\sup_{x\in \frac 1n \Z}q^n_0(x) +2\epsilon \\
& +3s_0\int_0^t \bigg( \sup_{x\in \frac1n \Z}|v^n_s(x)-v_s(x)|
+e^{(1+\alpha)s_0 t} \|v_0\|_\infty (\delta +2(t+1)n^{-1/2}+ n^{-1/3}\|\nabla g\|_\infty)\\
&\hspace{2cm}+(t^{-1/2}
+2s_0 (1+\alpha) e^{(1+\alpha)s_0 t}  t^{1/2}  )m^{-1/2} \|v_0\|_\infty n^{-1/3}\bigg) ds.
\end{align*}
The result follows by Gronwall's inequality.
\end{proof}
By the theory of speed and scale (see for example~\cite{karlin/taylor:1981}),
$(Z_t)_{t\ge 0}$ as defined in~\eqref{eq:SDE} has
scale function $S$ and speed measure density $M$ given by
\begin{equation} \label{eq:smdefn}
S(x)=\int_0^x \tfrac 14 e^{-\alpha \kappa y}g(y)^{-2}dy
\quad \text{and}\quad
M(x)=\frac 4m e^{\alpha \kappa x}g(x)^2.
\end{equation}
Therefore $(Z_t)_{t\ge 0}$ has
 a stationary distribution with density $\pi$
 as defined in~\eqref{eq:pidefn}.
 We now establish some useful upper bounds on the total variation distance between $\pi$ and the law of $Z_t$ at a large time $t$.
 Recall the definitions of $\gamma_n$ and $d_n$ in~\eqref{eq:paramdefns}.
\begin{lemma} \label{lem:Tbound}
Take $z_0\in \R$ and
suppose $(Z^{(1)}_t)_{t\ge 0}$ and $(Z^{(2)}_t)_{t\ge 0}$ solve the SDEs
\begin{align*}
dZ^{(1)}_t &= \nu  dt +\frac{m\nabla g(Z^{(1)}_t)}{g(Z^{(1)}_t)}dt +\sqrt m dB^{(1)}_t, \quad Z^{(1)}_0=z_0\\
\text{and }\quad dZ^{(2)}_t &= \nu  dt +\frac{m\nabla g(Z^{(2)}_t)}{g(Z^{(2)}_t)}dt +\sqrt m dB^{(2)}_t, \quad Z^{(2)}_0=Z,
\end{align*}
where $(B^{(1)}_t)_{t\ge 0}$ and $(B^{(2)}_t)_{t\ge 0}$ are independent Brownian motions and $Z$ is an independent random variable with density $\pi$.
Let
$$
T^Z=\inf\{t\ge 0: Z^{(1)}_t=Z^{(2)}_t\}.
$$
Then for $n$ sufficiently large, if $|z_0|\le d_n+1$,
\begin{equation} \label{eq:TZbound1}
\p{T^Z \ge \tfrac 12 \gamma_n }\le (\log N)^{-12C}.
\end{equation}
For $A<\infty$, for $t\ge 0$ sufficiently large, if $|z_0|\le A$,
\begin{equation} \label{eq:TZbound2}
\p{T^Z \ge t }\le 2m^{-1/2} t^{-1/4}.
\end{equation}
\end{lemma}
\begin{remark}
The first bound~\eqref{eq:TZbound1} will be used in the proof of Proposition~\ref{prop:eventE2}, and the weaker bound in~\eqref{eq:TZbound2} will be used in Section~\ref{sec:thmstatdist} in the proof of Theorem~\ref{thm:statdist}.
\end{remark}
\begin{proof}
Suppose first that $|z_0|\le d_n+1$.
Since $g(x)\le \min(e^{-\kappa x},1)$ $\forall x\in \R$, for $y_0>0$ we have
\begin{equation} \label{eq:pitail}
\begin{aligned}
\int_{y_0}^\infty g(y)^2 e^{\alpha \kappa y}dy &\le (2-\alpha)^{-1} \kappa^{-1} e^{-(2-\alpha)\kappa y_0}\\
 \text{ and } \quad 
\int_{-\infty}^{-y_0}g(y)^2 e^{\alpha \kappa y} dy &\le \alpha^{-1} \kappa^{-1} e^{-\alpha \kappa y_0}.
\end{aligned}
\end{equation}
It follows that
\begin{equation} \label{eq:lemTbound1}
\p{|Z^{(2)}_0|\ge 13\alpha^{-1} d_n}\le 2\alpha^{-1}\kappa^{-1} \left(\int_{-\infty}^\infty g(y)^2 e^{\alpha \kappa y}dy\right)^{-1} (\log N)^{-13C}.
\end{equation}
Take $(Z_t)_{t\ge 0}$ as defined in~\eqref{eq:SDE}, and for
$a\in \R$, let 
$$
\tau^a = \inf\{t\ge 0 : Z_t=a\}.
$$
By~\eqref{eq:smdefn} and the theory of speed and scale (see for example~\cite{karlin/taylor:1981}), and then since $g(y)\in [\frac 12 e^{-\kappa y},e^{-\kappa y}]$ $\forall y\ge 0$, for $x>0$,
\begin{align*}
\psub{x/2}{\tau^x <\tau^0}
=\frac{S(0)-S(x/2)}{S(0)-S(x)}
\le \frac{\int_0^{x/2}4e^{-\alpha\kappa y}e^{2\kappa y} dy}{\int_0^{x}e^{-\alpha \kappa y}e^{2\kappa y} dy}
&=4\frac{e^{(2-\alpha)\kappa x/2}-1}{e^{(2-\alpha)\kappa x}-1}\\
&\le 8 e^{-(2-\alpha)\kappa x/2}
\end{align*}
for $x\ge \kappa^{-1} \log 2$.
Similarly, since $g(y)\in [1/2,1]$ $\forall y\le 0$,
\begin{align*}
\psub{-x/2}{\tau^{-x} <\tau^0}
=\frac{S(0)-S(-x/2)}{S(0)-S(-x)}
\le \frac{\int_{-x/2}^0 4e^{-\alpha \kappa y} dy}{\int_{-x}^0 e^{-\alpha \kappa y} dy}
=4\frac{e^{\alpha \kappa x/2}-1}{e^{\alpha \kappa x}-1}
\le 8 e^{-\alpha \kappa x/2}
\end{align*}
for $x\ge \alpha^{-1} \kappa^{-1} \log 2$.
Hence for $n$ sufficiently large, 
\begin{equation} \label{eq:lemTbound2}
\max\left(
\psub{13\alpha^{-1} d_n}{\tau^{26\alpha^{-1} d_n} <\tau^0},\psub{-13\alpha^{-1} d_n}{\tau^{-26\alpha^{-1} d_n} <\tau^0}
\right)
\le 8(\log N)^{-13C}.
\end{equation}
Let $(B_t)_{t\ge 0}$ denote a Brownian motion.
Note that $\frac{\nabla g(y)}{g(y)}\in [-\kappa ,0]$ $\forall y\in \R$, and so $|\nu +\frac{m \nabla g(y)}{g(y)}|<\sqrt{2s_0 m}$. Hence
 for $x\in \R$ with $|x|\ge 13\alpha^{-1} d_n$,
\begin{equation} \label{eq:lemTbound3}
\psub{x}{\tau^0<1}\le \p{\sup_{t\in [0,1]}\sqrt m B_t \ge 13\alpha^{-1} d_n-\sqrt{2ms_0}}
\le 2e^{-\frac 1{2m}(13\alpha^{-1} d_n-\sqrt{2m s_0})^2}
\end{equation}
by the reflection principle and a Gaussian tail bound.
Therefore by a union bound,
\begin{align} \label{eq:X12notfar}
&\p{\exists j\in \{1,2\},t\in [0,\gamma_n]: |Z^{(j)}_t|\ge 26\alpha^{-1} d_n} \notag \\
&\le \p{|Z_0^{(2)}| \ge 13 \alpha^{-1} d_n} 
+2\lceil \gamma_n \rceil \max\left(
\psub{13 \alpha^{-1}d_n}{\tau^{26 \alpha^{-1}d_n} <\tau^0},\psub{-13\alpha^{-1}d_n}{\tau^{-26\alpha^{-1}d_n} <\tau^0}
\right) \notag \\
&\quad +2\lceil \gamma_n \rceil \max\left(
\psub{13 \alpha^{-1}d_n }{\tau^0<1},\psub{-13 \alpha^{-1}d_n}{\tau^0<1}
\right) \notag \\
&\le \tfrac 12 (\log N)^{-12C}
\end{align}
for $n$ sufficiently large, by~\eqref{eq:lemTbound1},~\eqref{eq:lemTbound2} and~\eqref{eq:lemTbound3}.

For $t\ge 0$, define the sigma-algebra $\mathcal F^Z_t=\sigma((Z^{(1)}_s)_{s\le t},(Z^{(2)}_s)_{s\le t})$.
Note that if $Z^{(1)}_t \le Z^{(2)}_t$ then for $s\in [t,T^Z \vee t]$,
\begin{align} \label{eq:X2sX1s}
&Z^{(2)}_s -Z^{(1)}_s \notag \\
&= (Z^{(2)}_t-Z^{(1)}_t)+m \int_t^s \left( \frac{\nabla g(Z^{(2)}_u)}{g(Z^{(2)}_u)}-\frac{\nabla g(Z^{(1)}_u)}{g(Z^{(1)}_u)}\right) du +\sqrt m ((B^{(2)}_s -B^{(2)}_t)-(B^{(1)}_s-B^{(1)}_t)) \notag \\
&\le (Z^{(2)}_t-Z^{(1)}_t)+\sqrt m ((B^{(2)}_s -B^{(2)}_t)-(B^{(1)}_s-B^{(1)}_t)),
\end{align}
since $y\mapsto \frac{\nabla g(y)}{g(y)}$ is decreasing.
Therefore, for $n$ sufficiently large,
for $t\ge 0$, if $|Z^{(1)}_t|\vee |Z^{(2)}_t|\le 26\alpha^{-1} d_n$ then
\begin{align} \label{eq:tau21B}
\p{T^Z> t + \gamma_n^{1/2} \Big| \mathcal F^Z_t}
&\le \psub{52\alpha^{-1} d_n}{\sqrt{2m} B_s \ge 0 \; \forall s\in [0,\gamma_n^{1/2}]} \notag \\
&\le \psub{52\alpha^{-1}\kappa^{-1} C+1}{\sqrt{2m} B_s \ge 0 \; \forall s\in [0,1]}:=p>0
\end{align}
by Brownian scaling and since $d_n =\kappa^{-1} C \log \log N$ and $\gamma_n=\lfloor (\log \log N)^4 \rfloor$.
Therefore by~\eqref{eq:X12notfar} and a union bound, for $n$ sufficiently large,
\begin{align*}
&\p{T^Z \ge \tfrac 12 \gamma_n }\\
&\le \tfrac 12 (\log N)^{-12C}
+\p{T^Z \ge \tfrac 12 \gamma_n  , |Z^{(1)}_{k\gamma_n^{1/2}}|\vee |Z^{(2)}_{k\gamma_n^{1/2}}|\le 26\alpha^{-1}d_n \; \forall k\in \N_0 \cap [0,\tfrac 12 \gamma_n^{1/2}]}\\
&\le \tfrac 12 (\log N)^{-12C}+p^{\lfloor \gamma_n^{1/2}/2\rfloor }
\end{align*}
by~\eqref{eq:tau21B}, which completes the proof of~\eqref{eq:TZbound1}.

Now take $A<\infty$ and suppose $|z_0|\le A$.
Then for $t\ge A^4$, by a union bound and~\eqref{eq:X2sX1s},
\begin{align*}
\p{T^Z \ge t}
&\le \p{|Z^{(2)}_0|\ge t^{1/4}}+\psub{2t^{1/4}}{\sqrt{2m} B_s \ge 0 \; \forall s\in [0,t]}\\
&\le 2\alpha^{-1}\kappa^{-1} \left( \int_{-\infty}^\infty g(y)^2 e^{\alpha \kappa y} dy\right)^{-1} e^{-\alpha \kappa t^{1/4}}+\psub{0}{|B_{2mt}|\le 2t^{1/4}}
\end{align*}
by~\eqref{eq:pitail} and the reflection principle.
Since $\psub{0}{|B_{2m t}|\le 2t^{1/4}}\le \frac{4t^{1/4}}{(4\pi m t)^{1/2}}$, the result follows by taking $t$ sufficiently large.
\end{proof}

Fix $x_0 \in \frac 1n \Z$, and take
$(v^n_t)_{t\geq 0}$ as in~\eqref{eq:vndef} with $v_0^n(x)=p^n_0(x_0)\1_{x=x_0}$, and where 
 $(u^n_t)_{t\ge 0}$ is defined in~\eqref{eq:undef}.
 The following result will be combined with a bound on $|q^n_{\gamma_n}-v^n_{\gamma_n }|$ to show that the event $A^{(1)}_t(x_1,x_2)$ occurs with high probability for suitable $t$, $x_1$ and $x_2$.
Recall that we fixed $c_2>0$.

\begin{lemma} \label{lem:vtstat}
Suppose $\sup_{x\in \frac 1n \Z, s\in [0,\gamma_n]}|u^n_s(x)-g(x-\nu  s)|\le e^{-(\log N)^{c_2}}$.
For $n$ sufficiently large, if $|x_0|\le d_n$
and $|x-\nu  \gamma_n|\le d_n+1$,
$$
\frac{v^n_{\gamma_n}(x)}{g(x-\nu \gamma_n )}=\frac{\pi(x_0)}{g(x_0)} p^n_0(x_0)n^{-1}(1+\mathcal O((\log N)^{-4C})).
$$
\end{lemma}
\begin{proof}
Let
$t_0=(\log N)^{-12C}$.
For $x\in \frac 1n \Z$, let
$P^n_{t_0,x_0}(x)=\psubb{x}{X^n_{mt_0}=x_0}$, and let $\bar{P}^n_{t_0,x_0}:\R\to [0,1]$ denote the linear interpolation of 
$P^n_{t_0,x_0}$.
Let $\bar{v}^n_{t_0}$ denote the linear interpolation of $v^n_{t_0}$.
For $t\ge t_0$ and $x\in \R$, let
\begin{equation} \label{eq:vteq}
v_t(x)=g(x-\nu t)\Esub{x-\nu t}{\frac{\bar{v}^n_{t_0}(Z_{t-t_0}+\nu t_0)}{g(Z_{t-t_0})}},
\end{equation}
where $(Z_t)_{t\ge 0}$ is defined in~\eqref{eq:SDE}.
By~\eqref{eq:veasybound},
for $t\ge 0$ and $y\in \frac 1n \Z$,
\begin{equation} \label{eq:lemvtstatvnb}
v^n_t(y)\le e^{(1+\alpha)s_0 t} p^n_0(x_0) \psubb{y}{X^n_{mt}=x_0},
\end{equation}
and so for $t\ge t_0$ and $x\in \R$,
\begin{align}
v_t(x)
&\le  g(x-\nu t)p^n_0(x_0)e^{(1+\alpha)s_0 t_0}
\Big(\Esub{x-\nu t}{g(Z_{t-t_0})^{-1}\bar{P}^n_{t_0,x_0}(Z_{t-t_0}+\nu t_0)\1_{|Z_{t-t_0}+\nu t_0-x_0|<n^{1/4}}} \notag \\
&\qquad \qquad \qquad \qquad +\Esub{x-\nu t}{g(Z_{t-t_0})^{-1}\bar{P}^n_{t_0,x_0}(Z_{t-t_0}+\nu t_0)\1_{|Z_{t-t_0}+\nu t_0-x_0|\ge n^{1/4}}} \Big). \label{eq:vtupper}
\end{align}
For the first term on the right hand side,
we have that if
$n$ is sufficiently large that
$n^{1/4}\le \frac 12 m nt_0$,
 then by Lemma~\ref{lem:lclt},
\begin{align*}
&\Esub{x-\nu t}{g(Z_{t-t_0})^{-1}\bar{P}^n_{t_0,x_0}(Z_{t-t_0}+\nu t_0)\1_{|Z_{t-t_0}+\nu t_0-x_0|< n^{1/4}}}\\
&\le n^{-1}(2\pi m t_0)^{-1/2} e^{\mathcal O(n^{-1/5})} 
\Esub{x-\nu t}{g(Z_{t-t_0})^{-1} e^{-(Z_{t-t_0}+\nu t_0-x_0)^2/(2m t_0)}}.
\end{align*}
For the second term on the right hand side of~\eqref{eq:vtupper},
by the definition of $\bar{P}^n_{t_0,x_0}$ and then by Markov's inequality, for $n$ sufficiently large,
\begin{align*}
&\Esub{x-\nu t}{g(Z_{t-t_0})^{-1}\bar{P}^n_{t_0,x_0}(Z_{t-t_0}+\nu t_0)\1_{|Z_{t-t_0}+\nu t_0-x_0|\ge n^{1/4}}}\\
&\le 
\Esub{x-\nu t}{(1+e^{\kappa Z_{t-t_0}}) \psubb{0}{X^n_{m t_0}\ge |Z_{t-t_0}+\nu t_0-x_0|-n^{-1}}
\1_{|Z_{t-t_0}+\nu t_0-x_0|\ge n^{1/4}}}\\
&\le  \Esub{x-\nu t}{(1+e^{\kappa Z_{t-t_0}}) 
e^{-3\kappa |Z_{t-t_0}+\nu t_0-x_0|}e^{3\kappa n^{-1}}\Esubb{0}{e^{3\kappa X^n_{mt_0}}}
\1_{|Z_{t-t_0}+\nu t_0-x_0|\ge n^{1/4}}}\\
&\le e^{10s_0 t_0}(e^{-3\kappa n^{1/4}}+e^{\kappa |x_0|}e^{-2\kappa n^{1/4}})
\end{align*}
by Lemma~\ref{lem:Xnmgf} and since $\frac 12 m \kappa^2 =s_0$ and $e^{\kappa Z_{t-t_0}} e^{-3 \kappa |Z_{t-t_0}+\nu t_0-x_0|}\le e^{(-\nu t_0 +x_0)\kappa} e^{-2\kappa |Z_{t-t_0}+\nu t_0-x_0|}$.
Substituting into~\eqref{eq:vtupper},  it follows that
\begin{align} \label{eq:vtupper2}
v_t(x)
&\le  g(x-\nu t)p^n_0(x_0)e^{(1+\alpha)s_0 t_0} n^{-1}(2\pi m t_0)^{-1/2} \notag \\
&\qquad \Big(\mathcal O(nt_0^{1/2}e^{\kappa |x_0|} e^{-2\kappa n^{1/4}}) 
+e^{\mathcal O(n^{-1/5})}\Esub{x-\nu t}{g(Z_{t-t_0})^{-1}e^{-(Z_{t-t_0}+\nu t_0-x_0)^2/(2mt_0)}
}\Big).
\end{align}
Note that for $y\in \R$,
\begin{align*}
g(y)^{-1} e^{-(y+\nu t_0-x_0)^2/(2m t_0)}
&\le 1 + e^{\kappa(x_0 -\nu t_0)} e^{(\kappa-(2mt_0)^{-1}(y+\nu t_0-x_0))(y+\nu t_0-x_0)}\\
&\le 1+e^{\kappa|x_0|+s_0 t_0}
\end{align*}
since $\frac 12 m \kappa^2 =s_0$.
Hence by Lemma~\ref{lem:Tbound}, for $n$ sufficiently large, if $t-t_0\ge \gamma_n/2$ and $|x-\nu t|\le d_n+1$, then
\begin{align} \label{eq:fromxstat*}
&\Esub{x-\nu t}{g(Z_{t-t_0})^{-1} e^{-(Z_{t-t_0}+\nu t_0-x_0)^2/(2m t_0)}} \notag \\
&\quad \le \int_{-\infty}^\infty \pi(y) g(y)^{-1} e^{-(y+\nu t_0-x_0)^2/(2m t_0)}  dy +3e^{\kappa|x_0|}(\log N)^{-12C}.
\end{align}
Note that $g(y) e^{\alpha \kappa y}\le \min(e^{\alpha \kappa y},e^{-(1-\alpha)\kappa y})\le 1$ $\forall y\in \R$.
Therefore, since $y\mapsto g(y)$ is decreasing, and letting $(B_s)_{s\ge 0}$ denote a Brownian motion,
\begin{align*}
& \int_{-\infty}^\infty g(y) e^{\alpha \kappa  y} e^{-(y+\nu t_0-x_0)^2/(2mt_0)}  dy\\
&\le 
g(x_0-\nu t_0-t_0^{1/3})
\int_{-\infty}^\infty e^{\alpha \kappa  y} e^{-(y+\nu t_0-x_0)^2/(2mt_0)} dy
+\int_{-\infty}^\infty e^{-(y+\nu t_0-x_0)^2/(2mt_0)} \1_{|y+\nu t_0-x_0|>t_0^{1/3}} dy \\
&\le 
(2\pi m t_0)^{1/2} \left( g(x_0-\nu t_0-t_0^{1/3})
\Esub{x_0-\nu t_0}{e^{\alpha \kappa  B_{mt_0}}} +\psub{0}{|B_{m t_0}|>t_0^{1/3}}\right)\\
&\le 
(2\pi m t_0)^{1/2} \left( g(x_0-\nu t_0-t_0^{1/3})
e^{\alpha \kappa  (x_0-\nu t_0)}e^{\frac 12 m\alpha ^2\kappa^2 t_0} +2e^{-t_0^{-1/3}/(2m)}\right)
\end{align*}
by a Gaussian tail bound.
Therefore if $|x_0|\le d_n$, by~\eqref{eq:fromxstat*} and since $|\frac{\nabla g(y)}{g(y)}|\le \kappa $ $\forall y\in \R$ and $g(y)^{-1} e^{-\alpha \kappa y}\le 2e^{\kappa |y|}$ $\forall y\in \R$,
\begin{align*}
& \Esub{x-\nu t}{g(Z_{t-t_0})^{-1} e^{-(Z_{t-t_0}+\nu t_0-x_0)^2/(2m t_0)}}\\
&\le (2\pi m t_0)^{1/2} \pi(x_0) g(x_0)^{-1} (1+\mathcal O(t_0^{1/3}) +\mathcal O(t_0^{-1/2}e^{2\kappa  d_n}(\log N)^{-12C})).
\end{align*}
Substituting into~\eqref{eq:vtupper2}, we have 
that if $t-t_0 \ge \gamma_n /2$, $|x-\nu t|\le d_n+1$ and $|x_0|\le d_n$,
\begin{align} \label{eq:lemvstatA}
\frac{v_t(x)}{g(x-\nu t)}
&\le n^{-1} p^n_0(x_0)\pi(x_0)g(x_0)^{-1}(1+\mathcal O((\log N)^{-4C})).
\end{align}
For a lower bound, note that
by~\eqref{eq:vngreena} with $a=(1-\alpha)s_0 $ and since $(1-u)(2u-1+\alpha )\ge \alpha -1$ $\forall u\in [0,1]$, for $y\in \frac 1n \Z$,
$$
v^n_{t_0}(y)\ge e^{-(1-\alpha)s_0 t_0} p^n_0(x_0) P^n_{t_0,x_0}(y).
$$
Suppose $n$ is sufficiently large that $t_0^{1/3} \le \frac 12 m nt_0$, and
then by~\eqref{eq:vteq},
\begin{align} \label{eq:lemvstatB}
v_t(x)
&\ge  g(x-\nu t)\Esub{x-\nu t}{g(Z_{t-t_0})^{-1} e^{-(1-\alpha)s_0 t_0} p^n_0(x_0)\bar{P}^n_{t_0,x_0}(Z_{t-t_0}+\nu t_0)\1_{|Z_{t-t_0}+\nu t_0-x_0|<t_0^{1/3}}} \notag \\
&\ge  g(x-\nu t)p^n_0(x_0)e^{-(1-\alpha)s_0 t_0} g(x_0-\nu t_0 -t_0^{1/3})^{-1} \notag \\
&\qquad  \Esub{x-\nu t}{n^{-1}(2\pi m t_0)^{-1/2}e^{-(Z_{t-t_0}+\nu t_0-x_0)^2/(2m t_0)}e^{\mathcal O(n^{-1} t_0^{-2})}\1_{|Z_{t-t_0}+\nu t_0-x_0|<t_0^{1/3}}}
\end{align}
by Lemma~\ref{lem:lclt}.
By Lemma~\ref{lem:Tbound},
for $n$ sufficiently large, if $t-t_0 \ge\gamma_n /2$ and $|x-\nu t|\le d_n+1$,
\begin{align} \label{eq:lemvstatC}
&\Esub{x-\nu t}{e^{-(Z_{t-t_0}+\nu t_0-x_0)^2/(2m t_0)}\1_{|Z_{t-t_0}+\nu t_0-x_0|<t_0^{1/3}}} \notag \\
&\ge \int_{-\infty}^\infty \pi(y) e^{-(y+\nu t_0-x_0)^2/(2m t_0)}\1_{|y+\nu t_0-x_0|<t_0^{1/3}} dy
-(\log N)^{-12C}.
\end{align}
Since $y\mapsto g(y)$ is decreasing,
\begin{align*}
&\int_{-\infty}^\infty g(y)^2 e^{\alpha \kappa  y} e^{-(y+\nu t_0-x_0)^2/(2m t_0)}\1_{|y+\nu t_0-x_0|<t_0^{1/3}} dy \\
&\ge g(x_0-\nu t_0+t_0^{1/3})^2 e^{\alpha \kappa  (x_0-\nu t_0-t_0^{1/3})}(2\pi m t_0)^{1/2} \left(1-\psub{0}{|B_{m t_0}|>t_0^{1/3}}\right)
\\
&\ge  g(x_0)^2 e^{\alpha \kappa  x_0} (2\pi m t_0)^{1/2}(1+\mathcal O(e^{-t_0^{-1/3}/(2m)})+\mathcal O(t_0^{1/3}))
\end{align*}
by a Gaussian tail bound and since $|\frac{\nabla g(y)}{g(y)}|\le \kappa $ $\forall y\in \R$.
Therefore if $t-t_0\ge \gamma_n/2$, $|x-\nu t|\le d_n+1$ and $|x_0|\le d_n$, by~\eqref{eq:lemvstatC} and~\eqref{eq:lemvstatB}, and since $(\log N)^{-12C} t_0^{-1/2}\pi(x_0)^{-1}=\mathcal O((\log N)^{-4C})$,
\begin{align} \label{eq:lemvstatD}
\frac{v_t(x)}{g(x-\nu t)}
&\ge 
p^n_0(x_0)n^{-1}\pi(x_0)g(x_0)^{-1}(1-\mathcal O((\log N)^{-4C})).
\end{align}
It remains to bound $|v^n_{\gamma_n}(x)-v_{\gamma_n}(x)|$.
By~\eqref{eq:lemvtstatvnb} and Lemma~\ref{lem:lclt},
for $z\in \frac 1n \Z$ and $t>0$,
\begin{equation} \label{eq:vnvcor*}
v^n_{t}(z)
\le e^{2s_0 t}p^n_0(x_0) n^{-1} (2\pi m t)^{-1/2} e^{\mathcal O(n^{-1} t^{-1/2})}.
\end{equation}
Therefore, by Lemma~\ref{lem:vnvbound},
for $n$ sufficiently large,
\begin{align} \label{eq:lemvnvcor1}
&\sup_{x\in \frac 1n \Z}|v^n_{\gamma_n}(x)-v_{\gamma_n}(x)| \notag \\
&\le 
 \Big(C_7 (n^{-1/3}+e^{-(\log N)^{c_2}}) e^{2s_0 t_0}p^n_0(x_0) (m t_0)^{-1/2} n^{-1}
+2n^{-1/3}\sup_{z\in \frac 1n \Z}|\nabla_n v^n_{t_0}(z)|\Big)  e^{5s_0{\gamma_n}}\gamma_n^2.
\end{align} 
Let $t_1=t_0/2$; then for $z\in \frac 1n \Z$, by~\eqref{eq:vngreena},
\begin{align*}
&|\nabla_n v^n_{t_0}(z)|\\
&=\Big| n \langle v^n_{t_1}, \phi^{t_1 ,z+n^{-1}}_0 -\phi^{t_1,z}_0 \rangle_n
+ns_0  \int_0^{t_1} \langle v^n_{t_1+s}(1-u^n_{t_1+s})(2u^n_{t_1+s}-1+\alpha ), \phi^{t_1,z+n^{-1}}_s-\phi^{t_1,z}_s \rangle_n ds \Big|\\
&\le \sup_{x\in \frac 1n \Z, s\in [0,t_1]} v^n_{t_1+s}(x)\left(n\langle 1,  |\phi^{t_1,z+n^{-1}}_0 -\phi^{t_1,z}_0 | \rangle_n
+ns_0 \int_0^{t_1} \langle 1+\alpha , |\phi^{t_1,z+n^{-1}}_s-\phi^{t_1,z}_s | \rangle_n ds \right)\\
&\le e^{2s_0 t_0}p^n_0(x_0)n^{-1}(m t_1)^{-1/2} \left(C_5 t_1^{-1/2}+\int_0^{t_1} 2s_0 C_5 (t_1-s)^{-1/2}ds \right)
\end{align*}
for $n$ sufficiently large,
by~\eqref{eq:vnvcor*} and Lemma~\ref{lem:nablaunbound}.
Hence 
\begin{align*}
\sup_{z\in \frac 1n \Z} |\nabla_n v^n_{t_0}(z)|
&\le e^{2s_0 t_0}p^n_0(x_0) n^{-1}m^{-1/2} C_5 (2t_0^{-1}+4s_0).
\end{align*}
By~\eqref{eq:lemvnvcor1} it follows that for $n$ sufficiently large,
$\sup_{x\in \frac 1n \Z}|v^n_{\gamma_n}(x)-v_{\gamma_n}(x)|\le p^n_0(x_0) n^{-1} (e^{-\frac 12 (\log N)^{c_2}} \vee n^{-1/6})$.
By~\eqref{eq:lemvstatA} and~\eqref{eq:lemvstatD}, this completes the proof.
\end{proof}
We now show that $|q^n_{\gamma_n}-v^n_{\gamma_n}|$ is small with high probability, which, combined with the previous lemma, will imply that $A^{(1)}_t(x_1,x_2)$ occurs with high probability for suitable $x_1$, $x_2$ and $t$.
This result is stronger than Proposition~\ref{prop:pnun} (but only applies when $q^n_0(x)=p^n_0(x_0)\1_{x=x_0}$ for some $x_0$), and will also be used to show that $A^{(4)}_t(x)$ occurs with high probability for suitable $x$ and $t$.
\begin{lemma} \label{lem:qnvnonepoint}
For $c, c'\in (0,1/2)$ and $\ell \in \N$,
 the following holds for $n$ sufficiently large.
Suppose $N\ge n^{3}$, and for some $x_0\in \frac 1n \Z$, $q^n_0(x)=p^n_0(x_0)\1_{x=x_0}$ and $p^n_0(x_0)\ge \big( \frac{n^2}N \big)^{1-c}$.
For $t\le \gamma_n $ and $z\in \frac 1n \Z$,
$$
\p{ |q^n_t(z)-v^n_t(z)|\ge \left(\frac n N \right)^{1/2-c'}p^n_0(x_0)^{1/2}n^{-1/2}}\le \left(\frac n N \right)^{\ell},
$$
where $(q^n_t)_{t\ge 0}$ and $(v^n_t)_{t\ge 0}$ are defined in~\eqref{eq:qndef} and~\eqref{eq:vndef} respectively.
\end{lemma}
\begin{proof}
By Lemma~\ref{lem:lclt}, there exists a constant $K_5>1$ such that
\begin{equation} \label{eq:qnvnonepoint*}
\psubb{0}{X^n_{mt}=0}\le K_5 n^{-1}t^{-1/2} \quad \forall \, n\in \N,\, t>0.
\end{equation}
By Corollary~\ref{cor:qnMa} with $a=-(1+\alpha)s_0$, for $t\ge 0$ and $z\in \frac 1n \Z$,
\begin{align} \label{eq:qntbound2}
q^n_t(z)
&\le e^{(1+\alpha )s_0 t } \langle q^n_0, \phi^{t,z}_0\rangle_n +M^n_t(\phi^{t,z,-(1+\alpha)s_0}) \notag \\
&\le  e^{(1+\alpha )s_0 t } p^n_0(x_0)\min (K_5 n^{-1}t^{-1/2}, 1) +M^n_t(\phi^{t,z,-(1+\alpha)s_0})
\end{align}
by~\eqref{eq:qnvnonepoint*}.
Let
$$
\tau=\inf\left\{t> 0 : \sup_{x\in \frac 1n \Z}q^n_t(x)\ge K_5 e^{2s_0 \gamma_n} p^n_0(x_0) n^{-1}t^{-1/2} \right\}.
$$
We will show that $\tau>\gamma_n$ with high probability.
By Lemma~\ref{lem:qnphi}, for $t>0$,
\begin{align*}
\sup_{s\in [0,t]}|M^n_s(\phi^{t,z,-(1+\alpha)s_0})-M^n_{s-}(\phi^{t,z,-(1+\alpha)s_0})|
=\sup_{s\in [0,t]}|\langle q^n_s-q^n_{s-},\phi_s^{t,z,-(1+\alpha)s_0}\rangle _n|\le e^{(1+\alpha)s_0t}N^{-1}.
\end{align*}
Therefore, by the Burkholder-Davis-Gundy inequality as stated in Lemma~\ref{lem:BDG}, for $t\ge 0$, $z\in \frac 1n \Z$ and $k\in \N$ with $k\ge 2$,
\begin{align} \label{eq:BDGqnvnonept}
\E {\sup_{s\in [0,t]}|M^n_{s\wedge \tau}(\phi^{t,z,-(1+\alpha)s_0})|^k }
\leq C(k) \E {\langle M^n(\phi^{t,z,-(1+\alpha)s_0}) \rangle_{t\wedge \tau}^{k/2} +e^{(1+\alpha)s_0 tk}N^{-k} }.
\end{align}
Then for $t\le \gamma_n$, by the definition of $\tau$ and by Lemma~\ref{lem:qnphi},
\begin{align} \label{eq:qnvnonepointd}
\langle M^{n}(\phi^{t,z,-(1+\alpha)s_0}) \rangle_{t\wedge \tau}
&\le \frac n N
\int_0^t \langle (1+2m) K_5 e^{2s_0 \gamma_n} p^n_0(x_0) n^{-1}s^{-1/2} , (\phi_s^{t,z})^2 e^{2(1+\alpha)s_0 (t-s)}\rangle_n
ds \notag\\
&\le \frac n N (1+2m) K_5 e^{6s_0 \gamma_n} p^n_0(x_0)  \int_0^t s^{-1/2} \psubb{0}{X^n_{2m(t-s)}=0}ds,
\end{align}
by Lemma~\ref{lem:Mqvarbound}.
Then by~\eqref{eq:qnvnonepoint*},
\begin{align*}
\int_0^t  s^{-1/2} \psubb{0}{X^n_{2m(t-s)}=0}ds 
&\le \int_{0}^{t} s^{-1/2}K_5 n^{-1}(2(t-s))^{-1/2}ds\\
&= K_5 n^{-1}2^{-1/2}\cdot 2 \int_0^{t/2} s^{-1/2}(t-s)^{-1/2} ds\\
&\le 2^{3/2}K_5 n^{-1}.
\end{align*}
Hence, by~\eqref{eq:qnvnonepointd}, for $t\le \gamma_n$,
\begin{align} \label{eq:qnvnonepointA}
\langle M^{n}(\phi^{t,z,-(1+\alpha)s_0 }) \rangle_{t\wedge \tau}
&\leq \frac{1}{N}(1+2m)
 2^{3/2}K_5^2 e^{6s_0 \gamma_n} p^n_0(x_0).
\end{align}
For $b \in (0,1/2)$ and $\ell_1 \in \N$, take $k\in \N$ with $k>\ell_1 /b$.
Then for $n$ sufficiently large, for $t\le \gamma_n$ and $z\in \frac 1n \Z$, by Markov's inequality and~\eqref{eq:BDGqnvnonept},
and since $p^n_0(x_0)^{1/2}N^{-1/2}\ge (\frac{n^2}N)^{1/2}N^{-1/2}=n N^{-1}$,
\begin{align} \label{eq:Mnttaubound}
&\p{|M^n_{t\wedge \tau}(\phi^{t,z,-(1+\alpha)s_0})|\ge \left( \frac n N \right)^{1/2-b}  p^n_0(x_0)^{1/2} n^{-1/2}} \notag\\
&\le \left( \frac n N \right)^{-k(1/2-b)}
p^n_0(x_0)^{-k/2}n^{k/2}C(k)\cdot 2
\left( \frac{1}{N}(1+2m)
2^{3/2} K_5^2 e^{6s_0 \gamma_n} p^n_0(x_0) \right)^{k/2} \notag 
\\
&\le \left( \frac n N \right)^{\ell_1}
\end{align}
for $n$ sufficiently large, since $bk>\ell_1$ and $\gamma_n =\lfloor (\log \log N)^4 \rfloor$.
Now let $b=c/4$.
Then for $n$ sufficiently large, since $N\ge n^3$ and then since $p^n_0(x_0)\ge (\frac{n^2}N)^{1-c}$,
\begin{equation} \label{eq:qnvnonea}
\left( \frac n N \right)^{1/2-b}n^{-1/2}\le \left( \frac{n^2}N \right)^{(1-c)/2}n^{-1}
\le \tfrac 13 K_5 e^{2s_0 \gamma_n} (\gamma_n+N^{-1})^{-1/2}p^n_0(x_0)^{1/2}n^{-1}.
\end{equation}
Since $p^n_0(x_0)\ge n^2 N^{-1}$, we can
take $n$ sufficiently large that
\begin{equation} \label{eq:qnvnoneb}
N^{-1}\le \tfrac 13 K_5 e^{2s_0 \gamma_n} (\gamma_n+N^{-1})^{-1/2}p^n_0(x_0)n^{-1}
\end{equation}
 and also, since $\alpha<1$ and $N\ge n^3$,
 \begin{equation} \label{eq:qnvnonec}
 e^{(1+\alpha)s_0 t}t^{-1/2}\le \tfrac 13 e^{2s_0 \gamma_n} (t+N^{-1})^{-1/2} \; \forall t\in [N^{-1},\gamma_n]
 \quad \text{ and }\quad  \tfrac 13 n^{-1} (2N^{-1})^{-1/2}\ge 1.
 \end{equation}
If $|M^n_{t\wedge \tau}(\phi^{t,z,-(1+\alpha)s_0})|\le \left( \frac n N \right)^{1/2-b}  p^n_0(x_0)^{1/2} n^{-1/2}$ and $t\in [0, \tau \wedge \gamma_n]$ then by~\eqref{eq:qntbound2}, and since $K_5>1$,
\begin{align} \label{eq:lemqnvnonequ}
q^n_t(z)
&\le K_5 e^{(1+\alpha)s_0 t} p^n_0(x_0) \min(n^{-1}t^{-1/2},1)+
\left( \frac n N \right)^{1/2-b} p^n_0(x_0)^{1/2}n^{-1/2} \notag \\
&\le K_5 e^{2s_0 \gamma_n} (t+N^{-1})^{-1/2}p^n_0(x_0) n^{-1}-N^{-1},
\end{align}
by~\eqref{eq:qnvnonea},~\eqref{eq:qnvnoneb} and~\eqref{eq:qnvnonec} (using the second equation in~\eqref{eq:qnvnonec} for the case $t\le N^{-1}$).
Take $\ell_2 \in \N$ and let $Y_n \sim \text{Poisson}((2m+1)N^{2-\ell_2}r_n)$.
Then for $t\ge 0$ and $z\in \frac 1n \Z$, since $(q^n_s(z))_{s\ge 0}$ jumps at rate at most $(2m+1)r_n N^2$,
\begin{equation} \label{eq:qndoublejump}
\p{\sup_{s\in [0,N^{-\ell_2}]}|q^n_{t+s}(z)-q^n_t(z)|>N^{-1}}
\le \p{Y_n \ge 2}\le (\tfrac 12 (2m+1)N^{1-\ell_2}n^2 )^2
\end{equation}
since $r_n = \frac 12 n^2 N^{-1}$.
Therefore, for $\ell_1,\ell_2 \in \N$,
letting $\mathcal A=N^{-\ell_2}\N_0 \cap [0,\gamma_n ]$,
by a union bound and~\eqref{eq:lemqnvnonequ},
\begin{align*} 
&\p{\tau \le \gamma_n } \notag \\
&\le \p{\exists t \in \mathcal A, z\in \tfrac 1n \Z: |z-x_0|\le N^5,
|M^n_{t\wedge \tau}(\phi^{t,z,-(1+\alpha)s_0})|\ge \left(\frac n N \right)^{1/2-b}p^n_0(x_0)^{1/2}n^{-1/2}} \notag \\
&\quad + \p{\exists t \in \mathcal A, z\in \tfrac 1n \Z: |z-x_0|\le N^5, \sup_{s\in [0,N^{-\ell_2}]} |q^n_{t+s}(z)-q^n_t(z)|>N^{-1}} \notag \\
 &\quad + \p{\exists z\in \tfrac 1n \Z, t\in [0,\gamma_n]: |z-x_0|> N^5, q^n_t(z)>0} \notag \\
&\le \sum_{t \in \mathcal A}(2nN^5+1)
\left( \frac n N \right)^{\ell_1} 
 +\sum_{t \in \mathcal A}(2nN^5+1)
(\tfrac 12 (2m+1)N^{1-\ell_2}n^2)^2 +2e^{-N^5} ,
\end{align*}
for $n$ sufficiently large, by~\eqref{eq:Mnttaubound} and~\eqref{eq:qndoublejump}, and by the same argument as Lemma~\ref{lem:p01} for the last term.
For $\ell '\in \N$, take $\ell_2$ sufficiently large that $\gamma_n N^{\ell_2+5}n(N^{1-\ell_2}n^2)^2 =\gamma_n N^{7-\ell_2}n^5 \le \left( \frac n N \right)^{\ell '+1}$ for $n$ sufficiently large, and then take $\ell_1$ sufficiently large that $\gamma_n N^{\ell_2+5}n \left( \frac n N \right)^{\ell_1} \le \left( \frac n N \right)^{\ell '+1}$ for $n$ sufficiently large.
It follows that for $n$ sufficiently large,
\begin{equation} \label{eq:Ataut}
\p{\tau \le \gamma_n } \le \left( \frac n N \right)^{\ell '}.
\end{equation}
Note that by~\eqref{eq:veasybound} and~\eqref{eq:qnvnonepoint*}, for $t\ge 0$ and $z\in \frac 1n \Z$,
\begin{equation} \label{eq:qnvnonepointB}
v^n_t(z)\le e^{(1+\alpha)s_0 t} \langle q^n_0, \phi^{t,z}_0\rangle_n \le e^{(1+\alpha)s_0 t} p^n_0(x_0)\min(K_5 n^{-1}t^{-1/2}, 1).
\end{equation}
Take $k\in \N$ with $k\ge 2$.
By Lemma~\ref{lem:qnvndet} and since $q^n_t,v^n_t\in [0,1]$, we have that for $t\ge 0$ and $z\in \frac 1n \Z$,
\begin{align*}
|q^n_t(z)-v^n_t(z)|^k
&\leq 3^{2k-1}s_0^k  t^{k-1} \left(\int_0^t \langle |q^n_s - v^n_s|^k , \phi^{t,z}_s \rangle_n  ds 
 +  \int_0^t  \sup_{x\in \frac 1n \Z} v^n_s(x)^k \langle |p^n_s- u^n_s|^k , \phi^{t,z}_s \rangle_n  ds \right)\\
&\qquad +\1_{\tau<t}+3^{k-1}|M^n_{t\wedge \tau}(\phi^{t,z})|^k.
\end{align*}
Therefore, by~\eqref{eq:gronwall1stat} in Proposition~\ref{prop:pnun} and by~\eqref{eq:qnvnonepointB} and~\eqref{eq:Ataut}, for $\ell ' \in \N$, for $n$ sufficiently large, for $t\le \gamma_n$ and $z\in \frac 1n \Z$,
\begin{align} \label{eq:qnvnkexp}
&\E{|q^n_t(z)-v^n_t(z)|^k} \notag \\
&\leq 3^{2k-1}s_0^k t^{k-1} \int_0^t \sup_{x\in \frac 1n \Z}\E{ |q^n_s(x)-v^n_s(x)|^k }  ds \notag \\
&\quad + 3^{2k-1}s_0^k t^{k-1}e^{(1+\alpha)s_0 tk}p_0^n(x_0)^k \int_0^t  (K_5 n^{-1}s^{-1/2}\wedge 1)^k C_1\left( \frac{n^{k/2}s^{k/4}}{N^{k/2}}+N^{-k}\right)  e^{C_1 s^k} ds \notag \\
&\qquad +\left(\frac n N \right)^{\ell '} +3^{k-1}
\E{|M^n_{t\wedge \tau}(\phi^{t,z})|^k}.
\end{align}
Take $\ell '$ sufficiently large that for $n$ sufficiently large,
$\left( \frac n N \right)^{\ell '}\le N^{-k/2} \left(\frac{n^2}N \right)^{k/2}\le N^{-k/2} p^n_0(x_0)^{k/2}$.
Note that for the second term on the right hand side of~\eqref{eq:qnvnkexp},
\begin{align*}
& \int_0^t (K_5 n^{-1}s^{-1/2}\wedge 1)^k C_1 \left( \frac{n^{k/2}s^{k/4}}{N^{k/2}}+N^{-k}\right) e^{C_1 s^k}ds\\
&\le C_1 \int_0^t (K_5^{k/2}N^{-k/2}+N^{-k}) e^{C_1 s^k}ds\\
&\le C_1 (K_5^{k/2}N^{-k/2}+N^{-k})te^{C_1 t^k}.
\end{align*}
By the same argument as in~\eqref{eq:BDGqnvnonept} and~\eqref{eq:qnvnonepointA},
since $t\le \gamma_n$,
$$
\E{|M^n_{t\wedge \tau}(\phi^{t,z})|^k}
\le C(k)
\left(\left( \frac{1}{N}(1+2m)
2^{3/2}K_5^2 e^{2s_0 \gamma_n} p^n_0(x_0) \right)^{k/2}
+N^{-k}
\right).
$$
Note that $N^{-1/2}p^n_0(x_0)^{1/2}\ge n N^{-1}$.
Hence substituting into~\eqref{eq:qnvnkexp} and then by Gronwall's inequality, there exists a constant $K_6=K_6(k)$ such that for $n$ sufficiently large, for $t\in [0,\gamma_n]$,
\begin{equation} \label{eq:vnqnonedagger}
\sup_{x\in \frac 1n \Z}\E{|q^n_t(x)-v^n_t(x)|^k}
\le K_6(\gamma_n^ke^{(1+\alpha)s_0 \gamma_n k} e^{C_1 \gamma_n^k}+1+e^{s_0\gamma_n k})N^{-k/2}p^n_0 (x_0)^{k/2}e^{3^{2k-1}s_0^k \gamma_n^{k-1}t}.
\end{equation}
The result now follows by Markov's inequality, taking $k\in \N$ sufficiently large that $kc'>\ell$, and then taking $n$ sufficiently large that~\eqref{eq:vnqnonedagger} holds with this choice of $k$.
\end{proof}
We are now ready to prove that $A^{(1)}_t(x_1,x_2)$ occurs with high probability for suitable $t$, $x_1$ and $x_2$.
For $t\ge 0$ and $x_1 \in \frac 1n \Z$, let
$(v^n_{t,t+s}(x_1,\cdot))_{s\ge 0}$ denote the solution of
\begin{equation} \label{eq:vnx1defn}
\begin{cases}
\partial_s v^n_{t,t+s}(x_1,\cdot)&=\tfrac 12 m\Delta_n v^n_{t,t+s}(x_1,\cdot)+s_0 v^n_{t,t+s}(x_1,\cdot)(1-u^n_{t,t+s})(2u^n_{t,t+s}-1+\alpha ) \quad \text{for }s>0,\\
v^n_{t,t}(x_1,x)&= p^n_t(x_1) \1_{x=x_1},
\end{cases}
\end{equation}
where $(u^n_{t,t+s})_{s\ge 0}$ is defined in~\eqref{eq:unttsdef}. 
Recall the definition of $q^n_{t_1,t_2}(x_1,x_2)$ in~\eqref{eq:qt1t2defn}.
\begin{prop} \label{prop:eventA1}
Suppose $N\ge n^{3}$ for $n$ sufficiently large.
For $\ell \in \N$, the following holds for $n$ sufficiently large.
For $t\in [(\log N)^2-\gamma_n,N^2]$ and $x_1,x_2 \in \frac 1n \Z$,
$$
\p{A^{(1)}_{t}(x_1,x_2)^c \cap \{|x_1-\mu^n_t|\vee |x_2 -\mu^n_{t+\gamma_n}|\le d_n \} \cap E'_1}
\le \left( \frac n N \right)^\ell.
$$
\end{prop}
\begin{proof}
Fix $c'\in (0,1/4)$.
By Lemma~\ref{lem:qnvnonepoint}, for $n$ sufficiently large,
\begin{equation} \label{eq:propA11}
\p{\left\{ |q^n_{t,t+\gamma_n}(x_1,x_2)-v^n_{t,t+\gamma_n}(x_1,x_2)|\ge \left( \frac n N \right)^{1/2-c'}n^{-1/2}\right\}
\cap \left\{ p^n_{t}(x_1)\ge \left( \frac{n^2}N \right)^{3/4}
\right\} }
\le \left( \frac n N \right)^\ell.
\end{equation}
Suppose $n$ is sufficiently large that $(\log N)^2-\gamma_n \ge \frac 12 (\log N)^2 \vee \log N$.
Recall the definition of $E'_1$ in~\eqref{eq:E1'defn}.
By Lemma~\ref{lem:vtstat}, if $E_1'$ occurs and $|x_1-\mu^n_t|\le d_n$, $|x_2 -\nu  \gamma_n -\mu^n_t|\le d_n+1$ then
\begin{align*}
\frac{v^n_{t,t+\gamma_n}(x_1,x_2)}{g(x_2-\nu  \gamma_n-\mu^n_t)}&=
\frac{\pi(x_1-\mu^n_t)}{g(x_1-\mu^n_t)}
p^n_{t}(x_1)n^{-1} (1+\mathcal O((\log N)^{-4C})).
\end{align*}
Suppose $|x_1-\mu^n_t|\vee |x_2-\mu^n_{t+\gamma_n}|\le d_n$ and $E'_1$ occurs. Then
if $n$ is sufficiently large, by the definition of $E_1$ in~\eqref{eq:eventE1}
we have
$p^n_t(x_1)\ge \frac 1 {10} (\log N)^{-C}$, $|x_2-\nu  \gamma_n -\mu^n_t|\le d_n+1$, $|p^n_t(x_1)-g(x_1-\mu^n_t)|\le e^{-(\log N)^{c_2}}$, $|p^n_{t+\gamma_n}(x_2)-g(x_2-\mu^n_{t+\gamma_n})|\le e^{-(\log N)^{c_2}}$
and $|\mu^n_{t+\gamma_n}-(\mu^n_t +\nu \gamma_n)|\le \gamma_n e^{-(\log N)^{c_2}}$.
Hence for $n$ sufficiently large, if  $|q^n_{t,t+\gamma_n}(x_1,x_2)-v^n_{t,t+\gamma_n}(x_1,x_2)|\le \left( \frac n N \right)^{1/2-c'}n^{-1/2}\le n^{-3/2+2c'}$, then $A^{(1)}_{t}(x_1,x_2)$ occurs.
By~\eqref{eq:propA11}, this completes the proof.
\end{proof}
The next two lemmas will be used to show $A^{(2)}_t(x_1,x_2)$ and $A^{(3)}_t(x_1,x_2)$ occur with high probability for suitable $t$, $x_1$ and $x_2$.
Recall that we fixed $c_1>0$, and recall the definition of $D_n^+$ in~\eqref{eq:Dn+-defn}.
\begin{lemma} \label{lem:vfromtipbound}
For $\epsilon>0$ sufficiently small, $t^* \in \N$ sufficiently large and $K \in \N$ sufficiently large (depending on $t^*$), the following holds for $n$ sufficiently large.
Suppose $\sup_{s\in [0,t^*],x\in \frac 1n \Z}|u^n_s(x)-g(x-\nu s)|<\epsilon$, and also $p^n_t(x)\in [\frac 16 g(x-\nu  t), 6g(x-\nu  t)]$ $\forall t\in [0,t^*]$, $x\le \nu  t+D^+_n+1$ and $p_t(x)\le 6g(D^+_n)$ $\forall t\in [0,t^*], x\ge \nu  t+D^+_n$.
Suppose $q^n_0(z)=p^n_0(z)\1_{z\ge \ell}$ for some $\ell \in \tfrac 1n \Z\cap  [K, D^+_n ]$. Then for $z\le \nu t^*+ D^+_n +1$,
$$
\frac{v^n_{t^*}(z)}{p^n_{t^*}(z)}\le \tfrac 12 c_1 e^{-(1+\frac 12 (1-\alpha))\kappa (\ell -(z-\nu  t^*)\vee K +2)},
$$
where $(v^n_t)_{t\ge 0}$ is defined in~\eqref{eq:vndef}.
\end{lemma}
\begin{proof}
Let $\lambda = \frac 12 (1-\alpha)$.
Note that since $(\alpha-2)^2>1$, we have $\frac 14 (1-\alpha ^2)<1-\alpha$.
Take $a\in (\frac 14 (1-\alpha ^2),1-\alpha)$ so that
$$
\lambda^2 +\lambda \alpha -a = \tfrac 12 (1-\alpha)(\tfrac 12 (1-\alpha)+\alpha)-a=\tfrac 14 (1-\alpha^2)-a<0.
$$
Take $t^* \in \N$ sufficiently large that $144 e^{(\lambda^2 +\lambda \alpha -a)s_0 t^* }\le \tfrac 13c_1 e^{-2\kappa (1+\lambda)}.$
Take
$\epsilon \in (0, \frac 12 (1-\alpha))$ sufficiently small that $(1-\epsilon)(2\epsilon-1+\alpha)<-a$.
Then take $K\in \N$ sufficiently large that $\nu t^*\le K/6$,
$2s_0 t^* e^{4s_0 t^*} e^{-\lambda \kappa  K/6}\le 1$, $72 e^{5s_0 t^*}e^{-(1-\lambda) \kappa K/2}\le \frac 12 c_1 e^{-2\kappa (1+\lambda)}$,
$2g(K/3)+2\epsilon <1-\alpha$ and
$$
(1-g(x)-\epsilon)(2(g(x)+\epsilon)-1+\alpha)\le -a \qquad \text{for }x\ge K/3.
$$
Then for $s\ge 0$ and $x\in \frac 1n \Z$, if $x-\nu  s \ge K/3$ and $|u^n_s(x)-g(x-\nu  s)|<\epsilon$ we have
$$
(1-u^n_s(x))(2u^n_s(x)-1+\alpha)+a\le 0.
$$
If instead $x-\nu  s\le K/3$, then by~\eqref{eq:veasybound},
\begin{align*}
v^n_s(x)\le e^{(1+\alpha)s_0 s}\Esubb{x}{p^n_0(X^n_{ms}) \1_{X^n_{ms}\ge \ell}}
\le e^{(1+\alpha)s_0 s} \max_{y\ge \ell} p^n_0(y) \psubb{0}{X^n_{ms}\ge \ell -\tfrac 13 K-\nu  s}.
\end{align*}
Moreover, for $u\in [0,1]$, we have $(1-u)(2u-1+\alpha)+a\le 2$.

Suppose $\ell \in [K,D^+_n]$ and $\sup_{s\in [0,t^*],x\in \frac 1n \Z}|u^n_s(x)-g(x-\nu  s)|<\epsilon$.
For $z\in \frac 1n \Z$ and $t\in [0,t^*]$ we have by~\eqref{eq:vngreena} that
\begin{align} \label{eq:vnupper*}
v^n_t(z)
&\le  e^{-as_0 t}\langle q^n_0, \phi_0^{t,z}\rangle_n 
+\int_0^t 2s_0 e^{-as_0 (t-s)}\sup_{x-\nu  s\le K/3}v^n_s(x) ds \notag \\
&\le \max_{x\ge \ell}p^n_0(x)\left( e^{-as_0 t} \psubb{z}{X^n_{mt}\ge \ell}
+2s_0 e^{(1+\alpha )s_0 t}\int_0^t \psubb{0}{X^n_{ms}\ge \ell -\tfrac 13 K-\nu s}ds \right).
\end{align}
By Markov's inequality and Lemma~\ref{lem:Xnmgf}, and since $\frac 12 m \kappa^2 =s_0$,
\begin{align*}
\psubb{z}{X^n_{mt}\ge \ell}
= \psubb{0}{X^n_{mt}\ge \ell -z}
&\le e^{-\lambda \kappa  (\ell-z)}\Esubb{0}{e^{\lambda \kappa  X^n_{mt}}}\\
&= e^{-\lambda \kappa  (\ell-z)}e^{(\lambda^2 +\mathcal O(n^{-1}))s_0 t}.
\end{align*}
Therefore, applying the same argument to the second term on the right hand side of~\eqref{eq:vnupper*},
\begin{align*}
v^n_t(z)
&\le \max_{x\ge \ell} p^n_0(x) (e^{-\lambda \kappa (\ell-z)} e^{(\lambda^2 -a +\mathcal O(n^{-1}))s_0 t} +2s_0 t e^{(1+\alpha)s_0 t} e^{-\lambda \kappa  (\ell -\frac 13 K-\nu t )} e^{(\lambda^2 +\mathcal O(n^{-1}))s_0t})\\
&\le \max_{x\ge \ell} p^n_0(x) e^{-\lambda \kappa (\ell-z)}e^{(\lambda^2 -a +\mathcal O(n^{-1}))s_0 t}
(1+2s_0 t e^{(1+\alpha +a +\lambda \alpha)s_0 t} e^{-\lambda \kappa  (z-\frac 13 K)}),
\end{align*}
since $\kappa \nu =\alpha s_0$.
Hence for $z \in [\frac 12 K+\nu  t^*, D^+_n +1+ \nu  t^*]$,
\begin{align} \label{eq:vnpnt*}
\frac{v^n_{t^*}(z)}{p^n_{t^*}(z)}
&\le  \frac{6g(\ell)}{\frac 16 g(z-\nu  t^*)}e^{-\lambda \kappa (\ell -z)} e^{(\lambda^2 -a+\mathcal O(n^{-1}))s_0 t^*}
(1+2s_0 t^* e^{4s_0 t^*} e^{-\lambda \kappa  K/6}) \notag \\
&\le 36 e^{-\kappa  \ell} \cdot 2 e^{\kappa (z-\nu  t^*)}e^{-\lambda \kappa (\ell -z)} e^{(\lambda^2 -a+\mathcal O(n^{-1}))s_0 t^*}\cdot 2 \notag \\
&= 144 e^{-(1+\lambda)\kappa (\ell -(z-\nu t^*))}  e^{(\lambda^2+\alpha \lambda -a+\mathcal O(n^{-1}))s_0 t^*} \notag \\
&\le \tfrac 12 c_1 e^{-(1+\lambda)\kappa (\ell -(z-\nu  t^*)+2)}
\end{align}
for $n$ sufficiently large,
where the second inequality follows by our choice of $K$, and the last inequality by our choice of $t^*$.
Also, for any $z\in \frac 1n \Z$ and $t\ge 0$, by~\eqref{eq:veasybound} and then by Markov's inequality and Lemma~\ref{lem:Xnmgf},
\begin{align*}
v^n_t(z)
\le 
e^{(1+\alpha )s_0 t}\max_{x\ge \ell}p^n_0(x)
\psubb{z}{X^n_{mt}\ge \ell} 
&\le e^{(1+\alpha)s_0 t} \max_{x\ge \ell}p^n_0(x) e^{-\kappa (\ell -z)} \Esubb{0}{e^{\kappa X^n_{mt}}}\\
&\le e^{(1+\alpha )s_0 t}\max_{x\ge \ell}p^n_0(x)e^{2s_0 t}e^{-\kappa (\ell -z)}
\end{align*}
for $n$ sufficiently large.
Therefore, for $z\le \frac 12 K+\nu  t^*\le \frac 23 K$, and then since $\kappa \nu =\alpha s_0$,
\begin{align*}
\frac{v^n_{t^*}(z)}{p^n_{t^*}(z)}&\le e^{(1+\alpha)s_0 t^*} \frac{6g(\ell)}{\frac 16 g(K/2)} e^{2s_0 t^*} e^{-\kappa (\ell -\frac 12 K-\nu  t^*)}\\
&\le 72 e^{5s_0 t^*} e^{-(1+\lambda)\kappa (\ell -\frac 12 K)}e^{-(1-\lambda) \kappa \cdot \frac 12 K}\\
&\le \tfrac 12 c_1 e^{-(1+\lambda)\kappa (\ell -\frac 12 K+2)},
\end{align*}
where the second inequality follows since $g(\ell)\le e^{-\kappa \ell}$, $g(K/2)^{-1}\le 2e^{\kappa K/2}$ and $\ell-\frac 12 K \ge \frac 12 K$ and the third inequality follows
by our choice of $K$.
By~\eqref{eq:vnpnt*}, this completes the proof.
\end{proof}

\begin{lemma} \label{lem:bulktail1}
For $\epsilon>0$ sufficiently small and
$t^* \in \N$ sufficiently large, for $K \in \N$ sufficiently large (depending on $t^*$), the following holds for $n$ sufficiently large.
Suppose $\sup_{s\in [0,t^*],\, x\in \frac 1n \Z}|u^n_s(x)-g(x-\nu  s)|< \epsilon$,
and $p^n_t(x)\ge \frac 16 g(x-\nu  t)$ $\forall t\in [0,t^*]$, $x\le \nu  t+D^+_n$.
Suppose $q^n_0(z)=p^n_0(z)\1_{z\le \ell}$ for some $\ell \in \frac 1n \Z$ with
$\ell \le -K$. Then for $z\le \nu  t^*+D^+_n$,
\begin{equation} \label{eq:lembulktail}
\frac{v^n_{t^*}(z)}{p^n_{t^*}(z)}\le \tfrac 12 c_1 e^{-\frac 12 \alpha\kappa  ((z-\nu  t^*)-\ell+1)},
\end{equation}
where $(v^n_t)_{t\ge 0}$ is defined in~\eqref{eq:vndef}.
\end{lemma}
\begin{proof}
Take $c\in (0,\alpha^2 /4)$.
Take $t^* \in \N$ sufficiently large that $e^{(c-\alpha^2/4)s_0 t^*}<\frac 1 {10} c_1 e^{-\kappa }$.
Suppose $\sup_{s\in [0,t^*], x\in \frac 1n \Z} |u^n_s(x)-g(x-\nu  s)|<c/4$.
Take $K\in \N$ sufficiently large that $g(-K/2)\ge 1-c/4$, $2s_0 t^* e^{13s_0 t^*} e^{-\kappa  K/2}<\frac 1 {10}c_1e^{-\kappa }$
and $e^{7 s_0 t^*}e^{-\kappa K}< \frac 1 {24} c_1 e^{-\kappa }$.
Then for $s\in [0,t^*]$ and $x\in \frac 1n \Z$ with $x\le -\frac 12 K+\nu  s$, we have
$(1-u^n_s(x))(2u^n_s(x)-1+\alpha)\le c$.
Take $\ell \in \frac 1n \Z$ with $\ell \le -K$.
By~\eqref{eq:vngreena} with $a=-cs_0 $,
and since $(1-u)(2u-1+\alpha)-c\le 2$ for $u\in [0,1]$, for $t\in [0,t^*]$ and $z\in \frac 1n \Z$,
\begin{align} \label{eq:bulkvnt*}
v^n_t(z)
&\le e^{cs_0 t} \langle q^n_0, \phi^{t,z}_0 \rangle_n +s_0 \int_0^t e^{cs_0 (t-s)} \langle 2 v^n_s(\cdot) \1_{\cdot \ge -\frac 12 K +\nu  s}, \phi^{t,z}_s \rangle_n ds \notag \\
&\le e^{cs_0 t} \psubb{z}{X^n_{mt}\le \ell} +2 s_0 e^{cs_0 t} \int_0^t \sup_{x\ge -\frac 12 K+\nu s}v^n_s(x) ds.
\end{align}
For $s\in [0,t]$ and $x\ge -\frac 12 K+\nu  s$, by~\eqref{eq:veasybound},
\begin{align*}
v^n_s(x)\le e^{(1+\alpha)s_0 s}\psubb{x}{X^n_{ms}\le \ell}
&\le e^{(1+\alpha)s_0 s}\psubb{0}{X^n_{ms}\ge -\ell -\tfrac 12 K+\nu  s}\\
&\le e^{(1+\alpha)s_0 s}e^{3\kappa (\ell+\frac 12 K-\nu  s)}e^{10s_0 s},
\end{align*}
for $n$ sufficiently large, by Markov's inequality and Lemma~\ref{lem:Xnmgf}, and since $\frac 12 m\kappa^2 =s_0$.
Hence by~\eqref{eq:bulkvnt*} and then by Lemma~\ref{lem:Xnmgf} and since $\kappa \nu =\alpha s_0$ and $\ell \le -K$, for $z\le \nu  t^*$,
\begin{align*}
v^n_{t^*}(z)&\le e^{cs_0 t^*} e^{-\frac 12 \alpha \kappa  (z-\ell)} \Esubb{0}{e^{\frac 12 \alpha \kappa  X^n_{mt^*}}}+2s_0 t^* e^{13s_0 t^*} e^{3\kappa (\ell +\frac 12 K)}\\
&\le e^{-\frac 12 \alpha \kappa  ((z-\nu  t^*)-\ell)} e^{(c-\frac 14 \alpha^2+\mathcal O(n^{-1}))s_0 t^*}+2s_0 t^* e^{13s_0 t^*} e^{\kappa \ell} e^{-\kappa  K/2}\\
&\le \tfrac 15 c_1 e^{-\frac 12 \alpha \kappa  ((z-\nu  t^*)-\ell +1)},
\end{align*}
where the last line follows
by our choice of $t^*$ and $K$ and since $z\le \nu t^*$.
Hence for $z\le \nu  t^*$,
since $p^n_{t^*}(z)\ge 1/2-c/4>2/5$, we have that~\eqref{eq:lembulktail} holds.
For $z\in [\nu  t^*,\nu  t^*+ D^+_n]$,
by~\eqref{eq:veasybound} and then by Markov's inequality and Lemma~\ref{lem:Xnmgf},
for $n$ sufficiently large,
\begin{align*}
v^n_{t^*}(z)\le e^{(1+\alpha)s_0 t^*} \psubb{z}{X^n_{mt^*}\le \ell}\le e^{(1+\alpha)s_0 t^*}e^{-2\kappa (z-\ell)}e^{5 s_0 t^*}&\le e^{7 s_0 t^*}e^{-\kappa K}e^{-\kappa z}e^{-\kappa (z-\ell)}\\
&\le \tfrac 1 {24}c_1 e^{-\kappa z} e^{-\frac 12 \alpha \kappa  ((z-\nu  t^*)-\ell +1)}
\end{align*}
by our choice of $K$.
The result follows since $p^n_{t^*}(z)\ge \frac 1 {12}e^{-\kappa (z-\nu  t^*)}\ge \frac 1 {12}e^{-\kappa z}$.
\end{proof}

For $t\ge 0$ and $x_1\in \frac 1n \Z$, let $(v^{n,+}_{t,t+s}(x_1,\cdot))_{s\ge 0}$ denote the solution of
\begin{equation*}
\begin{cases}
\partial_s v^{n,+}_{t,t+s}(x_1,\cdot)&=\tfrac 12 m\Delta_n v^{n,+}_{t,t+s}(x_1,\cdot)+s_0 v^{n,+}_{t,t+s}(x_1,\cdot)(1-u^n_{t,t+s})(2u^n_{t,t+s}-1+\alpha ) \quad \text{for }s>0,\\
v^{n,+}_{t,t}(x_1,x)&= p^n_t(x) \1_{x \ge x_1},
\end{cases}
\end{equation*}
where $(u^n_{t,t+s})_{s\ge 0}$ is defined in~\eqref{eq:unttsdef}. 
Similarly, let $(v^{n,-}_{t,t+s}(x_1,\cdot))_{s\ge 0}$ denote the solution of
\begin{equation*}
\begin{cases}
\partial_s v^{n,-}_{t,t+s}(x_1,\cdot)&=\tfrac 12 m\Delta_n v^{n,-}_{t,t+s}(x_1,\cdot)+s_0 v^{n,-}_{t,t+s}(x_1,\cdot)(1-u^n_{t,t+s})(2u^n_{t,t+s}-1+\alpha ) \quad \text{for }s>0,\\
v^{n,-}_{t,t}(x_1,x)&= p^n_t(x) \1_{x \le x_1}.
\end{cases}
\end{equation*}
We now use Lemmas~\ref{lem:vfromtipbound} and~\ref{lem:bulktail1} to prove the following result.
\begin{lemma} \label{lem:A2A3}
For $t^*\in \N$ sufficiently large, and $K\in \N$ sufficiently large (depending on $t^*$), for
$\ell \in \N$, the following holds for $n$ sufficiently large.
For $t\in [(\log N)^2-t^*,N^2]$ and $x_1,x_2 \in \frac 1n \Z$ with $x_1-x_2 \le (\log N)^{2/3}$,
\begin{equation} \label{eq:lemA2A31}
\p{A^{(2)}_{t}(x_1,x_2)^c \cap \{ x_1-\mu^n_t \in [K,D^+_n], x_2-\mu^n_{t+t^*}\le D^+_n \} \cap E'_1}
\le \left( \frac n N \right)^\ell.
\end{equation}
For $t\in [(\log N)^2-t^*,N^2]$ and $x_1,x_2 \in \frac 1n \Z$ with $x_2-x_1 \le (\log N)^{2/3}$,
\begin{equation} \label{eq:lemA2A32}
\p{A^{(3)}_{t}(x_1,x_2)^c \cap \{ x_1-\mu^n_t \le -K \} \cap E'_1}
\le \left( \frac n N \right)^\ell.
\end{equation}
\end{lemma}
\begin{proof}
Take $t^*,K\in \N$ sufficiently large that Lemmas~\ref{lem:vfromtipbound} and~\ref{lem:bulktail1} hold.
Recall the definition of $E'_1$ in~\eqref{eq:E1'defn}.
Suppose $n$ is sufficiently large that $(\log N)^2-t^*\ge \frac 12 (\log N)^2 \vee \log N$, and
 $E'_1$ occurs.
 Take $t\in [(\log N)^2 -t^*,N^2]$ and $x_1,x_2\in \frac 1n \Z$ with $x_1-x_2\le (\log N)^{2/3}$.
 Recall from~\eqref{eq:Dn+-defn} that $D^+_n =(1/2-c_0)\kappa^{-1} \log (N/n)$.
Take $c_3\in (0,c_0)$ and
suppose $|q^{n,+}_{t,t+t^*}(x_1,x_2)-v^{n,+}_{t,t+t^*}(x_1,x_2)|\le \left( \frac n N \right)^{1/2-c_3}$.
Then for $n$ sufficiently large, by Lemma~\ref{lem:vfromtipbound} and by the definition of the event $E_1$ in~\eqref{eq:eventE1},
if 
 $x_1-\mu^n_t \in [K, D^+_n]$ and $x_2 -\mu^n_{t+t^*} \le D^+_n$,
\begin{align*}
\frac{q^{n,+}_{t,t+t^*}(x_1,x_2)}{p^n_{t+t^*}(x_2)}
&\le \tfrac 12 c_1 e^{-(1+\frac 12 (1-\alpha))\kappa (x_1 -(x_2-\nu t^*)\vee (\mu^n_t+K) +2)}+5g(D^+_n)^{-1} \left( \frac n N \right)^{1/2-c_3}\\
&\le c_1 e^{-(1+\frac 12 (1-\alpha))\kappa (x_1 -(x_2-\nu  t^*)\vee (\mu^n_t+K) +2)}
\end{align*}
for $n$ sufficiently large, since $x_1-x_2 \le (\log N)^{2/3}$
and $g(D^+_n)^{-1}\le 2 \left( \frac N n \right)^{1/2-c_0}$ with $c_0>c_3$.
By Proposition~\ref{prop:pnun}, the first statement~\eqref{eq:lemA2A31} follows.

Now take $t\in [(\log N)^2 -t^*,N^2]$ and $x_1,x_2 \in \frac 1n \Z$ with $x_2-x_1 \le (\log N)^{2/3}$.
Suppose $E'_1$ occurs and
suppose $|q^{n,-}_{t,t+t^*}(x_1,x_2)-v^{n,-}_{t,t+t^*}(x_1,x_2)|\le \left( \frac n N \right)^{1/4}$.
If $x_1-\mu^n_t\le -K$, then $x_2 -\mu^n_{t+t^*}\le (\log N)^{2/3}$ and so $p^n_{t+t^*}(x_2)^{-1}\le 10 e^{\kappa (\log N)^{2/3}}$.
Hence by Lemma~\ref{lem:bulktail1},
\begin{align*}
\frac{q^{n,-}_{t,t+t^*}(x_1,x_2)}{p^n_{t+t^*}(x_2)}
&\le \tfrac 12 c_1 e^{-\frac 12 \alpha \kappa  ( (x_2-\nu  t^*)-x_1+1)}+10 e^{\kappa (\log N)^{2/3}}\left( \frac n N \right)^{1/4}\\
&\le c_1 e^{-\frac 12 \alpha \kappa  ( (x_2-\nu  t^*)-x_1+1)}
\end{align*}
for $n$ sufficiently large.
By Proposition~\ref{prop:pnun}, the second statement~\eqref{eq:lemA2A32} follows, which completes the proof.
\end{proof}
We now show that $A^{(4)}_t(x)$ and $A^{(5)}_t(x)$ occur with high probability for suitable $x$ and $t$.
\begin{lemma} \label{lem:A4A5}
For $\ell \in \N$, the following holds for $n$ sufficiently large.
For $x\in \frac 1n \Z$ and $t\ge 0$,
\begin{equation} \label{eq:lemA4A52}
\p{A^{(5)}_t(x)^c}\le \left( \frac n N \right)^\ell.
\end{equation}
If there exists $a_2>3$ such that $N\ge n^{a_2}$ for $n$ sufficiently large, then
for $t\in [(\log N)^2-\epsilon_n ,N^2]$ and $x \in \frac 1n \Z$,
\begin{equation} \label{eq:lemA4A51}
\p{A^{(4)}_t(x)^c  \cap \{ x-\mu^n_t \le D^+_n \} \cap E'_1}
\le \left( \frac n N \right)^\ell.
\end{equation}
\end{lemma}
\begin{proof}
For $t\ge 0$ and $x_1,x_2\in \frac 1n \Z$, by Corollary~\ref{cor:qnMa} with $a=-(1+\alpha)s_0$,
\begin{align*}
\E{q^n_{t,t+\epsilon_n}(x_1,x_2)}\le e^{(1+\alpha)s_0 \epsilon_n} \psubb{x_2}{X^n_{m\epsilon_n}=x_1}
\le e^{(1+\alpha)s_0 \epsilon_n} e^{-(\log N)^{3/2}|x_1-x_2|} e^{m(\log N)^3 \epsilon_n}
\end{align*}
for $n$ sufficiently large,
by Markov's inequality and Lemma~\ref{lem:Xnmgf}.
Recall from~\eqref{eq:paramdefns} that $\epsilon_n\le (\log N)^{-2}$.
Therefore, for $n$ sufficiently large, for $x\in \frac 1n \Z$, by a union bound and then by Markov's inequality,
\begin{align*}
\p{A^{(5)}_t(x)^c}
&\le \sum_{x'\in \frac 1n \Z,|x-x'|\ge 1} \p{q^n_{t,t+\epsilon_n}(x',x)\ge N^{-1} }\\
&\le N e^{(1+\alpha)s_0 \epsilon_n}N^{m} \sum_{x' \in \frac 1n \Z, |x-x'|\ge 1} e^{-(\log N)^{3/2}|x-x'|},
\end{align*}
which completes the proof of~\eqref{eq:lemA4A52}.

From now on, assume there exists $a_2>3$ such that $N\ge n^{a_2}$ for $n$ sufficiently large.
Suppose $n$ is sufficiently large that $(\log N)^2-\epsilon_n \ge \frac 12 (\log N)^2 \vee \log N$, and take
$t\in [(\log N)^2-\epsilon_n, N^2]$ and $x_1,x_2\in \frac 1n \Z$ with $|x_1-x_2|\le 1$.
Recall the definition of $(v^n_{t,t+s}(x_1,\cdot))_{s\ge 0}$ in~\eqref{eq:vnx1defn}.
By~\eqref{eq:veasybound},
and then by Lemma~\ref{lem:lclt}, there exists a constant $K_7<\infty$ such that for $n$ sufficiently large,
$$
v^n_{t,t+\epsilon_n}(x_1,x_2)\le e^{(1+\alpha)s_0 \epsilon_n} p^n_t(x_1) \psubb{x_2}{X^n_{m\epsilon_n}=x_1}
\le K_7 n^{-1} \epsilon_n^{-1/2} p^n_t(x_1).
$$
Suppose $E'_1$ occurs and $x_1 \le \mu^n_t+D^+_n$.
Then for $n$ sufficiently large, by the definition of the event $E_1$ in~\eqref{eq:eventE1} and
 since $|x_1-x_2|\le 1$, there exists a constant $K_8<\infty$ such that $\frac{p^n_t(x_1)}{p^n_{t+\epsilon_n}(x_2)}\le K_8$, and so
\begin{equation} \label{eq:useK8}
\frac{v^n_{t,t+\epsilon_n}(x_1,x_2)}{p^n_{t+\epsilon_n}(x_2)}\le K_7 K_8 n^{-1} \epsilon_n^{-1/2}.
\end{equation}
Recall from~\eqref{eq:Dn+-defn} that $D_n^+=(1/2-c_0)\kappa^{-1} \log (N/n)$.
Take $c' \in (0,c_0/2)$ and suppose $|q^n_{t,t+\epsilon_n}(x_1,x_2)-v^n_{t,t+\epsilon_n}(x_1,x_2)|\le \left( \frac n N \right)^{1/2-c'}p^n_t(x_1)^{1/2}n^{-1/2}$. By~\eqref{eq:useK8} and then since $x_2\le \mu^n_{t}+D^+_n+1$ and by the definition of $K_8$,
\begin{align} \label{eq:lemA4A5A}
\frac{q^n_{t,t+\epsilon_n}(x_1,x_2)}{p^n_{t+\epsilon_n}(x_2)}
&\le K_7 K_8 n^{-1} \epsilon_n^{-1/2}+ p^n_{t+\epsilon_n}(x_2)^{-1/2} \left( \frac n N \right)^{1/2-c'} 
\left( \frac{p^n_{t}(x_1)}{p^n_{t+\epsilon_n}(x_2)}\right)^{1/2} n^{-1/2} \notag \\
&\le K_7 K_8 n^{-1} \epsilon_n^{-1/2}+10^{1/2} e^{\frac 12 \kappa  (D^+_n +2)} \left( \frac n N \right)^{1/2-c'} K_8^{1/2} n^{-1/2} \notag \\
&\le (K_7 K_8 +1) n^{-1} \epsilon_n^{-1/2}
\end{align}
for $n$ sufficiently large,
since $N\ge n^3$ and so $e^{\frac 12 \kappa  D_n^+} \left( \frac n N \right)^{1/2-c'} =\left( \frac n N \right)^{1/4+c_0/2-c'}\le n^{-1/2}$.
For $c\in (0,\frac 12 (a_2-2)^{-1} (a_2-3))$, we have $3/2-2c< a_2 (1/2-c)$ and so since $N\ge n^{a_2}$ we have
 $p^n_t(x_1)\ge \frac 1 {10} e^{-\kappa D^+_n}\ge \frac 1 {10}  \left( \frac n N \right)^{1/2}\ge \left( \frac {n^2} N \right)^{1-c}$ for $n$ sufficiently large. Hence by Lemma~\ref{lem:qnvnonepoint}, for $n$ sufficiently large,
\begin{align*}
&\p{\{|q^n_{t,t+\epsilon_n}(x_1,x_2)-v^n_{t,t+\epsilon_n}(x_1,x_2)|\ge \left( \frac n N \right)^{1/2-c'}p^n_t(x_1)^{1/2}n^{-1/2} \} \cap \{x_1 \le \mu^n_t +D_n^+\} \cap E'_1}\\
&\qquad  \le \left( \frac n N \right)^{\ell+1} ,
\end{align*}
and by~\eqref{eq:lemA4A5A},
it follows that for $n$ sufficiently large,
$$
\p{\{q^n_{t,t+\epsilon_n}(x_1,x_2)> n^{-1} \epsilon_n^{-1} p^n_{t+\epsilon_n}(x_2)\}\cap \{x_1-\mu^n_t \le D_n^+\} \cap E'_1}\le \left( \frac n N \right)^{\ell+1}.
$$
By the same argument as for the proof of~\eqref{eq:lemA4A52}, the second statement~\eqref{eq:lemA4A51} now follows.
\end{proof}
Finally we show that $A^{(6)}_t(x)$ occurs with high probability; the proof is similar to the first half of the proof of Lemma~\ref{lem:A4A5}.
\begin{lemma} \label{lem:eventA6}
For $\ell \in \N$ and $t^*\in \N$, the following holds for $n$ sufficiently large.
For $t\ge 0$ and $x \in \frac 1n \Z$,
$$
\p{ A^{(6)}_t(x)^c} \le \left( \frac n N \right)^\ell.
$$
\end{lemma}
\begin{proof}
By Corollary~\ref{cor:qnMa} with $a=-(1+\alpha)s_0 $, for $k\in [ t^* \delta_n^{-1}]$ and $x'\in \frac 1n \Z$,
\begin{align*}
\E{q^n_{t,t+k\delta_n}(x',x)}
&\le e^{(1+\alpha)s_0 t^*} \psubb{x}{X^n_{mk\delta_n}=x'}\\
&\le e^{(1+\alpha)s_0 t^*} e^{-(\log N)^{1/2}|x-x'|} \Esubb{0}{e^{X^n_{mk\delta_n}(\log N)^{1/2}}}\\
&\le e^{(1+\alpha)s_0 t^*} e^{-(\log N)^{1/2} |x-x'|} e^{mt^* \log N} 
\end{align*}
for $n$ sufficiently large,
where the second inequality follows by Markov's inequality, and the third by Lemma~\ref{lem:Xnmgf}.
Therefore, by a union bound and Markov's inequality,
\begin{align*}
&\p{\exists x' \in \tfrac 1n \Z, k\in[t^* \delta_n^{-1}]:  |x-x'|\ge (\log N)^{2/3}, q^n_{t,t+k\delta_n}(x',x)\ge N^{-1}}\\
&\le t^* \delta_n^{-1} \cdot N e^{(1+\alpha)s_0 t^*} N^{mt^*} \sum_{x' \in \frac 1n \Z, |x-x'|\ge (\log N)^{2/3}}e^{-(\log N)^{1/2}|x-x'|}\\
&\le \left( \frac n N \right)^\ell
\end{align*}
for $n$ sufficiently large.
\end{proof}
We can now end this section by proving Proposition~\ref{prop:eventE2}.
\begin{proof}[Proof of Proposition~\ref{prop:eventE2}]
Note that if $x_1-x_2>(\log N)^{2/3}$ and $A_t^{(6)}(x_2)$ occurs, then $A^{(2)}_t(x_1,x_2)$ occurs.
Similarly, if $x_2-x_1>(\log N)^{2/3}$ and $A_t^{(6)}(x_2)$ occurs, then $A^{(3)}_t(x_1,x_2)$ occurs.
The result now follows directly from Proposition~\ref{prop:eventA1} and Lemmas~\ref{lem:A2A3},~\ref{lem:A4A5} and~\ref{lem:eventA6}.
\end{proof}

\section{Event $E_3$ occurs with high probability} \label{sec:eventE3}
In this section, we will prove the following result.
\begin{prop} \label{prop:eventE3}
For $K\in \N$ sufficiently large, for $c_2>0$, if $N\ge n^{3}$ for $n$ sufficiently large, then for $n$ sufficiently large, if $p^n_0(x)=0$ $\forall x\ge N$,
$$
\p{(E_3)^c \cap E_1}\le \left( \frac n N \right)^2.
$$
\end{prop}
By the definition of the events $E_1$ and $E_3$ in~\eqref{eq:eventE1} and~\eqref{eq:eventE3},
Proposition~\ref{prop:eventE3} follows directly from the following result.
\begin{lemma} \label{lem:coalCB}
For $\ell \in \N$, for $K\in \N$ sufficiently large, for $c_2>0$,  if $N\ge n^{3}$ for $n$ sufficiently large then  the following holds for $n$ sufficiently large.
If $p^n_0(y)=0$ $\forall y\ge N$ then
 for $t\in [(\log N)^2-\delta_n,N^2]$, $x\in \frac 1n \Z$ with $x\ge -N^5$ and $j\in \{1,2,3,4\}$,
\begin{equation} \label{eq:Bjbound}
\p{B^{(j)}_t(x) ^c \cap E_1 \cap \{x\le \mu^n_t+D^+_n+1\}}\le \left( \frac n N \right)^\ell .
\end{equation}
\end{lemma}
 \begin{proof}
 We begin by proving~\eqref{eq:Bjbound} with $j=1$.
For $x\in \frac 1n \Z$, $i \in [N]$ and $0\le t_1<t_2$, let
$\mathcal A^{x,i}[t_1,t_2)$ denote the total number of points in the time interval $[t_1,t_2)$ in the Poisson processes $(\mathcal P^{x,i,i'})_{i'\in [N]\setminus\{ i\}}$, $(\mathcal S^{x,i,i'})_{i'\in [N]\setminus\{ i\}}$, $(\mathcal Q^{x,i,i',i''})_{i', i''\in [N]\setminus \{i\},i'\neq i''}$ and $(\mathcal R^{x,i,y,i'})_{i'\in [N], y\in \{x\pm n^{-1}\}}$.
(These points correspond to the times at which the individual $(x,i)$ may be replaced by offspring of another individual.)
For $t\ge 0$ and $x\in \frac 1n \Z$, let
\begin{align*}
\mathcal C^{n,1}_t(x)
&=\{(i,j):i\neq j \in [N],
 \mathcal P^{x,i,j}[t,t+\delta_n)=1=\mathcal A^{x,i}[t,t+\delta_n), \, \mathcal A^{x,j}[t,t+\delta_n)=0,\\
&\qquad \qquad \hspace{10cm} \xi^n_{t}(x,j)=1\}.
\end{align*}
Recall the definition of $\mathcal C^n_t(x,x)$ in~\eqref{eq:Cntdefn}.
If $(i,j)\in \mathcal C^{n,1}_t(x)$, we have $(\zeta^{n,t+\delta_n}_{\delta_n}(x,i),\theta^{n,t+\delta_n}_{\delta_n}(x,i))=(x,j)=(\zeta^{n,t+\delta_n}_{\delta_n}(x,j),\theta^{n,t+\delta_n}_{\delta_n}(x,j))$, and so $(i,j),(j,i)\in \mathcal C^n_t(x,x)$.
Therefore,
since if $(i,j) \in \mathcal C^{n,1}_t(x)$ then $(j,i) \notin \mathcal C^{n,1}_t(x)$,
\begin{equation} \label{eq:lemcoalCB*}
|\mathcal C^{n}_t(x,x)|\ge 2 |\mathcal C^{n,1}_t(x)|.
\end{equation}
For $t\ge 0$, $x\in \frac 1n \Z$ and $i\in [N]$, let
\begin{equation} \label{eq:Dntxidef}
\mathcal D^n_t(x,i)
=\{(y,j)\in \tfrac 1n \Z \times [N] :(\zeta^{n,t+s}_s(y,j),\theta^{n,t+s}_s(y,j))=(x,i)\text{ for some }s\in [0,\delta_n]\},
\end{equation}
the set of labels of individuals whose time-$t$ ancestor at some time in $[t,t+\delta_n]$ is $(x,i)$.
Define 
\begin{equation} \label{eq:Mntdefn}
\mathcal M^n_t=\max_{x\in \frac 1n \Z \cap [-2N^5, N^5],\,  i\in [N]}|\mathcal D^n_t(x,i)|.
\end{equation} 
For $t\ge 0$ and $x\in \frac 1n \Z$, let
\begin{align} \label{eq:Cn2tdefn}
\mathcal C^{n,2}_t(x)
&=\Big\{(i,j) \in [N]^2 :
\Big( \mathcal P^{x,i,j}+\mathcal S^{x,i,j}+\sum_{k\in [N]\setminus \{ i,j\} } \mathcal Q^{x,i,j,k}\Big) [t,t+\delta_n)>0, \, \xi^n_{t}(x,j)=1\Big\}.
\end{align}
Suppose $(i,j)\in \mathcal C^n_t(x,x)$, and $(i,j),(j,i)\notin \mathcal C^{n,2}_t(x)$.
Then there exist $s\in [0,\delta_n]$, $(y,k)\notin\{ (x,i),(x,j)\}$ and $i'\in \{i,j\}$ such that 
$(\zeta^{n,t+\delta_n}_s(x,i'),\theta^{n,t+\delta_n}_s(x,i'))=(y,k)$.
Then letting
$(x_0,i_0)=(\zeta^{n,t+\delta_n}_{\delta_n}(x,i),\theta^{n,t+\delta_n}_{\delta_n}(x,i))$, we have
$(x,i)$, $(x,j)$, $(y,k)\in \mathcal D^n_t(x_0,i_0)$. 
Since $\zeta^{n,t+\delta_n}(x,i)$ only jumps in increments of $\pm n^{-1}$, and $(\zeta^{n,t+\delta_n}_s(x,i), \theta^{n,t+\delta_n}_s(x,i))\in \mathcal D^n_t(x_0,i_0)$ $\forall s\in [0,\delta_n]$, we have
$|x-x_0|<|\mathcal D^n_t(x_0,i_0)|n^{-1}. $ Hence if $x_0 \in [-2N^5,N^5]$ then $|x-x_0|< \mathcal M^n_t n^{-1}$.
Therefore, by the definition of $q^{n,-}$ in~\eqref{eq:qn+-defn}, if $q^{n,-}_{t,t+\delta_n}(-2N^5,x)=0$ and $p^n_t(y)=0$ $\forall y\ge N^5$, then
\begin{equation} \label{eq:coalA}
|\mathcal C^n_t(x,x)|
\le 2 |\mathcal C^{n,2}_t(x)|+2{\mathcal M^n_t \choose 2} |\{(x_0,i_0)\in \tfrac 1n \Z \times [N] :|x-x_0|< \mathcal  M^n_t n^{-1}, |\mathcal D^n_t (x_0,i_0)|\ge 3\}| .
\end{equation}
We now use the inequalities~\eqref{eq:lemcoalCB*} and~\eqref{eq:coalA} to give lower and upper bounds on $|\mathcal C^n_t(x,x)|$. 

We begin with a lower bound.
For $x\in \frac 1n \Z$, $i\in [N]$ and $0\le t_1<t_2$, let
$\mathcal A^{1,x,i}[t_1,t_2)$ denote the total number of points in the time interval $[t_1,t_2)$ in the Poisson processes $(\mathcal P^{x,i,j})_{j\in [N]\setminus\{i\}, \xi^n_{t_1}(x,j)=1}$.
Let $\mathcal A^{2,x,i}[t_1,t_2)$ denote the total number of points in the time interval $[t_1,t_2)$ in the Poisson processes $(\mathcal P^{x,i,j})_{j\in [N]\setminus\{i\}, \mathcal A^{x,j}[t_1,t_2)>0}$.
Now fix $t\ge 0$ and $x\in \frac 1n \Z$ and let
\begin{align*}
A^{(1)}&=  |\{i\in [N]:\xi^n_t(x,i)=1,\mathcal A^{x,i}[t,t+\delta_n)=1=\mathcal A^{1,x,i}[t,t+\delta_n)\}|,\\
A^{(2)}&=  |\{i\in [N]:\xi^n_t(x,i)=0,\mathcal A^{x,i}[t,t+\delta_n)=1=\mathcal A^{1,x,i}[t,t+\delta_n)\}|,\\
\text{ and } \qquad B &=|\{i\in [N]:\mathcal A^{x,i}[t,t+\delta_n)=1= \mathcal A^{2,x,i}[t,t+\delta_n)\}|.
\end{align*}
Then by~\eqref{eq:lemcoalCB*} and the definition of $\mathcal C^{n,1}_t(x)$,
\begin{equation} \label{eq:coalA12B1*}
|\mathcal C^{n}_t(x,x)| \ge 2 |\mathcal C^{n,1}_t(x)|\ge 2( A^{(1)}+A^{(2)}-B).
\end{equation}
Let $(X^n_j)_{j=1}^\infty$ be i.i.d.~with $X^n_1\sim \text{Poisson }(r_n \delta_n(1-(\alpha +1)s_n))$,
let $(Y^n_j)_{j=1}^\infty$ be i.i.d.~with $Y^n_1\sim \text{Poisson }(r_n \delta_n (\alpha s_n+N^{-1}s_n (N-2)))$,
and let $(Z^n_j)_{j=1}^\infty$ be i.i.d.~with $Z^n_1\sim \text{Poisson }(m r_n\delta_n)$.
Recall from~\eqref{eq:snrndefn} that $r_n = \frac 12 n^2 N^{-1}$ and $s_n=2s_0 n^{-2}$.
Then conditional on $p^n_t(x)$, 
$A^{(1)}\sim \text{Bin}(Np^n_t(x),p_1)$ and $A^{(2)}\sim \text{Bin}(N(1-p^n_t(x)),p_2)$, where 
\begin{align*}
p_1&=
\p{\sum_{j=1}^{Np^n_t(x)-1} X^n_j =1, \sum_{j=Np^n_t(x)}^{N-1} X^n_j +\sum_{j=1}^{N-1}Y^n_j +\sum_{j=1}^{2N} Z^n_j=0}\\
&=\Big(\tfrac 12 n^2 \delta_n(p^n_t(x)-N^{-1})(1+\mathcal O(n^{-2}))+\mathcal O((n^2 \delta_n(p^n_t(x)-N^{-1}))^2)\Big)
\big(1-\mathcal O(n^2 \delta_n)\big)\\
&=\tfrac 12 n^2 \delta_n (p^n_t(x)-N^{-1})(1+\mathcal O(n^{-2}+n^2 \delta_n))
\end{align*}
and
\begin{align*}
p_2 &=
\p{\sum_{j=1}^{Np^n_t(x)} X^n_j=1, \sum_{j=Np^n_t(x)+1}^{N-1} X^n_j +\sum_{j=1}^{N-1}Y^n_j +\sum_{j=1}^{2N} Z^n_j=0}\\
&=\tfrac 12 n^2 \delta_n p^n_t(x)(1+\mathcal O(n^{-2}+n^2 \delta_n)).
\end{align*}
Hence 
$$
\E{A^{(1)}+A^{(2)} \Big| p^n_t(x)}
=\tfrac 12 Nn^2 \delta_n p^n_t(x)(1+\mathcal O(n^{-2}+n^2 \delta_n +N^{-1} p^n_t(x)^{-1})).
$$
Recall from~\eqref{eq:paramdefns} that $\delta_n = \lfloor N^{1/2} n^2 \rfloor ^{-1}$.
Suppose $n$ is sufficiently large that $(\log N)^2 -\delta_n \ge \frac 12 (\log N)^2$. Then on
 the event $E_1$, for
$t\in [(\log N)^2 -\delta_n ,N^2]$ and $x\le \mu^n_t + D^+_n+1$, 
by~\eqref{eq:eventE1} and~\eqref{eq:Dn+-defn} we have 
$N^{-1} p^n_t(x)^{-1}\le 10N^{-1} e^{\kappa(D^+_n+1)}\le 10 e^\kappa N^{-1/2}n^{-1/2}$ and
\begin{equation} \label{eq:n12lower}
Nn^2 \delta_n p^n_t(x)\ge \tfrac 15 N^{1/2} g(x-\mu^n_t) \ge \tfrac 1 {10}
N^{1/2}e^{-\kappa  (D^+_n+1)}\ge 2n^{1/2}
\end{equation}
for $n$ sufficiently large.
Hence for $n$ sufficiently large, for $t\in [(\log N)^2-\delta_n,N^2]$ and $x\in \frac 1n \Z$,
by conditioning on $p^n_t(x)$ and then applying Theorem~2.3(c) in~\cite{mcdiarmid:1998},
\begin{align} \label{eq:coallemma1}
\p{\left\{A^{(1)}+A^{(2)} \le \tfrac 12 N n^2 \delta_n p^n_t(x)(1-n^{-1/5})\right\} \cap \{x\le \mu^n_t+D^+_n+1\}\cap E_1}
&\le e^{-\frac 13 n^{-2/5}n^{1/2}} \notag \\
&=e^{-\frac 13 n^{1/10}}.
\end{align}
For an upper bound on $B$, first let
$$
A'=|\{i\in [N]:\mathcal A^{x,i}[t,t+\delta_n)>0\}|.
$$
Then $A'\sim \text{Bin}(N,p)$
where
\begin{align*}
p&= \p{\sum_{j=1}^{N-1}(X^n_j +Y^n_j)+\sum_{j=1}^{2N} Z^n_j >0}
=\tfrac 12 n^2 \delta_n (1+2m)(1+\mathcal O(n^2 \delta_n+n^{-2})).
\end{align*}
Conditional on $A'$, we have $B\le \text{Bin}(A',\frac{A'-1}{(1+2m)N-1})$.
By Theorem~2.3(b) in~\cite{mcdiarmid:1998}, for $n$ sufficiently large,
\begin{equation} \label{eq:coal1}
\p{A'\ge N n^2 \delta_n (1+2m)}\le e^{-\frac 18 N n^2 \delta_n(1+2m)}.
\end{equation}
Moreover, since $\delta_n =\lfloor N^{1/2}n^2 \rfloor ^{-1}$,
letting $B' \sim \text{Bin}(\lfloor 2N^{1/2}(1+2m) \rfloor, 2N^{-1/2})$,
 for $n$ sufficiently large,
\begin{align} \label{eq:coal2}
\p{B\ge n^{1/4} , A'\le N n^2 \delta_n (1+2m) }\le \p{B' \ge n^{1/4}}
&\le e^{-n^{1/4}}(1+(e-1)2N^{-1/2})^{\lfloor 2N^{1/2}(1+2m)\rfloor } \notag \\
&\le e^{-\frac 12 n^{1/4}},
\end{align}
where the second inequality follows by Markov's inequality.
Therefore,
by~\eqref{eq:coalA12B1*},~\eqref{eq:n12lower},~\eqref{eq:coallemma1},~\eqref{eq:coal1} and~\eqref{eq:coal2},
 for $n$ sufficiently large, for $t\in [(\log N)^2-\delta_n ,N^2]$ and $x\in \frac 1n \Z$,
\begin{align} \label{eq:coallemma2}
&\p{
\left\{|\mathcal C^n_t(x,x)|\le  N  n^2 \delta_n p^n_t(x)(1-2n^{-1/5})\right\}
\cap \{x\le \mu^n_t+D^+_n+1\}\cap E_1} \notag \\
&\quad \le e^{-\frac 13 n^{1/10}}+e^{-\frac 18 N^{1/2}}+ e^{-\frac 12 n^{1/4}}.
\end{align}

For an upper bound on $|\mathcal C^n_t(x,x)|$, note that by the definition of $\mathcal C^{n,2}_t(x)$ in~\eqref{eq:Cn2tdefn}, conditional on $p^n_t(x)$,
$$
|\mathcal C^{n,2}_t(x)|\sim \text{Bin}(Np^n_t(x)(N-1),p'),
$$
where
\begin{align*}
p' &=\P\bigg(\Big(\mathcal P^{x,1,2}+\mathcal S^{x,1,2}+\sum_{k\in [N]\setminus \{1,2\}} \mathcal Q^{x,1,2,k}\Big)[0,\delta_n)>0\bigg)\\
&=r_n \delta_n(1+\mathcal O(r_n \delta_n + n^{-2}N^{-1})).
\end{align*}
Then $Np^n_t(x)(N-1)p'=\frac 12 N n^2 \delta_n p^n_t(x)(1+\mathcal O(n^2 N^{-1}\delta_n+N^{-1}))$.
Hence for $n$ sufficiently large, for $t\in [(\log N)^2-\delta_n ,N^2]$ and $x\in \frac 1n \Z$,
by Theorem~2.3(b) in~\cite{mcdiarmid:1998} and~\eqref{eq:n12lower}, 
\begin{align} \label{eq:coalupper1}
&\p{\left\{|\mathcal C^{n,2}_t(x)|\ge \tfrac 12 N  n^2 \delta_n p^n_t(x) (1+n^{-1/5}) \right\} \cap \{x\le \mu^n_t+D^+_n+1\} \cap E_1} \notag\\
&\quad \le e^{-\frac 13 n^{-2/5}\cdot n^{1/2}}
= e^{-\frac 13 n^{1/10}}.
\end{align}
We now bound the second term on the right hand side of~\eqref{eq:coalA}.
For $x\in \frac 1n \Z$, $i\in [N]$ and $0\le t_1 < t_2$,
let $\mathcal B^{x,i}[t_1,t_2)$ denote the total number of points in the time interval $[t_1,t_2)$ in the Poisson processes 
$(\mathcal P^{x,i',i})_{i'\in [N]\setminus \{i\}}$, $(\mathcal S^{x,i',i})_{i'\in [N]\setminus \{i\}}$, $(\mathcal Q^{x,i',i,i''})_{i', i''\in [N]\setminus \{i\},i'\neq i''}$ and $(\mathcal R^{y,i',x,i})_{i'\in [N], y\in \{x\pm n^{-1}\}}$.
(These points correspond to the times at which offspring of the individual $(x,i)$ may replace another individual.)
Let $\mathcal B^{1,x,i}[t_1,t_2)$ denote the total number of points in the time interval $[t_1,t_2)$
in $(\mathcal P^{x,i',i})_{i'\in [N]\setminus \{i\}, \mathcal B^{x,i'}[t_1,t_2)>0}$,
$(\mathcal S^{x,i',i})_{i'\in [N]\setminus \{i\}, \mathcal B^{x,i'}[t_1,t_2)>0}$, $(\mathcal Q^{x,i',i,i''})_{i',i''\in [N]\setminus \{i\}, i''\neq i', \mathcal B^{x,i'}[t_1,t_2)>0}$ and $(\mathcal R^{y,i',x,i})_{i'\in [N], y\in \{x\pm n^{-1}\}, \mathcal B^{y,i'}[t_1,t_2)>0}$.
Then fix $x\in \frac 1n \Z$ and $t\ge 0$, and let
\begin{align*}
C^{(1)}&= |\{i\in [N]:\mathcal B^{x,i}[t,t+\delta_n)\ge 2\}|\\
\text{and }\quad C^{(2)}&= |\{i\in [N]:\mathcal B^{x,i}[t,t+\delta_n)=1=\mathcal B^{1,x,i}[t,t+\delta_n)\}|.
\end{align*}
By the definition of $\mathcal D^n_t(x,i)$ in~\eqref{eq:Dntxidef}, we have that
\begin{equation} \label{eq:DC1C2bound}
|\{i\in [N]: |\mathcal D^n_t(x,i)|\ge 3\}| \le C^{(1)}+C^{(2)}.
\end{equation}
Then
$C^{(1)} \sim \text{Bin}(N, p'')$, where
\begin{align*}
p''=\p{\mathcal B^{x,1}[t,t+\delta_n)\ge 2}
&\le (r_n \delta_n N (1+2m))^2 
=\tfrac 14 n^4 \delta_n^2 (1+2m)^2.
\end{align*}
Therefore, by Markov's inequality and since
 $n^4 \delta_n^2 \le 2N^{-1}$ for $n$ sufficiently large,
$$
\p{C^{(1)} \ge n^{1/4}} \le e^{-n^{1/4}}(1+(e-1)\tfrac 14 n^4 \delta_n^2 (1+2m)^2)^N
\le e^{-\frac 12 n^{1/4}}
$$
for $n$ sufficiently large.
For $y \in \frac 1n \Z$, let $D_y =|\{i\in [N]:\mathcal B^{y,i}[t,t+\delta_n)>0\}|$.
Then conditional on $D_x$, $D_{x-n^{-1}}$ and $D_{x+n^{-1}}$ we have
$C^{(2)} \le \text{Bin}(D_x,\frac{(D_x-1)(1-2N^{-1}s_n)+m(D_{x-n^{-1}}+D_{x+n^{-1}})}{ (1-2N^{-1}s_n)(N-1)+2mN})$.
By the same argument as in~\eqref{eq:coal1} and~\eqref{eq:coal2}, it follows that
for $n$ sufficiently large,
$$
\p{C^{(2)} \ge n^{1/4}}\le 3e^{-\frac 18 N n^2 \delta_n(1+2m)}+e^{-\frac 12 n^{1/4}}.
$$
Therefore, by~\eqref{eq:DC1C2bound}, for $n$ sufficiently large, for $x\in \frac 1n \Z$ and $t\ge 0$,
\begin{align} \label{eq:Dmorethan3bound}
\p{|\{i \in [N]: |\mathcal D^n_t(x,i)|\ge 3\}|\ge 2n^{1/4}
}
\le 3e^{-\frac 18 N n^2 \delta_n(1+2m)}+2e^{-\frac 12 n^{1/4}}.
\end{align}
For $K\in \N$, let $S^K_n \sim \text{Poisson}((2m+1)N r_n  (K-1) \delta_n)$. Then
since a set of $k$ individuals produces offspring individuals at total rate at most $(2m+1)N r_n  k$, for $i\in [N]$,
\begin{align*}
\p{|\mathcal D^n_t(x,i)|\ge K}
\le \p{S^K_n\ge K-1}
&\le ((2m+1)N r_n  (K-1)\delta_n )^{K-1}\\
&\le ((2m+1) (K-1))^{K-1}N^{- (K-1)/2}
\end{align*}
for $n$ sufficiently large.
Therefore, by the definition of $\mathcal M^n_t$ in~\eqref{eq:Mntdefn}, for $\ell \in \N$, for $K\in \N$ sufficiently large that $7-\frac 12 (K-1)<-\ell$, for $t\ge 0$,
\begin{align} \label{eq:Dmaxbound}
\p{\mathcal M^n_t\ge K}
\le \sum_{x\in \frac 1n \Z \cap [-2N^5,N^5], i\in [N]} \p{|\mathcal D^n_t(x,i)|\ge K}
\le \tfrac 13 \left( \frac n N \right) ^\ell
\end{align}
for $n$ sufficiently large.
For $x\ge -N^5$ and $t\ge 0$, by Corollary~\ref{cor:qnMa} with $a=-(1+\alpha)s_0$,
and then by Markov's inequality,
\begin{align} \label{eq:coallemmanew2}
\E{q^{n,-}_{t,t+\delta_n}(-2N^5,x)}
\le e^{(1+\alpha)s_0 \delta_n}\langle \1_{\cdot \le -2N^5},\phi^{\delta_n,x}_0 \rangle_n
&\le e^{(1+\alpha)s_0 \delta_n}\Esubb{0}{e^{X^n_{m\delta_n}}}e^{-N^5} \notag \\
&\le e^{1-N^5}
\end{align}
for $n$ sufficiently large, by Lemma~\ref{lem:Xnmgf}.
By Lemma~\ref{lem:p01}, for $t\le N^2$,
$\p{p^n_t(y)=0 \, \forall y\ge N^5}\ge 1-e^{-N^5}$.
By~\eqref{eq:coalA},~\eqref{eq:n12lower},~\eqref{eq:coalupper1},~\eqref{eq:Dmorethan3bound} and~\eqref{eq:Dmaxbound}, it now follows that for $\ell \in \N$, for $n$ sufficiently large, for $x\in \frac 1n \Z$ with $x\ge -N^5$ and $t\in [(\log N)^2 -\delta_n ,N^2]$,
\begin{equation} \label{eq:Cntxxupper}
\p{
\left\{|\mathcal C^n_t(x,x)|\ge  N  n^2 \delta_n p^n_t(x)(1+2n^{-1/5})\right\}
\cap \{x\le \mu^n_t+D^+_n+1\}\cap E_1}
\le \tfrac 12 \left( \frac n N \right)^\ell.
\end{equation}
By~\eqref{eq:coallemma2}, we now have that~\eqref{eq:Bjbound} holds with  $j=1$.

For $t\ge 0$ and $x,y\in \frac 1n \Z$ with $|x-y|=n^{-1}$, let
\begin{align*}
\mathcal C^{n,1}_t(x,y) 
&= \{(i,j) \in [N]^2 : \mathcal R^{x,i,y,j}[t,t+\delta_n)=1=\mathcal A^{x,i}[t,t+\delta_n),
\mathcal A^{y,j}[t,t+\delta_n)=0, \xi^n_t(y,j)=1\},\\
\mathcal C^{n,2}_t (x,y)
&= \{(i,j) \in [N]^2 :\mathcal R^{x,i,y,j}[t,t+\delta_n)>0, \xi^n_t(y,j)=1\}.
\end{align*}
Then $|\mathcal C^{n}_t(x,x+n^{-1})|\ge |\mathcal C^{n,1}_t(x,x+n^{-1})|+ |\mathcal C^{n,1}_t(x+n^{-1},x)|$.
If $q^{n,-}_{t,t+\delta_n}(-2N^5,x)=0$ and $p^n_t(y)=0$ $\forall y\ge N^5$, then
by the same argument as for~\eqref{eq:coalA},
\begin{align*}
|\mathcal C^{n}_t(x,x+n^{-1})|
&\le |\mathcal C^{n,2}_t(x,x+n^{-1})|+|\mathcal C^{n,2}_t(x+n^{-1},x)|\\
&\quad +{\mathcal M^n_t \choose 2} |\{(x_0,i_0)\in \tfrac 1n \Z \times [N]:|x-x_0|< \mathcal  M^n_t n^{-1}, |\mathcal D^n_t (x_0,i_0)|\ge 3\}|.
\end{align*}
By the same argument as for~\eqref{eq:coallemma2} and~\eqref{eq:Cntxxupper}, it follows that for $n$ sufficiently large, for $x\in \frac 1n \Z$ with $x\ge -N^5$ and $t\in [(\log N)^2-\delta_n,N^2]$,~\eqref{eq:Bjbound} holds with  $j=2$.

Suppose for some $k>1$ that $x,y\in \frac 1n \Z$ with $x\ge -N^5$ and
$|x-y|=kn^{-1}$.
Take $(i,j)\in \mathcal C^n_t(x,y)$, 
and let $(x_0,i_0)=(\zeta^{n,t+\delta_n}_{\delta_n} (x,i),\theta^{n,t+\delta_n}_{\delta_n} (x,i))$.
Since $(\zeta^{n,t+\delta_n}_s(x,i),\theta^{n,t+\delta_n}_s(x,i))\in \mathcal D^n_t(x_0,i_0)$ and $(\zeta^{n,t+\delta_n}_s(y,j),\theta^{n,t+\delta_n}_s(y,j))\in \mathcal D^n_t(x_0,i_0)$  $\forall s\in [0,\delta_n]$,
we have $(x,i),(y,j)\in \mathcal D^n_t(x_0,i_0)$ and
$|\mathcal D^n_t(x_0,i_0)|\ge \max (k,n|x_0-x|)+1\ge 3.$
If
$p^n_t(y)=0$ $\forall y\ge N^5$
 and $q^{n,-}_{t,t+\delta_n}(-2N^5,x)=0$, then by~\eqref{eq:Mntdefn}
it follows that $k<\mathcal M^n_t$ and $|x_0-x|<\mathcal M^n_t n^{-1}$.
Therefore
$$
|\mathcal C^n_t(x,y)|
\le  \1_{|x-y|<\mathcal M^n_t n^{-1}}{\mathcal M^n_t \choose 2} |\{(x_0,i_0)\in \tfrac 1n \Z\times [N] : |x_0-x|< \mathcal M^n_t n^{-1}, |\mathcal D^n_t(x_0,i_0)|\ge 3\}|
.
$$
By Lemma~\ref{lem:p01},~\eqref{eq:coallemmanew2},~\eqref{eq:n12lower},~\eqref{eq:Dmorethan3bound} and~\eqref{eq:Dmaxbound},
it follows that for $K\in \N$ sufficiently large, for $n$ sufficiently large, for $x\ge -N^5$ and $t\in [(\log N)^2-\delta_n,N^2]$,~\eqref{eq:Bjbound} holds with  $j=3$.

Finally, suppose $x,y,y' \in \frac 1n \Z$ with $x \ge -N^5 $.
 Take $(i,j,j')\in \mathcal C^n_t(x,y,y')$, and let
 $(x_0,i_0)=(\zeta^{n,t+\delta_n}_{\delta_n}(x,i),\theta^{n,t+\delta_n}_{\delta_n}(x,i))$.
 Suppose that $p^n_t(y)=0$ $\forall y\ge N^5$
 and $q^{n,-}_{t,t+\delta_n}(-2N^5,x)=0$.
 Then $(x,i),(y,j),(y',j')\in \mathcal D^n_t(x_0,i_0)$, and moreover $|x-x_0|<\mathcal M^n_t n^{-1}$ and $|x-y|\vee |x-y'|<\mathcal M^n_t n^{-1}$.
 Therefore
\begin{align*}
& |\mathcal C^n(x,y,y')|\\
&\quad \le \1_{|x-y|\vee |x-y'|<\mathcal M^n_t n^{-1}} (\mathcal M^n_t)^3  |\{(x_0,i_0)\in \tfrac 1n \Z\times [N]: |x_0-x |< \mathcal M^n_t n^{-1}, |\mathcal D^n_t(x_0,i_0)|\ge 3\}|.
\end{align*}
By Lemma~\ref{lem:p01},~\eqref{eq:coallemmanew2},~\eqref{eq:n12lower},~\eqref{eq:Dmorethan3bound} and~\eqref{eq:Dmaxbound}, it follows that for $K\in \N$ sufficiently large, for $n$ sufficiently large, for $x\ge -N^5$ and $t\in [(\log N)^2 -\delta_n,N^2]$,~\eqref{eq:Bjbound} holds with $j=4$. This completes the proof.
\end{proof}

\section{Event $E_4$ occurs with high probability} \label{sec:eventE4}
In this section, we complete the proof of Proposition~\ref{prop:eventE} by proving the following result.
\begin{prop} \label{prop:eventE4}
Suppose for some $a_1>1$, $N\ge n^{a_1}$ for $n$ sufficiently large.
For $b_1>0$ sufficiently small, $b_2>0$ and $t^*\in \N$, for $K\in \N$ sufficiently large, 
then for $n$ sufficiently large, if condition~\eqref{eq:conditionA} holds,
$$
\p{(E_4)^c } \le \left( \frac n N \right)^2.
$$
\end{prop}
Proposition~\ref{prop:eventE} now follows directly from Propositions~\ref{prop:eventE1},~\ref{prop:eventE2},~\ref{prop:eventE3} and~\ref{prop:eventE4}.
From now on in this section, we assume that there exists $a_1>1$ such that $N\ge n^{a_1}$ for $n$ sufficiently large.
We begin by proving the following lemma, which we will then use iteratively to show that with high probability no lineages consistently stay far ahead of the front.
Fix $t^*\in \N$.
\begin{lemma} \label{lem:decayintip}
There exist $c\in (0,1)$ and $\epsilon \in (0,1)$ such that for $K\in \N$ sufficiently large, the following holds.
Suppose $q^n_0$ is random, and define the event
$$
A=\left\{ \sup_{t\in [0,t^*],\, x\in \frac 1n \Z}|p^n_t(x)-g(x-\mu^n_t)|\le \epsilon \right\}\cap \left\{\sup_{t\in [0,t^*]}\mu^n_t \le 2\nu t^*\right\}.
$$
Then
\begin{equation} \label{eq:decayintip}
\sup_{z\ge K}\E{q^n_{t^*}(z)}\le c \sup_{x\in \frac 1n \Z} \E{q^n_0(x)}+4s_0 t^* \p{A^c}.
\end{equation}
\end{lemma}

\begin{proof}
Let $\delta = \p{A^c}$.
For $a\in \R$, $t\ge 0$ and $z\in \frac 1n \Z$, by Lemma~\ref{lem:qnphi},
$(M^n_s(\phi^{t,z,a s_0}))_{s\ge 0}$ is a martingale with $M^n_0(\phi^{t,z,as_0 })=0$.
Hence by Corollary~\ref{cor:qnMa},
\begin{align} \label{eq:Eqn}
\E{q^n_t(z)}
&=e^{-as_0 t}\langle \E{q^n_0},\phi^{t,z}_0\rangle_n
+s_0 \int_0^t e^{-as_0 (t-s)}\langle \E{q^n_s((1-p^n_s)(2p^n_s-1+\alpha )+a)}, \phi^{t,z}_s\rangle_n ds.
\end{align}
Take $a\in (0,1-\alpha)$ and then take
$\epsilon \in (0,\frac 12 (1-\alpha))$ sufficiently small that $(1-\epsilon)(2\epsilon-1+\alpha )<-a$.
Take
 $K\in \N$ sufficiently large that 
$1-g(K/2-2t^* \nu )-\epsilon>0$, 
 $e^{-as_0 t^*}+2s_0 t^*e^{(2s_0 +m) t^*-K/2}<1$ and
$$
(1-g(x-2\nu t^*)-\epsilon)(2(g(x-2\nu t^*)+\epsilon)-1+\alpha)\le -a  \qquad \text{for }x\ge K/2.
$$
Then on the event $A$,
$$
(1-p^n_s(x))(2p^n_s(x)-1+\alpha )+a\le 0 \qquad \forall \, x\ge K/2, \; s\in [0,t^*].
$$
It follows that for $x\ge K/2$ and $s\in [0,t^*]$, since $p^n_s(x)\in [0,1]$,
\begin{align*}
\E{q^n_s(x)((1-p^n_s(x))(2p^n_s(x)-1+\alpha )+a)}
\le \E{q^n_s(x) (1+\alpha +a) \1_{A^c}}
&\le 2\delta ,
\end{align*}
and for $x\le K/2$ and $s\in [0,t^*]$, 
\begin{align*}
\E{q^n_s(x)((1-p^n_s(x))(2p^n_s(x)-1+\alpha )+a)} \le \E{q^n_s(x) (1+\alpha +a) }
&\le 2\E{q^n_s(x)}.
\end{align*}
Hence for $t\in [0,t^*]$ and $z\in \frac 1n \Z$, substituting into~\eqref{eq:Eqn},
\begin{align} \label{eq:Eqntbound}
\E{q^n_t(z)}
&\le e^{-as_0 t}\langle \E{q^n_0},\phi^{t,z}_0\rangle_n
+s_0 \int_0^t e^{-as_0 (t-s)}\langle 2\delta +2\sup_{y\in \frac 1n \Z}\E{q^n_s(y)}\1_{\cdot \leq K/2}, \phi^{t,z}_s\rangle_n ds \notag \\
&\le e^{-as_0 t}\sup_{x\in \frac 1n \Z} \E{q^n_0(x)}+2s_0 t^* \delta
+2s_0 \int_0^t \sup_{y\in \frac 1n \Z}\E{q^n_s(y)}\psubb{z}{X^n_{m(t-s)}\le K/2} ds.
\end{align}
In particular, for $t\in [0,t^*]$, since $a>0$,
\begin{align*}
\sup_{z\in \frac 1n \Z}\E{q^n_t(z)}
\le \sup_{x\in \frac 1n \Z}\E{q^n_0(x)}+2s_0 t^* \delta 
+2s_0 \int_0^t \sup_{y \in \frac 1n \Z}\E{q^n_s(y)}ds.
\end{align*}
By Gronwall's inequality, it follows that for $t\in [0,t^*]$,
\begin{equation} \label{eq:Eqngronwall}
\sup_{z\in \frac 1n \Z}\E{q^n_t(z)}
\le \left( \sup_{x\in \frac 1n \Z}\E{q^n_0(x)}+2s_0 t^* \delta  \right) e^{2s_0 t}.
\end{equation}
Therefore, substituting the bound in~\eqref{eq:Eqngronwall} into~\eqref{eq:Eqntbound}, for $t\in [0,t^*]$ and $z\in \frac 1n \Z$ with $z\ge K$,
\begin{align*}
\E{q^n_t(z)} 
&\le e^{-as_0 t}\sup_{x\in \frac 1n \Z} \E{q^n_0(x)}+2s_0 t^* \delta \\
&\quad +2s_0 \int_0^t e^{2s_0 s} \left( \sup_{x\in \frac 1n \Z}\E{q^n_0(x)}+2s_0 t^* \delta  \right)\psubb{K}{X^n_{m(t-s)}\le K/2} ds.
\end{align*}
For $0\le s \le t \le t^*$,
by Markov's inequality and Lemma~\ref{lem:Xnmgf},
$$
\psubb{K}{X^n_{m(t-s)}\le K/2}=\psubb{0}{X^n_{m(t-s)}\ge K/2} \le e^{-K/2}\E{e^{X^n_{m(t-s)}}}\le e^{mt^*-K/2}
$$
for $n$ sufficiently large.
Hence for $z\in \frac 1n \Z$ with $z\ge K$,
$$
\E{q^n_{t^*}(z)}\le (e^{-as_0 t^*}+2s_0 t^* e^{(2s_0+m) t^* -K/2})\sup_{x\in \frac 1n \Z} \E{q^n_0(x)} 
+2s_0 t^* \delta  (1+2s_0 t^* e^{(2s_0+m) t^* -K/2}),
$$
which completes the proof, since we chose $K$ sufficiently large that $e^{-as_0 t^*}+2s_0 t^* e^{(2s_0+m)t^*-K/2}<1$.
\end{proof}

Take $c\in (0,1)$ and $\epsilon \in (0,1)$ as in Lemma~\ref{lem:decayintip}.
For $t\ge 0$, define the sigma-algebra
$\mathcal F'_t=\sigma((p^n_s(x))_{s\in [0,t], x\in \frac 1n \Z})$.
The following result will easily imply Proposition~\ref{prop:eventE4}.
\begin{prop} \label{prop:eventE4int}
For $\ell \in \N$, there exists $\ell '\in \N$ such that for $K\in \N$ sufficiently large and $c_2>0$, the following holds for $n$ sufficiently large.
Take $t\in \delta_n \N_0\cap [0,T^-_n]$ and let $t'=T_n-t-t^* \lfloor (t^*)^{-1}  K\log N \rfloor$.
Suppose $p^n_{t'}(x)=0$ $\forall x\ge N^5$ and $\p{(E_1)^c | \mathcal F'_{t'}}\le \left( \frac n N \right)^{\ell '}$.
Then
$$
\p{r^{n,K,t^*}_{ K \log N , T_n-t }(x)=0 \; \forall x\in \tfrac 1n \Z \Big|\mathcal F'_{t'}}\ge 1-\left( \frac n N \right)^{\ell}.
$$
\end{prop}

\begin{proof}
Take $\ell '$ sufficiently large that $nN^6 \left( \frac n N \right)^{\ell '} \le \left( \frac n N \right)^{\ell+1}$ for $n$ sufficiently large.
Then take $c'\in (c,1)$ and take $K>t^* (\ell '+1)(-\log c')^{-1}$ sufficiently large that Lemma~\ref{lem:decayintip} holds.
Suppose 
\begin{equation} \label{eq:pE1cbound}
 \p{(E_1)^c | \mathcal F'_{t'}}\le \left( \frac n N \right)^{\ell '}.
 \end{equation}
For $k\in \N$ and $x\in \frac 1n \Z$, let
$r^n_k(x)=r^{n,K,t^*}_{kt^*,t'+kt^*}(x)$.
Take $k\in \N$ with $kt^*\le  K \log N$.
Then by the definition of $r^{n,y,\ell}_{s,t}$ in~\eqref{eq:rnystdefn},
\begin{align*}
\sup_{z\in \frac 1n \Z}\E{r^n_k(z) \Big|\mathcal F'_{t'}}
&=\sup_{z\in \frac 1n \Z}\E{r^n_k(z)\1_{z\ge \mu^n_{{t'}+kt^*}+K} \Big|\mathcal F'_{t'}}\\
&\le \sup_{z\in \frac 1n \Z,\,  z\ge \mu^n_{t'}+\nu  kt^* +K-\nu t^* }\E{r^n_k(z) \big|\mathcal F'_{t'}}+\p{(E_1)^c | \mathcal F'_{t'}}
\end{align*}
for $n$ sufficiently large, by the definition of the event $E_1$ in~\eqref{eq:eventE1}.
Therefore, by~\eqref{eq:pE1cbound} and then by Lemma~\ref{lem:decayintip} with $q^n_0=r^n_{k-1}(\cdot +\mu^n_{t'}+\lfloor \nu  (k-1)t^* n\rfloor n^{-1})$,
\begin{align} \label{eq:tipdecay*}
\sup_{z\in \frac 1n \Z} \E{r^n_k(z)\big|\mathcal F'_{t'}}
&\le \sup_{z\in \frac 1n \Z,\,  z\ge \mu^n_{t'}+\lfloor \nu  (k-1)t^* n\rfloor n^{-1}+K} \E{r^n_k(z)\big|\mathcal F'_{t'}}+\left( \frac n N \right)^{\ell '} \notag \\
&\le c \sup_{x\in \frac 1n \Z} \E{r^n_{k-1}(x)\big|\mathcal F'_{t'}}
+(1+4s_0t^*) \left( \frac n N \right)^{\ell '}
\end{align}
for $n$ sufficiently large.
Recall that we chose $c'\in (c,1)$, and let
$$
k^*=\min \left\{k\in \N_0 : \sup_{x\in \frac 1n \Z}\E{r^n_k(x)\big|\mathcal F'_{t'}}\le \frac {1+4s_0t^*} {c'-c} \left( \frac n N \right)^{\ell '} \right\}.
$$
Then for $k\in \N$ with $k\le \min(k^*, (t^*)^{-1} K\log N)$,
we have
$ (c'-c)\sup_{x\in \frac 1n \Z}\E{r^n_{k-1}(x) \big|\mathcal F'_{t'}}\ge (1+4s_0t^*) \left( \frac n N \right)^{\ell '}$ by the definition of $k^*$, 
and so by~\eqref{eq:tipdecay*},
$$
\sup_{x\in \frac 1n \Z}\E{r^n_{k}(x)|\mathcal F'_{t'}}\le 
c' \sup_{x\in \frac 1n \Z}\E{r^n_{k-1}(x)|\mathcal F'_{t'}}
\le \ldots \le (c')^{k} \sup_{x\in \frac 1n \Z}\E{r^n_{0}(x)|\mathcal F'_{t'}} \le (c')^{k}.
$$
Hence for $n$ sufficiently large, since $\lfloor  (t^*)^{-1}K \log N \rfloor > (\ell '+1)(-\log c')^{-1} \log (N/n)$ by our choice of $K$, we have $k^*< (t^*)^{-1}K \log N$.
For $k \in \N \cap [k^*+1,  (t^*)^{-1}K \log N]$, 
if $\sup_{x\in \frac 1n \Z}\E{r^n_{k-1}(x)|\mathcal F'_{t'}}\le \frac {1+4s_0t^*} {c'-c} \left( \frac n N \right)^{\ell '}$ then
by~\eqref{eq:tipdecay*},
\begin{equation} \label{eq:tipdecaydagger}
\sup_{x\in \frac 1n \Z}\E{r^n_k(x)\big| \mathcal F'_{t'}}\le \left( \frac {c} {c'-c}+1\right)(1+4s_0t^*) \left( \frac n N \right)^{\ell '} \le
\frac {1+4s_0t^*} {c'-c}\left( \frac n N \right)^{\ell '}
\end{equation}
since $c'<1$.
Therefore, by induction,~\eqref{eq:tipdecaydagger} holds for all $k \in \N \cap [k^*, (t^*)^{-1}K\log N]$.
By a union bound, and then by Lemma~\ref{lem:p01} and since $p^n_{t'}(x)=0$ $\forall x\ge N^5$,
and by~\eqref{eq:tipdecay*},
\begin{align*}
&\p{\sup_{x\in \frac 1n \Z}r^n_{\lfloor   (t^*)^{-1}K \log N \rfloor }(x)>0 \bigg| \mathcal F'_{t'}}\\
&\le \p{\exists x\ge 2N^5 : p^n_{T_n-t }(x)>0 \Big|\mathcal F'_{t'}}
+\p{\mu^n_{T_n-t}\le 0 \Big| \mathcal F'_{t'}}\\
&\qquad +\sum_{x\in \frac 1n \Z\cap  [K, 2N^5]}N \E{r^n_{\lfloor   (t^*)^{-1}K \log N \rfloor }(x)\Big|\mathcal F'_{t'}}\\
&\le e^{-N^5}+\p{(E_1)^c | \mathcal F'_{t'}}+ 2nN^5\cdot N  \frac {1+4s_0t^*} {c'-c} \left( \frac n N \right)^{\ell '}\\
&\le \left( \frac n N \right)^{\ell}
\end{align*}
for $n$ sufficiently large, by~\eqref{eq:pE1cbound} and our choice of $\ell '$.
\end{proof}
\begin{proof}[Proof of Proposition~\ref{prop:eventE4}]
Take $\ell \in \N$ sufficiently large that $\left( \frac n N \right)^{\ell-2}N^2 \delta_n^{-1}\le \left( \frac n N \right)^3$ for $n$ sufficiently large.
Take $\ell ' \in \N$  and $K\in \N$  sufficiently large that Proposition~\ref{prop:eventE4int} holds.
By Proposition~\ref{prop:eventE1}, by taking $b_1,c_2>0$ sufficiently small,
$\p{(E_1)^c}\le\left( \frac n N \right)^{\ell+\ell'}$ for $n$ sufficiently large.
For $t\in \delta_n \N_0 \cap [0,T^-_n]$, let 
$$
D_t =\Big\{r^{n,K,t^*}_{  K \log N , T_n-t}(x)=0 \; \forall x\in \tfrac 1n \Z \Big\}.
$$
Then by Proposition~\ref{prop:eventE4int}, letting $t'=T_n-t-t^* \lfloor  (t^*)^{-1} K \log N \rfloor$,
$$
\p{D^c_t \big| \mathcal F'_{t' }}\le \left( \frac n N \right)^{\ell }+\1_{\{\p{(E_1)^c |\mathcal F'_{t'}}>\left( \frac n N \right)^{\ell'}\}}
+\1_{\{\exists x \ge N^5:p^n_{t'}(x)>0\}}.
$$
Hence by Markov's inequality and Lemma~\ref{lem:p01},
$$
\p{D^c_t }\le \left( \frac n N \right)^{\ell }+\left( \frac N n \right)^{\ell '}\p{E^c_1} +e^{-N^5} \le 3\left( \frac n N \right)^{\ell }
$$
for $n$ sufficiently large.
Therefore, by a union bound and then by Markov's inequality,
$$
\p{(E_4)^c}\le \sum_{t\in \delta_n \N_0 \cap [0,T^-_n]} \p{\p{D^c_t | \mathcal F}\ge \left( \frac n N \right)^2}
\le  \sum_{t\in \delta_n \N_0 \cap [0,T^-_n]}\left( \frac N n \right)^2 \p{D^c_t}\le \left( \frac n N \right)^2
$$
for $n$ sufficiently large, by our choice of $\ell$,
which completes the proof.
\end{proof}

\section{Proof of Theorem~\ref{thm:statdist}} \label{sec:thmstatdist}
The proof of Theorem~\ref{thm:statdist} uses results from Sections~\ref{sec:mainproof},~\ref{sec:eventE1},~\ref{sec:eventE2} and~\ref{sec:eventE4}.
\begin{proof}[Proof of Theorem~\ref{thm:statdist}]
Recall from~\eqref{eq:paramdefns} that $\delta_n=\lfloor N^{1/2}n^2 \rfloor^{-1}$, and let $S_n=T_n-\delta_n \lfloor \delta_n^{-1} T'_n\rfloor $.
Take $b_1,c_2>0$ sufficiently small and $t^*,K\in \N$ sufficiently large that Proposition~\ref{prop:eventE1} holds with $\ell=1$ and Propositions~\ref{prop:eventE2} and~\ref{prop:eventE4} hold.
Assume $c_2<a_0$ (recall that $(\log N)^{a_0}\le \log n$ for $n$ sufficiently large).
Condition on $\mathcal F_0$, and
suppose the event $E'_1 \cap E'_2 \cap E_4$ occurs, so in particular by~\eqref{eq:eventE1} and~\eqref{eq:E1'defn},
\begin{equation} \label{eq:pnSn}
|p^n_{S_n}(x)-g(x-\mu^n_{S_n})|\le e^{-(\log N)^{c_2}}\; \forall x\in \tfrac 1n \Z.
\end{equation}
Fix $x_0\in \R$ and take $\epsilon >0$.
Define $v_0 :\frac 1n \Z \to [0,1]$ by letting
\begin{equation} \label{eq:v0defn}
v_0(y)=
\begin{cases}
p^n_{S_n}(y) \quad &\text{for }y< \mu^n_{S_n}+x_0,\\
\min(p^n_{S_n}(y), N^{-1} \lfloor N h(y)\rfloor ) \quad &\text{for }y\in [\mu^n_{S_n}+x_0,\mu^n_{S_n}+x_0+\epsilon],\\
0 &\text{for }y> \mu^n_{S_n}+x_0+\epsilon,
\end{cases}
\end{equation}
where $h : [\mu^n_{S_n}+x_0,\mu^n_{S_n}+x_0+\epsilon]\to [0,1]$ is linear with $h (\mu^n_{S_n}+x_0)=p^n_{S_n}(\mu^n_{S_n}+x_0)$ and $h (\mu^n_{S_n}+x_0+\epsilon)=0$.
For each $y\in \frac 1n \Z$, take $I_y \subseteq \{(y,i):\xi^n_{S_n}(y,i)=1\}$ such that
$|I_y|=N v_0(y)$.
Then let
$I = \cup_{y\in \frac 1n \Z}I_y$.
For $t\ge S_n$ and $x\in \frac 1n \Z$, let
$$
\tilde q^{n}_t(x)=N^{-1} |\{i\in [N]:(\zeta^{n,t}_{t-S_n}(x,i), \theta^{n,t}_{t-S_n}(x,i))\in I\}|,
$$
the proportion of individuals at $x$ at time $t$ which are descended from the set $I$ at time $S_n$.
Recall the definition of $q^{n,-}$ in~\eqref{eq:qn+-defn} and
note that for $t\ge S_n$ and $x\in \frac 1n \Z$,
\begin{equation} \label{eq:tildeqineqs}
q^{n,-}_{S_n,t}(\mu^n_{S_n}+x_0,x)\le \tilde q^{n}_t(x)
\le q^{n,-}_{S_n,t}(\mu^n_{S_n}+x_0+\epsilon ,x).
\end{equation}
Let $(\tilde v^{n}_t)_{t\ge S_n}$ solve
$$
\begin{cases}
\partial_t \tilde v^{n}_t =\tfrac 12 m \Delta_n \tilde v^{n}_t+s_0 \tilde v^{n}_t(1- u_{S_n,t}^n)(2u_{S_n,t}^n-1+\alpha )
\quad & \text{for }t>S_n,\\
\tilde v^{n}_{S_n}=
v_0,
\end{cases}
$$
where $(u^n_{S_n,t})_{t\ge S_n}$ is defined as in~\eqref{eq:unttsdef}.
Recall the definition of $\gamma_n$ in~\eqref{eq:paramdefns}.
Note that by Proposition~\ref{prop:pnun}, for $n$ sufficiently large, for $t\le S_n +\gamma_n$,
\begin{equation} \label{eq:thmstatqnvn}
\p{\sup_{x\in \frac 1n \Z\cap [-N^5,N^5]}|\tilde q^{n}_t(x)-\tilde v^{n}_t(x)|\ge \left( \frac n N \right)^{1/4}}
\le \frac n N.
\end{equation}
For $t\ge 0$ and $x\in \R$, let
$$
\tilde v _t(x)=g(x-\mu^n_{S_n}-\nu  t)\Esub{x-\mu^n_{S_n}-\nu  t}{\bar{v}_0(Z_t +\mu^n_{S_n})g(Z_t)^{-1}},
$$
where $\bar{v}_0$ is the linear interpolation of $v_0$, and $(Z_t)_{t\ge 0}$ is defined in~\eqref{eq:SDE}.
By Lemma~\ref{lem:vnvbound} and the definition of the event $E_1'$ in~\eqref{eq:E1'defn}, for $n$ sufficiently large,
\begin{align*}
&\sup_{x\in \frac 1n \Z,\, t\in [0,\gamma_n]}
|\tilde v^{n}_{S_n+t}(x)-\tilde v_t(x)|\\
&\le
(C_7 (n^{-1/3}+e^{-(\log N)^{c_2}})+2\sup_{x_1,x_2\in\frac 1n \Z,|x_1-x_2|\leq n^{-1/3}}|v_0(x_1)-v_0(x_2)|)e^{5s_0 \gamma_n}\gamma_n^2.
\end{align*}
By the definition of $v_0$ in~\eqref{eq:v0defn} and by~\eqref{eq:pnSn},
$$
\sup_{x_1,x_2\in\frac 1n \Z,|x_1-x_2|\leq n^{-1/3}}|v_0(x_1)-v_0(x_2)|
\le 2(2e^{-(\log N)^{c_2}}+n^{-1/3}\|\nabla g\|_\infty )+\epsilon^{-1} n^{-1/3}+N^{-1}.
$$
Therefore, for $n$ sufficiently large, for $t\in [0, \gamma_n]$ and $x\in \frac 1n \Z$ with $|x-\mu^n_{S_n+t}|\le d_n$,
\begin{equation} \label{eq:tildevvclose}
\Big| \frac{\tilde v^{n}_{S_n+t}(x)}{g(x-\mu^n_{S_n}-\nu  t)}
- \Esub{x-\mu^n_{S_n}-\nu  t}{\bar{v}_0(Z_t +\mu^n_{S_n})g(Z_t)^{-1}}\Big|\le  e^{-\frac 12 (\log N)^{c_2}}.
\end{equation}
From now on, we consider two different cases; suppose first that
 $T'_n\le \gamma_n$.
 Recalling~\eqref{eq:tildeqineqs} and~\eqref{eq:thmstatqnvn}, suppose for all $x\in \frac 1n \Z \cap [-N^5,N^5]$ that
$$
q^{n,-}_{S_n,T_n}(\mu^n_{S_n}+x_0,x)\le \tilde v^{n}_{T_n}(x)+\left( \frac n N \right)^{1/4}
\quad \text{and}\quad
q^{n,-}_{S_n,T_n}(\mu^n_{S_n}+x_0+\epsilon,x)\ge \tilde v^{n}_{T_n}(x)-\left( \frac n N \right)^{1/4}.
$$
By the definition of the event $E_1$ in~\eqref{eq:eventE1}, for
$n$ sufficiently large,
if $x\in \frac 1n \Z$ with $|x-\mu^n_{T_n}|\le K_0$ then since we are assuming $T'_n \le \gamma_n$ we have
 $|x-\mu^n_{S_n}-\nu  (T_n-S_n)  |\le 2K_0$, and so by~\eqref{eq:tildevvclose} and by~\eqref{eq:TZbound2} in Lemma~\ref{lem:Tbound},
\begin{align} \label{eq:thmstatpf1}
&\frac{q^{n,-}_{S_n,T_n}(\mu^n_{S_n}+x_0,x)}{g(x-\mu^n_{S_n}-\nu (T_n-S_n)  )} \notag \\
&\le \int_{-\infty}^\infty \pi(y) \bar{v}_0 (y+\mu^n_{S_n})g(y)^{-1}dy
+2m^{-1/2} (T_n-S_n)^{-1/4} \sup_{z\in \R}|\bar{v}_0(z+\mu^n_{S_n})g(z)^{-1}| \notag \\
&\qquad  +e^{-\frac 12 (\log N)^{c_2}}
+\left( \frac n N \right)^{1/4} g(2K_0)^{-1} \notag \\
&\le \int_{-\infty}^{x_0+\epsilon} \pi(y) dy +\epsilon
\end{align}
for $n$ sufficiently large, since by~\eqref{eq:pnSn} and by the definition of $v_0$ in~\eqref{eq:v0defn},
$v_0(y+\mu^n_{S_n})\le (g(y)+e^{-(\log N)^{c_2}})\1_{y\le x_0 +\epsilon}$ $\forall y\in \frac 1n \Z$, and since we are assuming that $T'_n \to \infty$ as $n\to \infty$.
Similarly, since $v_0(y+\mu^n_{S_n})\ge (g(y)-e^{-(\log N)^{c_2}})\1_{y\le x_0}$ $\forall y\in \frac 1n \Z$,
for $n$ sufficiently large we have
\begin{align}  \label{eq:thmstatpf2}
\frac{q^{n,-}_{S_n,T_n}(\mu^n_{S_n}+x_0+\epsilon,x)}{g(x-\mu^n_{S_n}-\nu  (T_n-S_n) )}
&\ge \int_{-\infty}^{x_0} \pi(y) dy -\epsilon.
\end{align}
For $n$ sufficiently large, since $|T_n-T'_n-S_n|\le \delta_n$ we have that $|\mu^n_{T_n-T'_n}-\mu^n_{S_n}|\le \epsilon$.
Recall the definition of $G_{K_0,T_n}$ in~\eqref{eq:Gdefn}.
Then 
for $(X_0,J_0)\in G_{K_0,T_n}$ we have $|X_0-\mu^n_{T_n}|\le K_0$,
and so for $n$ sufficiently large,
by the definition of the event $E_1$ in~\eqref{eq:eventE1} and by~\eqref{eq:thmstatpf2},
$$
\p{\zeta^{n,T_n}_{T_n-S_n}(X_0,J_0)\le \mu^n_{T_n-T'_n} + x_0+2\epsilon \Big| \mathcal F_0}\ge \frac{q^{n,-}_{S_n,T_n}(\mu^n_{S_n}+x_0+\epsilon, X_0)}{p^n_{T_n}(X_0)}
\ge \int_{-\infty}^{x_0} \pi(y) dy -2 \epsilon
$$
and by~\eqref{eq:thmstatpf1},
$$
\p{\zeta^{n,T_n}_{T_n-S_n}(X_0,J_0)\le \mu^n_{T_n-T'_n}+ x_0-\epsilon \Big| \mathcal F_0}\le \frac{q^{n,-}_{S_n,T_n}(\mu^n_{S_n}+x_0,X_0)}{p^n_{T_n}(X_0)}
\le \int_{-\infty}^{x_0+\epsilon} \pi(y) dy +2\epsilon.
$$
Hence letting $y_0=x_0 +2\epsilon$, by~\eqref{eq:tildeqineqs} and~\eqref{eq:thmstatqnvn}, for $n$ sufficiently large,
\begin{align} \label{eq:thmstatpf3}
\p{\zeta^{n,T_n}_{T_n-S_n}(X_0,J_0)-\mu^n_{T_n-T'_n} \le y_0 }&\ge \left(\int_{-\infty}^{y_0-2\epsilon} \pi(y) dy -2 \epsilon\right)
\left(1-\frac n N -\p{(E'_1\cap E'_2 \cap E_4)^c}\right) \notag \\
&\ge \int_{-\infty}^{y_0-2\epsilon} \pi(y) dy -3 \epsilon
\end{align}
for $n$ sufficiently large, by Propositions~\ref{prop:eventE1},~\ref{prop:eventE2} and~\ref{prop:eventE4}.
Similarly, for $n$ sufficiently large,
\begin{align} \label{eq:thmstatpf4}
 \p{\zeta^{n,T_n}_{T_n-S_n}(X_0,J_0)-\mu^n_{T_n-T'_n} \le y_0 }&\le \int_{-\infty}^{y_0+2\epsilon} \pi(y) dy +3 \epsilon.
\end{align}
Note that the rate at which $(\zeta^{n,T_n}_t(X_0,J_0))_{t\in [0,T_n]}$ jumps is bounded above by $2m r_n N=mn^2$, and so letting $Y_n \sim \text{Poisson}(mn^2 \delta_n)$,
\begin{equation} \label{eq:thmstatpf6}
\p{\zeta^{n,T_n}_{T'_n}(X_0,J_0)\neq \zeta^{n,T_n}_{T_n-S_n}(X_0,J_0)}
\le \p{Y_n\ge 1}
\le mn^2 \delta_n.
\end{equation}
Since $\epsilon>0$ can be taken arbitrarily small, 
this, together with~\eqref{eq:thmstatpf3} and~\eqref{eq:thmstatpf4}, completes the proof in the case $T'_n\le \gamma_n$.

Now suppose instead that $T_n'\ge \gamma_n$, and take $s\in t^* \N_0$ such that $T_n-s \in [S_n+\gamma_n-t^*, S_n +\gamma_n]$.
Recall from~\eqref{eq:paramdefns} that $d_n=\kappa^{-1} C \log \log N$.
 By Propositions~\ref{prop:intip} and~\ref{prop:RlogN},
if $(X_0,J_0)\in G_{K_0,T_n}$,
\begin{equation}  \label{eq:thmstatpf5}
\p{|\zeta^{n,T_n}_{s}(X_0,J_0)-\mu^n_{T_n-s} |\ge d_n \Big| \mathcal F_0}=\mathcal O((\log N)^{3-\frac 18 \alpha C})=\mathcal O((\log N)^{-1})
\end{equation}
since we chose $C>2^{13}\alpha^{-2}$.
Suppose for all $y\in \frac 1n \Z\cap [-N^5,N^5]$ that
$$
q^{n,-}_{S_n,T_n-s}(\mu^n_{S_n}+x_0,y)\le \tilde v^{n}_{T_n-s}(y)+\left( \frac n N \right)^{1/4}
\quad \text{and} \quad 
q^{n,-}_{S_n,T_n-s}(\mu^n_{S_n}+x_0+\epsilon,y)\ge \tilde v^{n}_{T_n-s}(y)-\left( \frac n N \right)^{1/4}.
$$
Take $x\in \frac 1n \Z$ with $|x-\mu^n_{T_n-s}|\le d_n$.
Then for $n$ sufficiently large, by the definition of the event $E_1$ in~\eqref{eq:eventE1}, and by~\eqref{eq:tildevvclose} and by~\eqref{eq:TZbound1} in Lemma~\ref{lem:Tbound},
\begin{align*} 
&\frac{q^{n,-}_{S_n,T_n-s}(\mu^n_{S_n}+x_0,x)}{g(x-\mu^n_{S_n}-\nu  (T_n-s-S_n) )} \notag \\
&\le \int_{-\infty}^\infty \pi(y) \bar{v}_0 (y+\mu^n_{S_n})g(y)^{-1}dy
+ (\log N)^{-12C} \sup_{z\in \R}|\bar{v}_0(z+\mu^n_{S_n})g(z)^{-1}| \notag \\
&\qquad  +e^{-\frac 12 (\log N)^{c_2}}
+\left( \frac n N \right)^{1/4} g(d_n+1)^{-1} \notag \\
&\le \int_{-\infty}^{x_0+\epsilon} \pi(y) dy +\epsilon
\end{align*}
for $n$ sufficiently large, as in~\eqref{eq:thmstatpf1}.
Hence for $n$ sufficiently large that $|\mu^n_{T_n-T'_n}-\mu^n_{S_n}|\le \epsilon$,
if $|\zeta^{n,T_n}_s (X_0,J_0)-\mu^n_{T_n-s}|\le d_n$ then
\begin{align*}
\p{\zeta^{n,T_n}_{T_n-S_n}(X_0,J_0)\le \mu^n_{T_n-T'_n} + x_0-\epsilon \Big| \mathcal F_s}
&\le \frac{q^{n,-}_{S_n,T_n-s}(\mu^n_{S_n}+x_0, \zeta^{n,T_n}_s(X_0,J_0))}{p^n_{T_n-s}(\zeta^{n,T_n}_s(X_0,J_0))}\\
&\le \int_{-\infty}^{x_0+\epsilon} \pi(y) dy +2 \epsilon
\end{align*}
for $n$ sufficiently large, and similarly
$$
\p{\zeta^{n,T_n}_{T_n-S_n}(X_0,J_0)\le \mu^n_{T_n-T'_n}+ x_0+2\epsilon \Big| \mathcal F_s}
\ge \int_{-\infty}^{x_0} \pi(y) dy -2\epsilon.
$$
As in~\eqref{eq:thmstatpf3} and~\eqref{eq:thmstatpf4}, it follows by~\eqref{eq:thmstatpf5},~\eqref{eq:tildeqineqs},~\eqref{eq:thmstatqnvn} and Propositions~\ref{prop:eventE1},~\ref{prop:eventE2} and~\ref{prop:eventE4} that for $n$ sufficiently large,
$$
\int_{-\infty}^{y_0-2\epsilon} \pi(y) dy -3\epsilon \le 
\p{\zeta^{n,T_n}_{T_n-S_n}(X_0,J_0)- \mu^n_{T_n-T'_n}\le  y_0}
\le \int_{-\infty}^{y_0+2\epsilon} \pi(y) dy +3\epsilon.
$$
By~\eqref{eq:thmstatpf6} and since $\epsilon>0$ can be taken arbitrarily small, this completes the proof.
\end{proof}

\appendix
\section{Proof of Proposition~\ref{prop:expconvtog}} \label{sec:append}

\begin{proof}[Proof of Proposition~\ref{prop:expconvtog}]
By rescaling time and space, we can assume $m=2$ and $s_0=1$.
In this proof, we use the notation and refer to results from~\cite{fife/mcleod:1977}.
The only change required in the proof is in Section~5, where we need to control $\sup_z |h(z,t)|$ at large times $t$.

Take $\delta>0$ and
suppose $|\varphi (z)-U(z)|\leq \delta $ $\forall z \in \R$.
Then by Lemma~4.2, for some constant $C_0$, if $\delta$ is sufficiently small then
$|u(x+ct,t)-U(x)|\leq C_0 \delta$ $\forall x\in \R, t>0$.
Therefore, by Lemma~4.5, there exists $z_0 \in \R$ such that
$\lim_{t\to \infty} \sup_{x\in \R} |u(x+ct,t)-U(x-z_0)|=0$
and so
$\sup_{x\in \R}|U(x)-U(x-z_0)|\leq C_0\delta$.
It follows that
$$|u(x+ct,t)-U(x-z_0)|\leq 2C_0 \delta \quad\forall x\in \R, \; t>0. $$
Hence by the definition of $w(z,t)$ in the proof of Lemma~4.5, and by the estimates in Lemma~4.3, for $t$ sufficiently large (depending on $\delta$),
\begin{equation} \label{eq:wUbound}
|w(z,t)- U(z-z_0)|\le 3C_0\delta \quad \forall z\in \R.
\end{equation}
By the definition of $\alpha (t)$ in~(5.1), for $t$ sufficiently large (depending on $\delta$), it follows that
\begin{align*}
0&=\int_{-\infty}^\infty e^{cz} h(z,t) U'(z-z_0-\alpha(t))dz\\
&\geq \int_{-\infty}^\infty e^{cz} U'(z-z_0-\alpha(t))(U(z-z_0)-3C_0\delta -U(z-z_0-\alpha(t))dz.
\end{align*}
There exists a constant $a>0$ such that if $\alpha(t)\geq \delta^{1/2}$ and if $\delta$ is sufficiently small then
\begin{align*}
&\int_{z_0+\alpha(t)-\delta^{1/2}}^{z_0+\alpha(t)} e^{cz} U'(z-z_0-\alpha(t))(U(z-z_0)-3C_0 \delta -U(z-z_0-\alpha(t))dz\\
&\qquad \geq a\delta e^{c(z_0+\alpha(t))}.
\end{align*}
For $R<\infty$, if $\delta$ is sufficiently small and $\alpha(t)\geq \delta^{1/2}$
then for $z\in \R$ with $|z-(z_0+\alpha(t))|\leq R$ we have
$U(z-z_0)-U(z-z_0-\alpha(t))\geq 3C_0\delta$.
Therefore
\begin{align*}
0
&\geq a\delta e^{c(z_0+\alpha(t))}
-3C_0\delta \Big(\int_{z_0+\alpha(t)+R}^\infty e^{cz} U'(z-z_0-\alpha(t))dz
+ \int^{z_0+\alpha(t)-R}_{-\infty} e^{cz} U'(z-z_0-\alpha(t))dz\Big),
\end{align*}
which, by the tail behaviour of $U'$, is a contradiction for $R$ sufficiently large.
By the same argument for the case $\alpha(t)\le -\delta^{1/2}$, it
 follows that if $\delta$ is sufficiently small, $|\alpha(t)|\leq \delta^{1/2}$ for $t$ sufficiently large (depending on $\delta$).

Hence by~\eqref{eq:wUbound}, for $b>0$, if $\delta$ is sufficiently small then for $t$ sufficiently large (depending on $\delta$ and $b$),
$\sup_z |h(z,t)|\leq b$. Therefore, if $\delta$ is sufficiently small then the inequality
$$
\frac12 \frac{d}{dt} \|y\|^2 \leq -\frac M 2 \|y\|^2 +\mathcal O (e^{-Kt})
$$  
(which appears before~(5.3)) holds for $t\geq T$, where $T=T(\delta)$ and $K=K(\delta)$.

This is the only modification required in the proof.
\end{proof}

\bibliographystyle{alpha}
\bibliography{Genealogies}

\end{document}